\documentclass[english]{memo-l}

\usepackage{microtype}
\usepackage[T1]{fontenc}
\usepackage[utf8]{inputenc}
\usepackage{babel}

\usepackage{color}
\usepackage{amssymb}
\usepackage{amsmidx}
\usepackage[colorlinks=true, pdfstartview=FitV, linkcolor=darkblue, citecolor=darkred, urlcolor=cyan]{hyperref}

\usepackage{amsrefs}
\usepackage{stmaryrd}
\usepackage{mathdots}

\usepackage{tikz}
\usetikzlibrary{decorations.markings}
\tikzset{middlearrow/.style={
        decoration={markings,
            mark= at position 0.5 with {\arrow{#1}} ,
        },
        postaction={decorate}
    }
}
\tikzset{firstthirdarrow/.style={
        decoration={markings,
            mark= at position 0.33 with {\arrow{#1}} ,
        },
        postaction={decorate}
    }
}
\tikzset{secondthirdarrow/.style={
        decoration={markings,
            mark= at position 0.66 with {\arrow{#1}} ,
        },
        postaction={decorate}
    }
}
\usetikzlibrary{cd}

\usepackage{csquotes}
\usepackage{enumerate}

\definecolor{darkred}{RGB}{203,65,84}
\definecolor{darkblue}{RGB}{70,130,180}

\newcommand{\bemph}[1]{{\upshape#1}} 
\newcommand{\ep}[1]{\bemph{(}#1\bemph{)}} 

\DeclareMathOperator{\tr}{tr}
\DeclareMathOperator{\sgn}{sgn}

\DeclareMathOperator{\Hom}{Hom}
\DeclareMathOperator{\Hommax}{Hom_{max}}
\DeclareMathOperator{\Homfmax}{Hom^{f}_{max}}
\DeclareMathOperator{\Homdelta}{Hom_{\delta}}
\DeclareMathOperator{\Homf}{Hom^{f}}

\DeclareMathOperator{\diag}{diag}
\DeclareMathOperator{\Id}{Id}
\DeclareMathOperator{\Sym}{Sym}
\DeclareMathOperator{\Span}{Span}
\DeclareMathOperator{\Conf}{Conf}
\DeclareMathOperator{\CR}{CR}

\DeclareMathOperator{\im}{Im}

\DeclareMathOperator{\PSp}{PSp}
\DeclareMathOperator{\Sp}{Sp}
\DeclareMathOperator{\GL}{GL}
\DeclareMathOperator{\PGL}{PGL}
\DeclareMathOperator{\SL}{SL}
\DeclareMathOperator{\PSL}{PSL}
\DeclareMathOperator{\PO}{PO}
\DeclareMathOperator{\OO}{O}
\DeclareMathOperator{\U}{U}

\DeclareMathOperator{\SO}{SO}

\DeclareMathOperator{\Rot}{Rot}

\DeclareMathOperator{\Loc}{Loc}
\DeclareMathOperator{\Locdelta}{Loc_{\delta}}
\DeclareMathOperator{\Locf}{Loc^f}
\DeclareMathOperator{\Locfdelta}{Loc^{f}_{\delta}}
\DeclareMathOperator{\Locfstdelta}{Loc^{f,st}_{\delta}}
\DeclareMathOperator{\LocfT}{Loc^{f}_{\mathcal{T}}}
\DeclareMathOperator{\LocfdeltaT}{Loc^{f}_{\delta,\mathcal{T}}}
\DeclareMathOperator{\Locd}{Loc^d}
\DeclareMathOperator{\Locddelta}{Loc^{d}_{\delta}}
\DeclareMathOperator{\LocddeltaTzero}{Loc^{d}_{\delta,\mathcal{T}_0}}
\DeclareMathOperator{\LocddeltaTone}{Loc^{d}_{\delta,\mathcal{T}_1}}
\DeclareMathOperator{\LocddeltaT}{Loc^{d}_{\delta,\mathcal{T}}}

\DeclareMathOperator{\Repf}{Rep^{f}}
\DeclareMathOperator{\Repdelta}{Rep_{\delta}}

\DeclareMathOperator{\M}{\mathcal{M}}
\DeclareMathOperator{\Mf}{\mathcal{M}^{f}}
\DeclareMathOperator{\Mfdelta}{\mathcal{M}^{f}_{\delta}}

\DeclareMathOperator{\DfSLdOn}{\mathcal{D}^{f}_{ \mathcal{T}}( \mathnormal{S}, \SL(2,\R)\otimes \OO(n))}
\DeclareMathOperator{\DSLdOn}{\mathcal{D}_{ \mathcal{T}}( \mathnormal{S}, \SL(2,\R)\otimes \OO(n))}

\DeclareMathOperator{\holXpT}{hol^{\mathcal{X},+}_{\mathcal{T}}}
\DeclareMathOperator{\holXpdT}{hol^{\mathcal{X},+}_{\mathcal{T},\Delta}}
\DeclareMathOperator{\holXT}{hol^{\mathcal{X}}_{\mathcal{T}}}
\DeclareMathOperator{\holXGT}{hol^{\mathcal{X}}_{\mathnormal{G},\mathcal{T}}}
\DeclareMathOperator{\holXpGT}{hol^{\mathcal{X},+}_{\mathnormal{G},\mathcal{T}}}
\DeclareMathOperator{\holXstT}{hol^{\mathcal{X},st}_{\mathcal{T}}}
\DeclareMathOperator{\hol}{hol}

\DeclareMathOperator{\holSaz}{hol^{\mathcal{X},+}_{\mathcal{S},\mathnormal{a}_0}}
\DeclareMathOperator{\holSp}{hol^{\mathcal{X},+}_{\mathcal{S}}}

\DeclareMathOperator{\Stab}{Stab}

\newcommand{\edgeT}{\overrightarrow{\mathcal{T}}}
\newcommand{\edgeTsub}[1]{\overrightarrow{\mathcal{T}_{#1}}}

\newcommand{\Lag}[1]{\mathcal{L}_{#1}}
\newcommand{\wideLag}[1]{\widetilde{\mathcal{L}}_{#1}}
\newcommand{\Lagd}[1]{\mathcal{L}^{\mathrm{d}}_{#1}}

\newcommand{\XplusTn}{\mathcal{X}^{+}( \mathcal{T}, n)}
\newcommand{\XTn}{\mathcal{X}( \mathcal{T}, n)}
\newcommand{\XplusTG}{\mathcal{X}^{+}( \mathcal{T}, G)}
\newcommand{\XTG}{\mathcal{X}( \mathcal{T}, G)}
\newcommand{\XplusDeltaTn}{\mathcal{X}^{+}_{\Delta}( \mathcal{T}, n)}
\newcommand{\XETn}{\mathcal{X}_{\mathcal{E}}( \mathcal{T}, n)}
\newcommand{\XEbfxTn}{\mathcal{X}_{\mathcal{E}, \mathbf{x}}( \mathcal{T}, n)}
\newcommand{\XplusDeltaT}[1]{\mathcal{X}^{+}_{\Delta_n}( \mathcal{T}, #1)}
\newcommand{\XplusSaz}{\mathcal{X}^{+}_{\mathcal{S}, a_0}}
\newcommand{\XplusSn}{\mathcal{X}^{+}_{\mathcal{S}}( \mathcal{T}, n)}
\newcommand{\XplusS}[1]{\mathcal{X}^{+}_{\mathcal{S}}( \mathcal{T}, #1)}
\newcommand{\XplusSLdOn}{\mathcal{X}^{+}_{\mathcal{S}}( \SL(2,\R) \otimes \OO(n))}
\newcommand{\XplusFTOn}{\mathcal{X}^{+}_{\mathcal{S}}( \mathcal{F}_\mathcal{T} \otimes \OO(n))}

\newcommand{\coloneqq}{\mathrel{\mathop:}=}
\renewcommand{\leq}{\leqslant}
\renewcommand{\geq}{\geqslant}

\newcommand{\R}{\mathbf R}
\newcommand{\CC}{\mathbf C}
\newcommand{\C}{\mathbf C}
\newcommand{\PP}{\mathbb P}
\newcommand{\N}{\mathbf N}
\newcommand{\Z}{\mathbf Z}
\newcommand{\Q}{\mathbf Q}

\providecommand{\abs}[1]{\lvert#1\rvert}

\newcommand{\setm}{\smallsetminus}

\renewcommand{\emptyset}{\varnothing}

\newcommand{\ra}{\rightarrow}

\newtheorem{teo}{Theorem}[chapter]
\newtheorem{lem}[teo]{Lemma}
\newtheorem{cor}[teo]{Corollary}
\newtheorem{prop}[teo]{Proposition}

\theoremstyle{definition}
\newtheorem{df}[teo]{Definition}

\theoremstyle{remark}
\newtheorem{rem}[teo]{Remark}

\numberwithin{section}{chapter}
\numberwithin{equation}{chapter}

\numberwithin{figure}{chapter}

\makeindex{notation}
\makeindex{definition}

\begin{document}

\frontmatter

\title{Noncommutative coordinates for symplectic representations}

\author[D. Alessandrini]{Daniele Alessandrini}
\address{Department of Mathematics, Columbia University, New York, USA}
\email{daniele.alessandrini@gmail.com}

\author[O. Guichard]{Olivier Guichard}
\address{IRMA UMR 7501, Université de
  Strasbourg et CNRS,
  Strasbourg, France}
\email{olivier.guichard@math.unistra.fr}

\author[E. Rogozinnikov]{Eugen Rogozinnikov}
\address{IRMA UMR 7501, Université de
  Strasbourg et CNRS,
  Strasbourg, France\\
Mathematisches
  Institut, Ruprecht-Karls-Universität Heidelberg, Germany}
\email{erogozinnikov@gmail.com}

\author[A. Wienhard]{Anna Wienhard}
\address{Mathematisches Institut, Ruprecht-Karls-Universität Heidelberg,
  Germany\hfill{}\ \linebreak
  HITS gGmbH, Heidelberg Institute for Theoretical Studies, Heidelberg,
  Germany}
\email{wienhard@uni-heidelberg.de}

\thanks{D.~A.\ and A.~W.\ acknowledge funding by the Deutsche
  Forschungsgemeinschaft within the Priority Program SPP 2026
  \enquote{Geometry at Infinity}.  D.~A., E.~R.\ and A.~W.\ acknowledge
  funding by the Deutsche Forschungsgemeinschaft within the RTG 2229
  \enquote{Asymptotic invariants and limits of groups and spaces}.  O.~G.\
  was partially supported by the Agence Nationale de la Recherche under the
  grant DynGeo (ANR-16-CE40-0025) by the Simion Stoilow Institute of
  Mathematics of the Romanian Academy (IMAR) in the framework of the Contract
  Bitdefender and by le Centre Francophone de Mathématiques in Bucharest, funded by l'Agence
  Universitaire de la Francophonie. O.~G.\ thanks the Institut Universitaire de
  France. A.~W.\ acknowledges funding by the National Science Foundation under
  Grant No.~1440140, by the European Research Council under ERC-Consolidator
  grant 614733, and by the Klaus-Tschira-Foundation.  Part of the work was
  done while E.~R.\ and A.~W.\ were in residence at the Mathematical Sciences
  Research Institute in Berkeley, California. This work is supported by
  Deutsche Forschungsgemeinschaft under Germany’s Excellence Strategy
  EXC-2181/1 - 390900948 (the Heidelberg STRUCTURES Cluster of Excellence). The authors acknowledge support from the NSF grants DMS-1107452, 1107263 and 1107367 RNMS: GEometric structures And Representation varieties (the GEAR Network).}

\date{September 15, 2020}




\begin{abstract}
  We introduce coordinates on the spaces of framed and decorated
  representations of the fundamental group of a surface with nonempty boundary into the
  symplectic group $\Sp(2n,\R)$. These coordinates provide a noncommutative
  generalization of the parametrizations of the spaces of representations
  into $\SL(2,\R)$ or $\PSL(2,\R)$ given by Thurston, Penner, Kashaev and
  Fock--Goncharov.  On the space of decorated symplectic representations the
  coordinates give a geometric realization of the noncommutative cluster-like
  structures introduced by Berenstein--Retakh.  The locus of positive
  coordinates maps to the space of framed maximal representations. We use this
  to determine an explicit homeomorphism between the space of framed maximal
  representations and a quotient by the group~$\OO(n)$. This allows us to
  describe the homotopy type and, when $n=2$, to give an exact description of
  the singularities. Along the way, we establish a complete classification of
  pairs of nondegenerate quadratic forms.
\end{abstract}

\maketitle

\tableofcontents


\mainmatter

\chapter{Introduction}
In their seminal paper~\cite{FG}, Fock and Goncharov introduced a pair of
moduli spaces, the $\mathcal{X}$-space and the $\mathcal{A}$-space,
which are closely related to the representation variety of the fundamental group of
a surface~$S$ with nonempty boundary and of negative Euler characteristic into a split real simple
Lie group~$G$. They introduced on these spaces explicit cluster $\mathcal{X}$-coordinates and
$\mathcal{A}$-coordinates associated to an ideal triangulation of~$S$.
Changing the triangulation, the coordinates change by positive
rational functions. Thus the locus of positive coordinates is independent of
the choice of triangulation.  When~$G$ is $\SL(2,\R)$, the positive locus in
the $\mathcal{X}$-space is closely related to the Teichmüller space, and the
positive locus in the $\mathcal{A}$-space to the decorated Teichmüller space
of~$S$, therefore the Fock--Goncharov coordinates are extensions of Thurston's
shear coordinates, respectively Penner's $\lambda$-lengths~\cite{Penner, Kashaev}.
When~$G$ is a
split real group of higher rank, these moduli spaces give higher Teichmüller
spaces, and the positive locus of the $\mathcal{X}$-space is closely related
to the Hitchin component in the representation variety.

The set of positive representations of Fock--Goncharov and the Hitchin
components account only for one family of higher Teichmüller spaces, another
family is given by maximal representations into Lie groups of Hermitian
type. The symplectic groups $\Sp(2n,\R)$ form essentially the only family of
Lie groups that are both split real forms and of Hermitian type.  In this
article we generalize the work of Fock--Goncharov in the following way. We
introduce two new moduli spaces, an $\mathcal{X}$-space and an
$\mathcal{A}$-space of representations of the fundamental group of~$S$
into the symplectic group $\Sp(2n,\R)$, and describe noncommutative
$A_1$-type cluster coordinates on them.  We show, on the one hand, that the
positive locus of the $\mathcal{X}$-space corresponds precisely to maximal
representations into $\Sp(2n,\R)$; 
we use this to determine the homeomorphism type of the
space of framed maximal representations, and for $\Sp(4,\R)$ a precise
understanding of its singularities.
On the other hand, we show that the $\mathcal{A}$-space
gives a geometric realization of the noncommutative cluster-like
structures
introduced by Berenstein--Retakh~\cite{BR}.

In Fock--Goncharov's work, an important role is played by Lusztig's total
positivity, in our work, a similar role is played by positivity related to the
Maslov index. As such, our work fits well in the framework of
$\Theta$-positivity, recently introduced by Guichard--Wienhard~\cite{WienhardICM, GW, GWpos, GLW}, 
that generalizes Lusztig's total
positivity and provides a unifying framework for the different higher
Teichmüller spaces.
There are four families of Lie groups which admit a $\Theta$-positive
structure, where $\Theta$ is a subset of the set of simple restricted
roots. One is the family of split real Lie groups, the second one is the
family of Hermitian Lie groups of tube type, the third family consists of the
groups $\mathrm{SO}(p,q)$ with $p<q$, and the fourth is an exceptional family
consisting of four groups which are real forms of real rank~$4$ of the complex
simple Lie groups of type $F_4, E_6, E_7, E_ 8$. In the case of Hermitian Lie
groups of tube type, positivity is governed by a Weyl group of type $A_1$,
giving $\Theta$-positivity in that case the flavor of a noncommutative
$A_1$-theory. This is precisely what is reflected in the structure of the
coordinates we define here.
This analogy has been made
even more clear for classical Hermitian Lie groups of tube
type in~\cite{ABRRW}, where such groups are realized as
$\Sp_2$ over a noncommutative ring.

When the Fock--Goncharov's approach is applied to the group $\Sp(2n,\R)$, they
define a positive locus in the space of symplectic representations. It is
important to remark that the positive locus that our approach gives in the
space of symplectic representations properly contains the Fock--Goncharov's one.
The
two theories are based on two different $\Theta$-positive structures on
$\Sp(2n,\R)$: respectively the one for split groups and the one for groups of
Hermitian type. In particular, the framings and decorations that play a role in the construction of the moduli spaces are with respect to different flag varieties.
The perspective chosen in the present paper is the one which
is suitable for describing the spaces of maximal representations.

We now describe our results in more detail.

\section{The pair of moduli spaces}
Let~$S$ be a compact surface with nonempty boundary and negative Euler
characteristic (we refer to Section~\ref{sec:topological-data} for the wider
generality that can be allowed for~$S$, for example marked disks, that
leads to spaces of configurations of Lagrangians).

We introduce two moduli spaces, the space of {\em framed symplectic
  representations} (i.e.\ representations
$\pi_1(S) \rightarrow \Sp(2n,\R)$ together with a $k$-tuples of
Lagrangian subspaces $(L_1, \dots, L_k)$, that are fixed by peripheral
elements
$c_1, \dots, c_k$ in
$\pi_1(S)$) which serves as our $\mathcal{X}$-space, and the space of
{\em decorated symplectic representations} (i.e.\
representations $\pi_1(S) \rightarrow \Sp(2n,\R)$ together with a $k$-tuple of
decorated Lagrangian subspaces $( (L_1, \mathbf{v}_1), \dots , (L_k,
\mathbf{v}_k))$, where $\mathbf{v}_i$ is a basis of~$L_i$, and the
isomorphism of~$L_i$ induced by~$c_i$ is $-\Id_{L_i}$, cf.\ Corollary~\ref{cor:framed-local-systems-rep-odd-punctured})
 which serves as our
$\mathcal{A}$-space.

Fixing an ideal triangulation~$\mathcal{T}$ of $S$, we introduce systems
of $\mathcal{X}$-coordinates, using invariants of triples, quadruples, and
quintuples of Lagrangian subspaces. A system of $\mathcal{X}$-coordinates
consists of a triangle invariant for each triangle, which is given by the
Maslov index of the triple of Lagrangians associated to the vertices of the
triangle, an edge invariant for every edge of the triangulation, which can be
seen as a cross ratio function of four Lagrangians, and an angle invariant,
associated to each corner of a triangle, which comes from an invariant of
quintuples of Lagrangians.
Along the way, we show that the (opposite of the) Toledo number of the
associated representation is the half sum of the triangle invariants (cf.\ Theorem~\ref{teo:toledo_maslov}).
We then describe in detail a map denoted by $\holXT$
from the set~$\XTn$ of $\mathcal{X}$-coordinates to the
space of framed representations. A special role is played by the
set~$\XplusDeltaTn$ of positive diagonal $\mathcal{X}$-coordinates, those
for which the triangle invariants are equal to~$n$, the edge invariants are
just $n$-tuples of positive real numbers, and the angle invariants take values
in $\OO(n)$.

\begin{teo}\label{teo.intro.X_to_max}
  The map~$\holXT$ induces a proper generically
  finite-to-one surjection from~$\XplusDeltaTn$ to the space
  of framed maximal representations.
\end{teo}

Maximal representations into Lie groups of Hermitian type have been
extensively studied
in~\cite{BIW}, and further considered in~\cite{BILW, Strubel}. All maximal
representations are discrete embeddings, and spaces of maximal representations
are examples of higher Teichmüller spaces.

Let us emphasize that the correspondence between positive
$\mathcal{X}$-coordinates and framed maximal representations is not
one-to-one.  To every framed maximal representation corresponds a system of
positive $\mathcal{X}$-coordinates, but in general only the edge invariants
are uniquely determined, the angle invariants involve some choices. We also
explicitly describe the fibers of the map $\holXT$
(cf.\ Theorem~\ref{teo:max-param-Xplus}).

In general (i.e.\ not restricting to the positive locus), the space
$\mathcal{X}(\mathcal{T}, n)$ pa\-ra\-me\-trizes framed
representations that are  $\mathcal{T}$-transverse, that means that they satisfy a transversality condition with respect to the triangulation. Note that maximal
framed representations are $\mathcal{T}$-transverse. We show that
the map $\hol^{\mathcal{X}}_{\mathcal{T}}$ is onto the space of $\mathcal{T}$-transverse
framed representations, generically finite-to-one, and we give an explicit
description of its fibers (Theorem~\ref{teo:from-coord-repr-X-general}). In
turn, topological conclusions are drawn concerning the space of $\mathcal{T}$-transverse
framed representations (cf.\
Corollary~\ref{cor:number-connected-components-LocdT}).

The $\mathcal{X}$-coordinates are quite geometric, and can be used to determine
the topology of the space of maximal representations, but they do not have
nice algebraic properties.  For example, we did not include in this paper the
complete explicit formulas for the change of coordinates for $\mathcal{X}$-coordinates
under a flip of the triangulation as they involve some unpleasant
operations
such as diagonalizing symmetric matrices. We give a formula for the change of some of the coordinates, the cross-ratio coordinates, in Proposition \ref{cross_ratio_flip}.

\smallskip
We introduce $\mathcal{A}$-coordinates on the space of decorated symplectic representations. These $\mathcal{A}$-coordinates have
better and cleaner algebraic properties.
To define the $\mathcal{A}$-coordinates on the space of decorated symplectic
representations, we introduce the symplectic $\Lambda$-length, which is an
invariant of pairs of decorated Lagrangians. Let~$\omega$ denote the symplectic
form, and let $(L,\mathbf{v})$, $(M, \mathbf{w})$ be a pair of decorated
Lagrangians, where $\mathbf{v} = (v_1, \dots, v_n)$ is a basis of~$L$ and
$\mathbf{ w} = (w_1, \dots, w_n)$ a basis of~$M$, then the symplectic
$\Lambda$-length is
\[\Lambda_{\mathbf{ v}, \mathbf{ w}}\coloneqq (\omega(v_i, w_j))_{i,j = 1,
  \dots , n}.\]
It takes values in $\GL(n,\R)$ if and only if the pair $(L,M)$
is transverse, and its square provides a noncommutative generalization of
Penner's $\lambda$-lengths, which are the special case when $n=1$. We show
that the symplectic $\Lambda$-lengths satisfy a noncommutative analogue of the
Ptolemy relation, as well as special triangle relations (which are trivially
satisfied for $n=1$). A system of $\mathcal{A}$-coordinates associates to
every oriented edge the symplectic $\Lambda$-length of the two decorated
Lagrangians at the vertices of the edge. The noncommutative Ptolemy equation
translates into an explicit formula for the changes of
$\mathcal{A}$-coordinates under a flip.

\begin{teo}\label{teo.intro.flip_Lambda}
  Let
  $(L_1,{\mathbf{v}_1}),
  (L_2,{\mathbf{v}_2}),(L_3,{\mathbf{v}_3}),(L_4,{\mathbf{v}_4})$ be four
  pairwise transverse decorated Lagrangians.  Let $\Lambda_{ij}$ denote $\Lambda_{\mathbf{v}_i,\mathbf{v}_j}$. 
  Then one has
  \[\Lambda_{24} = \Lambda_{23}\Lambda_{13}^{-1}\Lambda_{14} +
    \Lambda_{21}\Lambda_{31}^{-1} \Lambda_{34}.
  \]
\end{teo}

When $n=1$, this formula reduces to the formula for the flip in the
$\SL(2,\R)$-situation, and in general it is a noncommutative generalization
of it. This lets us view the theory of decorated symplectic representations as
a noncommutative $A_1$-theory. In fact, $\mathcal{A}$-coordinates
give a geometric realization of the noncommutative
algebras introduced by Berenstein--Retakh~\cite{BR}, see Theorem~\ref{thm:realization BR}.

There is a natural map from the $\mathcal{A}$-space to the
$\mathcal{X}$-space. Under this map, the formula for the flip for
$\mathcal{A}$-coordinates leads to a formula for the flip of the cross ratios in the $\mathcal{X}$-space (see Proposition~\ref{cross_ratio_flip}), which provides a noncommutative generalization of the well-known formula for the coordinate change of shear coordinates under a flip.

In Fock--Goncharov~\cite{FG}, an interesting application of the $\mathcal{A}$-coordinates is that they can be used to write a natural Poisson structure on the $\mathcal{A}$-space. In the noncommutative setting this is harder, and we are not attempting to generalize this property here. But this might be an interesting question for future research. In~\cite{AOS} a double quasi Poisson bracket is constructed on the space of noncommutative weights of arcs of a directed graph embedded in a disk or cylinder.

\section{Topology of the space of maximal representations}
We now discuss the applications to the topology of the space of (framed)
maximal representations. Let us point out that contrary to the space of
positive representations or the Hitchin component, which are contractible, the
space of maximal representations has nontrivial topology. For
 closed surfaces, the topology
of the space of maximal representations has been studied using the theory of
Higgs bundles in~\cite{ AC, BGPG, GPGM, Gothen}.  These techniques do not apply
easily to the case of maximal representations of fundamental groups of surfaces
with punctures, in particular since we do not fix the holonomy along
peripheral curves on the surface.

Here we rely on Theorem~\ref{teo.intro.X_to_max}
to determine topological properties of the space of
maximal framed representations.  Note that the positive locus of the
$\mathcal{X}$-coordinates does not parametrize the space of framed maximal
representations, but maps surjectively to it.
The fibers of this surjection
are complicated to describe, because they depend on the shape of the edge
invariants. However, they are generically finite.

There is a variant of~$\XplusDeltaTn$
from which we deduce a description of the space of maximal representations as
the quotient of an $\OO(n)$-action, see Theorem~\ref{teo:holonomy-ZS}. To
state the theorem we denote by $\Sym^+(n,\R)$ the space of positive definite
symmetric matrices and let $\OO(n)$ act on it by conjugation:

\begin{teo}\label{intro:homeo}
  The space of framed maximal representations into~$\Sp(2n,\R)$ is
  homeomorphic to the quotient of
  $\Sym^+(n,\R)^{-3 \chi(S)}\times\OO(n)^{1-{ \chi(S)}}$ by the
  diagonal $\OO(n)$-action by conjugation on each factor.
\end{teo}

We furthermore determine an explicit description  of the space of framed maximal
representations into any connected Lie group isogenic to $\PSp(2n,\R)$, cf.\
Theorem~\ref{teo:holonomy-ZGplus}.

As a corollary, we obtain a different proof of~\cite[Theorem~4]{Strubel} for
the value of the number of connected components.

\begin{cor}\label{intro:conncomp}
  The space of maximal representations and the space of framed maximal
  representations into $\Sp(2n,\R)$ have $2^{1-{ \chi(S)} }$ connected
  components.  The space of maximal representations and the space of framed
  maximal representations into $\PSp(2n,\R)$ have $2^{1-{ \chi(S)}}$
  connected components when $n$~is even; they are connected when $n$~is odd.
\end{cor}

There is a special subset of framed representations
(Section~\ref{sec:subsp-enqu-repr}), for which the edge invariants are
\enquote{totally singular}. It corresponds to the subset $\{ (\Id, \dots,
\Id)\} \times \OO(n)^{ 1-{ \chi(S)}}$ of $\Sym^+(n,\R)^{-3{
    \chi(S)}}\times\OO(n)^{1-{ \chi(S)} }$. From this, we obtain

\begin{teo}\label{intro:homotopy}
  The space of totally singular framed representations is a strong deformation
  retract of the space of framed maximal representations into~$\Sp(2n,\R)$; it
   is homeomorphic to $\OO(n)^{1-{ \chi(S)}}/\OO(n)$ \ep{where the action
    of~$\OO(n)$ is by simultaneous conjugation}.
\end{teo}

When $n=2$, we analyze the quotient of Theorem~\ref{intro:homeo} in more
detail and show that all connected components except one are orbifolds, one
connected component contains a non-orbifold singularity, see
Theorem~\ref{teo:singular-points-sp4}.

\section{Comparison with Fock--Goncharov's coordinates}

The Fock--Goncharov's approach~\cite{FG} covers all the split groups. Here we
work with $\Sp(2n,\R)$ or isogenous groups, the only family of groups that are
both split and of Hermitian type. Hence, it is possible to apply both our
construction and the Fock--Goncharov's construction in order to describe
spaces of representations into $\Sp(2n,\R)$. We would like to point out a few 
differences between the coordinates we develop here and Fock--Goncharov's coordinates.

Fock and Goncharov consider moduli spaces of framed and decorated
representations, where the framing is given by a full flag and the decoration
by a full flag and a generating basis, here our framing is given by a partial flag, namely a Lagrangian subspace and the decoration by a Lagrangian and a basis. A framing by a full flag clearly gives rise to a framing by a partial flag, but the converse is not true in general. This implies also that the transversality conditions are different: a framed or decorated representation in the Fock--Goncharov's sense is transverse if the associated full flags are transverse --- then of course also the framing or decoration by partial flags is transverse. But the converse is not true.
Thus the set of representations we parametrize is strictly bigger than the set of representations that is parametrized by Fock--Goncharov's coordinates.

In both cases, there is a special subset of representations, called positive, that behave especially well, and can be better understood. The positive representations are transverse with reference to all the triangulations, hence every choice of triangulation gives global coordinates that describe the topology of the full space of positive representations.
The main difference between our construction and the Fock--Goncharov's construction when
applied to $\Sp(2n,\R)$ is that the set of positive representations is
different. The Fock--Goncharov's coordinates are adapted to the positive structure
that $\Sp(2n,\R)$ has as a split group, while our coordinates are adapted to the
positive structure that $\Sp(2n,\R)$ has as a group of Hermitian type. In our
approach, the positive representations are precisely the maximal
representations, while in Fock--Goncharov's approach the positive representations are a smaller subset. It is possible to find some maximal representations that are not positive in the Fock-Goncharov's sense, for every choice of framing by a full flag. We can do this by explicitly computing the Fock--Goncharov's coordinates.
Our coordinates thus allows us to give some global charts to the space of maximal representations, but this is not possible using the Fock--Goncharov's coordinates.

\section*{Structure of the memoir} In
Chapter~\ref{sec:sympl-group-lagr} we recall classical facts on the symplectic group, the Maslov index and the Souriau
index.  In Chapter~\ref{sec:prelim} we introduce the invariants of
Lagrangians and decorated Lagrangians which are used to define coordinates. In
Chapter~\ref{sec:rep}, we introduce the spaces of framed and decorated local
systems, recall the definition and key properties of maximal local
systems. Chapter~\ref{sec:local-systems-their} gives an interpretation of
these spaces in terms of local systems on a quiver embedded in the surface. In
Chapter~\ref{sec_def_coord_max} we introduce positive
$\mathcal{X}$-coordinates, and construct the map to framed maximal
representations. Variants of $\mathcal{X}$-coordinates are introduced in
Chapter~\ref{sec:other-X-like}. The applications for the topology of the space
of maximal representations are proven in Chapter~\ref{sec:topol-space-maxim}
(homotopy type) and Chapter~\ref{sec:singularities} (singularities when~$n$
is~$2$). The general $\mathcal{X}$-coordinates are introduced in
Chapter~\ref{sec:generalX}, and in Chapter~\ref{cent_ext} we generalize them
to $G$-local systems for connected Lie groups~$G$ isogenic to $\Sp(2n,\R)$. Finally, in
Chapter~\ref{sec:Acoordinates} we introduce $\mathcal{A}$-coordinates,
describe the relations to the noncommutative
algebras of Berenstein
and Retakh, and give formulas for the coordinate changes under a flip of the
triangulation.  The Appendix~\ref{sec:normal-form-pair} contains a description
of the invariants of pairs of nondegenerate symmetric bilinear forms that are
used in Chapter~\ref{sec:generalX}.

{\bf Acknowledgments:} We thank Arkady Berenstein, Vladimir Fock, Michael
Gekhtman, Vladimir Retakh, and Michael Shapiro for helpful and interesting
discussions. We are also grateful to the anonymous referee whose remarks
helped us correcting some oversights and improving the exposition.

\chapter{Symplectic group, Lagrangians}
\label{sec:sympl-group-lagr}

This chapter introduces the symplectic group and the Lagrangian
Grassmannian. We also recall facts on the Maslov and Souriau indices, and the
translation number.

\section{Lagrangian Grassmannian and decorated Lagrangian Grassmannian}
\label{sec:lagr-grassm-its}

We consider the symplectic vector space $(\R^{2n},\omega)$ where~$\omega$ is
the standard symplectic form on~$\R^{2n}$, i.e.\
\begin{equation} \label{form:standard symplectic form}
  \omega(x,y)= {}^T\! x \begin{pmatrix}
    0 & \Id \\ -\Id & 0
  \end{pmatrix} y,
\end{equation}
for~$x$ and~$y$ in~$\R^{2n}$.

Every basis of $\R^{2n}$ such that $\omega$, expressed in that basis, has the
form~\eqref{form:standard symplectic form} is called a \emph{symplectic basis}
(hence the standard basis is a symplectic basis). We will usually write a
symplectic basis as a pair $(\mathbf{e},\mathbf{f})$, where
$\mathbf{e}=(e_1,\dots,e_n), \mathbf{f}=(f_1,\dots,f_n)$; thus, one has, for
all~$i, j$, $\omega(e_i,f_j)=\delta_{ij}$.  We will denote this last equality
more concisely by $\omega( \mathbf{e}, \mathbf{f})=\Id$. More generally, given
two families $\mathbf{v}=(v_1, \dots, v_n)$ and $\mathbf{w}=(w_1,\dots, w_m)$
we will write $\omega( \mathbf{v}, \mathbf{w})$ for the $n\times m$-matrix
whose coefficients are $\omega(v_i,w_j)$; one has $\omega(\mathbf{w},
\mathbf{v}) = -{}^T\! \omega(\mathbf{v},
\mathbf{w})$.\index{definition}{basis!symplectic ---}%
\index{definition}{symplectic!basis}%

The group
\[ \Sp(2n,\R)\coloneqq \{g\in\GL(2n,\R) \mid {}^T\! g \omega g=\omega\}\] is
the \emph{symplectic group}, and its adjoint form
$\PSp(2n,\R) \coloneqq \Sp(2n,\R)/\{\pm \Id\}$ is the \emph{projective
  symplectic group}.
\index{notation}{01@$\Sp(2n,\R)$ (the symplectic group)}
\index{notation}{02@$\PSp(2n,\R)$ (the projective symplectic group)}

\begin{df}
  A subspace $L$ of $\R^{2n}$ is called \emph{Lagrangian} if $\dim(L)=n$ and
  $\omega(u,v)=0$ for all $u,v\in L$. The space of all Lagrangian subspaces of
  $(\R^{2n},\omega)$ is called the \emph{Lagrangian Grassmannian}, and denoted
  $\Lag{n}$.
\end{df}
\index{notation}{03@$\Lag{n}$ (the Lagrangian Grassmannian)}%
\index{definition}{Grassmannian!Lagrangian ---}%
\index{definition}{Lagrangian}%
\index{definition}{Lagrangian!Grassmannian}%

\begin{df} A \emph{decorated Lagrangian} is a pair $(L,\mathbf{v})$, where
  $L\in\Lag{n}$ and $\mathbf{v}$ is a basis of $L$. The space of all decorated
  Lagrangians of $(\R^{2n},\omega)$ is called the \emph{decorated Lagrangian
    Grassmannian}, and denoted~$\Lagd{n}$.
\end{df}
\index{notation}{04@$\Lagd{n}$ (the decorated Lagrangian Grassmannian)}%
\index{definition}{Lagrangian!decorated ---}%
\index{definition}{decorated!Lagrangian}%
\index{definition}{decorated!Lagrangian Grassmannian}%
\index{definition}{Lagrangian! decorated --- Grassmannian}%
\index{definition}{Grassmannian! decorated Lagrangian ---}%

Our notation will often retain only the family~$\mathbf{v}$ for a decorated
Lagrangian as it determines the Lagrangian~$L$.

The natural projection to $\Lag{n}$ turns $\Lagd{n}$ into a right principal
$\GL(n,\R)$-bundle.

\begin{rem}
  \label{rem:action-on-basis}
  \begin{enumerate}
  \item Recall that for any $n$-dimensional real vector space~$L$, the
    group~$\GL(n,\R)$ acts on the right on the space of bases of $L$: if
    $\mathbf{v}=(v_1,\dots, v_n)$ is a basis of~$L$ and
    $g=(g_{i,j})_{i,j\in\{1,\dots, n\}}$ is in~$\GL(n,\R)$ then
    $\mathbf{w}= \mathbf{v}g$ is the family $(w_1,\dots, w_n)$ defined by the
    formula: $\forall j\in\{ 1,\dots, n\}$, $w_j= \sum_{i=1}^{n} v_i
    g_{i,j}$. This is the simply transitive action behind the principal bundle
    structure mentioned above.

  \item In the case of a symplectic vector space of dimension~$2n$, we get
    this way a simply transitive action of $\Sp(2n,\R)$ on the space of
    symplectic bases.

  \item We will use also the notation $\mathbf{v}\cdot g$ (or sometimes
    $\mathbf{v}g$) when $\mathbf{v}$ is a family (not necessarily free, nor
    generating) of~$n$ elements in a vector space~$V$ and when $g$ is a
    $m\times n$-matrix.

  \item\label{item:4:rem:action-on-basis} A morphism $\psi\colon L\to L'$ and
    its matrix~$g$ with respect to bases~$\mathbf{v}$ of~$L$ and~$\mathbf{v}'$
    of~$L'$ are related by the well-known formula
    $\psi(\mathbf{v}) =\mathbf{v}'g$.

  \item For two families $\mathbf{v}=(v_1,\dots,v_n)$ and $\mathbf{w}=(w_1,
    \dots, w_n)$, the family $(v_1+w_1, \dots, v_n+w_n)$ is denoted
    $\mathbf{v}+ \mathbf{w}$.
  \end{enumerate}
\end{rem}

The group $\Sp(2n,\R)$ has a natural left action on~$\Lag{n}$ and~$\Lagd{n}$:
\begin{align*}
  g\cdot L&\coloneqq \{g(x) \}_{  x\in L},\\
  g\cdot (L,(v_1,\dots,v_n)) &\coloneqq \bigl(g \cdot L,(g(v_1), \dots,g(v_n)) \bigr).
\end{align*}

These actions are transitive, hence the spaces $\Lag{n}$ and
$\Lagd{n}$ are homogeneous spaces under the symplectic group.
Furthermore, the map $\Lagd{n}\to \Lag{n}$ is $\Sp(2n,\R)$-equivariant.
The left action of~$\Sp(2n,\R)$ and the right action of~$\GL(n,\R)$ on
$\Lagd{n}$ commute.

Let $L_0=\Span(\mathbf{e}_0)$ (where
$(\mathbf{e}_0, \mathbf{f}_0)$ is the standard symplectic basis
of~$\R^{2n}$), and consider the stabilizers
\begin{equation*}
P = \Stab_{\Sp(2n,\R)}(L_0),\quad
U = \Stab_{\Sp(2n,\R)}((L_0,\mathbf{e}_0)).
\end{equation*}
\index{notation}{05@$P$ (standard parabolic subgroup in $\Sp(2n.\R)$)}%
\index{notation}{06@$U$ (standard unipotent subgroup in $\Sp(2n.\R)$)}%
The group $P$ is a parabolic subgroup of $\Sp(2n,\R)$, and $U \subset P$ is
its unipotent radical. As homogeneous spaces, we have
\[\Lag{n} = \Sp(2n,\R)/P,\quad
\Lagd{n} = \Sp(2n,\R)/U.\]
The action of $\Sp(2n,\R)$ on $\Lag{n}$ is not effective, its
kernel is $\{\pm \Id\}$. The group of symmetries of $\Lag{n}$ is the
projective symplectic group $\PSp(2n,\R)$. On the space~$\Lagd{n}$, the
action is effective.

\begin{df}
  \label{df:transverse-lag-and-lagfr}
  Two Lagrangians $L_1, L_2 $ are called \emph{transverse} if their
  intersection is trivial.
  Two decorated Lagrangians $\mathbf{v}_1,
  \mathbf{v}_2 $ are called \emph{transverse} if $\Span(\mathbf{v}_1)$ and
  $\Span(\mathbf{v}_2)$ are transverse.
\end{df}\index{definition}{Lagrangian!transverse ---}%
\index{definition}{transverse! Lagrangian}%

\section{Configurations of Lagrangians}
\label{sec:conf-lagr}

For every integer~$d\geq 2$, let us denote by $\Conf^d( \Lag{n})$ the moduli
space of $d$-tuples of Lagrangians, i.e.\ the quotient of
$(\mathcal{L}_{n})^{d}$ by the diagonal action of~$\Sp(2n, \R)$. The natural
action of the symmetric group~$\mathfrak{S}_d$ on~$(\mathcal{L}_{n})^{d}$
descends to $\Conf^d( \Lag{n})$.
\index{notation}{07@$\Conf^d( \Lag{n})$ (configuration space of $d$-tuples of Lagrangians)}

We will be particularly interested in $\Conf^3(\Lag{n}), \Conf^4( \Lag{n})$ and $\Conf^5( \Lag{n})$
 and certain of their subspaces. We will denote by
$\Conf^{3\ast}( \Lag{n})$ the configuration space of triples of pairwise
transverse Lagrangians. The subspace
$\Conf^{3\ast}( \Lag{n}) \subset \Conf^3( \Lag{n})$ is invariant by the action
of~$\mathfrak{S}_3$.

The subspace of $\Conf^4( \Lag{n})$ consisting of (orbits of) quadruples
$(L_1, M_1, L_2, M_2)$ of Lagrangians such that $(L_1, M_1, L_2)$ and $(L_1,
M_2, L_2)$ belong to $\Conf^{3\ast}( \Lag{n})$ will be denoted by
$\Conf^{4 \Diamond}( \Lag{n})$. It is not invariant by the full symmetry
group~$\mathfrak{S}_4$ but it is invariant by the subgroup generated by the transpositions~$(1,3)$ and~$(2,4)$.

The map $(L_1, M_1, L_2, M_2) \mapsto (L_2, M_2, L_1, M_1)$ induces an
automorphism of $\Conf^4( \Lag{n})$ denoted by~$\kappa$.
\index{notation}{08@$\kappa$ (the involution on $\Conf^4( \Lag{n})$)}

\section{Maslov index}
\label{sec:maslov-index}

In this section we review properties of the Maslov index, for a more general
discussion we refer the reader to~\cite[Part~I, Appendix~A]{Lion}.

Let $L_1, M, L_2$ be three pairwise transverse Lagrangians. There is a
unique linear map~$M_{L_1 \ra L_2}$ from~$L_1$ to~$L_2$ such that
$M = \{ v\in \R^{2n} \mid \exists e\in L_1,\, v=e+M_{L_1\to L_2}(e)\}$. When
this does not cause confusion, it will be denoted just by $M$.

Using the symplectic form $\omega$, we can define a bilinear form $\beta$ on
$L_1$ in the following way: for $v_1,v_2\in L_1$
\[\beta(v_1,v_2)\coloneqq \omega(v_1, M_{L_1 \ra L_2}(v_2)).\]

\begin{df}\label{def:maslov-form}
  The bilinear form $\beta$ is called the Maslov form and is denoted
  by~$[L_1,M,L_2]$.
\end{df}\index{notation}{09@$[L_1,M,L_2]$ (Maslov form associated with the
  triple $(L_1,M,L_2)$)}%
\index{definition}{Maslov!form}%
\index{definition}{form (Maslov ---)}%
The following is well known, see~\cite{Souriau}:

\begin{prop}
  The Maslov form $[L_1,M,L_2]$ is symmetric and nondegenerate.
\end{prop}

\begin{rem}
  \label{rem:maslov-index-matrix}
  Let $\mathbf{e}$ be a basis of $L_1$ and let~$\mathbf{f}$ be the basis of
  $L_2$ such that $(\mathbf{e}, \mathbf{f})$ is a symplectic basis. Then the
  matrix $[M]_{\mathbf{e}, \mathbf{f}}$ of $M$ in these bases and the matrix
  $[\beta]_{\mathbf{e}}$ are equal:
  $[M]_{\mathbf{e}, \mathbf{f}} = [\beta]_{\mathbf{e}}$.
\end{rem}

We will denote the signature of $\beta$ by
\[\sgn(\beta)=p-q,\]
where $p$ is the dimension of a maximal subspace of $L_1$ on which $\beta$
is positive definite and $q$ is the dimension of a maximal subspace of $L_1$
on which $\beta$ is negative definite. They satisfy $p+q=n$ so that
$\sgn(\beta) = n \mod 2$.

\begin{df}
  The \emph{Maslov index} of the triple of Lagrangians $(L_1,M,L_2)$ is the
  signature $\sgn([L_1,M,L_2])$ and is denoted by $\mu_n(L_1,M,L_2)$.
\end{df}
\index{notation}{10@$\mu_n$ (the Maslov index)}%
\index{definition}{Maslov!index}%
\index{definition}{index!Maslov ---}%

For $n=1$, the three Lagrangians $(L_1,M,L_2)$ are pairwise
distinct points in the circle $\R\PP^1$. The Maslov index is~$1$ if the three
points are cyclically ordered, and~$-1$ otherwise.

There is a slightly more general definition of the Maslov index that works for any triple of Lagrangians, not just for the pairwise transverse triples. It can be defined as the signature $\sgn(\gamma)$ of the quadratic form
\begin{align*}
  \gamma\colon  L_1 \oplus M \oplus L_2 &\longrightarrow \R\\
  (v,w,x) &\longmapsto \omega(v,w) +\omega(w,x) +\omega(x,v).
\end{align*}
When the triple is pairwise transverse, the two definitions agree.

\begin{prop}[Properties of Maslov index]\label{maslov_prop}
The Maslov index
\begin{itemize}
\item has range $\{-n, -n+1, \dots, n\}$ and its restriction to the set of triples of
  pairwise transverse Lagrangians has range $\{-n, -n+2, \dots, n\}$;
\item is invariant under the action of $\Sp(2n,\R)$ on $\Lag{n}^{3}$;
\item is antisymmetric and, as a result, is cyclically invariant;
\item satisfies the cocycle relation, namely for all
  $L_1,L_2,L_3,L_4\in \Lag{n}$
\[\mu_n(L_1,L_2,L_3) - \mu_n(L_1,L_2,L_4) + \mu_n(L_1,L_3,L_4) -
  \mu_n(L_2,L_3,L_4)=0\]
\item the group $\Sp(2n,\R)$ acts transitively on the set of triples of
  pairwise transverse Lagrangians with the same Maslov index, i.e.\
  the map $\Conf^{3\ast}( \Lag{n})\to \{-n, -n+2, \dots, n\}$ induced
  by~$\mu_n$ is a bijection.
\end{itemize}
\end{prop}

Note also that every decomposition of $n=h+k$ induces an injection
$\Lag{h} \times \Lag{k} \rightarrow \Lag{n}$ that is equivariant with respect
to the homomorphism $\Sp( 2h, \R) \times \Sp( 2k, \R) \to \Sp( 2n, \R)$ and
for which $\mu_n = \mu_h + \mu_k$.

\section{Universal coverings}
\label{sec:universal-coverings}

Let us also equip~$\R^{2n}$ with its standard Euclidean structure. Then a
maximal compact subgroup of $\Sp(2n, \R)$ is the intersection $\Sp(2n, \R)
\cap \OO(2n)$ that identifies with~$\U(n)$ via $A+i B \mapsto \bigl(
\begin{smallmatrix}
  A & -B \\ B & A
\end{smallmatrix}
\bigr)$.

The homomorphism $\det\colon \U(n) \to \C^*$ can be used to construct the
universal cover of~$\U(n)$:
\[ \widetilde{\U}(n) \coloneqq \{ (u,t)\in \U(n)\times \R \mid \det(u)=e^{it}\}.\]
Of course $\widetilde{\U}(n)$ is a Lie subgroup of the universal cover  $\widetilde{\Sp}(2n, \R)$ of $\Sp(2n,\R)$.
For any decomposition $n=h+k$ there is a homomorphism
$\widetilde{\Sp}(2h, \R)
\times \widetilde{\Sp}(2k, \R) \to \widetilde{\Sp}(2n, \R)$. Similar statements hold for
any decomposition of~$n$.

\begin{rem}\label{rem:lag-ortho-basis-ortho}
  Let $(\mathbf{e}_0, \mathbf{f}_0)$ be the standard symplectic basis
  of~$\R^{2n}$; it is also an orthonormal basis. The space
  $L_0\coloneqq \Span( \mathbf{e_0})$ is Lagrangian and its orthogonal
  complement is $L_{0}^{\perp}=\Span( \mathbf{f}_0)$. If, for a symplectic
  basis $(\mathbf{e}, \mathbf{f})$, one has $ \Span( \mathbf{e})= L_0$ and
  $\Span( \mathbf{f}) = L^{\perp}_{0}$, then $( \mathbf{e}, \mathbf{f})$ is an
  orthonormal basis of~$\R^{2n}$ if and only if $\mathbf{f}$ is an orthonormal
  family (this holds in fact as soon as that $\Span(\mathbf{e})$ and
  $\Span(\mathbf{f})$ are orthogonal).
\end{rem}

\begin{lem}
  \label{lem:ortho-basis-lag}
  For every~$M$ in $\Lag{n}$ that is transverse to~$L_0$, there is a unique
  $n$-tuple of real numbers $(\varphi_1, \varphi_2, \dots, \varphi_n)$ such that:
  \begin{itemize}
  \item $0< \varphi_1 \leq \varphi_2 \leq \cdots \leq \varphi_n < \pi$, and
  \item there exist symplectic orthonormal bases $(\mathbf{e}, \mathbf{f})$
    with
    \[L_0= \Span( \mathbf{e}),\ L_{0}^{\perp} = \Span( \mathbf{f}),\ \text{and}\ M =
    \Span \{ \cos( \varphi_i) e_i + \sin( \varphi_i) f_i\}_{1\leq i\leq n}.\]
  \end{itemize}
\end{lem}
\begin{rem} The bases constructed in this lemma consist of one orbit under the
  action of the group~$\Stab_{\OO(n)}(M)$.
\end{rem}

\begin{proof}
  Another way to state the conclusion is  that \[M =
  \Span \{ \cot( \varphi_i) e_i +  f_i\}_{1\leq i\leq n}.\]

  Since~$M$ is transverse to~$L_0$, it is the graph of a map
  $L_{0}^{\perp}\to L_0$ whose matrix in the bases~$\mathbf{f}_0$,
  $\mathbf{e}_0$ is symmetric. The ortho-diagonalization together with
  Remark~\ref{rem:lag-ortho-basis-ortho} furnishes a (unique) nonincreasing
  sequence $\lambda_1 , \dots, \lambda_n$ and a symplectic orthonormal basis
  $(\mathbf{e}, \mathbf{f})$ such that $L_0= \Span( \mathbf{e})$,
  $L_{0}^{\perp} = \Span( \mathbf{f})$ and
  $M= \Span \{ \lambda_i e_i + f_i\}_{1\leq i\leq n}$. Setting
  $\varphi_i = \cot^{-1} (\lambda_i)$ gives the result.
\end{proof}

For~$M$ transverse to~$L_0$, we will denote by $( \varphi_1(M), \dots,
\varphi_n(M))$ the $n$-tuple provided by the previous lemma. The~$\varphi_i$
are continuous functions on the open and dense subset
\[\mathcal{U} \coloneqq \{ M\in \Lag{n} \mid M \text{ is transverse to } L_0\}.\]
They admit (noncontinuous) extensions to $\Lag{n}$:

\begin{lem}
  \label{lem:ortho-basis-lag-non-transverse}
  For every~$M$ in $\Lag{n}$, there is a unique $n$-tuple $0\leq \varphi_1
  \leq \varphi_2 \leq \cdots \leq \varphi_n < \pi$ for which there exist
  symplectic orthonormal bases $(\mathbf{e, \mathbf{f}})$ such that $L_0 = \Span(
  \mathbf{e})$, $L_{0}^{\perp} = \Span( \mathbf{f})$ and $M = \Span \{ \cos(
  \varphi_i) e_i + \sin( \varphi_i) f_i\}_{1\leq i \leq n}$.
\end{lem}

\begin{proof}
  Choose first $( e_1, \dots, e_k)$ an orthonormal basis of $N=M\cap L_0$ and
  apply the previous lemma to $M \cap N^\perp = ( M^\perp \oplus N)^\perp$
  which is a Lagrangian subspace of the $(2n-2k)$-dimensional symplectic
  vector space $N^\perp \cap N^{\perp_\omega}$.
\end{proof}

The space:
\[ \wideLag{n} \coloneqq \Bigl\{ ( M, \theta) \in \Lag{n}\times
  \R \mid \sum_{i=1}^{n} \varphi_i(M) = \theta \mod \pi\Bigr\}\] is a
submanifold of $\Lag{n}\times \R$ and the natural map
$\wideLag{n} \to \Lag{n}$ is a covering. The action of
$\Sp(2n, \R)$ on $\Lag{n}$ lifts to an action of $\widetilde{\Sp}(2n,
\R)$ on $\wideLag{n}$;  the restriction of this action
to~$\widetilde{\U}(n)$ has an explicit expression:
\[ (u,t)\cdot (M,\theta) = ( u(M), t+\theta), \quad ((u,t)\in
  \widetilde{\U}(n), \ (M,\theta)\in \wideLag{n}).\]
From this expression, it follows easily that the action of~$\widetilde{\U}(n)$
is transitive and that the stabilizer of $(L_0, 0)$ is the subgroup
$\SO(n)\times \{0\}\subset \widetilde{\U}(n)$. This implies that $\wideLag{n}$ is a connected and simply connected manifold, thus
$\wideLag{n} \to \Lag{n}$ is the universal covering.
\index{notation}{11@$\wideLag{n}$ (the universal cover of $\Lag{n}$)}

\section{Souriau index}
\label{sec:souriau-index}

Two elements of $\wideLag{n}$ will be called \emph{transverse} if
their projections to $\Lag{n}$ are transverse.
For two transverse elements $\tilde{L}_{1}$, $\tilde{L}_{2}$ of
$\wideLag{n}$, there is~$\tilde{u}$ in~$ \widetilde{\U}(n)$, $M$ in~$
\mathcal{U}$, and~$\theta$ in~$\R$ such that
\[ \tilde{u} \cdot  \tilde{L}_{1} = (L_0, 0), \quad \text{and}, \quad \tilde{u} \cdot  \tilde{L}_{2} = (M, \theta).\]
\begin{df}
  \label{def:souriau-index}
  The \emph{Souriau index} (cf.~\cite[Section~4]{Souriau}) of  $(\tilde{L}_{1}, \tilde{L}_{2})$ is the integer number
  \[m_n (\tilde{L}_{1}, \tilde{L}_{2}) = n + \frac{1}{\pi} \Bigl( \theta -
    \sum_{i=1}^{n} \varphi_i(M)\Bigr). \]
\end{df}
\index{notation}{12@$m_n$ (the Souriau index)}%
\index{definition}{Souriau index}%
\index{definition}{index!Souriau ---}%
The Souriau index gives rise to a function on the space of
transverse pairs of~$\wideLag{n}$, which we will call the Souriau index as well. The reader may refer to~\cite{FrP} (in
particular Sections~5 and~6) for a more general discussion. For the reader's
convenience, we provide below short proofs of the main properties of the
Souriau index.

\begin{prop}
  \label{prop:souriau-index-first-properties}
  The Souriau index is well defined, antisymmetric, and
  $\widetilde{\Sp}(2n, \R)$-invariant.
\end{prop}
\begin{proof}
  We note that $\varphi_j( g\cdot M)= \varphi_j(M)$ for every~$M$
  in~$\mathcal{U}$ and every~$g$ in~$\OO(n)$ so that the above formula
  (Definition~\ref{def:souriau-index}) does not depend on the choices since the
  stabilizer in~$\widetilde{\U}(n)$ of $(L_0, 0)$ is~$\SO(n)$.

  To prove the antisymmetry, one can assume that $\tilde{L}_{1} = (L_0, 0)$ and
  $\tilde{L}_{2} = (M, \theta)$ with
  $M= \Span \{ \cos( \varphi_i) e_{0,i} + \sin( \varphi_i) f_{0,i}\}$ (as
  above $(\mathbf{e}_0, \mathbf{f}_0)$ is the standard basis
  of~$\R^{2n}$). Let $u\in \U(n)\subset \GL(n, \C)$ the diagonal element whose
  coefficients are
  $\pm e^{-i\varphi_1}, e^{-i\varphi_2}, \dots, e^{-i\varphi_n}$ where the
  sign is fixed so that $\tilde{u} \coloneqq (u, -\theta)$ belongs to
  $\widetilde{\U}(n)$ (i.e.\ it is~$+ e^{-i\varphi_1}$ in the case
  $\sum_{i=1}^{n} \varphi_i =\theta \mod 2\pi $ and~$- e^{-i\varphi_1}$ in the
  case $\sum_{i=1}^{n} \varphi_i =\theta+\pi \mod 2\pi $). Furthermore,
  $\tilde{u} \cdot \tilde{L}_{2}= (L_0, 0)$, and
  $\tilde{u} \cdot \tilde{L}_{1}= (N, -\theta)$ with $N$ a Lagrangian such that
  $\varphi_j(N)= \pi - \varphi_{n+1-j}(M)$ (for all~$j$). A small calculation
  then gives
  $m_n( \tilde{L}_{2}, \tilde{L}_{1}) = - m_n( \tilde{L}_{1}, \tilde{L}_{2})$.

  The Souriau index is continuous on the space of transverse pairs and thus locally constant which implies that
  it is constant on $\widetilde{\Sp}(2n, \R)$-orbits hence
  $\widetilde{\Sp}(2n, \R)$-invariant.
\end{proof}

Let~$T$ be the element $(\Id, 2\pi)$ of~$\widetilde{\U}(n)$. This element
generates the kernel of $\widetilde{\U}(n)\to \U(n)$. The following equality
is a
direct consequence of the definitions:
\begin{equation}
  \label{eq:trans-T-Souriau}
  m_n( \tilde{L}_{1}, T\cdot \tilde{L}_{2}) = m_n( \tilde{L}_{1}, \tilde{L}_{2})
  +2\quad (\tilde{L}_{1},\, \tilde{L}_{2} \in \wideLag{n}),
\end{equation}
more generally, for every element~$\tilde{u}= (\pm \Id, k\pi)$ (i.e.\
$\tilde{u}$ belongs to the center of~$\widetilde{\Sp}(2n, \R)$), one has
\begin{equation}
  \label{eq:trans-Z-Souriau}
  m_n( \tilde{L}_{1}, \tilde{u}\cdot \tilde{L}_{2}) = m_n( \tilde{L}_{1}, \tilde{L}_{2})
  +k\quad (\tilde{L}_{1},\, \tilde{L}_{2} \in \wideLag{n}).
\end{equation}

Furthermore, for every decomposition $n=h+k$, there is a natural map
$ \wideLag{h} \times \wideLag{k} \to \wideLag{n}$ that is equivariant with
respect to the homomorphism
$\widetilde{\Sp}(2h, \R) \times \widetilde{\Sp}(2k, \R) \to
\widetilde{\Sp}(2n, \R)$ mentioned earlier. The Souriau indices behave
naturally with respect to this map:
\begin{equation}
  \label{eq:Souriau-p-plus-q}
  m_n = m_h +  m_k.
\end{equation}
Similar statements hold for any decomposition of~$n$.

The Souriau index is strongly related to the Maslov index:
\begin{lem}
  \label{lem:souriau-maslov-for-transverse}
  Let~$L_1$, $L_2$, and~$L_3$ be three pairwise transverse Lagrangians. Let~$\tilde{L}_{1}$, $\tilde{L}_{2}$, and~$\tilde{L}_{3}$ be lifts to
  $\wideLag{n}$ of~$L_1$, $L_2$, and~$L_3$ respectively. Then
  \[ \mu_n( L_1, L_2, L_3) = m_n( \tilde{L}_{1}, \tilde{L}_{2}) +  m_n(
    \tilde{L}_{2}, \tilde{L}_{3}) +  m_n( \tilde{L}_{3}, \tilde{L}_{1}).\]
\end{lem}
\begin{proof}
  Consider~$d_n$ the difference of the two terms above seen as a function on
  the space of triples of pairwise transverse elements of
  $\wideLag{n}$. Since $\Lag{n}$ is the quotient of $\wideLag{n}$ by the
  center of $\widetilde{ \Sp}(2n, \R)$,
  Equation~\eqref{eq:trans-Z-Souriau} implies that this
  function descends to a function on the space of pairwise
  transverse  triples of Lagrangians. This last function is $\Sp(2n,\R)$-invariant so that
  the equality $d_n=0$ needs to be checked only on
  representatives of the finitely many $\Sp(2n, \R)$-orbits. Since one can choose
  representatives coming from the embedding $(\Lag{1})^n \to \Lag{n}$,
  thanks to Equation~\eqref{eq:Souriau-p-plus-q}, we can further
  assume that $n=1$. However, in the case $n=1$, $\Lag{1}$ can be
  identified with $S^1$ with $\mu_1$ being the orientation cocycle, and
  $\wideLag{1}$ is identified with~$\R$ with the Souriau index being
  the sign function; the equality $d_1 \equiv 0$ follows then from a direct
  calculation.
\end{proof}

Finally, we are able to extend the Souriau index to any pair in
$\wideLag{n}$:
\begin{prop}
  \label{prop:souriau-index-extension}
  For all~$\tilde{L}_{1}$ and~$\tilde{L}_{2}$ in $\wideLag{n}$, the following
  integer, defined for any~$\tilde{L}_{3}$ that is transverse to
  both~$\tilde{L}_{1}$ and~$\tilde{L}_{2}$,
  \[ \mu_n(L_1, L_2, L_3) -m_n( \tilde{L}_{2}, \tilde{L}_{3}) -m_n(
    \tilde{L}_{3}, \tilde{L}_{1})\]
  does not depend on~$\tilde{L}_{3}$.
\end{prop}
The resulting integer will be denoted by $m_n( \tilde{L}_{1}, \tilde{L}_{2})$
and called the \emph{Souriau index} of $(\tilde{L}_{1}, \tilde{L}_{2})$.  The
Souriau index is then a $\Z$-valued, antisymmetric, and
$\widetilde{\Sp}(2n, \R)$-invariant function on the space
$\wideLag{n} \times \wideLag{n}$; it satisfies
Equations~\eqref{eq:trans-T-Souriau} and~\eqref{eq:trans-Z-Souriau} for
all~$\tilde{L}_{1}$ and~$\tilde{L}_{2}$ in $\wideLag{n}$.

\begin{proof}
  Call $\delta( \tilde{L}_{1}, \tilde{L}_{2}, \tilde{L}_{3})$ the integer in the
  statement.

  We need to prove that for any other $\tilde{L}_{3}^{\prime\prime}$,
  $\delta( \tilde{L}_{1}, \tilde{L}_{2}, \tilde{L}_{3}) = \delta( \tilde{L}_{1},
  \tilde{L}_{2}, \tilde{L}_{3}^{\prime\prime})$. Choosing an element
  $\tilde{L}_{3}^{\prime}$ that is transverse to $\tilde{L}_{1}$, $\tilde{L}_{2}$,
  $\tilde{L}_{3}$, and $\tilde{L}_{3}^{\prime\prime}$, we will prove that
  $\delta( \tilde{L}_{1}, \tilde{L}_{2}, \tilde{L}_{3}) = \delta( \tilde{L}_{1},
  \tilde{L}_{2}, \tilde{L}_{3}^{\prime})$ and
  $\delta( \tilde{L}_{1}, \tilde{L}_{2}, \tilde{L}_{3}^{\prime}) = \delta(
  \tilde{L}_{1}, \tilde{L}_{2}, \tilde{L}_{3}^{\prime\prime})$. In fact, only the
  first equality needs a proof, the second being a consequence of the first
  applied to the quadruple
  $(\tilde{L}_{1}, \tilde{L}_{2}, \tilde{L}_{3}^{\prime} ,
  \tilde{L}_{3}^{\prime\prime})$.  The equality
  $\delta( \tilde{L}_{1}, \tilde{L}_{2}, \tilde{L}_{3}) - \delta( \tilde{L}_{1},
  \tilde{L}_{2}, \tilde{L}_{3}^{\prime})=0$ is the result of a direct calculation
  using the definitions, the cocycle property of the Maslov index,
  Lemma~\ref{lem:souriau-maslov-for-transverse}, and the antisymmetry of the
  Souriau index.
 \end{proof}

The relation between the Maslov index and the Souriau index is valid without
assuming the transversality of the Lagrangians.
\begin{prop}
  \label{prop:souriau-index-maslov-index-extended}
  Let $L_1$, $L_2$, and $L_3$ be
  Lagrangians and let $\tilde{L}_{1}$, $\tilde{L}_{2}$, and
  $\tilde{L}_{3}$ be lifts of $L_1$, $L_2$, $L_3$ in
  $\wideLag{n}$. Then
  \[ \mu_n(L_1, L_2, L_3) = m_n( \tilde{L}_{1},
    \tilde{L}_{2}) + m_n( \tilde{L}_{2}, \tilde{L}_{3}) + m_n( \tilde{L}_{3}, \tilde{L}_{1}).\]
\end{prop}
\begin{proof}
  By construction of the extension of the Souriau index, the identity holds
  true as soon as~$L_3$ is transverse to both~$L_1$ and~$L_2$. Let then
  $\tilde{M} =( M, \theta)$ be an element of $\wideLag{n}$ with~$M$ transverse
  to~$L_1$, $L_2$, and~$L_3$. Then, by the cocycle property of the Maslov
  index and its antisymmetry, one has
  $\mu_n( L_1, L_2, L_3) = \mu_n( L_1, L_2, M) + \mu_n( L_2, L_3, M)+ \mu_n(
  L_3, L_1, M)$ and the result follows from a successive application of
  Proposition~\ref{prop:souriau-index-extension} and the antisymmetry of the
  Souriau index.
\end{proof}

By an easy induction one gets:
\begin{lem}
  \label{lem:souriau-index-sum-triangle}
  Let $L_1$, $L_2, \dots{}, L_r$ be
  Lagrangians and let $\tilde{L}_{1}$, $\tilde{L}_{2}, \dots{},
  \tilde{L}_{r}$ be lifts of $L_1$, $L_2, \dots{}, L_r$ to
  $\wideLag{n}$. Then
  \[ \sum_{j=2}^{r-1} \mu_n(L_1, L_j, L_{j+1}) =
     m_n( \tilde{L}_{r}, \tilde{L}_{1}) + \sum_{j=1}^{r-1} m_n( \tilde{L}_{j},
    \tilde{L}_{j+1}).\]
\end{lem}

A consequence of this last lemma, the invariance of~$m_n$ and its antisymmetry is:
\begin{lem}
  \label{lem:souriau-index-sum-triangle-v2}
  Let~$M$, $L_1, \dots{}, L_{r-1}$ be Lagrangians. Let~$g$ be an element of
  $\Sp(2n,\R)$ fixing~$M$, let~$\tilde{M}$ be a lift of~$M$ and
  let~$\tilde{g} \in \widetilde{\Sp}( 2n, \R)$ be the lift of~$g$
  fixing~$\tilde{M}$.  Let~$\tilde{L}_{1}, \dots{}, \tilde{L}_{r-1}$ be lifts
  of~$L_1, \dots{}, L_{r-1}$ in $\wideLag{n}$ and set
  $L_r\coloneqq g\cdot L_1$ so that
  $\tilde{L}_{r}\coloneqq \tilde{g} \cdot \tilde{L}_{1}$ is a lift
  of~$L_r$. Then
  \[ \sum_{j=1}^{r-1} \mu_n(M, L_j, L_{j+1}) =
    \sum_{j=1}^{r-1} m_n( \tilde{L}_{j},
    \tilde{L}_{j+1}).\]
\end{lem}

\section{Translation number}
\label{sec:rotation-number}

The translation number $\widetilde{\Rot} \colon \widetilde{\Sp}(2n,\R) \to \R$ is a
conjugation invariant function defined in~\cite{BIW} using bounded cohomology. We will need the following properties:
\index{notation}{13@$\widetilde{\Rot}$ (the translation number)}
\index{definition}{translation number}%
\index{definition}{number (translation ---)}%

\begin{lem}[\cite{Strubel}]
  \label{lem:rotation-number}
  Let $h$ be in $\widetilde{\Sp}(2n,\R)$, $ \tilde{L}$ be in~$\wideLag{n}$ such
  that the projection of~$ \tilde{L}$ in~$\Lag{n}$ is fixed by the projection
  of~$h$ in $\Sp(2n, \R)$. Then
  \begin{equation*}
    \widetilde{\Rot}( h) =\frac{1}{2} m_n( h\cdot \tilde{L}, \tilde{L}).
  \end{equation*}
\end{lem}

\begin{rem}
  \label{rem:rotation-number}
  For any $h$ in $\widetilde{\Sp}(2n, \R)$ and any $\tilde{L}$ in
  $\wideLag{n}$ one has
  \[ \widetilde{\Rot}( h) = \lim_{k\to \infty} \frac{1}{2k} m_n(h^k \cdot
    \tilde{L}, \tilde{L}), \] and this equality could serve as a definition for the
  translation number (cf.~\cite[Section~10]{FrP}).
\end{rem}

\chapter{Invariants of Lagrangian subspaces}\label{sec:prelim}

The action of $\Sp(2n,\R)$ on pairs of transverse Lagrangians is
transitive. However the action of $\Sp(2n,\R)$ on triples of pairwise
transverse Lagrangians is not transitive and all the more for the actions on
quadruples and quintuples. In this chapter we describe invariants of such tuples of
Lagrangians.

Similarly, the action of $\Sp(2n,\R)$ on pairs of decorated Lagrangians is not
transitive. We introduce below the symplectic $\Lambda$-length and investigate its properties.

\section{Cross ratio}

Let $L_1, M_1, L_2, M_2$ be four Lagrangians such that~$M_1$ is transverse
to~$L_2$, $M_2$ is transverse to~$L_1$, and $L_1$ is transverse to~$L_2$.  We
use the linear isomorphisms $M_{1,L_1 \ra L_2}$ and $M_{2,L_2 \ra L_1}$,
defined in Section \ref{sec:maslov-index}, to introduce the
map
\[ [L_1,M_1,L_2,M_2]\coloneqq - M_{2,L_2 \ra L_1} \circ M_{1,L_1 \ra L_2} \colon L_1\to L_1\]
which is a linear endomorphism of~$L_1$.

\begin{df}
  The map $[L_1,M_1,L_2,M_2]\colon L_1\to L_1$ is called the \emph{cross
    ratio} of the quadruple of Lagrangians $(L_1,M_1,L_2,M_2)$.
\end{df}
\index{notation}{14@$[L_1,M_1,L_2,M_2]$  (the cross ratio of the quadruple of
  Lagrangians)}%
\index{definition}{cross ratio}

For related invariants of four Lagrangians,
see~\cite{Biallas,BB,HS,Siegel}. Cross ratios for quadruples of matrices have
been defined by Hua~\cite[Section~5]{Hua-GeoMAtI} and later extended to operator
by Zelikin~\cite{Zelikin}; these can be used to describe a cross ratio for
quadruples of $n$-planes in~$\R^{2n}$ (cf.\ Section~5.7
in~\cite[Chapter~2]{NeretinLecturesGaussian}). Noncommutative cross ratios
have been defined by Retakh~\cite{RetakhNonComCR}, and later related
in~\cite{Retakh2019NoncommutativeCA} to the work of Berenstein and
Retakh~\cite{BR} (cf.\ also Section~\ref{sec:geom-real-nonc} below).

For $n=1$, the cross ratio is a linear map from a line to itself. This is just
the multiplication by a scalar, which is exactly the classical cross ratio of four lines
in~$\R^2$.

\begin{prop}[Properties of cross ratio]\label{prop:properties_CR}
  \mbox{ }
\begin{enumerate}
\item The cross ratio is equivariant under the action of $\Sp(2n,\R)$ on
  $\Lag{n}$, that is, for any $g\in \Sp(2n,\R)$,
  \[[gL_1,gM_1,gL_2,gM_2] = h[L_1,M_1,L_2,M_2]h^{-1},\]
  where $h\colon L_1\to g(L_1)$ is the restriction of~$g$ to~$L_1$.
\item If furthermore~$M_1$ is transverse to~$L_1$ and~$M_2$ is transverse
to~$L_2$, then the maps $M_{1,L_1 \ra L_2}$ and $M_{2,L_2 \ra L_1}$
are bijective and one has:
\[[L_1,M_1,L_2,M_2]=[L_1,M_2,L_2,M_1]^{-1};\]
\item \label{item-prop-inv-double-transp} Under the same hypothesis,
\begin{align*}
[L_1,M_1,L_2,M_2] &= M_{1,L_1 \ra L_2}^{-1} \circ [L_2,M_2,L_1,M_1]\circ M_{1,L_1 \ra L_2}\\
                  &= M_{2,L_2 \ra L_1} \circ [L_2,M_2,L_1,M_1]\circ M_{2,L_2 \ra L_1}^{-1}.
\end{align*}
\item \label{item:prop-CR-cyclic} Denote
  $\psi_1\colon L_1 \to M_1 \mid v \mapsto v+M_{1,L_1 \ra L_2}(v)$, then, when~$M_1$
  and~$M_2$ are transverse,
  \[[M_1, L_2, M_2, L_1]= \psi_1 [L_1,M_1, L_2, M_2]^{-1} \psi_{1}^{-1}.\]
\end{enumerate}
\end{prop}

\begin{prop}
  The cross ratio $B\coloneqq [L_1,M_1,L_2,M_2]$ is selfadjoint with respect
  to the Maslov forms $[L_1,M_1,L_2]$ and $[L_1,M_2,L_2]$.
\end{prop}

\begin{proof} Let $\beta_1 = [L_1,M_1,L_2]$ and $\beta_2 =
  [L_2,M_2,L_1]$. Thus $\beta_2$~is a symmetric bilinear form on~$L_2$. Let
  $v,w\in L_1$. Then:
  \begin{align}\label{eq:1}
    \beta_1(Bv,w) &= \omega(-M_{2,L_2 \ra L_1} M_{1,L_1 \ra L_2}v, M_{1,L_1 \ra L_2} w) \nonumber \\
                  &=\omega(M_{1,L_1 \ra L_2}w, M_{2,L_2 \ra L_1}M_{1,L_1 \ra L_2}v)\\
                  &=\beta_2(M_{1,L_1 \ra L_2}w, M_{1,L_1 \ra L_2}v),\nonumber
  \end{align}
  and, exchanging~$v$ and~$w$, $\beta_1(Bw,v) =\beta_2(M_{1,L_1 \ra L_2}v,M_{1,L_1 \ra L_2}w)$.
  Since~$\beta_1$ and~$\beta_2$ are symmetric, $B$~is selfadjoint with respect
  to~$\beta_1$.  Exchanging the Lagrangians, we obtain that
  $[L_2,M_2,L_1,M_1]$ is selfadjoint with respect to $[L_2,M_2,L_1]$. By the
  second equality in~(\ref{item-prop-inv-double-transp}) of
  Proposition~\ref{prop:properties_CR},
  $B=M_{2,L_2 \ra L_1} \circ [L_2,M_2,L_1,M_1]\circ M_{2,L_2 \ra L_1}^{-1}$ and since
  $M_{2,L_2 \ra L_1}$ is isometric with respect to $[L_2,M_2, L_1]$
  (on~$L_2$) and $[L_1,M_2, L_2]$ (on~$L_1$), we get that $B$ is selfadjoint
  with respect to $[L_1,M_2, L_2]$.
\end{proof}

\begin{cor} \label{cor:diagonalizable}
  If the Maslov form $[L_1,M_1,L_2]$ is positive definite, then
  $[L_1,M_1,L_2,M_2]$ is diagonalizable. If the Maslov form $[L_2,M_2,L_1]$ is also
  positive definite, then $[L_1,M_1,L_2,M_2]$ has positive eigenvalues.
\end{cor}

\begin{proof}
  We set as before $\beta_1 = [L_1,M_1,L_2]$ and $\beta_2 = [L_2,M_2,L_1]$.
  Since $B=[L_1,M_1,L_2,M_2]$ is selfadjoint with respect to~$\beta_1$, there is a
  $\beta_1$-orthonormal
  basis~$\mathbf{e}$ of~$L_1$ diagonalizing~$B$, i.e.\ such that
  $[\beta_1]_\mathbf{e}=\Id$ and
  $[B]_\mathbf{e}=\diag(\lambda_1, \dots,\lambda_n)$.

  Let~$\mathbf{f}$ be the unique basis of~$L_2$ such that
  $\omega(\mathbf{e},\mathbf{f})=\Id$. By
  Remark~\ref{rem:maslov-index-matrix}, $[M_1]_{\mathbf{e},\mathbf{f}}=\Id$,
  i.e.\ for all~$i$, $f_i=M_{1}e_i$. By Equation~\eqref{eq:1}, for all~$i$,
  one has
  \begin{equation*}
    \lambda_i=\beta_1(Be_i,e_i)=\beta_2(f_i,f_i)>0,
  \end{equation*}
  since $\beta_2$ is positive definite.
\end{proof}

\begin{df}\label{df:positive-quad}
  A quadruple of Lagrangians $(L_1, M_1, L_2, M_2)$ is said to be
  \emph{positive} if $[L_1,M_1,L_2]$ and
  $[L_2,M_2,L_1]$ are positive definite.

  A symplectic basis $(\mathbf{e}, \mathbf{f})$ will be said to be
  in \emph{standard position} with respect to a positive quadruple of Lagrangians
  $(L_1, M_1, L_2, M_2)$ if there exists a diagonal matrix~$D$ with
  positive nondecreasing coefficients such that
  \[ L_1=\Span(\mathbf{e}),\ L_2=\Span(\mathbf{f}),\
    M_1=\Span(\mathbf{e}+\mathbf{f}),\ M_2=\Span(\mathbf{e}-\mathbf{f}\cdot D).\]
\end{df}\index{definition}{Lagrangian!positive quadruple of ---s}%
\index{definition}{positive! quadruple of Lagrangians}%
\index{definition}{quadruple (positive --- of Lagrangians)}%
\index{definition}{standard position (basis in ---)}%
\index{definition}{basis! in standard position}%
The matrix $D$ is uniquely determined by the quadruple
$(L_1, M_1, L_2, M_2)$ since $D^{-1}$ is equal to the cross ratio
$[L_1, M_1, L_2, M_2]$ (in the basis~$\mathbf{e}$) and is in fact the unique
diagonal matrix with nondecreasing coefficients representing this  endomorphism.

The configuration space of positive quadruples is denoted by $\Conf^{4+}(
\Lag{n}) \subset \Conf^4( \Lag{n})$; it is contained in the space of pairwise
transverse quadruples, and in particular
in~$\Conf^{4\Diamond}( \Lag{n})$.
\index{notation}{15@$\Conf^{4+}(\Lag{n})$ (the configuration space of positive quadruples)}

Note that $(L_2, M_2, L_1, M_1)$ is positive as soon as $(L_1, M_1, L_2, M_2)$
is positive. In fact the cocycle property of the Maslov index
(Proposition~\ref{maslov_prop}) implies that  $(L_1, M_1, L_2, M_2)$
is positive if and only if  $(M_1, L_2, M_2, L_1)$
is positive (see also Lemma~\ref{lem:maxim-decor-local-every-triangulation}).

Corollary~\ref{cor:diagonalizable} implies the existence of symplectic bases in standard
position (actually its proof constructs one).
The following proposition complements the conclusion of the corollary (see
also Proposition~\ref{prop:quadr-transv-lagr-normal-form} for a more general
statement without the positivity assumption):
\begin{prop}\label{prop:standard-basis-positive-4-uple}
  Let $(L_1, M_1, L_2, M_2)$ be a positive quadruple of Lagrangians. Then:
  \begin{enumerate}
  \item\label{item:1:prop:standard-basis-positive-4-uple} There exists a
    symplectic basis $( \mathbf{e}, \mathbf{f})$ in standard position.
  \item\label{item:2:prop:standard-basis-positive-4-uple} With respect to $(L_1, M_1, L_2, M_2)$, any other symplectic
    basis in standard position is of
    the form $( \mathbf{e}\cdot h, \mathbf{f}\cdot h )$ for $h$ in $\OO(n)$
    commuting with the diagonal matrix~$D$.
  \item\label{item:3:prop:standard-basis-positive-4-uple} Let
    $u=D^{1/2}$. Then
    \[ (\mathbf{e}', \mathbf{f}') \coloneqq (\mathbf{e}, \mathbf{f}) \cdot
      \begin{pmatrix}
        0 & -u^{-1}\\ u & 0
      \end{pmatrix} =( \mathbf{f}\cdot u, -\mathbf{e}\cdot u^{-1})
    \]
    is a symplectic basis that, with respect to $(L_2,
    M_2, L_1, M_1)$, is in standard position.
  \end{enumerate}
\end{prop}
\begin{proof}
  Point~(\ref{item:1:prop:standard-basis-positive-4-uple}) was already proved. Let us make first two observations: for any $n$-tuple $\mathbf{v}$ and any
  $g$ in $\GL(n,\R)$, $\Span(\mathbf{v}\cdot g) = \Span(\mathbf{v})$; for any
  pairs $(\mathbf{v}, \mathbf{w})$ of $n$-tuples such that the family
  $(\mathbf{v}, \mathbf{w})$ is free and for any~$m$, $m'$ in~$M_n(\R)$,
  $\Span( \mathbf{v}+ \mathbf{w}\cdot m)= \Span( \mathbf{v}+ \mathbf{w}\cdot
  m')$ if and only if $m=m'$.

  Let $h$ be as in~(\ref{item:2:prop:standard-basis-positive-4-uple}), and
  let $(\mathbf{e}_1, \mathbf{f}_1) = (\mathbf{e}\cdot h, \mathbf{f}\cdot
  h)$. Then $\Span( \mathbf{e}_1) =\Span( \mathbf{e}\cdot h)= \Span(
  \mathbf{e}) =L_1$, similarly $\Span( \mathbf{f}_1)=L_2$, and $\Span(
  \mathbf{e}_1+ \mathbf{f}_1) =\Span\bigl( (\mathbf{e}+\mathbf{f})\cdot h
  \bigr)=M_1$; and finally, since $h$ commutes with $D$, $\Span(
  \mathbf{e}_1 - \mathbf{f}_1 \cdot D)= \Span\bigl(
  (\mathbf{e} - \mathbf{f} \cdot D) \cdot h\bigr)=M_2$.

  Conversely, let $(\mathbf{e}_1, \mathbf{f}_1)$ be a symplectic basis in
  standard position with respect to the quadruple $(L_1, M_1, L_2, M_2)$. There is a unique
  symplectic matrix $g$ such that $(\mathbf{e}_1, \mathbf{f}_1) = (\mathbf{e},
  \mathbf{f})\cdot g$. The equalities $\Span( \mathbf{e}_1) = \Span(
  \mathbf{e})$ and $\Span( \mathbf{f}_1) = \Span( \mathbf{f})$ imply that the
  matrix~$g$ is block diagonal: $g=
  \bigl(\begin{smallmatrix}
    h & 0\\ 0 & {}^T\! h^{-1}
  \end{smallmatrix}\bigr)
  $ with $h$ in $\GL(n,\R)$. Since
  \[ \Span( \mathbf{e}_1+\mathbf{f}_1)=\Span( \mathbf{e}\cdot h+ \mathbf{f}\cdot
    {}^T\! h^{-1}) = \Span( \mathbf{e}\cdot h\, {}^T\! h + \mathbf{f})= \Span(
    \mathbf{e}+\mathbf{f}), \]
  one has $h \,{}^T\! h=\Id$, i.e.\ $h\in \OO(n)$ and $g=
  \bigl(\begin{smallmatrix}
    h & 0\\ 0 & h
  \end{smallmatrix}\bigr)
  $. Similarly, one obtains \[\Span( \mathbf{e} - \mathbf{f}\cdot hD
  h^{-1}) =\Span( \mathbf{e} - \mathbf{f}\cdot D)\] so that $h$ and
  $D$ commute.

  Point~(\ref{item:3:prop:standard-basis-positive-4-uple}) follows from
  similar considerations.
\end{proof}

Let $\Delta_n$ be the space of diagonal matrices with positive nondecreasing
coefficients. Another way to phrase the conclusion of the proposition is:
\index{notation}{17@$\Delta_n$ (the diagonal matrices with positive nondecreasing
  coefficients)}

\begin{cor}
  \label{cor:cross-ratio-positiv-one-to-one-Deltan}
  The map that associates to a quadruple $(L_1, M_1, L_2, M_2)$ its
  invariant~$D$ induces an homeomorphism between $\Conf^{4+}( \Lag{n})$
  and~$\Delta_n$.
\end{cor}

\section{Triples and symplectic bases}
\label{sec:maxim-tripl-fram}
\index{definition}{maximal! triples of decorated Lagrangians}%
\index{definition}{triple (maximal --- of decorated Lagrangians)}%
\index{definition}{decorated! maximal triples of --- Lagrangians}%
\index{definition}{Lagrangian! maximal triples of decorated ---s}%

In order to later analyze parametrizations of maximal framed representations, it is instructive to analyze the behavior of symplectic bases, adapted to a triple of Lagrangian subspaces:
\begin{lem}\label{lem:triple-of-lagrangian}
  Let $L_a$, $L_b$, and $L_c$ be Lagrangians. Let $(\mathbf{e}_a,
  \mathbf{f}_a)$, $(\mathbf{e}_b,   \mathbf{f}_b)$, and $(\mathbf{e}_c,
  \mathbf{f}_c)$ be symplectic bases such that
  \begin{align*}
    L_a &= \Span( \mathbf{e}_c) = \Span( \mathbf{f}_b) = \Span( \mathbf{e}_a+
          \mathbf{f}_a)\\
    L_b &= \Span( \mathbf{e}_a) = \Span( \mathbf{f}_c) = \Span( \mathbf{e}_b+
         \mathbf{f}_b)\\
    L_c &= \Span( \mathbf{e}_b) = \Span( \mathbf{f}_a) = \Span( \mathbf{e}_c+
          \mathbf{f}_c).
  \end{align*}
  Let the matrices $A$, $B$, and~$C$ be defined uniquely by the equalities
  \[ \mathbf{f}_b = -\mathbf{e}_c \cdot A, \ \mathbf{f}_c = -\mathbf{e}_a
    \cdot B, \ \mathbf{f}_a = -\mathbf{e}_b \cdot C.\] Then they are orthogonal and
  satisfy the equation $CBA=-\Id$.

  Furthermore,
  \[ \mathbf{e}_b = ( \mathbf{e}_c + \mathbf{f}_c) \cdot A, \ \mathbf{e}_c =
    ( \mathbf{e}_a + \mathbf{f}_a) \cdot B, \ \mathbf{e}_a = ( \mathbf{e}_b
    + \mathbf{f}_b) \cdot C. \]
\end{lem}
\begin{rem}\label{rem:rewrite-conc-lem:triple-of-lagrangian}
  This result will be used in Section \ref{sec:maxim-decor-sympl} in the presence of three other Lagrangians
  $M_a, M_b, M_c$ such that the three quadruples $(L_a, L_b, L_c, M_b)$,
  $(L_b, L_c, L_a, M_c)$, and $(L_c, L_a, L_b, M_a)$ are positive and applied
  to symplectic bases in standard position with respect to these quadruples.

  We will in particular use the conclusion in the situation where
  \[(\mathbf{e}_b,
  \mathbf{f}_b) =(\mathbf{e}_c, \mathbf{f}_c)\cdot
  \bigl(\begin{smallmatrix}
    A & -A \\ A & 0
  \end{smallmatrix}\bigr)
  .\]
\end{rem}

\begin{proof}
  The pair $(\mathbf{f}_b, -\mathbf{e}_b -\mathbf{f}_b)$ is a symplectic basis.
  One has $\Span( \mathbf{f}_b) = L_a = \Span( \mathbf{e}_c)$ and  $\Span( -
  \mathbf{e}_b -\mathbf{f}_b) = L_b = \Span( \mathbf{f}_c)$. Thus, for the
  uniquely defined matrix~$A$ above,
  \[ (\mathbf{f}_b, - \mathbf{e}_b -\mathbf{f}_b) = (\mathbf{e}_c,
    \mathbf{f}_c)\cdot
    \begin{pmatrix}
      -A & 0 \\ 0 & -{}^T\!\! A^{-1}
    \end{pmatrix}.
  \]
  Thus $\mathbf{e}_b=\mathbf{e}_c \cdot A + \mathbf{f}_c \cdot {}^T\!\! A^{-1}$. In
  particular, $\Span (\mathbf{e}_c \cdot A\, {}^T\!\! A + \mathbf{f}_c) = \Span(
  \mathbf{e}_b) =  L_c = \Span (\mathbf{e}_c + \mathbf{f}_c)$ which implies $A
  \,{}^T\!\! A=\Id$, i.e.\ $A$ is orthogonal. The same property holds for~$B$
  and~$C$. To conclude, the composition of the three changes of symplectic
  bases must be the identity, that is
  \[\begin{pmatrix} C & -C \\ C & 0  \end{pmatrix}
    \begin{pmatrix} B & -B \\ B & 0  \end{pmatrix}
    \begin{pmatrix} A & -A \\ A & 0  \end{pmatrix} =\Id,
  \]
  and the identity $CBA=-\Id$ follows.
\end{proof}
\begin{rem}
  We will often use the fact that if, for two symplectic bases
  $(\mathbf{e}, \mathbf{f})$ and $(\mathbf{e}', \mathbf{f}')$, one has
  $\Span( \mathbf{e}) =\Span( \mathbf{e}')$,
  $\Span( \mathbf{f}) =\Span( \mathbf{f}')$,
  $\Span( \mathbf{e} + \mathbf{f}) =\Span( \mathbf{e}' + \mathbf{f}')$, then
  $\mathbf{e}'= \mathbf{e}A$ and $\mathbf{f}'= \mathbf{f}A$ for $A$ in
  $\OO(n)$.
\end{rem}

The converse statement is:
\begin{lem}\label{lem:CBA-to-triple-of-lagrangian}
  Let~$A$, $B$, and~$C$ be orthogonal matrices such that $CBA=-\Id$. Let
  $(\mathbf{e}_c, \mathbf{f}_c)$ be a symplectic basis and set
  \[ (\mathbf{e}_b,
    \mathbf{f}_b) \coloneqq (\mathbf{e}_c, \mathbf{f}_c)\cdot
    \begin{pmatrix}
      A & -A \\ A & 0
    \end{pmatrix} \text{ and } (\mathbf{e}_a,
    \mathbf{f}_a) \coloneqq(\mathbf{e}_b, \mathbf{f}_b)\cdot
    \begin{pmatrix}
      B & -B \\ B & 0
    \end{pmatrix}.\]
  Then $\displaystyle (\mathbf{e}_c,
  \mathbf{f}_c) =(\mathbf{e}_a, \mathbf{f}_a)\cdot
  \begin{pmatrix}
    C & -C \\ C & 0
  \end{pmatrix}$ and the three Lagrangians
    \begin{align*}
    L_a &\coloneqq \Span( \mathbf{e}_c) = \Span( \mathbf{f}_b) = \Span( \mathbf{e}_a+
          \mathbf{f}_a)\\
    L_b &\coloneqq \Span( \mathbf{e}_a) = \Span( \mathbf{f}_c) = \Span( \mathbf{e}_b+
         \mathbf{f}_b)\\
    L_c &\coloneqq \Span( \mathbf{e}_b) = \Span( \mathbf{f}_a) = \Span( \mathbf{e}_c+
          \mathbf{f}_c)
  \end{align*}
  are pairwise transverse.
\end{lem}

\section{Positive quintuples}
\label{sec:positive-quintuples}

When one considers a two-dimensional symplectic vector space, understanding
configurations of quintuples of Lagrangian subspaces (i.e.\ of lines) is
reduced to the cross ratios of four tuples. For symplectic vector spaces of
higher dimension, there is an additional invariant, which can be considered as
an angle invariant. Since this invariant is more involved for general configurations of five Lagrangians, we describe it here only for positive quintuples.

A quintuple $(L_a, M_c, L_b, L_c, M_b)$ of Lagrangians is said \emph{positive}
if the quadruples $(L_a, L_b, L_c, M_b)$ and  $(L_a, M_c, L_b, L_c)$ are
positive or equivalently if the triples  $(L_a, L_b, L_c)$,  $(L_a, L_c, M_b)$, and  $(L_a, M_c, L_b)$ are
positive.\index{definition}{positive!quintuple of Lagrangians}%
\index{definition}{quintuple (positive --- of Lagrangians)}%
\index{definition}{Lagrangian! positive quintuple of ---s}%

Using Proposition~\ref{prop:standard-basis-positive-4-uple}, we get
$( \mathbf{e}_b, \mathbf{f}_b)$ a symplectic basis in standard position with
respect to $(L_a, L_b, L_c, M_b)$ and $( \mathbf{e}_c, \mathbf{f}_c)$ a
symplectic basis in standard position with respect to $(L_b, L_c, L_a, M_c)$.

The matrix~$A$ in $\GL(n,\R)$ such that $ \mathbf{f}_c = \mathbf{e}_b \cdot A$
belongs to~$\OO(n)$. It is called the \emph{angle invariant} of the positive
quintuple $(L_a, M_c, L_b, L_c, M_b)$ (we refer to
Sections~\ref{sec:conf-assoc-with-1} and~\ref{sec:conf-assoc-with} for the
choice of this terminology); only its class modulo right
multiplication by $\Stab_{\OO(n)} (D_b)$ and modulo left multiplication
by $\Stab_{\OO(n)} (D_c)$ is well defined (where $D_b$ and
$D_c$ are the elements of~$\Delta_n$ associated with the positive
quadruples $(L_a, L_b, L_c, M_b)$ and  $(L_b, L_c, L_a,
M_c)$).\index{definition}{angle!invariant}%
 \index{definition}{invariant (angle ---)}%

 \section{Symplectic $\Lambda$-lengths}\label{Lambda_l}
In this section we introduce an invariant of a pair of transverse decorated
Lagrangians. Since this invariant is closely related to Penner's
$\lambda$-lengths in the case when $n=1$, we call it the symplectic
$\Lambda$-length.

\begin{df}
  The \emph{symplectic $\Lambda$-length} of two decorated Lagrangians
  $\mathbf{v}, \mathbf{w}$ is the $n\times n$-matrix
  \[\Lambda_{\mathbf{v}, \mathbf{w}} \coloneqq \omega(\mathbf{v},\mathbf{w}) =
    (\omega(v_i, w_j))_{i,j\in\{1,\dots,n\}}.\]
  To shorten the notation, when $\{\mathbf{v}_i\}_{ i\in I}$ is a family of
  decorated Lagrangians, we will write
  \[\Lambda_{ij} \coloneqq  \Lambda_{\mathbf{v}_i,\mathbf{v}_j}, \text{ for }
    i,j \in I. \]
\end{df}
\index{notation}{19@$\Lambda_{\mathbf{v}, \mathbf{w}}$, $\Lambda_{ij}$ (the symplectic $\Lambda$-length)}%
\index{definition}{$\Lambda$-length (symplectic ---)}%
\index{definition}{symplectic!$\Lambda$-length}%

\begin{lem}
  For all $\mathbf{v}_1 , \mathbf{v}_2 \in\Lagd{n}$, we have
  $\Lambda_{12}=-{}^T\!\Lambda_{21}$; $\mathbf{v}_1$ and $\mathbf{v}_2$ are
  transverse if and only if $\Lambda_{12}$ is not singular, in which case
  \[ [\omega]_{\mathbf{v}_1,\mathbf{v}_2} =    \begin{pmatrix}
      0 & \Lambda_{12} \\ \Lambda_{21} & 0
    \end{pmatrix}.\]
\end{lem}

\begin{rem}
  The symplectic $\Lambda$-lengths generalize Penner's $\lambda$-lengths for
  the decorated Teichmüller space (\cite{Penner}), and one can check when
  $n=1$, the symplectic $\Lambda$-length is a square root of Penner's
  $\lambda$-length.
\end{rem}

\section{Ptolemy equation, exchange and triangle relations}
\label{sec:Ptolemy}

Penner's $\lambda$-lengths satisfy the famous Ptolemy equation. Let's fix a bilinear form $b$ on $\R^3$ with signature $(2,1)$. The isotropic cone of $b$ has two components, and its projection on the projective plane is a circle.
Given four
$b$-isotropic vectors $w_1, w_2, w_3, w_4$, contained in the same
component of the isotropic cone and whose projective images are cyclically ordered (see Figure~\ref{Prol}),
we have the relation
\[\sqrt{b(w_2,w_4) b(w_1,w_3)} = \sqrt{b(w_2, w_3) b(w_1, w_4)} - \sqrt{b(w_1, w_2) b(w_3, w_4)},\]
where the terms $b(w_i,w_j)$ are Penner's $\lambda$-lengths.
  Our symplectic $\Lambda$-lengths satisfy a
noncommutative version of the Ptolemy equation. We also call this identity
the exchange relation. Moreover, they satisfy a triangle
relation, which is trivial
in the case of~$\SL(2,\R)$.

\begin{figure}[ht]
\begin{center}
\begin{tikzpicture}
  \coordinate (L1) at (0,-1) ;
  \coordinate (L2) at (2,.5) ;
  \coordinate (L3) at (-0.5,2) ;
  \coordinate (L4) at (-2,0) ;
  \draw (L1) node[below]{$w_1$} ;
  \draw (L2) node[right]{$w_2$} ;
  \draw (L3) node[above]{$w_3$} ;
  \draw (L4) node[left]{$w_4$} ;
  \draw (L1)--(L2) ;
  \draw (L1)--(L3) ;
  \draw (L1)--(L4) ;
  \draw (L2)--(L3) ;
  \draw[dashed] (L2)--(L4) ;
  \draw (L3)--(L4) ;
\end{tikzpicture}
\caption{The tetrahedron illustrating the exchange relation}
\label{Prol}
\end{center}
\end{figure}

\begin{lem}\label{lamb}
  Let $\mathbf{v}_1$, $\mathbf{v}_3$ be two transverse decorated
  Lagrangians. Consider a third decorated Lagrangian $\mathbf{v}_2$, so that
  there exists a unique pair $(A,B)$ of $n\times n$-matrices such that
  \[ \mathbf{v}_2=\mathbf{v}_1 A+\mathbf{v}_3 B.\]
  Then
  \[  A=\Lambda_{31}^{-1}\Lambda_{32}, \quad
    B=\Lambda_{13}^{-1}\Lambda_{12},\]
  and the matrix $ \Lambda_{23}  \Lambda_{13}^{-1}\Lambda_{12}$ is
  symmetric.
\end{lem}
\begin{proof}
  One has
  $\Lambda_{12} = \omega( \mathbf{v}_1, \mathbf{v}_2) = \omega( \mathbf{v}_1,
  \mathbf{v}_1 A+ \mathbf{v}_3 B) = \omega( \mathbf{v}_1, \mathbf{v}_3 B)=
  \Lambda_{13} B$, so $ \Lambda_{13}^{-1}\Lambda_{12}=B$. Similarly
  $A=\Lambda_{31}^{-1}\Lambda_{32}$.

  Also the equality $\omega( \mathbf{v}_2, \mathbf{v}_2)=0$ gives ${}^T\!\! A
  \Lambda_{13} B + {}^T\! B \Lambda_{31} A=0$. Since $\Lambda_{13}=-{}^T\!
  \Lambda_{31}$, we get that ${}^T\!\! A
  \Lambda_{13} B= {}^T\! \Lambda_{32}\, {}^T\! \Lambda_{31}^{-1} \Lambda_{12}$ is
  symmetric and the last result follows from the equalities ${}^T\!
  \Lambda_{32}=-\Lambda_{23}$ and ${}^T\! \Lambda_{31}=-\Lambda_{13}$.
\end{proof}

The symmetry of the matrix $ \Lambda_{23}  \Lambda_{13}^{-1}\Lambda_{12}$ can be restated as follow (thanks to the
relations ${}^T\! \Lambda_{ij}=-\Lambda_{ji}$):
\begin{cor}[Triangle relation]
  \label{coro:tri-rel}
  One has
  $\Lambda_{23} \Lambda^{-1}_{13} \Lambda_{12} + \Lambda_{12}
  \Lambda^{-1}_{31} \Lambda_{32}=0$, and when~$\mathbf{v}_2$ is transverse
  to~$\mathbf{v}_1$ and to~$\mathbf{v}_3$,
  $\Lambda_{32}^{-1} \Lambda_{13} \Lambda_{21}^{-1} \Lambda_{23}
  \Lambda^{-1}_{13} \Lambda_{12}=-\Id$.
\end{cor}\index{definition}{relation!triangle ---}%
\index{definition}{triangle!relation}%

\begin{prop}[Ptolemy relation]\label{prop:Ptolemy}
  Let $\mathbf{v}_i$, $i\in\{1,2,3,4\}$, be decorated Lagrangians such
  that~$\mathbf{v}_1$ and~$\mathbf{v}_3$ are transverse. Then
  \[ \Lambda_{24}= \Lambda_{23} \Lambda_{13}^{-1} \Lambda_{14} + \Lambda_{21}
    \Lambda_{31}^{-1} \Lambda_{34}.\]
\end{prop}\index{definition}{relation!Ptolemy ---}%
\index{definition}{Ptolemy relation}%

\begin{proof}
Using Lemma~\ref{lamb}, we have
\begin{align*}
  \Lambda_{24}
  &=\omega(\mathbf{v}_2,\mathbf{v}_4) =
    \omega(\mathbf{v}_1\Lambda_{31}^{-1}\Lambda_{32} , \mathbf{v}_3\Lambda_{13}^{-1}\Lambda_{14})
    + \omega(\mathbf{v}_3\Lambda_{13}^{-1}\Lambda_{12}, \mathbf{v}_1\Lambda_{31}^{-1}\Lambda_{34})\\
  &={}^T\!(\Lambda_{31}^{-1} \Lambda_{32}) \omega(\mathbf{v}_1 , \mathbf{v}_3)
    \Lambda_{13}^{-1}\Lambda_{14}
    +{}^T\!(\Lambda_{13}^{-1}\Lambda_{12}) \omega(\mathbf{v}_3 , \mathbf{v}_1)
    \Lambda_{31}^{-1}\Lambda_{34}\\
  &= \Lambda_{23}\Lambda_{13}^{-1}\Lambda_{14} + \Lambda_{21}\Lambda_{31}^{-1}\Lambda_{34}.\qedhere
\end{align*}
\end{proof}

\section{Symplectic $\Lambda$-lengths, Maslov
  index and cross ratios}\label{sec:symp_cross}

If we choose bases for all the Lagrangian subspaces, the Maslov index and the
cross ratio can be expressed in terms of the symplectic $\Lambda$-lengths.

\begin{lem}\label{lem:symp_Maslov}
  Let $(L_i,\mathbf{v}_i)\in \Lagd{n}$, for $i\in\{1,2,3\}$, be three
  pairwise transverse decorated Lagrangians. Then the matrix
  $\Lambda_{12}\Lambda_{32}^{-1}\Lambda_{31}$ is symmetric and the Maslov
  index of $(L_1,L_2,L_3)$ is given by its signature:
  \[\mu_n(L_1,L_2,L_3) = \sgn(\Lambda_{12} \Lambda_{32}^{-1} \Lambda_{31}).\]
\end{lem}

\begin{proof}
  The symmetry of the matrix was established in Lemma~\ref{lamb}.
  Changing the bases of~$L_1$, $L_2$, and~$L_3$ transconjugates the matrix
  $\Lambda_{12} \Lambda_{32}^{-1} \Lambda_{31}$ and in particular it does not
  change its signature. Thus we can assume that $(\mathbf{v}_1, \mathbf{v_3})$
  is a symplectic basis and that there is $p\in\{0,\dots, n\}$ such that
  $v_{2,j} = v_{1,j} + v_{3,j}$ if $j\leq p$ and   $v_{2,j} = v_{1,j} -
  v_{3,j}$ if $j> p$ so that $\mu_n( L_1,L_2,L_3 ) = p-(n-p)=2p-n$. In this
  situation $\Lambda_{31}=\Lambda_{32}=-\Id$ and $\Lambda_{12}= \bigl(
  \begin{smallmatrix}
    \Id_p & 0\\ 0& -\Id_{n-p}
  \end{smallmatrix}
\bigr).$
\end{proof}

\begin{lem}[Cross ratio in terms of symplectic $\Lambda$-lengths]\label{lem:symp_cross}
  Let $(L_i,\mathbf{v}_i) $, for $i\in\{1,2,3,4\}$, be four pairwise
  transverse decorated Lagrangians. Then
  \[[L_1,L_2,L_3,L_4]_{\mathbf{v}_1} = -\Lambda_{31}^{-1} \Lambda_{34}
    \Lambda_{14}^{-1} \Lambda_{12} \Lambda_{32}^{-1} \Lambda_{31} = -
    \Lambda_{41}^{-1} \Lambda_{43} \Lambda_{23}^{-1} \Lambda_{21},\] where
  $[L_1,L_2,L_3,L_4]_{\mathbf{v}_1}$ denotes the matrix of the cross ratio in
  the basis $\mathbf{v}_1$.
\end{lem}
\begin{proof}
  By Lemma~\ref{lamb}, if $A$ and $B$ are the matrices such that
  $\mathbf{v}_2 = \mathbf{v}_1 A + \mathbf{v}_3 B$, then
  $A=\Lambda_{31}^{-1}\Lambda_{32}$ and $B=\Lambda_{13}^{-1}
  \Lambda_{12}$. Since $L_2 =\Span(\mathbf{v}_1 A + \mathbf{v}_3 B)
  =\Span(\mathbf{v}_1 + \mathbf{v}_3 BA^{-1})$, the matrix of the linear map
  $L_2\colon L_1\to L_3$ is $BA^{-1}$, that is
  \begin{align*}
    [L_2]_{\mathbf{v}_1,\mathbf{v}_3}
    & =\Lambda_{13}^{-1}\Lambda_{12}\Lambda_{32}^{-1}\Lambda_{31}\\
    \intertext{and using the triangle relation:}
    &= -\Lambda_{13}^{-1}\Lambda_{13}\Lambda_{23}^{-1}\Lambda_{21}
      = -\Lambda_{23}^{-1}\Lambda_{21}.\\
\intertext{  Similarly}
   [L_4]_{\mathbf{v}_3,\mathbf{v}_1} &=
    \Lambda_{31}^{-1}\Lambda_{34}\Lambda_{14}^{-1}\Lambda_{13} =
    -\Lambda_{41}^{-1}\Lambda_{43}.
  \end{align*}
  Therefore, on one hand
  \begin{align*}
  [L_1,L_2,L_3,L_4]_{\mathbf{v}_1} &= -\Lambda_{31}^{-1} \Lambda_{34}
  \Lambda_{14}^{-1} \Lambda_{13} \Lambda_{13}^{-1} \Lambda_{12}
  \Lambda_{32}^{-1} \Lambda_{31}= -
  \Lambda_{31}^{-1} \Lambda_{34} \Lambda_{14}^{-1} \Lambda_{12}
  \Lambda_{32}^{-1} \Lambda_{31},\\
\intertext{and on the other hand}
  [L_1,L_2,L_3,L_4]_{\mathbf{v}_1} & = - \Lambda_{41}^{-1} \Lambda_{43}
                                     \Lambda_{23}^{-1} \Lambda_{21}.\qedhere
  \end{align*}
\end{proof}

\chapter{Moduli spaces of framed and decorated local systems }
\label{sec:moduli-spac-decor}
\label{sec:rep}

This chapter describes the framed and decorated local systems whose moduli
spaces are later parametrized.  The interpretation of these spaces in term of
\enquote{genuine} representations are also given.  We introduce Maslov
indices for triples associated with framed local systems and explain how
these can be used to calculate the Toledo number. The characterization of
maximal representations in term of these Maslov indices will lead to the
definition of maximal framed local systems.

\section{Topological data}
\label{sec:topological-data}

From now on, $S$ will denote an oriented, connected, finite-type
surface with nonempty boundary and without punctures. There is an essentially
unique connected oriented compact surface  $\bar{S}$ with nonempty boundary
$\partial \bar{S}$ and a finite set  $R \subset \partial \overline{S}$ so that
$S$ is the complement $\bar{S} \setm R$. The boundary of $S$ can be noncompact.
\index{notation}{22@$S$ (the finite type surface)}
\index{notation}{23@$\bar{S}$, $R$, $r$ (its compactification, $R\subset \partial
  \bar{S}$ of cardinal~$r$)}

We denote by~$g$ the genus of~$\bar{S}$, by $k$ the number of connected components of~$\partial \bar{S}$, and by~$r$ the cardinality of~$R$. We will also denote by~$p\in \{ 0, \dots, k\}$ the number of compact connected
components of~$\partial S$, i.e.\ the number of connected components
of~$\partial \bar{S}$ not meeting~$R$.

We will always assume that
\[ 4g - 4 + 2k + r > 0.\]
This precisely means that the double of~$S$ along
its boundary $\partial S = \partial \bar{S} \setm R$ has negative Euler characteristic.

\begin{rem}
The conditions we assume on~$S$ can be equivalently expressed in the following geometric way: $S$ admits a complete hyperbolic structure of finite volume with nonempty, possibly noncompact, geodesic boundary and without cusps.
\end{rem}

As examples, let's see some special and \enquote{extreme} cases. On the one hand, $r$~can be zero. This is the case when $S$~is compact, with negative Euler characteristic. On the other hand, the surface~$\bar{S}$ can be a disk ($g=0, k=1$), or an annulus ($g=0, k=2$). In these cases, our hypothesis implies that $r \geq 3$ for the disc, and $r \geq 1$ for the annulus.

\section{Fundamental group}
\label{sec:fundamental-group}

The fundamental group of~$S$ is a free group of rank~$2g+k-1$.

For definiteness, we denote by~$b_0$ the base point of the surface~$S$, i.e.\
$\pi_1(S)= \pi_1(S, b_0)$. Thus elements of~$\pi_1(S)$ are classes of loops
based at~$b_0$ and, for~$\alpha$ and~$\gamma$ in~$\pi_1(S)$, represented by
loops~$a$ and~$g$ respectively, their product~$\alpha \gamma$ or
$\alpha * \gamma$ is represented by the loop $a* g$ obtained juxtaposing~$g$
then~$a$, in that order. More generally, for two paths $a\colon [0,1]\to S$,
$g\colon [0,1]\to S$ such that $g(1)=a(0)$, we will denote by $a* g$ the
juxtaposition of~$g$ and~$a$. With this convention, the fundamental
group~$\pi_1(S)$ acts on the \emph{right} on the universal
cover~$\widetilde{S}$.

For later purpose, we \emph{fix} a path $\alpha_C\colon [0,1]\to S$
between~$b_0$ and~$C\cap S$ ($=C\setm R$) for every component~$C$ of~$\partial
\bar{S}$. This enables us in particular to give
$p$~elements of~$\pi_1(S)$, $c_1, \dots, c_p$, representing the compact
components of~$\partial
S$. Obviously, only the conjugacy class of~$c_j$ is intrinsically defined.
The other components of~$\partial \bar{S}$ will be numbered $C_1, \dots,
C_{k-p}$; they contain respectively $r_1, \dots, r_{k-p}$ elements of~$R$;
 $r_\ell$ is a positive integer and $\sum_{\ell} r_\ell =r$. The arc~$\alpha_{C_\ell}$ joining~$b_0$
and~$C_\ell$ will also be denoted by~$\alpha_\ell$.

\section{Triangulations}
\label{sec:triangulations-general}

A continuous one-to-one map $\alpha\colon [0,1]\to S$ with $\alpha(0),
\alpha(1)\in \partial S$
will be called an \emph{arc}; the map $\bar{\alpha}\colon [0,1]\to
S\mid t\mapsto \alpha(1-t)$ is also an arc and is called the
\emph{reverse} of~$\alpha$. Arcs will be considered up to homotopy: two arcs
$\alpha_0$ and~$\alpha_1$ are called \emph{homotopic relative to the boundary}
(or simply \emph{homotopic}) if there is a continuous map $h\colon [0,1]\times
[0,1]\to S$ with $\alpha_0 = h|_{ \{0\}\times [0,1]}$, $\alpha_1 = h|_{
  \{1\}\times [0,1]}$, and $h( [0,1]\times \{0,1\})\subset \partial S$; an arc~$\alpha$
is called \emph{homotopically trivial} if there exists a continuous map
$h\colon [0,1]\times [0,1]\to S$ with $\alpha = h|_{ \{0\}\times [0,1]}$,  $h(
[0,1]\times \{0,1\})\subset \partial S$, and $ h|_{
  \{1\}\times [0,1]}$ is the constant application equal to $h(1,1)\in \partial
S$.

For every arc~$\alpha$, the pair $\{\alpha, \bar{\alpha}\}$ is called a
\emph{nonoriented arc}; two nonoriented arcs $\{\alpha_1, \bar{\alpha}_1\}$,
$\{\alpha_2, \bar{\alpha}_2\}$ are called \emph{homotopic} if $\alpha_1$ and
$\alpha_2$ (or $\alpha_1$ and $\bar{\alpha}_2$) are homotopic relative to the boundary.
A nonoriented arc $\{\alpha,\bar{\alpha}\}$ is \emph{homotopically trivial} if
$\alpha$~is homotopically trivial.

An \emph{ideal triangulation}~$\mathcal{T}$ of~$S$ is
a family of nonoriented homotopically nontrivial arcs in~$S$ 
 satisfying
the following maximality
condition: every nonoriented homotopically nontrivial arc 
not intersecting any
of the arcs of~$\mathcal{T}$ is homotopic (relative to
the boundary) to one of the arcs of~$\mathcal{T}$.  It is convenient, though
not necessary, that all the endpoints in a given boundary component
of~$\partial S$ coincide
(cf.\ Figure~\ref{fig:triangulation}). Triangulations are considered up to
homotopy (relative to the boundary).

\index{notation}{24@$\mathcal{T}$ (ideal triangulation)}%
\index{definition}{ideal triangulation}%
\index{definition}{triangulation (ideal ---)}%

\begin{figure}[ht]
\begin{tikzpicture}[scale=1.3]
\coordinate (A) at (-7,2.5) ;
\coordinate (B) at (-2,2.5) ;
\coordinate (C) at (-4.5,1) ;
\coordinate (D) at (-4.2,1.5);

\draw[thick]  plot[smooth, tension=.7] coordinates {(-6,3) (-5.5,3.5) (-6,5.5) (-4.5,6.5) (-3,5.5) (-3.5,3.5) (-3,3)};
\draw[thick]  plot[smooth, tension=.7] coordinates {(-5.3,5.5) (-4.5,5) (-3.7,5.5)};
\draw[thick] plot[smooth, tension=.7] coordinates {(-5.1,5.3) (-4.5,5.5) (-3.9,5.3)};
\draw[thick]  plot[smooth, tension=.7] coordinates {(-5.55,3.8) (-6.5,3.4) (A) (-6.5,1.6) (C) (-2.5,1.6) (B)(-2.5,3.4) (-3.45,3.8)};
\draw[thick]  (-4,1.5) ellipse (0.2 and 0.15);

\draw[purple]  plot[smooth, tension=.7] coordinates {(D) (-4,1.8) (-3,2) (B)};
\draw[purple] plot[smooth, tension=.7] coordinates {(D) (-6.2,2) (A)};
\draw[blue] plot[smooth, tension=.7] coordinates {(D) (C)};
\draw[blue] plot[smooth, tension=.7] coordinates {(C) (-6,1.5) (A)};

\draw[blue]  plot[smooth, tension=.7] coordinates {(C) (-3,1.5) (B)};
\draw[blue]  plot[smooth, tension=.7] coordinates {(A) (-6,3.3) (-5.5,3.5)};
\draw[blue]  plot[smooth, tension=.7] coordinates {(-3.5,3.5) (-3,3.3) (B)};
\draw[blue,densely dashed]  plot[smooth, tension=.7] coordinates {(-5.5,3.5) (-3.5,3.5)};

\draw[red]  plot[smooth, tension=.7] coordinates {(A) (-5,3) (-4.85,5.12)};
\draw[red]  plot[smooth, tension=.7] coordinates {(-5.84,4.5) (-5.5,3.2) (A)};
\draw[red,densely dashed]  plot[smooth, tension=.7] coordinates {(-5.84,4.5) (-5.5,4.9) (-4.85,5.12)};

\draw[red]  plot[smooth, tension=.7] coordinates {(B) (-4,3) (-4.15,5.12)};
\draw[red]  plot[smooth, tension=.7] coordinates {(-3.16,4.5) (-3.5,3.2) (B)};
\draw[red,densely dashed]  plot[smooth, tension=.7] coordinates {(-3.16,4.5) (-3.5,4.9) (-4.15,5.12)};

\draw[cyan]  plot[smooth, tension=.7] coordinates {(A) (-5.5,3) (-5.5,4.5) (-5.5,6) (-3.8,6) (-3.9,4.5) (-3.8,3) (B)};

\draw[green]  plot[smooth, tension=.7] coordinates {(B) (-3.5,3) (-3.7,4.5) (-3.5,6.15)};
\draw[green,densely dashed]  plot[smooth, tension=.7] coordinates {(-3.5,6.15) (-4.5,6) (-5,5.35)};
\draw[green]  plot[smooth, tension=.7] coordinates {(-5,5.35) (-5.2,5.15) (-5,4) (-5.5,2.85) (A)};

\draw[yellow]  plot[smooth, tension=.7] coordinates {(B) (-3.49,3.35)};
\draw[yellow,densely dashed]  plot[smooth, tension=.7] coordinates {(-3.49,3.35) (-4,4) (-4.3,5.05)};
\draw[yellow]  plot[smooth, tension=.7] coordinates {(-4.3,5.05) (D)};

\draw[yellow]  plot[smooth, tension=.7] coordinates {(A) (-5.51,3.35)};
\draw[yellow,densely dashed]  plot[smooth, tension=.7] coordinates {(-5.51,3.35) (-5,4) (-4.7,5.05)};
\draw[yellow]  plot[smooth, tension=.7] coordinates {(-4.7,5.05) (D)};

\draw (-2.5,3.4) node {$\bullet$};
\draw (-6.5,1.6) node {$\bullet$};
\draw (-2.5,1.6) node {$\bullet$};

\end{tikzpicture}
\caption{An example of an ideal triangulation of a surface; here $g=1$,
  $k=4$, $r=3$,
  and $p=1$; the set~$R\subset \bar{S}$ consists of the three dotted points;
  there are $12=2r-3\chi(\bar{S})$ nonoriented edges in the triangulation.}   \label{fig:triangulation}
\end{figure}

The nonoriented arcs in~$\mathcal{T}$ are called the \emph{nonoriented
  edges} of~$\mathcal{T}$. For any such nonoriented edge $\{\alpha,
\bar{\alpha}\}$, the arcs $\alpha$ and $\bar{\alpha}$ are called
\emph{oriented edges} of~$\mathcal{T}$. The set of oriented edges
of~$\mathcal{T}$ will be denoted $\edgeT$.
Until
Chapter~\ref{sec:Acoordinates} we will not need to consider these oriented
edges and {nonoriented edges} will sometimes
 simply be called \emph{edges}.

 The union (of the images) of the edges of~$\mathcal{T}$ cuts the surface~$S$ into
 simply connected components; among them exactly~$r$ contain in their closure
 (in~$\bar{S}$)
 an element of~$R$; the other connected components will be called \emph{faces}
 (or \emph{triangles}) of~$\mathcal{T}$.

It is well known that triangulations can be obtained from one another by a
series of elementary moves called \emph{flips}. Two
triangulations~$\mathcal{T}_0$ and~$\mathcal{T}_1$ are related by a flip if
there exist (nonoriented) edges~$e_0$ in~$\mathcal{T}_0$ and~$e_1$ in~$\mathcal{T}_1$ such
that $\mathcal{T}_0 \setm \{ e_0\} = \mathcal{T}_1 \setm \{ e_1\}$ (cf.\ Figure~\ref{unoriented_flip}).

\begin{figure}[ht]
\begin{center}
\begin{tikzpicture}
  \coordinate (L1) at (0,-1) ;
  \coordinate (L2) at (2,.5) ;
  \coordinate (L3) at (-0.5,2) ;
  \coordinate (L4) at (-2,0) ;
  \draw (L1)--(L2) ; 
  \draw (L1)--(L3) node[pos=0.6, right]{$e_1$};
  \draw (L4)--(L1) ; 
  \draw (L3)--(L2) ;
  \draw (L4)--(L2) node[pos=0.65,below]{$e_0$} ;
  \draw (L4)--(L3) ;
\end{tikzpicture}
\caption{The edges involved in a flip}
\label{unoriented_flip}
\end{center}
\end{figure}

\begin{lem}
The number of vertices of~$\mathcal{T}$ is $p+r$, the number of nonoriented edges is
$2r-3{ \chi( \bar{S})}$, among those there are~$r$ \emph{external} edges, i.e.\ contained in exactly one face of the
triangulation, and there are $ r-3{ \chi(\bar{S})}$ \emph{internal} edges \ep{contained in two faces}, finally the number of
faces \ep{or triangles} is $r-2{ \chi( \bar{S})}$.\index{definition}{internal!edge}%
\index{definition}{edge!internal ---}%
\index{definition}{edge!external ---}%
\index{definition}{external!edge}%
\end{lem}

The~$r_\ell$ edges connecting cyclically the components of
$C_\ell \setm R$ (for every~$\ell$ in $\{1, \dots, k-p\}$) always belongs to
the triangulations~$\mathcal{T}$ of~$S$. These are the edges that account for
the~$r$ external edges.

\section{Framed symplectic local systems}
\label{sec:decor-sympl-local}

Let~$\mathcal{F}$ be a $\Sp(2n,\R)$-local system on the surface~$S$,
i.e.~$\mathcal{F}$ is a right principal $\Sp(2n,\R)_d$-bundle where
$\Sp(2n,\R)_d$ is the topological group~$\Sp(2n,\R)$ equipped with the
discrete topology (or, what amounts to the same, the change of trivializations
are locally constant maps into~$\Sp(2n,\R)$).\index{definition}{local system}
 Using the
parallel transport, we find a homomorphism~$\rho\colon \pi_1(S) \to
\Sp(2n,\R)$ so
that~$\mathcal{F}$ is the quotient
$\pi_1(S) \backslash ( \widetilde{S} \times \Sp(2n, \R))$ by the left diagonal
action on $\widetilde{S} \times \Sp(2n, \R)$ determined by~$\rho$; namely
$\gamma\cdot (p,g) = (p\cdot \gamma^{-1}, \rho(\gamma)g)$ for all~$\gamma$ in
$\pi_1(S)$ and all $(p,g) \in \widetilde{S} \times \Sp(2n, \R)$. Given~$\mathcal{F}$ there is an associated bundle with fiber~$\Lag{n}$, the  Lagrangian Grassmannian, which we denote by~$\mathcal{F}_{\Lag{n}}$. It is the quotient
$\pi_1(S) \backslash ( \widetilde{S} \times \Lag{n})$.

\begin{df}
  \label{df:decor-sympl-local}
  A \emph{framing} of a symplectic local system~$\mathcal{F}$ is a flat
  section of $\mathcal{F}_{\Lag{n}}|_{\partial S}$. A \emph{framed
    symplectic local system} is a pair $(\mathcal{F}, \sigma)$ where~$\sigma$
  is a framing of the symplectic local system~$\mathcal{F}$.
\end{df}
\index{notation}{26@$(\mathcal{F},\sigma)$ (a framed symplectic local
  system)}%
\index{definition}{framed!symplectic local system}%
\index{definition}{framing}%
\index{definition}{symplectic! framed --- local system}%
\index{definition}{local system! framed symplectic --- }%

\begin{rem}
  \label{rem:fram-sympl-local-universal}
  The framing~$\sigma$ is equivalently given as a locally constant and
  $\rho$-equivariant map $\tilde{\sigma}\colon \partial \tilde{S} \to \Lag{n}$; this means
  that, for all~$\gamma$ in~$\pi_1(S)$ and all~$p$ in~$\partial \tilde{S}$,
  $\gamma \cdot ( p, \tilde{\sigma}(p)) = (p\cdot \gamma^{-1}, \tilde{\sigma}
  (p\cdot \gamma^{-1}))$. This last equality can be rewritten $\tilde{
    \sigma}( p\cdot \gamma^{-1}) = \rho(\gamma) \cdot \tilde{ \sigma}(p)$.
\end{rem}

We will denote by $\Locf( S, \Sp(2n,\R))$\index{notation}{29@$\Locf( S, \Sp(2n,\R))$ (moduli of framed symplectic local systems)} the moduli space of framed symplectic local systems. A framed local system gives rise to the following data (with
the notation introduced in Section~\ref{sec:fundamental-group})
\begin{itemize}
\item a representation $\rho\colon \pi_1(S)\to \Sp(2n,\R)$,
\item for every $j=1, \dots, p$,  a $\rho(c_j)$-invariant Lagrangian~$L_j$. It is obtained using  the parallel transport along the
  arc~$\alpha_{C}$ where~$C$ is the compact boundary component homotopic to~$c_j$.
\item for every $\ell=1, \dots, k-p$, a $r_\ell$-tuple
  $(L_{1,\ell}, \dots, L_{r_\ell, \ell})$ of Lagrangians, obtained again
  using the parallel transport along the arc~$\alpha_\ell$ and along arcs
  connecting cyclically the~$r_\ell$ components of~\mbox{$C_\ell \setm R$}.
\end{itemize}
The tuple
$(\rho, \{L_j\}_{j=1}^{ p}, \{ (L_{j, \ell})_{j=1}^{r_\ell}\}_{\ell
  =1}^{k-p})$
will be called the \emph{holonomy} of the framed local system
$(\mathcal{F}, \sigma)$; it is only defined up to the action of the symplectic
group and it determines completely the isomorphism class of the framed
symplectic local system $(\mathcal{F}, \sigma)$. We are thus led to consider the following definition.\index{definition}{holonomy}
\begin{df}
A \emph{framed symplectic representation} is a tuple
  \[ \bigl(\rho, \{L_j\}_{j=1}^{ p}, \{ (L_{j, \ell})_{j=1}^{r_\ell}\}_{\ell
      =1}^{k-p}\bigr) \in \Hom( \pi_1(S), \Sp(2n, \R)) \times \Lag{n}^p \times
    \prod_{\ell=1}^{k-p} \Lag{n}^{r_\ell}\] satisfying the condition
  $\rho(c_j)\cdot L_j=L_j$ for all~$j=1, \dots, p$.
\end{df}\index{definition}{framed! symplectic representation}%
\index{definition}{symplectic! framed --- representation}%
\index{definition}{representation! framed symplectic ---}%

We denote by $\Homf(S, \Sp(2n, \R))$ the space of all framed symplectic representations. This space carries a natural action of~$\Sp(2n, \R)$, and we will denote the quotient by
\[\Repf(S, \Sp(2n, \R)) = \Homf(S, \Sp(2n, \R)) / \Sp(2n, \R). \]
\index{notation}{31@$\Repf(S, \Sp(2n, \R))$ (moduli of framed symplectic representations)}

\begin{rem}
  \label{rem:moduli-not-Hausdorff}
  \label{rem:Hausdorff}
  When $\Repf(S, \Sp(2n, \R))$ is endowed with the quotient topology, it is in general not a Hausdorff space.
  In this work, we will be led to consider \emph{noncontinuous}
  ($\Sp(2n,\R)$-invariant) functions on $\Homf(S, \Sp(2n, \R))$, and it is therefore more relevant not
  to consider any Hausdorffication of the moduli space.
\end{rem}

\begin{rem}
  \label{rem:moduli-more-intrinsic}
One can give a more intrinsic way to describe the moduli space
  $\Repf(S, \Sp(2n, \R))$ by considering at the same time all the
  possible arcs~$\alpha_\ell$ and all the possible representatives~$c_j$ of the
  compact boundary components. We will not write this definition here explicitly.
\end{rem}

We can identify $\Locf(S, \Sp(2n,\R))$ and $\Repf(S, \Sp(2n, \R))$:

\begin{lem}
  \label{lem:decor-sympl-local-holonomy}
  The holonomy map is a bijection between $\Locf(S, \Sp(2n,\R))$ and
  $\Repf(S, \Sp(2n, \R))$.
\end{lem}

We use this identification to endow the moduli space $\Locf(S, \Sp(2n,\R))$ with a (non-Hausdorff) topology.

\section{Twisted local systems}
\label{sec:twist-local-syst}

We now introduce twisted local systems. The moduli space of twisted local system is in bijection with the moduli space of local systems, but this bijection is not natural: it depends on a choice, see Remark \ref{rem:use-twist-local-syst}. The coordinates we introduce later are naturally defined on spaces of twisted local systems. Only after making the choice in Remark \ref{rem:use-twist-local-syst} we can identify the moduli space of twisted local systems with the moduli space of ordinary local systems, and thus get a parametrization of the latter.

We denote by $T'S$ the \emph{punctured} tangent bundle
 of~$S$, i.e.\ the tangent bundle~$T S$ with the zero section removed. Its fundamental
group is a central extension of~$\pi_1(S)$\index{notation}{34@$T'S$ (the punctured
  tangent bundle)}%
\index{definition}{punctured tangent bundle}%
\[ \Z \longrightarrow \pi_1( T'S ) \longrightarrow \pi_1(S)\]
that necessarily splits, since $\pi_1(S)$ is a free group. The generating
element~$\delta$ of the kernel of the extension is the class in~$\pi_1( T' S)$
of any loop representing the homotopy of a fiber of $T'S \to S$.

Let~$G$ be a group and let~$\delta_G$ be a central element in~$G$. For the
case when~$G=\Sp(2n,\R)$ we will always take $\delta_G=-\Id$. For the case
when~$G$ is a connected Lie group locally isomorphic to~$\Sp(2n,\R)$, the
relevant element~$\delta_G$ is described in Section~\ref{sec:some-elements}.

\begin{df}
  A $G$-local system~$\mathcal{F}$ on~$T'S$ is said to be \emph{twisted} (or
  $\delta$-twisted) if its holonomy in restriction to any fiber of $T'S \to S$
  is equal to~$\delta_G$.  The moduli space of $\delta$-twisted $G$-local
  systems is denoted by $\Locdelta( S,G)$.
  \index{notation}{36@$\Locdelta( S,G)$ (moduli of twisted local systems)}%

  A representation $\rho\colon \pi_1( T'S) \to G$ is said
  $\delta$-\emph{twisted} if $\rho(\delta)=\delta_G$.
\end{df}\index{definition}{twisted!local system}%
\index{definition}{local system!twisted ---}%
\index{definition}{twisted!representation}%
\index{definition}{representation!twisted ---}%
The space of $\delta$-twisted representations is denoted by $\Homdelta(S, G)$,
the space of their $G$-conjugacy classes by $\Repdelta(S,G)$.
\index{notation}{38@$\Repdelta(S,G)$ (moduli of twisted representations)}

The holonomy map gives an identification between $\Locdelta( S,G)$ and $\Repdelta(S,G)$.

\begin{rem}
  \label{rem:use-twist-local-syst}
  The choice of an isomorphism $\pi_1( T'S)\simeq \Z \times \pi_1(S)$ gives a  one-to-one correspondence between
  $\Homdelta(S, G)$ and  $\Hom( \pi_1(S), G)$.
  This identification is mapping class group equivariant but is not canonical
  since it relies on the choice of isomorphism.
   We will rely on this correspondence and describe in
  Chapter~\ref{sec_def_coord_max} and subsequent chapters moduli spaces of (decorated) local
  systems.
\end{rem}

\section{Vector fields}
\label{sec:vector-fields}
Since~$S$ has nonempty boundary there exists a nonvanishing vector field
on~$S$. Such a vector field gives rise to an isomorphism
$\pi_1( T'S)\simeq \Z \times \pi_1(S)$. We describe now a slightly different
identification, which relies on the existence of a vector field with isolated
zeros. This is a bit more involved but useful for later considerations.

Let~$\vec{x}$ be a vector field with isolated zeros and without any zero on
the compact components of~$\partial S$. We denote by $Z( \vec{x})$ the set of
its zeroes in the \emph{interior} of~$S$. For every~$u$ in $Z( \vec{x})$, the
index $\mathrm{ind}(u, \vec{x})$ (later simply denoted $\mathrm{ind}(u)$) is
the number of \enquote{turns} that~$\vec{x}$ does around~$u$ (e.g.\ it is $+1$
if $\vec{x}$ is tangent to a small circle around~$u$, it is~$-1$ if~$\vec{x}$
presents a saddle singularity). For every~$u\in Z( \vec{x})$, let us
fix~$\gamma_u$ an element in $\pi_1( S\setm Z( \vec{x}))$ that is freely
homotopic to a small (directly oriented) circle around~$u$.

The vector field~$\vec{x}$ provides a trivialization of
$T' ( S\setm Z( \vec{x}))$, that we can use to obtain:

\begin{lem}
  \label{lem:vector-fields-pi_TS}
  The homomorphism $\pi_1( T'( S\setm Z( \vec{x}))) \to \pi_1(
  T'S)$ induces an isomorphism between~$\pi_1( T'S)$ and the quotient of $\Z
  \times \pi_1( S\setm Z( \vec{x})) \simeq \pi_1( T'( S\setm Z(
  \vec{x})))$ by the normal subgroup generated by the elements $(
  -\mathrm{ind}(u), \gamma_u)$ for~$u$ in $Z( \vec{x})$.
\end{lem}

As a consequence:
\begin{cor}
  \label{cor:vector-fields-twisted-rep}
  The isomorphism of Lemma~\ref{lem:vector-fields-pi_TS} gives a one-to-one $G$-equivariant
  correspondence between the space $\Homdelta(S, G)$ and the space
  \[ \Homdelta( S \setm Z( \vec{x}), G) \coloneqq \{ \rho \in
    \Hom( \pi_1( S\setm Z( \vec{x})),G) \mid \forall u\in Z(
    \vec{x}), \  \rho( \gamma_u)= \delta_{G}^{\mathrm{ind(u)}}\}.\]
\end{cor}
We call the representations in $\Homdelta( S \setm Z( \vec{x}), G)$  also  twisted representations.

\begin{rem}
  \label{rem:vector-fields-boundary-elements}
  The surface $S\setm Z( \vec{x})$ is (almost) of the same type as the
  surface~$S$ so that the discussion of Sections~\ref{sec:topological-data}
  and~\ref{sec:fundamental-group} applies. In particular we will also denote
  by~$c_1, \dots, c_p$ representatives in $\pi_1( S\setm Z( \vec{x}))$ of the
  compact boundary components of~$S$. These elements depend now on choices of
  arcs in~$S\setm Z( \vec{x})$. Similarly, paths~$\alpha_\ell$ connecting in
  $S \setm Z( \vec{x})$ the base point to $C_\ell\setm R$ will be chosen.
\end{rem}

\begin{rem}
  \label{rem:vector-fields-index-order-delta_G}
  One can go a little further in the conclusion of
  Corollary~\ref{cor:vector-fields-twisted-rep} and remove only the zeroes
  of~$\vec{x}$ whose index is not a multiple of the order of~$\delta_G$. When
  $G=\Sp(2n,\R)$, so that $\delta_G$ is of order~$2$, we will sometimes be
  able to use vector fields whose indices are all even. In this case, the
  space $\Homdelta(S, \Sp(2n,\R))$ is in fact a space of
  representations of~$\pi_1(S)$.
\end{rem}
\section{Triangulations and vector fields}
\label{sec:vect-fields-triang}
Given a triangulation~$\mathcal{T}$ of~$S$ we can construct a vector field~$\vec{x}_\mathcal{T}$ on~$S$ that satisfies the condition required in Section~\ref{sec:vector-fields}.
It is constructed from a fixed model on
each triangle of~$\mathcal{T}$ that we now describe:
\begin{itemize}
\item The zeroes of the vector field are
  \begin{itemize}
  \item the vertices of the triangles, and, in a neighborhood of these
    vertices, the vector field is tangent to the clockwise oriented circles
    centered at the vertices;
  \item and the center of the triangle, and it is tangent to the
    counter-clockwise oriented circles in a neighborhood of this center
    (index~$1$).
  \item the midpoints of its edges, where it is \enquote{half} of a saddle
    connection.
  \end{itemize}
\item there are~$3$ complete flow lines emanating and ending at the midpoints
  of the edges; these flow lines cut the triangle into~$4$ regions, $3$ of
  them contain exactly one vertex and the last one contains the center.
  Furthermore, in each of the~$3$ regions containing a vertex, the flow lines
  foliate the region into a family of homotopic curves going from one half
  edge to another half edge, the last region contains a zero of the vector
  field and the other flow lines are counter-clockwise circles around this
  zero.
\end{itemize}
This basically determines uniquely the vector field, it is illustrated in Figure~\ref{fig:vector-field-xT}.

\begin{figure}[ht]
\centering
\begin{tikzpicture}
  \coordinate (A) at (5,0) ;
  \coordinate (B) at (5.5,5) ;
  \coordinate (C) at (9,2.6) ;
  \draw (A)--(B);
  \draw (A)--(C);
  \draw (C)--(B);
  \coordinate (ma) at (barycentric cs:A=0,B=1,C=1) ;
  \coordinate (mb) at (barycentric cs:A=1,B=0,C=1) ;
  \coordinate (mc) at (barycentric cs:A=1,B=1,C=0) ;
  \draw[middlearrow={latex},dotted] (ma) to[bend right,looseness=.4] (mc) ;
  \draw[middlearrow={latex},dotted] (mb) to[bend right,looseness=.4] (ma) ;
  \draw[middlearrow={latex},dotted] (mc) to[bend right,looseness=.4] (mb) ;
  \coordinate (Atb) at (barycentric cs:A=2,B=1,C=0) ;
  \coordinate (Atc) at (barycentric cs:A=2,B=0,C=1) ;
  \draw[middlearrow={latex}] (Atb) to[bend left,looseness=.2] (Atc) ;
  \coordinate (Asb) at (barycentric cs:A=5,B=1,C=0) ;
  \coordinate (Asc) at (barycentric cs:A=5,B=0,C=1) ;
  \draw[middlearrow={latex}] (Asb) to[bend left,looseness=.7] (Asc) ;
  \coordinate (Bta) at (barycentric cs:A=1,B=2,C=0) ;
  \coordinate (Btc) at (barycentric cs:A=0,B=2,C=1) ;
  \draw[middlearrow={latex}] (Btc) to[bend left,looseness=.2] (Bta) ;
  \coordinate (Bsa) at (barycentric cs:A=1,B=5,C=0) ;
  \coordinate (Bsc) at (barycentric cs:A=0,B=5,C=1) ;
  \draw[middlearrow={latex}] (Bsc) to[bend left,looseness=.7] (Bsa) ;
  \coordinate (Cta) at (barycentric cs:A=1,B=0,C=2) ;
  \coordinate (Ctb) at (barycentric cs:A=0,B=1,C=2) ;
  \draw[middlearrow={latex}] (Cta) to[bend left,looseness=.2] (Ctb) ;
  \coordinate (Csa) at (barycentric cs:A=1,B=0,C=5) ;
  \coordinate (Csb) at (barycentric cs:A=0,B=1,C=5) ;
  \draw[middlearrow={latex}] (Csa) to[bend left,looseness=.7] (Csb) ;
  \coordinate (C1a) at (barycentric cs:A=1,B=3,C=3) ;
  \coordinate (C1b) at (barycentric cs:A=3,B=1,C=3) ;
  \coordinate (C1c) at (barycentric cs:A=3,B=3,C=1) ;
  \draw [middlearrow={latex}] (C1b) to[out=30, in=-40, looseness=.3] (C1a);
  \draw  (C1a) to[out=140, in=90, looseness=.3] (C1c);
  \draw  (C1c) to[out=-90, in=210, looseness=.3] (C1b);
  \coordinate (C2a) at (barycentric cs:A=2,B=3,C=3) ;
  \coordinate (C2b) at (barycentric cs:A=3,B=2,C=3) ;
  \coordinate (C2c) at (barycentric cs:A=3,B=3,C=2) ;
  \draw [middlearrow={latex}] (C2b) to[out=30, in=-40, looseness=.7] (C2a);
  \draw  (C2a) to[out=140, in=90, looseness=.7] (C2c);
  \draw  (C2c) to[out=-90, in=210, looseness=.7] (C2b);
  \coordinate (cent) at (barycentric cs:A=3,B=3,C=3) ;
  \draw (cent) node {$.$};
\end{tikzpicture}
\caption{The model vector field; dotted lines are complete flow lines}\label{fig:vector-field-xT}
\end{figure}

To be complete, we remove from this local model a small neighborhood of
every vertex of the triangle if this vertex belongs to a compact component
of~$\partial S$.

We use this vector field  in Chapter~\ref{sec:local-systems-their} to describe the spaces of twisted
local systems introduced
here in terms of twisted local systems on a quiver.

\section{Decorated twisted local systems}
\label{sec:framed-local-systems-1}
In this section we introduce decorated twisted symplectic systems. Given a symplectic local system~$\mathcal{F}$
  on~$T'S$ we denote by $\mathcal{F}_{\Lagd{n}}$ the associated bundle over~$T'S$ with fiber ${\Lagd{n}}$, the space of decorated Lagrangians.

We use the orientation of~$S$ and an auxiliary Riemannian metric to define a section of $T'S|_{\partial S}$.
We denote the image of~$\partial S$ under this section by $\vec{\partial} S \subset T'S|_{\partial S}$ and normalize the section so
that tangent vectors forming an angle of $\pi/2$ with $\vec{\partial} S$ are
pointing inward.\index{notation}{39@$\vec{\partial} S$ (\enquote{tangential} boundary)}

\begin{df}
  \label{df:framed-local-systems-on-S}
  A \emph{decoration} of a twisted symplectic local system~$\mathcal{F}$
  on~$T'S$ is a flat section~$\beta$ of the restriction
  $\mathcal{F}_{\Lagd{n}}|_{ \vec{\partial} S}$. The pair $(\mathcal{F},
  \beta)$ is called a \emph{decorated twisted symplectic local system}.
\end{df}\index{definition}{decorated! twisted symplectic local system}%
\index{definition}{decoration}%
\index{definition}{twisted!decorated --- symplectic local system}%
\index{definition}{symplectic!decorated twisted --- local system}%
\index{definition}{local system!decorated twisted symplectic ---}%
The moduli space of decorated twisted symplectic local systems is denoted
$\Locddelta(S, \Sp(2n,\R))$.
\index{notation}{40@$(\mathcal{F}, \beta)$ (a decorated twisted local system)}%
\index{notation}{41@$\Locddelta(S, \Sp(2n,\R))$ (moduli of decorated twisted local systems)}

In order to translate this notion in terms of representations, it is convenient to choose vector fields~$\vec{x}$ that have some compatibility with
$\vec{\partial} S$. The vector field~$\vec{x}$ will be said to be \emph{tangent to
the compact boundary of~$S$} if, for every~$b$ belonging to a compact component
of~$\partial S$, $\vec{x}(b)$ is tangent to~$\partial S$ and the vector
forming an angle of~$\pi/2$ with it points inward the surface. In particular,
the vector field~$\vec{x}_\mathcal{T}$ constructed in the previous paragraph
has this property.\index{definition}{tangent to
the compact boundary}

Using the elements~$c_1, \dots, c_p$
in~$\pi_1( S\setm Z( \vec{x}))$ (cf.\
Remark~\ref{rem:vector-fields-boundary-elements}) as well as arcs (in
$S\setm Z( \vec{x})$) connecting the base point to the components
of $\partial \bar{S}$ containing elements of~$R$, one gets the following

\begin{prop}
  \label{prop:framed-local-systems-rep}
  Let~$\vec{x}$ be a vector field on~$S$ with isolated zeros and no zeros on
  the compact components of~$\partial S$, which furthermore is tangent to the
  compact boundary of~$S$.  The parallel transport gives rise to a one-to-one
  correspondence between the space $\Locddelta(S, \Sp(2n,\R))$
  and the space of conjugacy classes of tuples
  \[ ( \rho, \{ \mathbf{v}_j\}_{j=1}^{p}, \{ (v_{j,
      \ell})_{j=1}^{r_\ell}\}_{\ell=1}^{k-p}) \in \Homdelta( S \setm Z(
    \vec{x}), \Sp(2n,\R))\times (\Lagd{n})^p \times
    \prod_{\ell=1}^{k-p} (\Lagd{n})^{r_\ell}\]
  such that, for all~$j=1, \dots, p$, $\rho( c_j)\cdot \mathbf{v}_j=
  \mathbf{v}_j$.
\end{prop}

\begin{rem}
  Note that when $r>0$ (i.e.\ when $R\neq \emptyset$), there are always
  nonvanishing vector fields that are tangent to the compact boundary.
\end{rem}

Under some parity conditions, there are vector fields tangent to the compact
boundary and whose indices are all even (cf.\
Remark~\ref{rem:vector-fields-index-order-delta_G}) so that the space can be
interpreted as representations of~$\pi_1(S)$:

\begin{cor}
  \label{cor:framed-local-systems-rep-even-punctured}
  Suppose that $p=k$ \ep{i.e.\ $R=\emptyset$} and that $k$ is even. Then there is
  a one-to-one correspondence between $\Locddelta(S,
  \Sp(2n,\R))$ and conjugacy classes of tuples
  \[ ( \rho, \{ \mathbf{v}_j\}_{j=1}^{k}) \in \Hom( S, \Sp(2n,\R))\times
    (\Lagd{n})^k\] such that, for all~$j=1, \dots, k$,
  $\rho( c_j)\cdot \mathbf{v}_j= \mathbf{v}_j$.
\end{cor}

Without parity condition, we can use a vector field that
has index~$0$ at the boundary components and whose all other indices are
even to get the following.
\begin{cor}
  \label{cor:framed-local-systems-rep-odd-punctured}
  Suppose that $p=k$ \ep{i.e.\ $R=\emptyset$}. Then there is
  a one-to-one correspondence between $\Locddelta(S,
  \Sp(2n,\R))$ and conjugacy classes of tuples
  \[ ( \rho, \{ \mathbf{v}_j\}_{j=1}^{k}) \in \Hom( S, \Sp(2n,\R))\times
    (\Lagd{n})^k\] such that, for all~$j=1, \dots, k$,
  $\rho( c_j)\cdot \mathbf{v}_j= - \mathbf{v}_j$.
\end{cor}

\section{Framed twisted local systems}
\label{sec:decor-twist-local}

\begin{df}
  \label{df:decor-twist-local}
  A \emph{framing} of a twisted local system~$\mathcal{F}$ is a flat
  section~$\sigma$ of the restriction of $\mathcal{F}_{\Lag{n}}$ to
  $T'S|_{\partial S}$. The pair $(\mathcal{F}, \sigma)$ is called a
  \emph{framed twisted local system}.
\end{df}\index{definition}{framed! twisted local system}%
\index{definition}{framing}%
\index{definition}{twisted!framed --- local system}%
\index{definition}{local system!framed twisted ---}%

We denote by $\Locfdelta( S, \Sp(2n, \R))$\index{notation}{43@$\Locfdelta( S, \Sp(2n, \R))$ (moduli of
  framed twisted local systems)} the moduli space of framed twisted local systems.
Since the element~$-\Id$ of $\Sp(2n, \R)$ acts trivially on~$\Lag{n}$, the
restriction of the $\Lag{n}$-local system $\mathcal{F}_{\Lag{n}}$ to a fiber
of $T'S\to S$ is the trivial bundle and the section~$\sigma$ is constant in
restriction to such a fiber of $T'S\to S$ (with base point in $\partial S$). Thus the section is entirely
determined by its restriction to $\vec{\partial} S \subset T'S|_{\partial
  S}$. Therefore we also call a flat section~$\sigma$ of the
restriction of $\mathcal{F}_{\Lag{n}}$ to $\vec{\partial} S$ a \emph{framing}. With this point
of view, using the natural projection $\Lagd{n} \to \Lag{n}$ and the induced
map $\mathcal{F}_{\Lagd{n}} \to \mathcal{F}_{\Lag{n}}$, there is a natural
map
\[ \Locddelta (S, \Sp(2n,\R)) \longrightarrow \Locfdelta( S, \Sp(2n,\R)).\]
Recall that, after choosing a nonvanishing vector field, we found an identification
$\Locfdelta( S, \Sp(2n,\R)) \simeq \Locf( S, \Sp(2n,\R))$. In a similar way, we find an identification
$\Locddelta (S, \Sp(2n,\R)) \simeq \Locd( S, \Sp(2n,\R))$.

\section{Transverse local systems}
\label{sec:transv-local-syst}

Let~$\alpha\colon ( [0,1], \{ 0,1\}) \to (S, \partial S)$ be an arc and let
$(\mathcal{F}, \sigma)$ be a framed symplectic local system on~$S$. The
restriction of~$\mathcal{F}$ to~$\alpha$ (more precisely, its pull-back
by~$\alpha$) is the trivial local system $[0,1]\times \Sp(2n,\R)$ and the
decoration gives a pair of Lagrangians $(L^t, L^b)$ ($L^t$ coming from the
fiber above~$\alpha(0)$ and~$L^b$ from~$\alpha(1)$). Only the $\Sp(2n, \R)$-orbit of
the pair $(L^t, L^b)$ is well defined. However it makes sense to say when this
pair is transverse in which case we will say that the framed local system
$( \mathcal{F}, \sigma)$ is $\alpha$-\emph{transverse}.

As the pair associated with $\bar{ \alpha}$ is $(L^b, L^t)$, a framed local
system is $\bar{\alpha}$-transverse if and only if it is $\alpha$-transverse.

\begin{df}
  \label{df:transv-local-syst-T}
  Let~$\mathcal{T}$ be an ideal triangulation of~$S$. A framed symplectic
  local system is said to be \emph{$\mathcal{T}$-transverse} if it is
  $\alpha$-transverse for every edge~$\alpha$ in~$\mathcal{T}$.
\end{df}\index{definition}{transverse!local system}%
\index{definition}{local system!transverse ---}%

We will denote by $\LocfT(S, \Sp(2n,\R))$ the space of
$\mathcal{T}$-transverse decorated symplectic local systems.
\index{notation}{44@$\LocfT(S, \Sp(2n,\R))$ (moduli of transverse framed local
  systems)}%

Let $(\mathcal{F}, \sigma)$ be a framed \emph{twisted} local system. If
$(\mathcal{F}', \sigma')$ is the corresponding framed local system (obtained via the
choice of a nonvanishing vector field), we will say that
$(\mathcal{F}, \sigma)$ is $\alpha$-\emph{transverse} or
$\mathcal{T}$-\emph{transverse} if $(\mathcal{F}', \sigma')$ is so. As
different nonvanishing vector fields $\vec{x}_1$, $\vec{x}_2$ differ by
\enquote{twists} in the fiber of $T'S\to S$ (more precisely, for any
arc~$\alpha$, one can, up to homotopy, assume that~$\vec{x}_1$ and~$\vec{x}_2$
coincide at the extremities of~$\alpha$, and the loop
$\vec{x}_1({\alpha}) \sqcup \vec{x}_2({\bar{\alpha}})$ is homotopic to a power
of~$\delta$), and again as~$-\Id$ acts trivially on~$\Lag{n}$, this condition
does not depend on the choice of the nonvanishing vector field. Hence the
space $\LocfdeltaT(S, \Sp(2n,\R))$ of \index{notation}{46@$\LocfdeltaT(S, \Sp(2n,\R))$
  (moduli of framed twisted transverse local systems)}
$\mathcal{T}$-transverse framed twisted symplectic local system is well
defined and isomorphic to $\LocfT(S, \Sp(2n,\R))$.

\smallskip

Similarly,
a decorated $\delta$-twisted symplectic local system $(
\mathcal{F}, \beta)$ is said to be $\alpha$\emph{-transverse} or
$\mathcal{T}$\emph{-transverse} if the associated framed local system is
so. Their moduli space is denoted $\LocddeltaT( S, \Sp(2n, \R))$.
\index{notation}{48@$\LocddeltaT(S, \Sp(2n,\R))$
  (moduli of decorated twisted transverse local systems)}
\section{Configurations associated with framed local systems}
\label{sec:conf-assoc-with-1}

For $\ell\geq 2$, an \emph{$\ell$-gon} is (the homotopy class of) a map $(
\mathbb{D}, \mu_\ell) \to (S, \partial S)$ where $\mathbb{D}$~is the closed unit
disk in~$\C$, $\mathbb{D}=\{ z\in \C \mid \abs{z}\leq 1\}$ and~$\mu_\ell$ is the
set of $\ell$-roots of unity, $\mu_\ell = \{ z\in \C \mid z^\ell =1\}$. A $3$-gon will
be sometime called a \emph{triangle} (although this may cause confusion with
the triangles of a triangulation) and a $4$-gon will be called a
\emph{quadrilateral}. A $2$-gon is the same thing (up to homotopy) than an arc in~$S$.\index{definition}{triangle}%
\index{definition}{quadrilateral}%
\index{definition}{lgon@$\ell$-gon}%

Restricting an $\ell$-gon to segments contained in $\partial \mathbb{D}$ and
whose endpoints are successive elements of~$\mu_\ell$ defines a family of arcs
associated with the $\ell$-gon. An arc will be said to \emph{belong to the
  $\ell$-gon} if it is (homotopic to an arc) obtained in this way.

The pull back of a framed local system by an $\ell$-gon is the trivial local
system on~$\mathbb{D}$ together with the data of a Lagrangian for each element
in~$\mu_\ell$, i.e.\ it gives a well defined element in $\Conf^\ell( \Lag{n})$. We
have thus, associated with any $\ell$-gon~$\tau$, a map
\[ f_\tau\colon \Locf( S, \Sp(2n, \R)) \longrightarrow \Conf^\ell(\Lag{n} ).\]
For $\ell=2$, the maps $f_\alpha$ for $\alpha$ an edge of a triangulation are precisely the maps we
used above to define transverse local systems.

In a similar way, an $\ell$-gon provides a map from the space of decorated symplectic local systems to the configuration space of decorated Lagrangians.

Let~$\mathcal{T}$ be a triangulation of~$S$. Any face~$T$ (i.e.\ triangle)
of~$\mathcal{T}$ gives rise to three $3$-gons, $\tau_1$, $\tau_2$, $\tau_3$ obtained one from the other via
precomposition by the rotation of angle $2\pi/3$. In other words, for every
framed symplectic local system $(\mathcal{F}, \sigma)$, the three
configurations of triples of Lagrangians $f_{\tau_i}(\mathcal{F}, \sigma)$ ($i=1,2,3$)
are obtained one from the other by cyclic permutation.
\begin{df}
  \label{df:maslov-index-triang}
  The common value of the Maslov index
  $\mu_n ( f_{\tau_i}(\mathcal{F}, \sigma))$ depends only on~$T$ and
  $(\mathcal{F}, \sigma)$ (cf.\ Proposition~\ref{maslov_prop}) and will be
  called \emph{Maslov index of the triangle}~$T$ for~$(\mathcal{F}, \sigma)$
  and denoted by $\mu^T(\mathcal{F}, \sigma)$.
\end{df}\index{notation}{50@$\mu^T$ (the Maslov index of the triangle~$T$)}%
\index{definition}{Maslov!index of a triangle}%
\index{definition}{index!Maslov --- of a triangle}%
\index{definition}{triangle!Maslov index of a ---}%

In a similar vein, any internal (oriented) edge~$\alpha$ of~$\mathcal{T}$ is the
diagonal of a quadrilateral in~$\mathcal{T}$ and gives rise to a $4$-gon~$\tau$. In this
situation the map~$f_\tau$ will be denoted by~$q_\alpha$:
\[ q_\alpha \colon\Locf( S, \Sp(2n, \R)) \longrightarrow \Conf^4(\Lag{n} ).\]

As the $4$-gon associated with~$\bar{\alpha}$ is obtained from~$\tau$ by the
precomposition with the rotation of angle~$\pi$, the map~$q_{\bar{\alpha}}$ is
equal to $\kappa\circ q_\alpha$ ($\kappa$ is the automorphism of $\Conf^4(
\Lag{n})$ induced by the permutation~$(13)(24)$, see Section~\ref{sec:conf-lagr}).
\index{notation}{52@$q_\alpha$ (map to $\Conf^4(
\Lag{n})$ associated with an edge~$\alpha$)}

\begin{figure}[ht]
\begin{center}
\begin{tikzpicture}
  \coordinate (L1) at (0,-1) ;
  \coordinate (L2) at (-1,2) ;
  \coordinate (L3) at (1,2) ;
  \coordinate (L4) at (-2,0) ;
  \coordinate (L5) at (2,0) ;
  \draw (L1) node[below]{$L_1$} ;
  \draw (L2) node[above]{$L_2$} ;
  \draw (L3) node[above]{$L_3$} ;
  \draw (L4) node[left]{$L_4$} ;
  \draw (L5) node[right]{$L_5$} ;
  \draw (barycentric cs:L1=7,L2=1,L3=1) node{$\theta$} ;
  \draw (L1)--(L2) node[midway,left]{$\alpha$} ;
  \draw (L1)--(L3) node[midway,right]{$\alpha'$};
  \draw (L1)--(L4) ;
  \draw (L2)--(L3) ;
  \draw (L2)--(L4) ;
  \draw (L3)--(L5) ;
  \draw (L1)--(L5) ;

\end{tikzpicture}
\caption{Configuration of 5 Lagrangians corresponding to an angle $\theta$}
\end{center}
\end{figure}

An \emph{angle}~$\theta$ of~$\mathcal{T}$ is a pair of edges~$\{ \alpha, \alpha'\}$ of~$\mathcal{T}$
contained in the same face and having the same endpoint. It is called
\emph{internal} if both edges~$\alpha$, $\alpha'$ are internal. If the
angle~$\theta$ is internal, it gives rise to a well defined $5$-gon~$\tau$
(composed with the three faces of~$\mathcal{T}$ containing $\{ \alpha,
\alpha'\}$) and hence a map
\[ c_\theta\coloneqq f_\tau \colon\Locf( S, \Sp(2n, \R)) \longrightarrow \Conf^5(\Lag{n} ).\]
\index{notation}{54@$c_\theta$ (the map to $\Conf^5( \Lag{n})$ associated with an angle~$\theta$)}
\index{definition}{internal!angle}
\index{definition}{angle!internal ---}

\section{Toledo number and maximal representations}

In this section, we assume that~$r=0$ (i.e.\ $R=\emptyset$), so that $p=k$.

An important invariant for a symplectic local system~$\mathcal{F}$ on~$S$ (or
for the associated holonomy representation
$\rho\colon \pi_1(S)\to \Sp(2n,\R)$) is the Toledo number, here denoted
\index{notation}{56@$T_\mathcal{F}$, $T_\rho$ (Toledo number)}%
by~$T_\mathcal{F}$ (or~$T_\rho$), which was defined in~\cite{BIW} using
bounded cohomology. Note that the Toledo number depends on the topological
surface~$S$ and not only on its fundamental group. It is a real number which
satisfies the Milnor--Wood inequality:\index{definition}{Toledo number}
\[ -n\abs{\chi(S)} \leq T_\mathcal{F} \leq n\abs{\chi(S)}.\]

The representations where this invariant
achieves its maximum have particularly nice geometric properties,
see~\cite{BIW}.

\begin{df}
  If $R=\emptyset$, a symplectic local system $\mathcal{F}$ is called
  \emph{maximal} if \[T_\mathcal{F}=n\chi(S)=-n \abs{\chi(S)}.\]
\end{df}\index{definition}{maximal!symplectic local system}%
\index{definition}{symplectic!maximal --- local system}%
\index{definition}{local system!maximal symplectic ---}%
\begin{rem}
  The choice of the sign in the definition is not really relevant
  (pulling back with an orientation reversing diffeomorphism
  changes the Toledo number to its opposite). The chosen sign here makes
  Corollary~\ref{cor:max-iff-triangle-are-max} below look more natural;
  equally uniformizations are maximal for this choice.
\end{rem}

We denote by  $\M(S,\Sp(2n,\R))$ the subspace of $\Loc(S,\Sp(2n,\R))$
consisting of maximal local systems.  In a similar fashion, we denote by
$\Mf(S,\Sp(2n,\R))$ the subspace of $\Locf(S,\Sp(2n,\R))$ of
framed maximal local systems.  The following facts are proven
in~\cite{BIW}.\index{notation}{55@$\M(S,\Sp(2n,\R))$ (moduli of maximal local systems)}%
\index{notation}{57@$\Mf(S,\Sp(2n,\R))$ (moduli of maximal framed local systems)}%
\begin{prop}
  \label{maxrep}
  Suppose that~$R$ is empty.
  \begin{enumerate}[(a)]
  \item
    The natural map
    \begin{equation*}
      \Mf(S,\Sp(2n,\R)) \longrightarrow  \M(S,\Sp(2n,\R))
    \end{equation*}
    is surjective. In other words, every maximal local system admits a
    framing.
  \item\label{item-b-prop-maxrep} Framed maximal local systems are
    transverse with respect to any ideal triangulation~$\mathcal{T}$:
    \[\mathcal
    \Mf(S,\Sp(2n,\R)) \subset \LocfT( S,
    \Sp(2n,\R)).\]
  \item Maximal representations are reductive; hence the two spaces
    $\M(S,\Sp(2n,\R))$ and $\Mf(S,\Sp(2n,\R))$
    are Hausdorff \ep{cf.\ Remark~\ref{rem:Hausdorff}}.
\end{enumerate}
\end{prop}

\begin{rem}
The argument given in~\cite{BIW} implies that not only maximal representations are reductive, but also representations that are
  \emph{almost maximal}, i.e.\ where $T_\rho<-(n-1)\abs{\chi(S)}$.
  \end{rem}\index{definition}{almost maximal}%
\index{definition}{maximal!almost ---}%

The next result shows that the Toledo number of a framed local system can
be computed easily using an ideal triangulation. In the special case of a pair
of pants this has been proven
in~\cite{Strubel}.

\begin{teo} \label{teo:toledo_maslov} Suppose
that~$R$ is empty. Let~$\mathcal{T}$ be an ideal
  triangulation of~$S$ and
  $(\mathcal{F}, \sigma)\in \Locf(S,\Sp(2n,\R))$. The Toledo number
  $T_\mathcal{F}$
  can be computed via the following formula:
  \[T_\mathcal{F}=-\frac{1}{2} \sum_{T\in\mathcal{T}} \mu^T(\mathcal{F}, \sigma).\]
\end{teo}

This proposition implies the Milnor--Wood inequality and the
integrality
property of the Toledo invariant for framed representations.
Another consequence is that framed maximal representations can be detected
using a triangulation:

\begin{cor}\label{cor:max-iff-triangle-are-max}
Let $\mathcal{F}$ be a local system admitting a framing, and let $\mathcal{T}$ be an ideal
  triangulation of~$S$. For any framing~$\sigma$
  of~$\mathcal{F}$, we have that $\mathcal{F}$~is maximal if and only if, for
  every triangle~$T$ in~$\mathcal{T}$, the Maslov index
  $\mu^T(\mathcal{F}, \sigma)$ is~$n$.
\end{cor}

Note also that this corollary implies point~\eqref{item-b-prop-maxrep} of Proposition~\ref{maxrep}.

The proof of Theorem~\ref{teo:toledo_maslov} will take the rest of this
section. It will use the Souriau index (see
Section~\ref{sec:souriau-index}) and the translation number (see Section~\ref{sec:rotation-number}).

\begin{lem}\label{lem_sum_mu_T_indep}
  Let $(\mathcal{F}, \sigma)$ be a framed local system. Then the sum of the triangle invariants
  $\sum_{T\in\mathcal{T}} \mu^T(\mathcal{F}, \sigma)$ does not depend on the
  triangulation~$\mathcal{T}$.
\end{lem}
\begin{rem}
  Of course, a consequence of Theorem~\ref{teo:toledo_maslov} is that
  this sum does not depend on~$\sigma$ either. This can be proved a priori thanks
  to Lemma~\ref{lem:souriau-index-sum-triangle-v2}.
\end{rem}

\begin{proof}
  The result of the lemma follows from
  \begin{enumerate}[(1)]
  \item Any two ideal triangulations are related by a sequence of flips.
  \item The invariance of the sum under a flip is equivalent to the cocycle
    condition satisfied by the Maslov index (see
    Proposition~\ref{maslov_prop}).\qedhere
  \end{enumerate}
\end{proof}

The following result will be our starting point in the proof of Theorem~\ref{teo:toledo_maslov}.

\begin{lem}[{\cite[Thm.~12]{BIW}}] \label{lem:ToledoRotation} Let
  $\rho\in \Hom(\pi_1(S),\Sp(2n,\R))$ and let $a_1, \dots, a_g$,
  $b_1, \dots, b_g$, $c_1, \dots , c_k$ be based loops in~$S$ such that the
  complement $S \smallsetminus \bigl( a_1 \cup \cdots \cup c_k\bigr)$ is the
  disjoint union of~$k$ punctured disks bounded by $c_{1}, \dots, c_{k}$ and a
  disk bounded by $c_{1}^{-1} \cdots c_{k}^{-1}
  [b_g,a_g]\cdots[b_1,a_1]$.

  Since $\pi_1(S)$ is a free group, the representation~$\rho$ can be lifted
  to~$\widetilde{\Sp}(2n, \R)$, and let~$\tilde \rho$ be such a lift. The
  Toledo number of $\rho$ can be computed as:
  \[T_\rho=-\sum_{i=1}^k \widetilde{\Rot}(\tilde\rho(c_i)).\]
\end{lem}

\begin{proof}[Proof of Theorem~\ref{teo:toledo_maslov}]
Via the
choice of a finite volume, complete, hyperbolic structure on the
interior~$S\setm \partial S$ of~$S$, we can identify the interior
of~$\widetilde{S}$ with the hyperbolic plane~$\mathbb{H}^2$.
The fundamental group~$\pi_1(S)$ acts on the right. The boundary at infinity
of the interior of~$\widetilde{S}$, which we denote by $\partial_\infty \mathbb{H}^2$,
 is also endowed with a right $\pi_1(S)$-action. It contains
points whose stabilizers are generated by one element that is a conjugate of
one of the  boundary representatives
$c_1, \dots, c_k$ (cf.\ Section~\ref{sec:fundamental-group}, recall that we
are assuming here that $p=k$).
Note that, if the stabilizer of~$p$ is generated by~$\gamma$, then, for
every~$g\in \pi_1(S)$, $g^{-1} \gamma g$ generates the stabilizer of~\mbox{$p\cdot
g$.}

\begin{figure}[ht]
  \begin{tikzpicture}[scale=1]
    \newcommand{\rone}{3};
    \newcommand{\rtwo}{3.3};
    \newcommand{\rtwoprime}{3.15};
    \newcommand{\rthree}{3.1};
    \newcommand{\rfour}{2.6};
    \coordinate (p1) at (180:\rone);
    \coordinate (P1) at (180:\rtwo);
    \coordinate (S1) at (168:\rthree);
    \coordinate (s1) at (156:\rone);
    \coordinate (R1) at (144:\rthree);
    \coordinate (r1) at (132:\rone);
    \coordinate (U1) at (120:\rthree);
    \coordinate (v1) at (108:\rone);
    \coordinate (V1) at (96:\rthree);
    \coordinate (q1) at (84:\rone);
    \coordinate (Q1) at (84:\rtwo);
    \coordinate (a1start) at (168:\rfour);
    \coordinate (a1end) at (120:\rfour);
    \coordinate (b1start) at (96:\rfour);
    \coordinate (b1end) at (144:\rfour);
    \draw (p1) to[bend right] (s1);
    \draw (s1) to[bend right] (r1);
    \draw (r1) to[bend right] (v1);
    \draw (v1) to[bend right] (q1);
    \draw (P1) node {$p_1$};
    \draw (Q1) node {$q_1$};
    \draw (S1) node {$S_1$};
    \draw (R1) node {$R_1$};
    \draw (U1) node {$U_1$};
    \draw (V1) node {$V_1$};
    \draw[->] (a1start) to[bend right] node[midway,
    below, sloped] {$a_1$} (a1end) ;
    \draw[->] (b1start) to[bend left] node[midway,
    below, sloped] {$b_1$} (b1end) ;

    \coordinate (pg) at (60:\rone);
    \coordinate (Pg) at (60:\rtwo);
    \coordinate (Sg) at (48:\rthree);
    \coordinate (sg) at (36:\rone);
    \coordinate (Rg) at (24:\rthree);
    \coordinate (rg) at (12:\rone);
    \coordinate (Ug) at (0:\rthree);
    \coordinate (vg) at (-12:\rone);
    \coordinate (Vg) at (-24:\rthree);
    \coordinate (qg) at (-36:\rone);
    \coordinate (Qg) at (-36:\rtwo);
    \coordinate (agstart) at (48:\rfour);
    \coordinate (agend) at (0:\rfour);
    \coordinate (bgstart) at (-24:\rfour);
    \coordinate (bgend) at (24:\rfour);
    \draw (pg) to[bend right] (sg);
    \draw (sg) to[bend right] (rg);
    \draw (rg) to[bend right] (vg);
    \draw (vg) to[bend right] (qg);
    \draw (Pg) node[right] {$\!\!\!q_{g-1}$};
    \draw (Qg) node {$q_g$};
    \draw (Sg) node {$S_g$};
    \draw (Rg) node {$R_g$};
    \draw (Ug) node {$U_g$};
    \draw (Vg) node {$V_g$};
    \draw[->] (agstart) to[bend right] node[midway,
    below, sloped] {$a_g$} (agend) ;
    \draw[->] (bgstart) to[bend left] node[midway,
    above, sloped] {$b_g$} (bgend) ;

    \draw[loosely dotted] (q1) to (pg);

    \coordinate (p2) at (204:\rone);
    \coordinate (pkm) at (228:\rone);
    \coordinate (pk) at (252:\rone);
    \coordinate (qkm) at (276:\rone);
    \coordinate (q2) at (300:\rone);
    \coordinate (P2) at (204:\rtwo);
    \coordinate (Pkm) at (228:\rtwo);
    \coordinate (Pk) at (252:\rtwo);
    \coordinate (Qkm) at (276:\rtwo);
    \coordinate (Q2) at (300:\rtwoprime);

    \coordinate (Q1) at (192:\rthree);
    \coordinate (Qsegkm) at (240:\rthree);
    \coordinate (Psegkm) at (264:\rthree);
    \coordinate (P1) at (312:\rthree);

    \coordinate (ck2end) at (192:\rfour);
    \coordinate (ckend) at (240:\rfour);
    \coordinate (ckstart) at (263:\rfour);
    \coordinate (ck2start) at (312:\rfour);

    \draw[->] (ckstart) to[bend right] node[midway,
    above, sloped] {$c_k$} (ckend) ;
    \draw[->] (ck2start) to[bend right] node[midway,
    below, sloped] {$c_k\cdots c_2$} (ck2end) ;

    \draw (P2) node {$p_2$};
    \draw (Q2) node[right] {$\!\! p_2 \cdot c_{2}^{-1} \cdots c_{k}^{-1} $};
    \draw (Pkm) node[left] {$p_{k-1}\!\!\!$};
    \draw (Pk) node {$p_{k}$};
    \draw (Qkm) node[right] {$\!\!\!\!p_{k-1} \cdot c_{k}^{-1}$};

    \draw (Q1) node {$P_1$};
    \draw (P1) node {$Q_1$};
    \draw (Qsegkm) node {$P_{k-1}$};
    \draw (Psegkm) node {$Q_{k-1}$};

    \draw (p2) to[bend right] (p1) ;
    \draw (qg) to[bend right] (q2) ;
    \draw (qkm) to[bend right] (pk) ;
    \draw (pk) to[bend right] (pkm) ;

    \draw[loosely dotted] (p2) to (pkm);
    \draw[loosely dotted] (qkm) to (q2);

  \end{tikzpicture}
  \caption{The fundamental ideal polygon and its sides identifications} \label{fig:polygonsurface}
\end{figure}

A fundamental domain for the right action of~$\pi_1(S)$ on~$\mathbb{H}^2$ is
an ideal $(4g+2k-2)$-gon (see Figure~\ref{fig:polygonsurface}), which we now
describe in some details. The sides of this polygon are geodesic rays denoted
(in counter-clockwise order and with their \enquote{counter-clockwise}
orientation):
\begin{quote}
  $P_1, \dots, P_{k-1}, Q_{k-1}, \dots, Q_1, V_g, U_g, R_g, S_g, \dots, V_1, U_1, R_1, S_1,$
\end{quote}
and are subject to the following identifications (where a bar means the
opposite orientation):
\begin{quote}
  For all $i=1,\dots, g$, $S_i\cdot a_i = \overline{U}_{i}$, $V_i \cdot b_i =
  \overline{R}_{i}$ and for all $j=1,\dots, k-1$, $Q_j \cdot (c_k \cdots c_{j+1}) =
  \overline{P}_{j}$.
\end{quote}

It means that, on the interior of the surface~$S$, the rays~$P_j$ and $Q_j$ have the same
image, as well as the rays~$S_i$ and~$U_i$, and the rays~$V_i$ and~$R_i$. The
images of the rays~$P_1, \dots, P_{k-1}$ connect the boundary components and the
complement in the interior of~$S$ of the images of all the rays is a topological
disk.

Let $p_1, \dots, p_k$ be the elements of the boundary $\partial_\infty \mathbb{H}^2$ that
are the extremities of the sides~$P_1, \dots, P_{k-1}$: for all~$j=1, \dots,
k-1$, $P_j$ goes form~$p_j$ to~$p_{j+1}$. The stabilizer of~$p_j$
in~$\pi_1(S)$ is the subgroup generated by~$c_j$.

The ideal extremities of the fundamental polygon are then, in the
counter-clockwise order (cf.\ Figure~\ref{fig:polygonsurface}):
\begin{center}
 \parbox{0.8\textwidth}{
   \noindent
   $p_1, \dots, p_k, p_{k-1}\cdot c_{k}^{-1}, \dots, p_2\cdot (c_k \cdots
   c_3)^{-1}, q_g \coloneqq p_1\cdot (c_k \cdots c_2)^{-1} , q_g \cdot b_g a_g
   b_{g}^{-1}, q_g \cdot b_g a_g, {q_g \cdot b_g}$,
   $q_{g-1} \coloneqq q_g \cdot b_g a_g b_{g}^{-1} a_{g}^{-1}, \dots, q_{1}
   \coloneqq q_2 \cdot b_2 a_2 b_{2}^{-1} a_{2}^{-1}, q_1 \cdot b_1 a_1
   b_{1}^{-1}, q_1 \cdot b_1 a_1, q_1 \cdot b_1$.
 }
 \end{center}
The equality $p_1= q_1 \cdot b_1 a_1 b_{1}^{-1} a_{1}^{-1}$ follows from the
relation satisfied by $a_1, \dots, a_g$, $b_1, \dots, b_g$, $c_1, \dots, c_k$.

Generators of the corresponding stabilizers are respectively (using the
notation $x^y\coloneqq y^{-1} x y$ and $[x,y] = x y x^{-1} x^{-1}$):
\begin{center}
  \parbox{0.9\textwidth}{
  $c_1, \dots, c_k$, $c_k c_{k-1} c_{k}^{-1}= c_{k-1}^{c_{k}^{-1}}, \dots, c_{2}^{(c_k \cdots
    c_2)^{-1}}$, $c_{1}^{(c_k \cdots c_1)^{-1}} = c_{1}^{ [a_1, b_1]\cdots [a_g,
    b_g]}$, \linebreak $c_{1}^{ [a_1, b_1]\cdots [a_g, b_g] a_g}$, $c_{1}^{ [a_1,
    b_1]\cdots
    [a_g, b_g] a_g b_g}$, $c_{1}^{ [a_1, b_1]\cdots [a_g, b_g] a_g b_g
    a_{g}^{-1}}$, $c_{1}^{ [a_1, b_1]\cdots [a_{g-1}, b_{g-1}]}, \linebreak \dots, c_{1}^{
    [a_1, b_1]}, c_{1}^{a_1}, c_{1}^{a_1 b_1}, c_{1}^{a_1 b_1 a_{1}^{-1}}$.
  }
\end{center}

Let now $(\mathcal{F}, \sigma)$ be a framed local system and call~$(\rho,
\{ L_1, \dots, L_k\})$
its holonomy (cf.\ Lemma~\ref{lem:decor-sympl-local-holonomy}). For all~$j=1, \dots, k$,
$\rho(c_j)$ fixes the Lagrangian~$L_j$. For $j=1, \dots, k$, let $z_j\in
\wideLag{n}$ be a lift of~$L_j$. Since $\pi_1(S)$ is freely
generated by $\{ a_i, b_i\}_{i=1, \dots, g} \cup \{ c_j\}_{j=2, \dots, k}$,
one can choose a lift $\tilde{\rho} \colon \pi_1(S) \to \widetilde{\Sp}(2n, \R)$
such that, for all $j=2, \dots, k$, $\tilde{\rho}(c_j)$ fixes~$z_j$. Thus, for
all~$j=2, \dots, k$, $\widetilde{\Rot}(\tilde\rho(c_j))=0$ and
$\widetilde{\Rot}(\tilde\rho(c_1))= 1/2\, m_n( \tilde\rho(c_1)\cdot z_1, z_1
)$ (see Lemma~\ref{lem:rotation-number}). Therefore, by
Lemma~\ref{lem:ToledoRotation}, \[2T_\rho = -m_n( \tilde\rho(c_1)\cdot z_1, z_1).\]

For all~$i=1,\dots, g$, define
\[M_i= \rho( [b_{i}, a_{i}] \cdots [b_{1}, a_{1}]) \cdot L_1 \text{ and } y_i=
  \tilde{\rho}( [b_{i}, a_{i}] \cdots [b_{1}, a_{1}]) \cdot z_1.\]
The Lagrangians
associated ($\rho$-equivariantly) to the ideal vertices of the fundamental
polygon are thus (cf.\ Remark~\ref{rem:fram-sympl-local-universal}):
\begin{quote}
  $L_1, \dots, L_k$,
  $
  \rho(c_k) \cdot L_{k-1}, \dots, \rho(c_k
  \cdots c_1) L_1=M_g$, $\rho( b_g a_{g}^{-1} b_{g}^{-1}) \cdot M_g$,
  $\rho( a_{g}^{-1} b_{g}^{-1}) \cdot M_g$, $\rho( b_{g}^{-1}) \cdot M_g$,
  $M_{g-1}, \dots, M_1$, $\rho( b_1 a_{1}^{-1} b_{1}^{-1}) \cdot M_1$,
  $\rho( a_{1}^{-1} b_{1}^{-1}) \cdot M_1$, $\rho( b_{1}^{-1}) \cdot M_1$.
\end{quote}
Lifts of those to $\wideLag{n}$ are hence:
\begin{quote}
  $z_1, \dots, z_k=\tilde{\rho}(c_k)\cdot z_k$, $\tilde{\rho}( c_k c_{k-1})\cdot z_{k-1}, \dots, \tilde{\rho}(
      c_k\cdots c_{1})\cdot z_{1} = y_g$,
  $\tilde{\rho}( b_g a_{g}^{-1} b_{g}^{-1}) \cdot y_g$,
  $\tilde{\rho}( a_{g}^{-1} b_{g}^{-1}) \cdot y_g$,
  $\tilde{\rho}( b_{g}^{-1}) \cdot y_g$, $y_{g-1}, \dots, y_1$,
  $\tilde{\rho}( b_1 a_{1}^{-1} b_{1}^{-1}) \cdot y_1$,
  $\tilde{\rho}( a_{1}^{-1} b_{1}^{-1}) \cdot y_1$,
  $\tilde{\rho}( b_{1}^{-1}) \cdot y_1$.
\end{quote}
Let now~$\mathcal{T}$ be the ideal triangulation of~$S$ induced by the ideal
triangulation of the fundamental polygon obtained by adding edges
between~$p_1$ and every other vertex.

Thanks to Lemma~\ref{lem:souriau-index-sum-triangle}, one has
\[    \sum_{T\in\mathcal{T}}\mu^T(\mathcal{F}, \sigma) =
     \sum_{j=1}^{k-1} m_n(z_{j},z_{j+1})+ \sum_{j=k-1}^{1} m_n\bigl( \tilde{\rho}(
      c_k\cdots c_{j+1})\cdot z_{j+1},\tilde{\rho}(
      c_k\cdots c_{j})\cdot z_{j}\bigr)\]
\[    +\sum_{i=g}^{1}
      \Bigl( m_n\bigl( y_{i},\tilde{\rho} (b_g a_{g}^{-1} b_{g}^{-1})\cdot y_{i}\bigr)
      +  m_n\bigl( \tilde{\rho} (b_g a_{g}^{-1} b_{g}^{-1}) \cdot
       y_{i},\tilde{\rho} (a_{g}^{-1} b_{g}^{-1}) \cdot y_{i}\bigr)\Bigr.\]
\[       \qquad \Bigl.  +  m_n\bigl( \tilde{\rho} (a_{g}^{-1} b_{g}^{-1})
        \cdot y_{i},\tilde{\rho} ( b_{g}^{-1}) \cdot y_{i}\bigr) +  m_n\bigl( \tilde{\rho} (b_{g}^{-1})
       \cdot y_{i},\tilde{\rho} ( a_{g}^{-1} b_{g}^{-1} a_{g}^{-1} b_{g}^{-1})
        \cdot y_{i}\bigr) \Bigr).\]

  By the equivariance and the antisymmetry of the Souriau index, for all~$i=1, \dots, g$, one has
  \[m_n\bigl( \tilde{\rho} (a_{g}^{-1} b_{g}^{-1})\cdot y_{i},\tilde{\rho} ( b_{g}^{-1})
 \cdot y_{i}\bigr) = m_n\bigl( \tilde{\rho} ( b_{g} a_{g}^{-1} b_{g}^{-1})
 \cdot y_{i}, y_{i}\bigr) =
  -m_n\bigl( y_{i},\tilde{\rho} (b_g a_{g}^{-1} b_{g}^{-1}) \cdot y_{i}\bigr)\] and
\begin{align*}
   m_n\bigl( \tilde{\rho} (b_{g}^{-1}) \cdot y_{i},\tilde{\rho} ( a_{g}^{-1} b_{g}^{-1}
  a_{g}^{-1} b_{g}^{-1}) \cdot y_{i}\bigr) &=  m_n\bigl( \tilde{\rho} ( a_{g}
                                             b_{g}^{-1}) \cdot y_{i},\tilde{\rho} ( b_{g}^{-1}
  a_{g}^{-1} b_{g}^{-1}) \cdot y_{i}\bigr)\\ & =  -m_n\bigl( \tilde{\rho} ( b_{g}^{-1}
  a_{g}^{-1} b_{g}^{-1}) \cdot y_{i}, \tilde{\rho} ( a_{g} b_{g}^{-1}) \cdot y_{i}\bigr).
\end{align*}
Thus the last sum cancels out. For similar reasons
\[\sum_{j=2}^{k-1} \bigl( m_n(z_{j},z_{j+1})+  m_n( \tilde{\rho}(
      c_k\cdots c_{j+1})\cdot z_{j+1},\tilde{\rho}(
      c_k\cdots c_{j})\cdot z_{j})\bigr)=0;\]it remains
      \begin{align*}
    \sum_{T\in\mathcal{T}}\mu^T(\rho, D) =
        & m_n( z_1, z_2) + m_n( z_2, \tilde{\rho}(c_1) \cdot z_1)\\
        \intertext{applying
        Proposition~\ref{prop:souriau-index-maslov-index-extended}:}
        =& m_n( \tilde{\rho}(c_1)\cdot z_1, z_1) + \mu_n( D(c_1), D(c_2),
           D(c_1))\\
        =& -2 T_\rho  +0. \qedhere
      \end{align*}
\end{proof}

\section{Maximal framed local systems}
\label{sec:maxim-decor-local}

We return now to the general situation (i.e.\ $R$ can be empty or not).

Based on Corollary~\ref{cor:max-iff-triangle-are-max}, we introduce the following definition.
\begin{df}
  \label{df:maxim-decor-local-general}
  A framed local system $(\mathcal{F}, \sigma)$ is said to be  \emph{maximal} if
  there exists a triangulation~$\mathcal{T}$ of~$S$ for which $\mu^T(
  \mathcal{F}, \sigma)=n$ for every triangle~$T$ of~$\mathcal{T}$.
\end{df}\index{definition}{maximal!framed local system}%
\index{definition}{framed!maximal --- local system}%
\index{definition}{local system!maximal framed ---}%
We denote by $\Mf( S, \Sp(2n, \R))$
\index{notation}{57@$\Mf(S,\Sp(2n,\R))$ (moduli of maximal framed local systems)}%
the space of maximal framed symplectic local systems. Reasoning similarly to the proof of Lemma~\ref{lem_sum_mu_T_indep} and with
the help of finite coverings of~$S$, one has
\begin{lem}
  \label{lem:maxim-decor-local-every-triangulation}
For a framed symplectic local system, the following are equivalent:
  \begin{itemize}
  \item it is maximal;
  \item for every triangulation~$\mathcal{T}$ and every
    triangle~$T$ of~$\mathcal{T}$, $\mu^T(\mathcal{F}, \sigma)=n$;
  \item for every $3$-gon~$\tau$, the Maslov index of $f_\tau(
    \mathcal{F}, \sigma)$ is maximal.
  \end{itemize}
\end{lem}

Since every arc belongs 
to at least one $3$-gon, one gets:
\begin{cor}
  \label{cor:maxim-decor-local-transverse}
  A maximal framed local system is $\alpha$-transverse for every arc~$\alpha$.
\end{cor}

\begin{rem}
  When $r>0$, the notion of maximality really depends on the pair
  $(\mathcal{F}, \sigma)$ and not only on~$\mathcal{F}$; this is flagrant
  when~$S$ is a disk. Summing the Maslov indices for all the triangles
  of~$\mathcal{T}$, one could define a notion of Toledo number in this wider
  setting; the interest of this notion seems nevertheless limited.
\end{rem}

\begin{rem}
  \label{rem:maxim-decor-local-disk}
  When~$\bar{S}$ is a disk with~$r\geq 3$ points in its boundary, the
  space~$\Locf(S, \Sp(2n,\R))$ is isomorphic to the space $\Conf^r(
  \Lag{n})$. The space of maximal framed
  local systems $\Mf(S, \Sp(2n,\R))$ is sometimes denoted $\Conf^{r+}( \Lag{n})$ and its elements
  are called \emph{positive} configurations of Lagrangians. When~$r$ is~$3$,
  $\Conf^{3+}( \Lag{n})$ contains one element: the orbit of triples with
  Maslov index equal to~$n$. When~$r$ is~$4$, $\Conf^{4+}( \Lag{n})$ is the
  space of positive quadruples introduced above (cf.\
  Definition~\ref{df:positive-quad}).
\end{rem}

\section{Maximal framed twisted local systems}
\label{sec:maxim-fram-twist}

We have an identification
$\Locf(S, \Sp(2n,\R)) \simeq \Locfdelta(S, \Sp(2n,\R))$ that depends on the choice of a nonvanishing vector field~$\vec{x}$. By the same
arguments as in Section~\ref{sec:transv-local-syst}, the image of $\Mf(S,
\Sp(2n,\R))$ by this isomorphism does not depend on~$\vec{x}$. It will be
denoted by $\Mfdelta(S, \Sp(2n,\R))$ and its elements are call \emph{maximal
  framed twisted} local systems.
\index{notation}{59@$\Mfdelta(S, \Sp(2n,\R))$ (moduli of maximal
  framed twisted local systems)}

\chapter{Local systems on quivers and their framings}
\label{sec:local-systems-their}
In order to control and parametrize the spaces introduced in the previous chapter we relate them to local
systems on graphs embedded in the surface~$S$. For every ideal triangulation we define a quiver (i.e.\ an oriented graph), which will be used systematically in the rest of the paper.  We recall the
correspondence between local systems on~$S$ and local systems on this quiver. We introduce twisted local
systems and framed local systems on this quiver which then provide the counterparts of the moduli spaces introduced in the previous chapter.

\section{The quiver $\Gamma_\mathcal{T}$}
\label{sec:orient-graph-gamm}

Let~$\mathcal{T}$ be an ideal triangulation of~$S$ (cf.\ Section~\ref{sec:triangulations-general}).

We consider the following quiver (aka.\ oriented graph whose oriented edges
are called arrows)
$\Gamma_\mathcal{T}$ embedded in the
surface~$S$ (see Figure~\ref{fig:quadrilateral_intro}).

\begin{figure}[ht]
\centering
\begin{tikzpicture}
  \coordinate (A) at (5,0) ;
  \draw (A) node[below]{$b(v)=t(v')$} ;
  \coordinate (B) at (5.5,5) ;
  \draw (B) node[above]{$t(v)=b(v')$} ;
  \coordinate (C) at (9,2.6) ;
  \coordinate (D) at (1.4,2.9) ;
  \coordinate (a1) at (barycentric cs:A=3,B=3,C=1) ;
  \coordinate (a2) at (barycentric cs:A=1,B=3,C=3) ;
  \coordinate (a3) at (barycentric cs:A=3,B=1,C=3) ;
  \coordinate (c2) at (barycentric cs:A=-1,B=3,C=3) ;
  \coordinate (c3) at (barycentric cs:A=3,B=-1,C=3) ;
  \coordinate (b1) at (barycentric cs:A=3,B=3,D=1) ;
  \coordinate (b3) at (barycentric cs:A=1,B=3,D=3) ;
  \coordinate (b2) at (barycentric cs:A=3,B=1,D=3) ;
  \coordinate (d3) at (barycentric cs:A=-1,B=3,D=3) ;
  \coordinate (d2) at (barycentric cs:A=3,B=-1,D=3) ;
  \draw (b1) node[left]{$v$} ;
  \draw (b2) node[above]{$\ x$} ;
  \draw (b3) node[below]{$\ w$} ;
  \draw (a1) node[right]{$v'$} ;
  \draw (a1) node {$\bullet$} ;
  \draw (a2) node {$\bullet$} ;
  \draw (a3) node {$\bullet$} ;
  \draw (b1) node {$\bullet$} ;
  \draw (b2) node {$\bullet$} ;
  \draw (b3) node {$\bullet$} ;
  \draw (c2) node {$\bullet$} ;
  \draw (c3) node {$\bullet$} ;
  \draw (d2) node {$\bullet$} ;
  \draw[middlearrow={latex}] (a1) to[bend left] (a2) ;
  \draw[middlearrow={latex}] (a2) to[bend left] (a3) ;
  \draw[middlearrow={latex}] (a3) to[bend left] (a1) ;
  \draw[middlearrow={latex}] (b1) to[bend left] node[pos=0.4,below] {$a_1$}(b2) ;
  \draw[middlearrow={latex}] (b2) to[bend left] node[pos=0.4,left] {$a_2$} (b3) ;
  \draw[middlearrow={latex}] (b3) to[bend left] node[pos=0.4,above] {$a_3$}(b1) ;
  \draw[middlearrow={latex}] (b1) to[bend left] node[pos=0.3,above] {$a$} (a1) ;
  \draw[middlearrow={latex}] (a1) to[bend left] node[pos=0.3,below] {$a'$} (b1) ;
  \draw[middlearrow={latex}] (c2) to[bend left] (a2) ;
  \draw[middlearrow={latex}] (a2) to[bend left] (c2) ;
  \draw[middlearrow={latex}] (c3) to[bend left] (a3) ;
  \draw[middlearrow={latex}] (a3) to[bend left] (c3) ;
  \draw[middlearrow={latex}] (d2) to[bend left] (b2) ;
  \draw[middlearrow={latex}] (b2) to[bend left] (d2) ;
   \draw[gray] (A)--(B) ;
  \draw[gray] (C)--(B) ;
  \draw[gray] (D)--(B) ;
  \draw[gray] (A)--(C) ;
  \draw[gray] (A)--(D) ;
    \draw [dotted] (B) to[out=240, in=85,looseness=1] (b1);
     \draw [middlearrow={latex},dotted] (B) (b1)
  to[out=265,in=110,looseness=0.8] node[midway,left] {$\alpha_v$} (A);
\end{tikzpicture}
\caption{The quiver~$\Gamma_{\mathcal{T}}$; among the~$5$ drawn (nonoriented)
  edges 
  of the triangulation~$\mathcal{T}$, one is
  external (on the upper left), the other~$4$ are internal  }\label{fig:quadrilateral_intro}
\end{figure}
\index{notation}{61@$\Gamma_\mathcal{T}$ (the quiver associated with~$\mathcal{T}$)}

The vertex set~$V$ of the quiver~$\Gamma_{\mathcal{T}}$ is constructed as
follow.  For each face~$f$ of~$\mathcal{T}$ and each (nonoriented) edge of~$f$ we
add a point inside~$f$ near the midpoint of the edge. The~$r$ external edges
contribute to~$r$ vertices of~$\Gamma_\mathcal{T}$, we will call these
vertices \emph{external} as well; the $r-3{\chi(\bar{S})}
$ internal edges contribute to  $2r-6{\chi(\bar{S})}
$ vertices that we call \emph{internal}.
\index{definition}{internal!vertex}%
\index{definition}{external!vertex}%
\index{definition}{vertex!internal ---}%
\index{definition}{vertex!external ---}%

The arrows of~$\Gamma_\mathcal{T}$ have two
origins:
\begin{enumerate}[(A)]
  \setcounter{enumi}{4}
\item\label{item:s} for every internal (nonoriented) edge~$e$
  of~$\mathcal{T}$, we have two vertices $v$, $v'$
  of~$\Gamma_\mathcal{T}$, belonging respectively to triangles $T$,
  $T'$. We add the cycle~$C_e=\{a, a'\}$ made up of two arrows~$a$, $a'$
  connecting~$v$ to~$v'$ and~$v'$ to~$v$ within \index{notation}{63@$C_e=\{a,a'\}$ (the
    cycle in~$\Gamma_\mathcal{T}$ associated with the edge~$e$)}%
  $T\cup T'$. The arrows~$a$ and~$a'$ both intersect the edge~$e$ exactly once.
\item\label{item:f} for every face~$f$ of~$\mathcal{T}$ we have the three
  vertices of~$\Gamma_\mathcal{T}$ that are contained in~$f$. We add the
  cycle~$C_f=\{a_1, a_2,a_3\}$ made up of three arrows~$a_1$, $a_2$, $a_3$ \index{notation}{64@$C_a=\{a_1,a_2,a_3\}$ (the
    cycle in~$\Gamma_\mathcal{T}$ associated with the face~$f$)}
  connecting these three vertices within the face~$f$ and clockwise oriented.
\end{enumerate}

We denote by~$A_2$ the arrows that are part of such a $2$-cycle and by~$A_3$
the arrows that are part of a $3$-cycle.  The set~$A$ of arrows
of~$\Gamma_\mathcal{T}$ is then the disjoint union $A_2 \sqcup A_3$.

As the quiver~$\Gamma_\mathcal{T}$ is embedded in~$S$, it has a natural ribbon structure.
We consider the boundary components of the ribbon graph~$\Gamma_\mathcal{T}$ \emph{not}
containing external vertices. There are three possible types, which can be indexed by:
\begin{enumerate}
\item\label{item:type1} the compact boundary components of~${S}$. There
  are~$k$ such cycles.
\item the edges of~$\mathcal{T}$; for each (nonoriented) internal edge~$e$, we have the
  $2$-cycle~$C_e$. There are $r-3{ \chi(\bar{S})} $ such cycles.
\item the faces of~$\mathcal{T}$; for each face~$f$, one has the
  $3$-cycle~$C_f$. There are  $r-2{ \chi(\bar{S})} $ such cycles.
\end{enumerate}

Each vertex~$v$ of $\Gamma_\mathcal{T}$ has, by construction, a closest (nonoriented) edge~$e$ in~$\mathcal{T}$ and the endpoints of this edge are then contained in a component
of~$\partial S$.
 The orientation of~$S$ can be used to
distinguish between them:
\begin{quote}
  If $v$ is to the left of $e$, then
  $t(v)\in \pi_0( \partial S)$ is the component up and $b(v)$ is down (see
  Figure~\ref{fig:quadrilateral_intro}).
\end{quote}\index{notation}{66@$b(v)$, $t(v)$ (the components of $\partial S$ down and
  up the vertex $v$ of $\Gamma_\mathcal{T}$)}

For every internal vertex~$v$ in~$V$, we denote by $\alpha_v$ the (oriented)
edge of~$\mathcal{T}$ such that $r(\alpha_{v})$ \enquote{goes through}~$v$;
precisely~$\alpha_v$ is the edge next to~$v$ going from the component~$t(v)$ to
the component~$b(v)$.\index{notation}{67@$\alpha_v$ (the arc in~$S$ associated with the
  vertex $v$)}

\begin{rem}
  \label{rem:quiv-in-TprimeS}
  Using the vector field~$\vec{x}_\mathcal{T}$ (cf.\
  Section~\ref{sec:vect-fields-triang}), the quiver~$\Gamma_\mathcal{T}$ is
  also embedded in~$T' S$.
\end{rem}

\section{Local systems on a quiver}
\label{sec:local-systems-graph}

Local systems on quivers can be described concretely.

Let~$\Gamma$ be a quiver with vertex set~$V$ and arrow set~$A$. Any
arrow~$a$ in~$A$ has a start point~$v^-(a)$ and an endpoint~$v^+(a)$.

\begin{df}
  A \emph{rank-$d$ local system} on~$\Gamma$ is the data of
  \begin{enumerate}
  \item for each $v\in V$, a vector space~$F_v$ of dimension~$d$;
  \item for each $a\in A$, a linear isomorphism $g_a\colon  F_{v^-(a)} \to
    F_{v^+(a)}$.
  \end{enumerate}
  Such a local system will be denoted simply as a pair $( \{F_v\}_{v\in V}, \{
  g_a\}_{a\in A})$ or
  $(F_v, g_a)$.
\end{df}\index{definition}{local system}

\begin{rem}
  Representations of quivers are a more general notion: the dimensions of the
  vector spaces can vary and the linear morphisms are not necessarily invertible.
\end{rem}

Two local systems $(F_v,g_a)$ and $(F'_v, g'_a)$ are \emph{equivalent} if
there is a family of linear isomorphisms $\psi_v\colon  F_v \to
F'_v$ ($v\in V$) such that, for every arrow~$a$ of~$\Gamma$, the following
diagram commutes:
\begin{center}
  \begin{tikzcd}
    F_{v^-(a)} \arrow[r, "g_a"] \arrow[d, "\Psi_{v^-(a)}"'] & F_{v^+(a)}
    \arrow[d,
    "\Psi_{v^+(a)}"] \\
    F'_{v^-(a)} \arrow[r, "g^{\prime}_a"] & F_{v^+(a)}
  \end{tikzcd}\ .
\end{center}
We will denote by $\Loc(\Gamma, \GL(d,\R))$ the moduli space of rank-$d$ local
systems on~$\Gamma$.\index{notation}{69@$\Loc(\Gamma, \GL(d,\R))$ (moduli of local
  systems on a quiver)}

From time to time, the elements~$g_a$ will be called \emph{transition maps}.
\index{definition}{transition!map}%

A \emph{symplectic local system} (of rank $d=2n$) is a local system $(F_v,
g_a)$ and the data of a symplectic form~$\omega_v$ on~$F_v$ (for each~$v$
in~$V$) such that the maps~$g_a$ are symplectic isomorphisms. Equivalence of
symplectic local systems is defined likewise. Local systems with respect to a
general Lie group are discussed in Section~\ref{sec:representations}.
\index{definition}{symplectic!local system}%
\index{definition}{local system!symplectic ---}%

\section{Local systems and bases}
\label{sec:local-systems-basis}

Given a local system $(F_v, g_a)$ on~$\Gamma$ it is often more convenient to
give the elements~$g_a$ in matrix coordinates, i.e.\ to fix a
basis~$\mathbf{b}_v$ of~$F_v$ for each~$v$ in~$V$. Hence a rank-$d$ local
system can be thought of as the data $( (F_v, \mathbf{b}_v), G_a)$, where for
each~$a$ in~$A$, $G_a$ belongs to $\GL(d, \R)$; $G_a$ is the
matrix of the linear transformation $g_a\colon F_{v^-(a)} \to F_{v^+(a)}$ in
the bases~$\mathbf{b}_{v^-(a)}$ and~$\mathbf{b}_{v^+(a)}$.

We sometimes call~$G_a$ \emph{transition matrices}.\index{definition}{transition!matrices}

 We will often say that the family of bases $(\mathbf{b}_v)_{v\in V}$ is a
\emph{basis} of the local system~\mbox{$(F_v, g_a)$.}\index{definition}{basis}

One can even retain only the data $\{G_a\}_{a\in A}$ for a local system
on~$\Gamma$. However, it will be useful to keep the complete data
$(F_v, \mathbf{b}_v, G_a)$ (or even $(F_v, \mathbf{b}_v, g_a, G_a)$) in order
to later give the framing or a decoration and to have a precise understanding
of equivalent local systems.  We call these tuples also a local system
on~$\Gamma$.

\begin{lem}\label{lem:equi-loc-sys-1}
  Let $(F_v, g_a)$ and $(F'_v, g'_a)$ be two rank-$d$ local systems
  on~$\Gamma$. Then the following statements are equivalent:
  \begin{enumerate}[(i)]
  \item \label{item:i} $(F_v, g_a)$ and $(F'_v, g'_a)$ are equivalent local
    systems.
  \item \label{item:ii} For every basis $(\mathbf{b}_v)$ of $(F_v, g_a)$, there
    is a basis  $(\mathbf{b}'_v)$ of $(F'_v, g'_a)$ such that, for every
    arrow~$a$ in~$\Gamma$, the matrices~$G_a$ and~$G'_a$ of~$g_a$ and~$g'_a$
    respectively are equal.
  \item \label{item:iii} There are a basis $(\mathbf{b}_v)$ of $(F_v, g_a)$
    and a basis $(\mathbf{b}'_v)$ of $(F'_v, g'_a)$ such that, for every
    arrow~$a$, the matrices~$G_a$ and~$G'_a$ are equal.
  \end{enumerate}
\end{lem}

The proofs of (\ref{item:i}) $\Rightarrow$ (\ref{item:ii}) $\Rightarrow$
(\ref{item:iii}) $\Rightarrow$ (\ref{item:i}) are straightforward.

The lemma will be used in the following (equivalent) form:
\begin{lem}\label{lem:equi-loc-sys-2}
  Two local systems $(F_v, \mathbf{b}_v, g_a, G_a)$ and
  $(F'_v, \mathbf{b}'_v, g'_a, G'_a)$ are equivalent if and only if there is a
  basis $(\mathbf{c}_v)_{v\in V}$ of $(F_v, g_a)$ such that, for all arrow~$a$ of~$\Gamma$, $G'_a$ is the matrix of the map $g_a\colon  F_{v^-(a)} \to
  F_{v^+(a)}$ in the bases $\mathbf{c}_{v^-(a)}$ and $\mathbf{c}_{v^+(a)}$.
\end{lem}

Of course, our main concern here is symplectic local systems, for those we
work with symplectic bases of~$F_v$ given as pairs
$(\mathbf{e}_v, \mathbf{f}_v)$ and the matrices~$G_a$ will be elements of the
symplectic group $\Sp(2n,\R)$.

\section{Local systems on $\Gamma_\mathcal{T}$ and local systems on~$S$}
\label{sec:local-syst-gamm}

A local system on~$S$ can be restricted to~$\Gamma_\mathcal{T}$ and thus gives rise to a local system
on~$\Gamma_\mathcal{T}$. As the embedding $\Gamma_\mathcal{T} \hookrightarrow
S$ induces an epimorphism between the fundamental groups, the restriction map
\[ \Loc( S, \Sp(2n,\R)) \longrightarrow \Loc( \Gamma_\mathcal{T}, \Sp(2n,\R))\]
is injective. In order to describe its image, we need to keep track which
cycles in~$\Gamma_\mathcal{T}$ become trivial in~$S$. This leads to the following definition:

\begin{df}\label{df:pi1Scompat-local-sys}
  A local system $(F_v, g_a)$ on~$\Gamma_\mathcal{T}$ is said to be
  \emph{$S$-compatible} if
  \begin{enumerate}
  \item for all $\{a, a'\}$ as in~\eqref{item:s}
    (Section~\ref{sec:orient-graph-gamm}, p.~\pageref{item:s}),
    $g_{a} g_{a'}= \Id$.
  \item for all $\{a_1,a_2,a_3\}$ as in~\eqref{item:f}
    (Section~\ref{sec:orient-graph-gamm}, p.~\pageref{item:f}),
    $g_{a_3} g_{a_2} g_{a_1}= \Id$.
  \end{enumerate}
\end{df}\index{definition}{compatible local system}%
\index{definition}{local system!compatible ---}%

We denote by $\Loc_{S}( \Gamma_\mathcal{T}, \Sp(2n, \R))$ the moduli space of
$S$-compatible local systems.\index{notation}{71@$\Loc_{S}( \Gamma_\mathcal{T}, \Sp(2n,
  \R))$ (moduli of $S$-compatible local system)}

\begin{prop}
  \label{prop:local-syst-gamm-rep-pi1S}
There is an isomorphism between $\Loc( S, \Sp(2n, \R))$ and
  $\Loc_{ S}( \Gamma_\mathcal{T}, \Sp(2n, \R))$ given by the restriction map.
  \end{prop}

\section{Twisted local systems on $\Gamma_\mathcal{T}$}
\label{sec:local-syst-gamm-widehat}

As with local systems on~$S$, one can use local systems on~$\Gamma_\mathcal{T}$
to understand $\delta$-twisted local systems on~$T'S$.
Since~$\Gamma_\mathcal{T}$ can be embedded in~$T'S$
(Remark~\ref{rem:quiv-in-TprimeS}) we can define $\delta$-twisted local
systems on~$\Gamma_\mathcal{T}$ as follows.

\begin{df}\label{df:pi1S-hat-compat-local-sys}
  A symplectic local system $(F_v, g_a)$ on~$\Gamma_\mathcal{T}$ is said to be
  \emph{$\delta$-twisted} (or \emph{twisted})
  if
  \begin{enumerate}
  \item for all $\{a,a'\}$ as in~\eqref{item:s}, p.~\pageref{item:s},
    $g_{a} g_{a'}= -\Id$.
  \item for all $\{a_1,a_2,a_3\}$ as in~\eqref{item:f}, p.~\pageref{item:f},
    $g_{a_3} g_{a_2} g_{a_1}= -\Id$.
  \end{enumerate}
\end{df}\index{definition}{twisted!local system}%
\index{definition}{local system!twisted ---}%
Then, along the same line as the results above, we get the following, where $\Locdelta( \Gamma_\mathcal{T}, \Sp(2n, \R))$ denotes the moduli space of
$\delta$-twisted local systems: \index{notation}{73@$\Locdelta( \Gamma_\mathcal{T},
  \Sp(2n, \R))$ (moduli of twisted local systems)}

\begin{prop}\label{prop:local-syst-gamm-rep-widehat}
  There is an isomorphism between
  $\Locdelta( S, \Sp(2n, \R))$ and $\Locdelta( \Gamma_\mathcal{T}, \Sp(2n, \R))$ given by the restriction map.
\end{prop}

\section{Framed twisted local systems on $\Gamma_{\mathcal{T}}$}
\label{sec:decor-local-syst-gamma_T}

In the above correspondence (Proposition~\ref{prop:local-syst-gamm-rep-widehat}),
we would like to keep track of framed twisted local systems. For this we will use the ribbon structure on~$\Gamma_\mathcal{T}$, in particular the two maps $t\colon V\to \pi_0( \partial S)$ and $b\colon
V\to \pi_0( \partial S)$ defined in
Section~\ref{sec:orient-graph-gamm}; recall that $V$~is the set of vertices
of~$\Gamma_\mathcal{T}$.

A framing of a twisted local system on~$\Gamma_\mathcal{T}$ is an
equivariant data of Lagrangians, more precisely:
\begin{df}
  Let $(F_v, g_a)$ be a symplectic $\delta$-twisted local system. A \emph{framing} of
  $(F_v, g_a)$ is the data of two families of Lagrangians~$L^{t}_{v}$ and~$L^{b}_{v}$
  of~$F_v$ ($v$ in~$V$), such that, for every
  arrow~$a$ of~$\Gamma_\mathcal{T}$,
  one has:
  \begin{itemize}
  \item if $a$ is in $A_2$, i.e.\ if $a$ crosses an edge of~$\mathcal{T}$,
    $g_a( L^{t}_{v^-(a)}) = L^{b}_{v^+(a)}$ and
    $g_a( L^{b}_{v^-(a)}) = L^{t}_{v^+(a)}$.
  \item if $a$ is in $A_3$, i.e.\ if 
    $a$ is contained in a face
    of~$\mathcal{T}$, $g_a( L^{b}_{v^-(a)}) = L^{t}_{v^+(a)}$.
  \end{itemize}
\end{df}\index{definition}{framing}%
\index{definition}{framed!twisted local system}%
\index{definition}{twisted!framed --- local system}%
\index{definition}{local system!framed twisted ---}%
The tuple $(F_v, g_a, L^{t}_{v}, L^{b}_{v})$ is called a \emph{framed twisted local
  system} (cf.\ Figure~\ref{fig:quadrilateral_laocal_system}).

\begin{figure}[ht]
\centering
\begin{tikzpicture}
  \coordinate (A) at (5,0) ;
  \coordinate (B) at (5.5,5) ;
  \coordinate (D) at (1.4,2.9) ;
  \coordinate (a1) at (barycentric cs:A=3,B=3,D=-1) ;
  \coordinate (labelv) at (barycentric cs:A=4.5,B=1.5,D=-1) ;
  \coordinate (b1) at (barycentric cs:A=3,B=3,D=1) ;
  \coordinate (b3) at (barycentric cs:A=1,B=3,D=3) ;
  \coordinate (b2) at (barycentric cs:A=3,B=1,D=3) ;
  \draw (labelv) node[right]{$F_v\supset L^{b}_{v}, L^{t}_{v}$} ;
  \draw (b2) node[left]{$F_x\supset L^{b}_{x}, L^{b}_{x}$} ;
  \draw (b3) node[above left]{$F_w\supset L^{b}_{w}, L^{t}_{w}$} ;
  \draw (a1) node[right]{$F_v'\supset L^{b}_{v'}, L^{t}_{v'}$} ;
  \draw (a1) node {$\bullet$} ;
  \draw (b1) node {$\bullet$} ;
  \draw (b2) node {$\bullet$} ;
  \draw (b3) node {$\bullet$} ;
  \draw[middlearrow={latex}] (b1) to[bend left] node[pos=0.5,below] {$g_1$}(b2) ;
  \draw[middlearrow={latex}] (b2) to[bend left] node[pos=0.5,left] {$g_2$} (b3) ;
  \draw[middlearrow={latex}] (b3) to[bend left] node[pos=0.6,above] {$\ g_3$}(b1) ;
  \draw[middlearrow={latex}] (b1) to[bend left] node[pos=0.3,above] {$g_a$} (a1) ;
  \draw[middlearrow={latex}] (a1) to[bend left] node[pos=0.3,below] {$g_{a'}$} (b1) ;
   \draw[gray] (A)--(B) ;
  \draw[gray!50] (D)--(B) ;
  \draw[gray!50] (A)--(D) ;
    \draw [dotted] (labelv) to[out=150, in=270,looseness=1] (b1);
\end{tikzpicture}
\caption{The framed twisted local system on the
  quiver~$\Gamma_{\mathcal{T}}$. With respect to
  Figure~\ref{fig:quadrilateral_intro}, the vertices have been labeled by the
  corresponding symplectic vector spaces and the~$2$ Lagrangians they
  contained, whereas every arrow has been labeled by the corresponding
  symplectic isomorphism between the symplectic vector spaces associated with
  its endpoints (we also write $g_i$ instead of $g_{a_i}$). By definition, the
following relations hold: $g_{a'}=g_{a}^{-1}$, $g_3 g_2 g_1 =\Id$,
$L_{v'}^{b}= g_a( L^{t}_{v})$,
$L_{v'}^{t}= g_a( L^{b}_{v})$, $L^{b}_{v}= g_3( L^{t}_{w})$, $L^{t}_{v}=
g_{1}^{-1}( L^{b}_{x})$, etc. }\label{fig:quadrilateral_laocal_system}
\end{figure}

Equivalence of framed twisted local systems has to take care of the framings:
\begin{df}
  Let $(F_v, g_a, L^{t}_{v}, L^{b}_{v})$
  and $(F'_v, g'_a, L^{\prime t}_{v}, L^{\prime b}_{v})$ be two framed symplectic twisted local systems. They are defined  to be
  \emph{equivalent} if there is a family
  $\{ \psi_v\colon F_v \to F'_v\}_{v\in V}$ that is an isomorphism between the
  symplectic local systems $(F_v, g_a)$ and $(F'_v, g'_a)$, and such that, for
  all~$v$ in~$V$, $\psi_v( L^{t}_{v}) = L^{\prime t}_{v}$ and
  $\psi_v( L^{b}_{v}) = L^{\prime b}_{v}$.
\end{df}

The restriction of a framed twisted local system on~$S$ gives a framed
twisted local system on~$\Gamma_\mathcal{T}$: starting from a framed
twisted symplectic local system $(\mathcal{F}, \sigma)$, the restriction gives
first a local system $(F_v, g_a)$ on~$\Gamma_\mathcal{T}$, then, for every
vertex~$v$ in~$V$, using parallel transport along the arc~$\alpha_v$ described
in Section~\ref{sec:orient-graph-gamm}, and the section~$\sigma$ defined at
the extremities of~$\alpha_v$, we obtain two Lagrangians~$L^{t}_{v}$
and~$L^{b}_{v}$ in~$F_v$. One has

\begin{prop}
  \label{prop:decorated-restriction}
  The restriction map induces an isomorphism between the space
  $\Locfdelta(S, \Sp(2n,\R))$ and the space
  $\Locfdelta( \Gamma_\mathcal{T}, \Sp(2n, \R))$ of framed twisted
  symplectic local systems on~$\Gamma_\mathcal{T}$.
\end{prop}\index{notation}{75@$\Locfdelta( \Gamma_\mathcal{T}, \Sp(2n, \R))$ (moduli of
  framed twisted local systems)}

We will denote by
$\Mfdelta( \Gamma_\mathcal{T}, \Sp(2n, \R))$ the framed
local systems on~$\Gamma_\mathcal{T}$ corresponding to maximal
ones.\index{notation}{76@$\Mfdelta( \Gamma_\mathcal{T}, \Sp(2n, \R))$ (subspace of
  maximal framed local systems)}

\section{Symplectic bases of framed local systems}
\label{sec:sympl-basis-decor}

Local systems on $\Gamma_\mathcal{T}$ are easier to tackle when given as a family $\{G_a\}_{a\in A}$
of elements of $\Sp(2n,\R)$, i.e.\ the linear isomorphism $g_a$ associated to an arrow of the quiver is given by an explicit transition matrix. As explained in Section~\ref{sec:local-systems-basis}, this is equivalent to equipping each of the
 spaces $F_v$ with a symplectic basis.

In the presence of a framing, it is desirable that the symplectic bases
are adapted to the Lagrangians. This motivates the following definition:
\begin{df}\label{df:sympl-basis-decor}
  A framed $\delta$-twisted symplectic local system
  $(F_v, g_a, L^{t}_{v}, L^{b}_{v})$ is called \emph{transverse} if, for all
  $v\in V$, the two Lagrangians~$L^{t}_{v}$ and~$L^{b}_{v}$ are transverse:
\[  F_v =L^{t}_{v} \oplus L^{b}_{v}.\]
  A symplectic basis $\{ (\mathbf{e}_v, \mathbf{f}_v)\}_{v\in V}$ is said to
  \emph{generate} (or \emph{generating}) the framing if, for all~$v$,
  $L^{t}_{v} = \Span( \mathbf{e}_v)$ and $L^{b}_{v} = \Span( \mathbf{f}_v)$;
  in this case $(L^{t}_{v}, L^{b}_{v})$ will be called the \emph{generated
  framing}.
\end{df}\index{definition}{generating basis}%
\index{definition}{basis!generating ---}%
\index{definition}{transverse!framed twisted local system}%
\index{definition}{framed!transverse --- twisted local system}%
\index{definition}{twisted!transverse framed --- local system}%
\index{definition}{local system!transverse framed twisted ---}%
\begin{lem}\label{lem:sympl-basis-decor}
  \begin{enumerate}
  \item \label{item:1:lem:sympl-basis-decor} A framed
    $\delta$-twisted symplectic local system
    on~$\Gamma_\mathcal{T}$ is transverse if and only if the corresponding
    framed twisted local system on~$S$ is transverse with respect to~$\mathcal{T}$
    \ep{Definition~\ref{df:transv-local-syst-T}}.
  \item \label{item:2:lem:sympl-basis-decor} A framed $\delta$-twisted
    symplectic local system on~$\Gamma_\mathcal{T}$ is transverse if and only
    if it admits a generating basis.
  \end{enumerate}
\end{lem}

In other words, the isomorphism of Proposition~\ref{prop:decorated-restriction}
restricts to an isomorphism between the space $\LocfdeltaT(
\Gamma_\mathcal{T}, \Sp(2n,\R))$ of transverse framed \index{notation}{77@$\LocfdeltaT(
  \Gamma_\mathcal{T}, \Sp(2n,\R))$ (moduli of transverse framed twisted local systems)}
$\delta$-twisted symplectic local
system and the space $\LocfdeltaT( S,
\Sp(2n, \R))$ of transverse framed twisted local systems on~$S$.

For a framed twisted local system with a generating symplectic basis, the matrices
$\{ G_a\}_{a\in A}$ have a special form which we describe in the following lemma.
\begin{lem}
  Let $(F_v, g_a, L^{t}_{v}, L^{b}_{v})$ be a transverse framed symplectic
  twisted local system on~$\Gamma_\mathcal{T}$. Let
  $\{(\mathbf{e}_v, \mathbf{f}_v)\}_v$ be a generating symplectic basis and, for
  each $a\in A$, $G_a$ the matrix of the transformation~$g_a$ in these bases.

  Then, if the arrow~$a$ belongs to~$A_2$, the matrix~$G_a$ has
  the form $ \bigl( \begin{smallmatrix}
    0 & -{}^T\! B^{-1}\\ B & 0
  \end{smallmatrix}\bigr)$ for some~$B$ in~$\GL(n,\R)$; and if the arrow~$a$ belongs
  to~$A_3$, the matrix~$G_a$ has the form
  $ \bigl(\begin{smallmatrix} MC & -{}^T\! C^{-1}\\ C & 0
  \end{smallmatrix}\bigr)$ for some $C$ in~$\GL(n,\R)$ and $M$ in~$\Sym(n,\R)$.
\end{lem}
\begin{proof}
   Let $v_1$ and $v_2$ be the start point and endpoint of~$a$.
   In the case when $a\in A_2$ the conditions
     \[g_a( \Span( \mathbf{e}_{v_1})) = g_a( L^{t}_{v_1}) = L^{b}_{v_2} =
     \Span( \mathbf{f}_{v_2}) \text{ and } g_a( \Span( \mathbf{f}_{v_1})) = \Span(
     \mathbf{e}_{v_2})\]
   translate into the following conditions on~$G_a$, with $(\mathbf{e}_0,
   \mathbf{f}_0)$ the standard symplectic basis of~$\R^{2n}$:
       \[\Span( \mathbf{e}_0\cdot G_a)= \Span( \mathbf{f}_0) \text{ and }
     \Span( \mathbf{f}_0\cdot G_a)= \Span( \mathbf{e}_0),\]
   which precisely means that the matrix~$G_a$ has the specified off diagonal
   form.

   In the second case ($a\in A_3$), the condition on~$G_a$ is that $\Span( \mathbf{f}_0\cdot G_a)=
   \Span( \mathbf{e}_0)$ and this implies the result.
\end{proof}
Conversely, if a twisted local system is given by a family $\{G_a\}_{a\in A}$ of
symplectic matrices satisfying the special form given in the conclusions of the preceding lemma then it
comes from a unique generating symplectic basis on a framed
local system.

\begin{prop}\label{prop:decor-loc-from-Ga}
  Let $\{G_a\}_{a\in A}$ be a family of elements of~$\Sp(2n,\R)$ such that
  \begin{enumerate}
  \item $\{G_a\}_a$ is a $\delta$-twisted local system \ep{i.e.\
    for each edge~$e$ of~$\mathcal{T}$, the product of the $G_a$s along the
    cycle~$C_e$ is~$-\Id$ and for each triangle~$f$ of~$\mathcal{T}$, the
    product of the $G_a$s along~$C_f$ is~$-\Id$; see
    Section~\ref{sec:local-syst-gamm-widehat}.}
  \item for all arrow~$a$ of~$\Gamma_\mathcal{T}$ in~$A_2$, the matrix~$G_a$ has the form $ \bigl( \begin{smallmatrix}
    0 & -{}^T\! B^{-1}\\ B & 0
  \end{smallmatrix}\bigr)$ for some $B\in \GL(n,\R)$.
  \item for all arrow~$a$ in~$A_3$, the matrix~$G_a$ has the form $  \bigl(\begin{smallmatrix}
    MC & -{}^T\! C^{-1}\\ C & 0
  \end{smallmatrix}\bigr)$ for some $C\in \GL(n,\R)$ and $M\in \Sym(n,\R)$.
  \end{enumerate}
  Then there is a unique up to equivalence framed
  $\delta$-twisted symplectic local system on~$\Gamma_\mathcal{T}$ with
  generating symplectic basis such that the transition matrices are precisely
  the~$G_a$ \ep{$a\in A$}.
\end{prop}
\begin{proof}
  Uniqueness is clear since the bases on two local systems will entirely
  determine the isomorphism between them.

  Let us prove existence. For each~$v$ in~$V$, set $(F_v, (\mathbf{e}_v,
  \mathbf{f}_v))$ to be the space~$\R^{2n}$ with its canonical symplectic
  basis. Let $(L^{t}_{v}, L^{b}_{v})$ the Lagrangians generated by these
  bases. Finally let $g_a\colon  F_{v^-(a)}\to F_{v^+(a)}$ the symplectic
  isomorphism whose matrix is~$G_a$ (again in the given bases). The hypothesis
  on~$G_a$ precisely means that the family $(L^{t}_{v}, L^{b}_{v})$ is a
  framing of the $\delta$-twisted local system $(F_v, g_a)$.
\end{proof}

As a conclusion, a framed transverse local system (symplectic and
$\delta$-twisted) is entirely determined by a family
$\{G_a\}_{a\in A}$ satisfying the hypothesis of the proposition. By a little
abuse of terminology, such a family $\{G_a\}_{a\in A}$ will be also called a
transverse framed twisted local system.

\section{Configurations associated with framed local systems on~$\Gamma_\mathcal{T}$}
\label{sec:conf-assoc-with}

The maps from local systems to configurations of
Section~\ref{sec:conf-assoc-with-1} can be seen as defined on the space of
framed twisted local systems on~$\Gamma_\mathcal{T}$.

When~$v$ is a vertex of~$\Gamma_\mathcal{T}$, hence contained in some
triangle~$f$ of~$\mathcal{T}$, we can define a $3$-gon $\tau\colon
(\mathbb{D}, \mu_3)\to (S, \partial S)$ (cf.\
Section~\ref{sec:conf-assoc-with-1}) such that $\tau(1)= t(v)$. We will rather
denote
\[ f_v \coloneqq f_\tau \colon \Locddelta( \Gamma_\mathcal{T},
  \Sp(2n,\R)) \longrightarrow \Conf^3( \Lag{n}),\]
the map defined previously.\index{notation}{79@$f_v$ (map to $\Conf^3( \Lag{n})$
  associated with a vertex~$v$)}

When~$v$ is an internal vertex, the map $q_{\alpha_v}$ will be denoted
\index{notation}{80@$q_v$ (map to $\Conf^4( \Lag{n})$
  associated with a vertex~$v$)}
by~$q_v$ (cf.\ Section~\ref{sec:orient-graph-gamm} for the definition
of~$\alpha_v$). When~$\theta$ is an internal angle, we obtain as well a map
$\Locddelta( \Gamma_\mathcal{T}, \Sp(2n,\R)) \rightarrow \Conf^5( \Lag{n})$
that comes from the map defined in Section~\ref{sec:conf-assoc-with-1} and
will be denoted as well~$c_\theta$.\index{notation}{81@$c_\theta$ (map to $\Conf^5( \Lag{n})$
  associated with an internal angle~$\theta$)}

\smallskip

For a triangle~$T$ of~$\mathcal{T}$, we will denote by the same letter~$\mu^T$
the Maslov index of the triangle~$T$.

\section[Maslov indices, cross ratios, and Toledo number]{Maslov indices, cross ratios, and Toledo number for framed local systems}
\label{sec:maslov-indices-cross}

We end this chapter by explaining how to calculate the mentioned quantities
for a framed $\delta$-twisted local system
$x=(F_v, g_a, L^{t}_{v}, L^{b}_{v} )$ given explicitly as a family of
matrices $\{G_a\}_{a\in A}$ thanks to $\{ ( \mathbf{e}_v, \mathbf{f}_v)\}_v$
a generating symplectic basis.

Explicitly, the given data is:
\begin{itemize}
\item for every arrow~$a$ in~$A_2$, $G_a =
  \begin{pmatrix}
    0 & -{}^T\! B_{a}^{-1} \\ B_a & 0
  \end{pmatrix}
  $ for an element~$B_a$ in $\GL(n,\R)$. These
  elements satisfy $B_a= B_{a'}$ for every cycle $\{a,a'\}$ in~$A_2$.
\item for every arrow~$a$ in~$A_3$, $G_a=
  \begin{pmatrix}
    M_a C_a & -{}^T\! C_{a}^{-1} \\ C_a & 0
  \end{pmatrix}
  $ for a symmetric matrix~$M_a$ and an element~$C_a$ in $\GL(n,\R)$. These satisfy the identities given in
  Remark~\ref{rem:triple-lag} below.
\end{itemize}

\begin{lem}
  \label{lem:muT-in-loc-sys}
  Let $T$ be a triangle of~$\mathcal{T}$ and $a\in A_3$ an arrow of $\Gamma_\mathcal{T}$
  contained in~$T$, set $v=v^+(a)$.
  Then the configuration $f_v(x ) \in  \Conf^3(\Lag{n} )$
  is the class of
  $\bigl(\Span(\mathbf{e}_0), \Span(\mathbf{f}_0), \Span(\mathbf{e}_0 +
  \mathbf{f}_0\cdot M_a) \bigr)$, where $( \mathbf{e}_0, \mathbf{f}_0)$ is the
  standard basis of~$\R^{2n}$. The Maslov index $\mu^{T}(x)$ is
  equal to the signature of the symmetric matrix~$M_a$.
\end{lem}
\begin{proof}
  The triple of Lagrangians is
  \[ L_1= \Span( \mathbf{e}_v), \ L_2 =\Span( \mathbf{f}_v), \text{ and } M
    =g_{a}(\Span( \mathbf{e}_{v^-(a)})).\] Since
  $g_a( \mathbf{e}_{v^-(a)}) = \mathbf{e}_v \cdot M_a C_a+ \mathbf{f}_v\cdot
  C_a = \bigl(\mathbf{e}_v \cdot M_a + \mathbf{f}_v \bigr)\cdot C_a$, one has
  $M = \Span (\mathbf{e}_v \cdot M_a + \mathbf{f}_v)$ and the result follows
  identifying $(F_v, \mathbf{e}_v, \mathbf{f}_v)$ with
  $(\R^{2n}, \mathbf{e}_0, \mathbf{f}_0)$ (cf.\ Section~\ref{maslov_prop}).
\end{proof}
\begin{rem}\label{rem:triple-lag}
  As a consequence, the signature of the $M_a$'s have to coincide along the
  arrows of a $3$-cycle in $A_3$. This property can be also obtained from the equality
  \[ \begin{pmatrix}
      M_c C_c & -{}^T\! C_{c}^{-1} \\ C_a & 0
    \end{pmatrix} \begin{pmatrix}
      M_b C_b & -{}^T\! C_{b}^{-1} \\ C_b & 0
    \end{pmatrix} \begin{pmatrix}
      M_a C_a & -{}^T\! C_{a}^{-1} \\ C_c & 0
    \end{pmatrix}=-\Id \]
  which holds if and only if $M_c C_c {}^T\! C_{b}^{-1} C_a=\Id$, $M_{c}^{-1} = C_c M_b\, {}^T\! C_{c}$, and $M_{b}^{-1}
  = C_b M_a\, {}^T\! C_{b}$.
  Another consequence is that the symmetric matrices~$M_a$,
  $M_b$, and~$M_c$ must be invertible.
\end{rem}

\begin{cor}
  \label{cor:toledo-and-max-loc-sys}
  Let $A' \subset A_3$ be a subset containing exactly one arrow of every
  $3$-cycle in~$A_3$. Then
  \begin{enumerate}
  \item if $R=\emptyset$ \ep{so that the Toledo number is defined on
    $\Locddelta( \Gamma_\mathcal{T}, \Sp(2n, \R))$}, $T_x = -\frac{1}{2} \sum_{a\in A'} \sgn( M_a) = -\frac{1}{6} \sum_{a\in A_3} \sgn( M_a)$;
  \item \label{item:cor:max-loc-sys} $x$ is maximal if and only if for all
    $a$ in~$A'$, $M_a$ is positive definite, and if and only if, for all $a\in
    A_3$, $M_a$ is positive definite.
  \end{enumerate}
\end{cor}
\begin{proof}
  Using the preceding lemma, this results from
  Theorem~\ref{teo:toledo_maslov} and its
  corollary~\ref{cor:max-iff-triangle-are-max} or the lemma~\ref{lem:maxim-decor-local-every-triangulation}.
\end{proof}

We also have an understanding of the cross ratios which we state under the
additional condition that $M_a=\Id$ for all $a\in A_3$. This computation will be useful in Section~\ref{rec_rep_max} and
Section~\ref{sec:maxim-decor-sympl}.

\begin{lem}
  \label{lem:cross-ratio-m_aId}
  Suppose furthermore that $M_a=\Id$ for all $a\in A_3$, then for every internal
  vertex~$v$ of~$\Gamma_\mathcal{T}$, the cross ratio of the quadruple
  $q_v(x)$ is the class of the symmetric matrix $B_a\, {}^T\! B_a$, where $a$
  is one of the arrows in~$A_2$ whose endpoint is the vertex~$v$.
\end{lem}
\begin{proof}
  Let~$w$ be the vertex of~$\Gamma_\mathcal{T}$ connected to~$v$ by~$a$. The
  two symplectic vector spaces~$F_v$ and~$F_w$ are equipped with symplectic
  bases $(\mathbf{e}_v, \mathbf{f}_v)$ and $(\mathbf{e}_w, \mathbf{f}_w)$
  respectively.  By Lemma~\ref{lem:muT-in-loc-sys} the triple of Lagrangians in $F_v$ is
  $\Span( \mathbf{e}_v)$, $\Span( \mathbf{f}_v)$,
  $\Span( \mathbf{e}_v + \mathbf{f}_v)$, and similarly in $F_w$.   The matrix of
  $g_a\colon  F_w \to F_v$ in the symplectic basis is $\bigl(  \begin{smallmatrix}
    0 & -{}^T\! B_{a}^{-1} \\ B_a & 0
  \end{smallmatrix}\bigr)
  $, and $g_a( \Span( \mathbf{e}_w))= \Span( \mathbf{f}_v \cdot B_a)= \Span(
  \mathbf{f}_v)$, $g_a( \Span( \mathbf{f}_w))= \Span(
  \mathbf{e}_v)$, and  $g_a( \Span( \mathbf{e}_w + \mathbf{f}_w))= \Span( -
  \mathbf{e}_v \cdot {}^T\! B_{a}^{-1} +\mathbf{f}_v \cdot B_a) = \Span(
  \mathbf{e}_v -\mathbf{f}_v \cdot B_a\, {}^T\! B_{a})$, hence the result.
\end{proof}

The angle invariant (cf.\ Section~\ref{sec:positive-quintuples}) can also be
determined easily in the case when the matrices~$B_a$ are diagonal.

\begin{lem}
  \label{lem:angle-invariant-is-Ca}
  Suppose that $M_a=\Id$ for all~$a$ in~$A_3$ and that $B_a\in \Delta_n$ for
  all~$a$ in~$A_2$. Let~$\theta=(\alpha,\alpha')$ be an internal angle
  \ep{$\alpha$ and~$\alpha'$ are internal edges of the triangulation and are
    sides of the same triangle}. Let~$b$ be the arrow in~$A_3$ that is
  \enquote{next} to~$\theta$, i.e.\ $b$ joins~$\alpha$ and~$\alpha'$.

  Then the matrix~$C_b$ is orthogonal and represents the angle invariant of
  the quintuple $c_\theta( x)$.
\end{lem}

\begin{proof}
  The fact that~$C_b$ is orthogonal follows directly from the hypothesis on
  the matrices~$M_a$(cf.\ Remark~\ref{rem:triple-lag}). Let $v=v^-(b)$ and
  $v'=v^+(b)$ so that~$v$ is next to~$\alpha$ and~$v'$ is next
  to~$\alpha'$. The hypothesis on the matrices~$B_a$ implies that the
  symplectic basis $(\mathbf{e}_v, \mathbf{f}_v)$ (resp.\ $(\mathbf{e}_v,
  \mathbf{f}_v)$) is in standard position with respect to the quadruple
  $q_v(x)$ (resp.\ $q_{v'}(x)$). Since $\mathbf{e}_{v'} = - \mathbf{f}_v \cdot
  C_b$, the result follows from the definition of the angle invariant
  (Section~\ref{sec:positive-quintuples}).
\end{proof}

\chapter{$\mathcal{X}$-coordinates for maximal representations}\label{sec_def_coord_max}
In this chapter we introduce positive $\mathcal{X}$-coordinates. They
give a parametrization of the space of maximal representations: we restrict
our attention here to this special case because the definition is
significantly simpler than in the general case.
We refer the reader to Chapter~\ref{sec:generalX} for the definition of general $\mathcal{X}$-coordinates.

\section{A space of coordinates}
\label{sec:space-coordinates-1}

Let~$\mathcal{T}$ be an ideal triangulation of~$S$ and
let~$\Gamma_\mathcal{T}$ the quiver constructed in
Section~\ref{sec:orient-graph-gamm}. The
set of arrows (oriented edges) of~$\Gamma_\mathcal{T}$
is denoted by~$A$. Recall that arrows of~$\Gamma_\mathcal{T}$ are of two types:
$A=A_2 \sqcup A_3$ where $A_2$ is the set of arrows crossing an edge
of~$\mathcal{T}$ and $A_3$ is the set of arrows that are contained in one
triangle of~$\mathcal{T}$.

We denote by $\XplusTn$ the space of tuples\index{notation}{83@$\XplusTn$ (space of
  coordinates)}%
\index{definition}{$\mathcal{X}$-coordinates!positive ---}%
\index{definition}{positive!$\mathcal{X}$-coordinates}
  \[  x = \bigl( \{x(a)\}_{a\in A_2}, \{x(a)\}_{a\in A_3}\bigr) \in \bigl( \Sym^+(n,\R)\bigr)^{A_2} \times \OO(n)^{A_3} \]
such that
\begin{itemize}
\item for every $3$-cycle $(a_1,a_2,a_3)$ in $A_3$, $x(a_3)x(a_2)x(a_1)=\Id$.
\item for every $2$-cycle $(a,a')$ in~$A_2$, $x(a)= x(a')$.
\end{itemize}

To every $x\in \XplusTn$, we associate a framed $\delta$-twisted symplectic local
system $\holXpT(x)$ as follows: the local system is given by the matrices $\{
G_a(x)\}_{a\in A}$, where\index{notation}{85@$\holXpT$ (the holonomy map from the space
  $\XplusTn$)}
\begin{itemize}
\item for every~$a$ in~$A_2$ (using the square root function on the set of
  positive definite symmetric matrices)
  \[ G_a(x) =
    \begin{pmatrix}
      0 & -x(a)^{-1/2}\\ x(a)^{1/2} & 0
    \end{pmatrix};
\]
\item and for every~$a$ in~$A_3$, $G_a(x) =
    \begin{pmatrix}
      x(a) & -x(a)\\ x(a) & 0
    \end{pmatrix}
    $.
\end{itemize}

This local system $\{G_a(x)\}$ is
  $\delta$-twisted: this
  follows from the conditions on the $\{z(a)\}_{a\in A_3}$ and the fact
  that $\bigl(
  \begin{smallmatrix}
    0 & -\Lambda^{-1}\\ \Lambda & 0
  \end{smallmatrix}\bigr)^2 = -\Id
  $ for every matrix~$\Lambda$. Therefore, by Proposition~\ref{prop:decor-loc-from-Ga},
  the family $\{G_a(x)\}_{a\in A}$ defines a framed $\delta$-twisted
  symplectic local system on~$\Gamma_\mathcal{T}$, which is denoted by~$\holXpT(x)$.

\begin{lem}
  \label{lemma:holo-z-max}
  For every $x\in \XplusTn$, $\holXpT(x)$ is a maximal
  framed $\delta$-compatible local system on $\Gamma_\mathcal{T}$.
\end{lem}

\begin{proof}
  This follows from Corollary~\ref{cor:toledo-and-max-loc-sys} since, in
  the notation of the corollary, $M_a=\Id$ for every $a\in A_3$.
\end{proof}

In fact, only the equivalence class of the framed local system $\holXpT(x)$ is well defined,
so we get an element in
$\Mfdelta( \Gamma_\mathcal{T}, \Sp(2n,\R))$.  Thus we have a well
defined map
\begin{equation}
  \label{eq:hol-Z}
  \holXpT\colon  \XplusTn \longrightarrow \Mfdelta( \Gamma_\mathcal{T}, \Sp(2n,\R)).
\end{equation}

We will see shortly (Proposition \ref{prop:holXpT-surj}) that this map is surjective. This map is instead not injective, and the rest of this chapter will be an investigation of this lack of injectivity. We will restrict the map~$\holXpT$ to some suitable subsets of $\XplusTn$, and we will describe the fibers of the restricted maps.

\section{Positive $\mathcal{X}$-coordinates}
\label{rec_rep_max}

Let~$\Delta_n$ be the set of diagonal matrices with positive nondecreasing
entries.

\begin{df}[Positive $\mathcal{X}$-coordinates]\label{def_coord_max}
We will call \emph{positive $\mathcal{X}$-coordinates} and denote by
$\XplusDeltaTn$
the space of tuples\index{notation}{87@$\XplusDeltaTn$ (the space of positive coordinates)}%
\index{definition}{$\mathcal{X}$-coordinates!positive ---}%
\index{definition}{positive!$\mathcal{X}$-coordinates}
\[x =\bigl( \{x(a)\}_{a\in A_2}, \{x(a)\}_{a\in A_3}\bigr) \in \XplusTn \]
such that for every~$a$ in~$A_2$, $x(a)$ belongs to~$\Delta_n$.

 Sometimes we will say that $x$ is a \emph{system of positive
  $\mathcal{X}$-coordinates} of rank~$n$ on~\mbox{$(S,\mathcal{T})$.}
\end{df}\index{definition}{system of positive
  $\mathcal{X}$-coordinates}%
\index{definition}{$\mathcal{X}$-coordinates!system of positive ---}%
\index{definition}{positive!system of --- $\mathcal{X}$-coordinates}%

Let~$A'$ be a subset of~$A_3$ such that, for each
$3$-cycle~$C$ in $A_3$, $C\cap A'$ has~$2$ elements and let~$E$ be a subset
of~$A_2$ containing exactly one element in every $2$-cycle in~$A_2$. Then
the map
\begin{align*}
  \XplusDeltaTn & \longrightarrow \Delta_{n}^{E}\times
                                 \OO(n)^{A'} \\
  x & \longmapsto ( \{x(a)\}_{a\in E}, \{x(a)\}_{a\in A'})
\end{align*}
is a diffeomorphism. Note that $\sharp E = r-3{ \chi(\bar{S})} $, and $\sharp
A' = 2\sharp \mathcal{T} = 2r -4 { \chi(\bar{S})}$.

Since $\XplusDeltaTn \subset \XplusTn$, by restriction of the map $\holXpT$ we
have a map\index{notation}{88@$\holXpdT$ (restriction of $\holXpT$ to $\XplusDeltaTn$)}
\[ \holXpdT \colon \XplusDeltaTn \longrightarrow
\Mfdelta( \Gamma_\mathcal{T}, \Sp(2n,\R)). \]

The geometric significance of the positive $\mathcal{X}$-coordinates is the
following statement:

\begin{lem}
  \label{lem:posit-X-geom}
  Let~$x$ be in $\XplusDeltaTn$ and let $(\mathcal{F}, \sigma)=
  \holXpdT(x)$. Then, for every~$a$ in~$A_2$, the cross ratio of the quadruple
  $q_{v^+(a)}( \mathcal{F}, \sigma)$ is~$x(a)^{-1}$ and, for every internal
  angle $\theta=\{ \alpha, \alpha'\}$, the angle invariant of the quintuple
  $c_\theta( \mathcal{F}, \sigma)$ is~$x(a)$ where~$a$ is the arrow in~$A_3$
  going from~$\alpha$ to~$\alpha'$.
\end{lem}

\begin{proof}
  The statement concerning the cross ratio results from
  Lemma~\ref{lem:cross-ratio-m_aId}, the one concerning the angle invariant
  results from Lemma~\ref{lem:angle-invariant-is-Ca}.
\end{proof}

\section{Maximal framed symplectic local systems}
\label{sec:maxim-decor-sympl}

In this section, we will prove the following result:
\begin{teo}\label{teo:max-param-Xplus}
  The map~$\holXpdT$
\begin{enumerate}
\item is onto the space $\Mfdelta( \Gamma_\mathcal{T},
\Sp(2n,\R))$ of maximal twisted framed representations.
\item For $x$ and $x'$ in $\XplusDeltaTn$, one has $\holXpdT(x)
  = \holXpdT(x')$ if and only if
  \begin{itemize}
  \item For every $a$ in~$A_2$, $x(a)=x'(a)$
  \item There is a family $(r_v)_{v\in V}$ of orthogonal matrices such that:
    \begin{itemize}
    \item For all $a\in A_2$,  one has $r_{v^{-}(a)} = r_{v^{+}(a)}$, and
      $r_{v^{-}(a)}$ commutes with~$x(a)$.
    \item for every~$a$ in $A_3$, $x'(a) = r_{v^{+}(a)} x(a) r_{v^{-}(a)}^{-1}$.
    \end{itemize}
  \end{itemize}
\end{enumerate}
\end{teo}

\begin{cor} \label{cor:finite-to-one}
The map~$\holXpdT$ is
generically finite-to-one. More precisely, it is finite-to-one on the open and
dense subspace consisting of the elements $x\in\XplusDeltaTn$ such that for
all $a\in A_2$, $x(a)$ is a diagonal matrix with positive increasing entries.
\end{cor}

The surjectivity of the map $\holXpdT$ is the following proposition.
\begin{prop}
  \label{prop:holXpT-surj}
  Let $x=(F_v, g_a, L^{t}_{v}, L^{b}_{v})$ be a maximal framed
  $\delta$-twisted symplectic local system on the
  quiver~$\Gamma_\mathcal{T}$.

  Then there exists a symplectic basis $( \mathbf{e}_v, \mathbf{f}_v)_{v\in
    V}$ generating the framing and, denoting $G_a\in \Sp(2n, \R)$ the
  matrices of~$g_a$, one has:
  \begin{enumerate}
  \item For every~$a$ in~$A_2$, then $G_a =
    \begin{pmatrix}
      0 & -\Lambda_{a}^{-1} \\ \Lambda_{a} & 0
    \end{pmatrix}
    $ for some $\Lambda_{a}\in \Delta_n$.
  \item For each $a\in A_3$, $G_a =
    \begin{pmatrix}
      u_a & -u_a \\ u_a & 0
    \end{pmatrix}
    $ for some $u_a\in \OO(n)$.
  \end{enumerate}
\end{prop}

We will say that a symplectic basis of $\{F_v\}_{v\in V}$ is in \emph{standard $\mathcal{X}$}-position if it satisfies the properties of the above
conclusion.\index{definition}{standard $\mathcal{X}$-position}
The tuple
 $(( \Lambda^{2}_{a})_{a\in A_2} , (u_a)_{a\in A_3})$ is then an
element of $\XplusDeltaTn$ (see
Lemma~\ref{lem:triple-of-lagrangian}) and will be called \emph{the system of
  positive coordinates} associated with the
basis~$( \mathbf{e}_v, \mathbf{f}_v)$.

\begin{proof}
  Let first~$v$ be an external vertex of~$\Gamma_\mathcal{T}$. Let~$a$ be the
  arrow in~$A_3$ such that $v=v^+(a)$. As the local system~$x$ is maximal, the
  Maslov index of the triple $( L^{t}_{v}, g_a( L^{t}_{v^{-}(a)}), L^{b}_{v})$
  is maximal and there is thus a symplectic basis
  $(\mathbf{e}_v, \mathbf{f}_v)$ such that $L^{t}_v = \Span( \mathbf{e}_v)$,
  $g_a(L^{t}_{v^-(a)}) = \Span( \mathbf{e}_v + \mathbf{f}_v)$, and
  $L^{b}_v = \Span( \mathbf{f}_v)$

  Let $E\subset A_2$ be a subset containing exactly one of the two arrows for
  every cycle in $A_2$.  The set~$V$ of vertices of $\Gamma_\mathcal{T}$ is
  the disjoint union of the doubletons $\{ v^-(a), v^+(a)\}$ for $a\in E$.

  Let $a\in E$ and call $v=v^+(a)$, $v'=v^-(a)$. The quadruple
  $(L_1, M_1, L_2, M_2) = q_v(x)$ is then positive since the two triples
  $f_v(x)$ and $f_{v'}(x)$ are maximal (cf.\
  Section~\ref{sec:conf-assoc-with} for the definition of~$f_v$). By
  Proposition~\ref{prop:standard-basis-positive-4-uple}.(\ref{item:1:prop:standard-basis-positive-4-uple}),
  there is a symplectic basis $(\mathbf{e}_v, \mathbf{f}_v)$ in standard
  position with respect to that quadruple: there is a matrix $\Lambda$ in
  $\Delta_n$ so that $L_1=\Span( \mathbf{e}_v)$, $L_2=\Span( \mathbf{f}_v)$,
  $M_1=\Span( \mathbf{e}_v+ \mathbf{f}_v)$ and
  $M_2=\Span( \mathbf{e}_v- \mathbf{f}_v \cdot \Lambda)$.

  By construction of the bases, the matrices of the~$g_a$, $a\in A'$, have the
  desired form. The property $g_{a_1} g_{a_2}=-\Id$ for all $2$-cycles
  $(a_1, a_2)$ in $A_2$ implies that the same holds for every $a\in A_2$. The
  fact that the matrices of the~$g_a$, $a\in A_3$, have the requested form is
  a consequence of Lemma~\ref{lem:triple-of-lagrangian} and Remark~\ref{rem:rewrite-conc-lem:triple-of-lagrangian}.
\end{proof}

Using Lemma~\ref{lem:equi-loc-sys-2}, the description of the fibers of
$\holXpdT$ is a consequence of the following statement.
\begin{prop}
  \label{prop:fiber-holXpT}
  Let $(F_v, g_a, L^{t}_{v}, L^{b}_{v})$ be a maximal framed
  $\delta$-twisted symplectic local system on the
  quiver~$\Gamma_\mathcal{T}$.  Let $(\mathbf{e}_v , \mathbf{f}_v)$ and $(\mathbf{e}'_v , \mathbf{f}'_v)$ be
  two symplectic bases in standard $\mathcal{X}$-position and let $x$ and
  $x'$ be the two associated systems of positive coordinates. Finally, let $\{ \psi_v\}_{v\in V}$ be the family of elements of $\Sp(2n, \R)$
  defined by the equalities: $(\mathbf{e}_v , \mathbf{f}_v) = (\mathbf{e}'_v ,
  \mathbf{f}'_v) \cdot \psi_v$ \ep{$v\in V$}.

  Then
  \begin{enumerate}
  \item\label{item1:prop:fiber-holXpT} For all~$a$ in~$A_2$, $x'(a)=x(a)$.
  \item\label{item2:prop:fiber-holXpT} For all~$v$ in~$V$, there is an orthogonal matrix $u_v$ such that
    $\psi_v=
    \begin{pmatrix}
      u_v & 0 \\ 0 & u_v
    \end{pmatrix}
    $;
  \item\label{item3:prop:fiber-holXpT} For every 
    $a$ in~$A_2$,
    then $u_{v^+(a)} = u_{v^-(a)}$ and $u_{v^+(a)}$ commutes with $x(a)$.
  \item\label{item4:prop:fiber-holXpT} For every~$a$ in~$A_3$, $x'(a) = u_{v^+(a)} x(a)
    u_{v^-(a)}^{-1}$.
  \end{enumerate}
\end{prop}
\begin{proof}
  (\ref{item1:prop:fiber-holXpT}): By Lemma~\ref{lem:cross-ratio-m_aId}, the cross ratio of the quadruple
  $q_{v^+(a)}(x)$ is given by the matrices $x(a)$ and $x'(a)$, hence the
  equality.

  (\ref{item2:prop:fiber-holXpT}, \ref{item3:prop:fiber-holXpT}): Let $v$ be a vertex of $\Gamma_\mathcal{T}$ and $a$ such that $v^+(a) = v$. The symplectic bases $( \mathbf{e}_v, \mathbf{f}_v)$ and
  $( \mathbf{e}'_v, \mathbf{f}'_v)$ are in standard position. Applying
  Proposition~\ref{prop:standard-basis-positive-4-uple}.(\ref{item:2:prop:standard-basis-positive-4-uple}), there is an
  element~$u_v$ in $\OO(n)$ commuting with $x(a)$ and such that
  $( \mathbf{e}_v, \mathbf{f}_v) =( \mathbf{e}'_v \cdot u_v, \mathbf{f}'_v
  \cdot u_v)$.

  (\ref{item4:prop:fiber-holXpT}): Using the fact that the transition matrices with respect to the basis $(
  \mathbf{e}_v, \mathbf{f}_v)$ are deduced from the other by multiplying by
  the matrices $\psi_v$, we get the other statements.
\end{proof}

\section{Reparametrization of  the $\mathcal{X}$-coordinates}
\label{sec:other-X-like}

When $R=\emptyset$,
the space $\XplusDeltaTn$ gives a generically
finite-to-one parametrization of the space of maximal framed
local systems; this is not anymore the case when $R\neq \emptyset$. In the present section, we  introduce first another
parametrization of the space of maximal framed
local systems that is always finite-to-one.

In order to describe another parametrization for maximal representations we make use of a spanning tree $\mathcal{S}$ of the graph~$\Gamma_\mathcal{T}/A_2$,
i.e.\ the graph where all the arrows in~$A_2$ have been collapsed. Thus
$\mathcal{S}$ can be thought of as a subset of~$A_3$, though the orientation of
the arrows is not relevant here. Geometrically, the graph
$\Gamma_\mathcal{T}/A_2$ is obtained in the following way: for every (nonoriented) edge~$e$
of~$\mathcal{T}$ mark a midpoint in~$e$ and, for every triangle~$f$
of~$\mathcal{T}$, connect the three midpoints of the three sides of~$f$. Thus
the number of vertices of $\Gamma_\mathcal{T}/A_2$ is $2r-3{ \chi( \bar{S})} $ and
its number of edges is $3\sharp \mathcal{T} = 3r-6{ \chi( \bar{S})} $.

The spanning tree has to connect all the vertices, hence the cardinality of the set of edges
of~$\mathcal{S}$ is equal the number of vertices minus one, i.e.\
$\sharp \mathcal{S}=2r-3{ \chi( \bar{S})} -1$ (in a tree the number of
edges is the number of vertices minus~$1$).

We fix an arrow~$a_0$ in~$A_2$ and denote by~$\XplusSaz$ the subset of
$\XplusTn$ consisting of the tuples\index{notation}{89@$\XplusSaz$ (subspace of
  $\XplusTn$ associated with a spanning tree and an arrow)}%
\index{definition}{$\mathcal{X}$-coordinates!positive ---}%
\index{definition}{positive!$\mathcal{X}$-coordinates}
\[ y= ( \{ y(a)\}_{a\in A_2}, \{ y(a)\}_{a\in A_3}) \in  \Sym^+(n,\R)^{A_2} \times \OO(n)^{A_3}\]
such that
\begin{itemize}
\item $y(a_0)$ belongs to~$\Delta_n$;
\item for all~$a$ in~$\mathcal{S}$, $y(a)=\Id$.
\end{itemize}
Of course, the equations $y(a)=y(a')$ ($\{a,a'\}$ $2$-cycle in~$A_2$) and
$y(a_3)y(a_2)y(a_1)=\Id$ ($\{a_1, a_2, a_3\}$
cycle in~$A_3$) are satisfied.

We choose another set $\mathcal{R} \subset A_3$, disjoint from~$\mathcal{S}$ and
such that $\mathcal{S} \sqcup \mathcal{R}$ contains two of the arrows of every
$3$-cycle in~$A_3$ and a subset~$E \subset A_2$ containing~$a_0$ and containing
exactly one of the arrows of every $2$-cycle in~$A_2$. With this we have that the map
\begin{align*}
  \XplusSaz
  & \longrightarrow \Delta_n \times  \Sym^+(n,\R)^{ E
    \smallsetminus \{a_0\}} \times \OO(n)^\mathcal{R}\\
  y & \longmapsto \bigl(y({a_0}), \{y(a)\}_{a \in E
    \smallsetminus \{a_0\}}, \{y(a)\}_{a \in \mathcal{R}}\bigr)
\end{align*}
is a diffeomorphism. Recall that $\sharp ( E \smallsetminus \{a_0\}) =
r -3{ \chi( \bar{S})}  -1$ and $\sharp \mathcal{R} = 1-{ \chi( \bar{S})} $ (since
$\sharp \mathcal{S} = 2r-3{ \chi(\bar{S})}  -1$ and $\sharp ( \mathcal{S} \cup \mathcal{R}) = 2\sharp \mathcal{T} =
2r -4{ \chi(\bar{S})} $).

We have $\XplusSaz \subset \XplusTn$, hence by restriction of the map
$\holXpT$ we have a map\index{notation}{90@$\holSaz$ (restriction of $\holXpT$ to $\XplusSaz$)}
\[ \holSaz \colon \XplusSaz \longrightarrow
\Mfdelta( \Gamma_\mathcal{T}, \Sp(2n,\R)). \]

\begin{teo}
  \label{thm:holonomy-sym-Y}
  The map $\holSaz$ is onto the space of maximal framed $\delta$-twisted local
  system
  $\Mfdelta( \Gamma_\mathcal{T}, \Sp(2n,\R)) \subset
  \Locfdelta( \Gamma_\mathcal{T}, \Sp(2n,\R))$. Two elements~$y$
  and~$y'$ in $\XplusSaz$ have the same holonomy if
  and only if, $y(a_0) =y'(a_0)$ and there is an element $u\in \OO(n)$
  commuting with~$y(a_0)$ such that, for every~$a$ in~$A$,
  $y'(a) = uy(a) u^{-1}$.
\end{teo}

\section{Image and fibers of the holonomy map}
\label{sec:image-fibers-holon}

This section is devoted to the proof of Theorem~\ref{thm:holonomy-sym-Y}.

\begin{prop}
  For all $x$ in
  $\Mfdelta( \Gamma_\mathcal{T}, \Sp(2n,\R))$
  there is $y$ in $\XplusSaz$ such that
  $\holSaz(y)=x$.
\end{prop}

\begin{proof}
  By Theorem~\ref{teo:max-param-Xplus}, there is  $z\in \XplusTn$ such that
  $\holXpT(z) = x$ and $z(a_0)\in \Delta_n$. Let $(F_v, g_a, L^{t}_{v},
  L^{b}_{v}, (\mathbf{e}_v, \mathbf{f}_v))$ be the local system with
  generating symplectic basis defined by $z$.

  By induction on $\sharp \mathcal{S}'$ we prove:
  \begin{center}
    \parbox{0.8\textwidth}{
    For every subtree $\mathcal{S}'$ of $\mathcal{S}$ (connected to $a_0$),
    there is~$z' $ in $ \XplusTn$ and a generating basis
    $(\mathbf{e}'_v, \mathbf{f}'_v)$ of $(F_v, g_a, L^{t}_{v}, L^{b}_{v})$
    such that~$z'(a_0) $ belongs to~$\Delta_n$, the matrices of~$g_a$ in these bases are $\{G_a(z')\}$ and~$z'$ satisfies the following:
    for all $a\in \mathcal{S}'$, $z'(a)=\Id$.
    }
  \end{center}
  For $\mathcal{ S}'=\emptyset$, we set $z'=z$.

  Now let $\mathcal{S}' = \mathcal{S}^{\prime\prime} \cup \{b\}$ for some
  $b\in A_3 \smallsetminus \mathcal{S}^{\prime\prime}$ (and with
  $\mathcal{S}^{\prime\prime}$ a tree, i.e.\ $b$ is a leaf
  of~$\mathcal{S}'$). Let~$z^{\prime\prime}$ and
  $(\mathbf{e}^{\prime\prime}_v, \mathbf{f}^{\prime\prime}_v)$ be given by the
  induction step for $\mathcal{S}^{\prime\prime}$. We denote by $v_0$ and $v_1$
  the extremities of $b$. For definiteness, suppose that $v_0$ is connected to
  $\mathcal{S}^{\prime\prime}$ and $v_0=v^-(b)$, hence $v^+(b)=v_1$.
  In
  the case where~$v_1$ is an internal vertex,
  let $v_2$ be the vertex of~$\Gamma_{\mathcal{T}}$ connected to it by a
  cycle~$\{c,c'\}$ of~$A_2$. Since $b$ is a leaf of $\mathcal{S}'$, the vertex $v_2$ is not the
  extremity of any of the arrows in $\mathcal{S}'$, so that we can change the
  basis in $F_{v_2}$ without affecting the transitions
  matrices for $a\in \mathcal{S}'$.

  We now define the symplectic basis $(\mathbf{e}'_v, \mathbf{f}'_v)$. For
  $v\neq v_1, v_2$, set $(\mathbf{e}'_v, \mathbf{f}'_v) =
  (\mathbf{e}^{\prime\prime}_v, \mathbf{f}^{\prime\prime}_v)$. For~$v_1$, we set:
  \begin{align*}
   (\mathbf{e}'_{v_1}, \mathbf{f}'_{v_1})
   & = g_a(\mathbf{e}'_{v_0}, \mathbf{f}'_{v_0}) \cdot
     \begin{pmatrix}
       \Id & -\Id \\ \Id & 0
     \end{pmatrix}^{-1}\\
    & = g_a(\mathbf{e}^{\prime\prime}_{v_0}, \mathbf{f}^{\prime\prime}_{v_0})
      \cdot
     \begin{pmatrix}
       0 & \Id \\ -\Id & \Id
     \end{pmatrix}
     = (\mathbf{e}^{\prime\prime}_{v_1}, \mathbf{f}^{\prime\prime}_{v_1})
      \cdot
      \begin{pmatrix}
        z^{\prime\prime}(b) & -z^{\prime\prime}(b) \\ z^{\prime\prime}(b) & 0
      \end{pmatrix}
     \begin{pmatrix}
       0 & \Id \\ -\Id & \Id
     \end{pmatrix}\\
    & = (\mathbf{e}^{\prime\prime}_{v_1}, \mathbf{f}^{\prime\prime}_{v_1})
      \cdot
      \begin{pmatrix}
        z^{\prime\prime}(b) & 0 \\ 0 &z^{\prime\prime}(b)
      \end{pmatrix}
    = (\mathbf{e}^{\prime\prime}_{v_1}\cdot z^{\prime\prime}(b) ,
                                       \mathbf{f}^{\prime\prime}_{v_1} \cdot z^{\prime\prime}(b)).
  \end{align*}
  Finally, when~$v_1$ is internal, let
  $ (\mathbf{e}'_{v_2}, \mathbf{f}'_{v_2}) =
  (\mathbf{e}^{\prime\prime}_{v_2}\cdot z^{\prime\prime}(b) ,
  \mathbf{f}^{\prime\prime}_{v_2} \cdot z^{\prime\prime}(b))$. It is easily
  checked that the transition matrices are $\{G_a(z')\}$ with
  $z'\in \XplusTn$ such that
  $z'(a)= z^{\prime\prime}(a)$ for all arrows~$a$ distinct from~$b$ (and
  from~$c$ or~$c'$) and $z'(b)=\Id$ (and, in the case when $v_1$~is internal,
  $z'(c')=z'(c)= z^{\prime\prime}(b) z^{\prime\prime}(c)
  z^{\prime\prime}(b)^{-1}$ again symmetric).
\end{proof}

The uniqueness in Theorem~\ref{thm:holonomy-sym-Y} follows from:
\begin{prop}
  \label{prop:unique-Y}
  Let $(F_v, g_a, L^{t}_v, L^{b}_v)$ be a maximal framed local system. Let
  $( \mathbf{e}_v, \mathbf{f}_v)$ be a symplectic basis generating the
  framing and such that the transition matrices are $\{ G_a(y)\}$ for some
  $y$ in $\XplusSaz$. Then
  \begin{enumerate}
  \item \label{item:1:prop:unique-Y} For all $u\in \OO(n)$ commuting
    with $y(a_0)$, the symplectic basis
    $(\mathbf{e}_v \cdot u, \mathbf{f}_v \cdot u)$ generates the framing,
    the tuple $y'$ defined by
  $y'(a) = u y(a) u^{-1}$ \ep{$a\in A$} is in
    $\XplusSaz$, and the transition matrices with
    respect to this symplectic basis are~$\{ G_a(y')\}$.
  \item\label{item:2:prop:unique-Y} If $( \mathbf{e}'_v, \mathbf{f}'_v)$ is
    another generating basis and if there is an element $y'$ in
    $\XplusSaz$ such that the transition matrices are
    $\{G_a(y')\}$, then there is an element $u\in \OO(n)$, commuting
    with~$y(a_0)$ and such that, for every vertex~$v$
    in~$\Gamma_{\mathcal{T}}$,
    $(\mathbf{e}'_v, \mathbf{f}'_v) = (\mathbf{e}_v \cdot u, \mathbf{f}_v
    \cdot u)$.
  \end{enumerate}
\end{prop}

\begin{proof}
  The first part of the statement is clear. Let us address the second part.
  Applying Proposition~\ref{prop:standard-basis-positive-4-uple}, up to
  conjugating by an element commuting with~$y(a_0)$ we can assume that the
  generating bases $( \mathbf{e}_v, \mathbf{f}_v)$ and
  $( \mathbf{e}'_v, \mathbf{f}'_v)$ coincide at one of the endpoint~$v$
  of~$a_0$. We show now that they coincide at every vertex by travelling in the
  spanning tree~$\mathcal{S}$. For this it is enough to note that
  \begin{itemize}
  \item If the basis coincide at one of the extremity of an arrow~$a$ in
    $\mathcal{S}$ then they coincide at the other extremity (since the
    transition matrices coincide).
  \item Let $a\in A_2$ and suppose that the basis coincide at $v=v^-(a)$. At
    $w=v^+(a)$, since the basis are generating, we get, from the analysis in
    Lemma~\ref{lem:muT-in-loc-sys},
    \[ \Span(\mathbf{e}_v) = \Span(\mathbf{e}'_v),\ \Span(\mathbf{f}_v) =
      \Span(\mathbf{f}'_v),\ \Span(\mathbf{e}_v + \mathbf{f}_v) =
      \Span(\mathbf{e}'_v + \mathbf{f}'_v).\] This implies that there is
    $u\in \OO(n)$ such that
    $(\mathbf{e}'_v , \mathbf{f}'_v) = (\mathbf{e}_v \cdot u , \mathbf{f}_v
    \cdot u)$. Thus $y'(a) = y(a) u$ and the matrices $y(a)$, $y'(a)$
    are symmetric and positive definite. Uniqueness in the polar decomposition
    implies that $u=\Id$ and the bases coincide at~$w$. \qedhere
  \end{itemize}
\end{proof}

Similarly to Corollary~\ref{cor:finite-to-one},  we have here the following.

\begin{cor}
Generically a fiber of the map
\[\holSaz \colon \XplusSaz \longrightarrow \Mfdelta( \Gamma_\mathcal{T}, \Sp(2n,\R))\]
is finite.
\end{cor}

\section{Over-parametrizations}
\label{sec:over-param}

In the preceding section, the space of parameters has the same dimension as
the space of framed representations and the holonomy maps were generically
finite-to-one. In this section we give an (over)-parametrization which has too many parameters, but becomes injective after taking the quotient by the action of a group.

Let us denote by $\XplusSn$ the subset of
$\XplusTn$
consisting of tuples~$z$ with $z(a)=\Id$ for all~$a$ in the spanning tree~$\mathcal{S}$. We
denote by $\holSp$ the map to framed local systems; it is the restriction of $\holXpT$. It has already
been observed that $\holSp(z)$ is maximal (Lemma~\ref{lemma:holo-z-max}).\index{notation}{91@$\XplusSn$ (subspace of
  $\XplusTn$ associated with a spanning tree)}\index{notation}{92@$\holSp$ (restriction
  of $\holXpT$ to $\XplusSn$)}%
\index{definition}{$\mathcal{X}$-coordinates!positive ---}%
\index{definition}{positive!$\mathcal{X}$-coordinates}

\begin{teo}
  \label{teo:holonomy-ZS}
  The map $\holSp\colon  \XplusSn \to
  \Mfdelta( \Gamma_\mathcal{T},
  \Sp(2n,\R))$ is onto. Two tuples~$z$ and~$z'$ have the same image by $\holSp$
  if and only there is~$u$ in~$\OO(n)$ such that $z'(c) = uz(c)u^{-1}$ for
  all~$c$ in~$  A$.
\end{teo}
\begin{proof}
  Surjectivity is assured by Theorem~\ref{thm:holonomy-sym-Y} since
  $\XplusSaz \subset \XplusSn$. It is
  clear that if two elements of $\XplusSn$ are conjugate by an
  element in $\OO(n)$ then they will give rise to equivalent framed local
  systems.

  Conversely, let $z$ and $z'$ be in $\XplusSn$ such that
  $\holSp(z)=\holSp(z')$. By definition of $\XplusSn$, the elements $z(a_0)$
  and $z'(a_0)$ are positive definite matrices. This implies that there are
  orthogonal matrices~$r$ and~$r'$ such that $r z(a_0)r^{-1}$ and $r'
  z'(a_0)r^{\prime -1}$ are in~$\Delta_n$. Thus the elements $y= (r z(a)
  r^{-1})_{a\in A}$ and
  $y'= ( r z'(a)
  r^{\prime -1})_{a\in A}$ belong to $\XplusSaz$. Furthermore $\holSp( y) =
  \holSp(z)=\holSp(z') =\holSp(y')$, hence
  Theorem~\ref{thm:holonomy-sym-Y} applies and gives an element $v\in \OO(n)$
  such that $y'(a) = v y(a)v^{-1}$ for all $a\in A$. Setting $u= r^{\prime -1}
  v r$, one has thus $z'(a) = u z(a)u^{-1}$ for all $a\in A$.
\end{proof}

\chapter{Topology of the space of maximal framed representations}
\label{sec:topol-space-maxim}

From the parametrizations introduced in the previous chapter, we construct
explicit retractions of the space of maximal framed representations to subspaces which have a simpler parametrization and whose topological properties are easier to describe. This enables us to give the homotopy type of the space of
maximal framed local systems as well as some topological properties of the
space of maximal representations.

This chapter relies on the previous chapters for the parametrizations of
maximal framed twisted local systems and on
Section~\ref{sec:decor-twist-local} for their correspondence with nontwisted
local systems. Hence all results here will be expressed for the space
$\Mf(S, \Sp(2n,\R))$.

\section{Subspace of \enquote{singular} local systems}
\label{sec:subsp-enqu-repr}

We consider here subsets of maximal local systems where the holonomy elements
have as few different eigenvalues as possible.

For any framed $\SL(2,\R)$-local system $(\mathcal{F}, \sigma)$, there is an associated framed $\Sp(2n,\R)$-local system
$(\breve{\mathcal{F}}, \breve{\sigma})$ defined as follows. If
$g= \left(\begin{smallmatrix} a & b \\ c & d
\end{smallmatrix}\right)
$ belongs to~$\SL(2,\R)$ and $\ell = \Span( e\cdot x+ f\cdot y)$ belongs to $\R \PP^1$, with $(e,f)$ being the canonical
basis of~$\R^2$ and $x,y \in \R$, then we get an embedding $\SL(2,\R) \rightarrow
\Sp(2n,\R)$ (the diagonal embedding) and a $\SL(2,\R)$-equivariant application
$\R\PP^1 \to \Lag{n}$ by setting
\[ \breve{g}\coloneqq
\begin{pmatrix}
  a\Id & b\Id \\ c\Id & d\Id
\end{pmatrix}, \text{ and }
  \breve{\ell} \coloneqq \Span( \mathbf{e}\cdot x+ \mathbf{f}\cdot y).\]
The local system~$\breve{\mathcal{F}}$ is
obtained from~$\mathcal{F}$ by \enquote{composing} with the above homomorphism
and the section~$\breve{\sigma}$ is the composition of~$\sigma$ with the
natural map $\mathcal{F}_{\R\PP^1} \to \breve{\mathcal{F}}_{\Lag{n}}$ induced
by the map $\R\PP^1 \to \Lag{n}$.

Moreover  the formula
\[ \Bigl(
  \begin{pmatrix}
    a & b \\ c & d
  \end{pmatrix},u \Bigr) \longmapsto
\begin{pmatrix}
  a u & b u \\ c u & d u
\end{pmatrix}\]
defines an homomorphism $\SL_2(\R) \times \OO(n)\to \Sp(2n,\R)$ and the above
embedding $\R\PP^1 \to \Lag{n}$ is also equivariant with respect to this
homomorphism. This implies that if~$\mathcal{E}$ is a $\OO(n)$-local system on~$S$,
then $\mathcal{F} \otimes
\mathcal{E}$ is a twisted $\Sp(2n, \R)$-local system, and $\breve{\sigma}$~is a framing of this local system.

The following is easily checked:
\begin{lem}
  The framed local system
  $(\mathcal{F} \otimes \mathcal{E}, \breve{\sigma})$ is transverse with
  respect to a triangulation~${\mathcal{T}}$ if and only if
  $(\mathcal{F}, \sigma)$ is tranverse with respect to~${\mathcal{T}}$.

  The local system $(\mathcal{F} \otimes \mathcal{E}, \breve{\sigma})$ is
  maximal if and only $(\mathcal{F}, \sigma)$ is maximal.
\end{lem}

As a consequence, there is a well defined application
\[ \Mf(S, \SL(2,\R)) \times \Loc( S, \OO(n)) \to
  \Mf(S, \Sp(2n,\R))\]
whose image is denoted by $\Mf(S, \SL(2,\R)\otimes \OO(n))$. Equally there is
an application \[\M(S, \SL(2,\R)) \times \Loc( S, \OO(n)) \to
  \M(S, \Sp(2n,\R))\]
whose image is denoted by $\M(S, \SL(2,\R)\otimes \OO(n))$.\index{notation}{93@$\Mf(S,
  \SL(2,\R)\otimes \OO(n))$, $\M(S, \SL(2,\R)\otimes \OO(n))$ (subspaces of
  singular representations)}

Last let us introduce a \enquote{base point} in
$\Mf( S, \SL(2,\R))$. For a triangulation ${\mathcal{T}}$, let
$(\mathcal{F}_\mathcal{T}, \sigma_\mathcal{T})$ be the framed twisted local
system associated with the element $x$ in $\XplusDeltaT{1}$ such that, for all
$c\in  A$, $x(c)=1$.  We will denote by
\[ \DfSLdOn \subset \Mf(S, \SL(2,\R)\otimes \OO(n)) \]
the subset which is the image of
$\{ (\mathcal{F}_\mathcal{T}, \sigma_\mathcal{T}) \} \times \Loc(S, \OO(n))$.
Also we denote by\index{notation}{94@$\DfSLdOn$, $\DSLdOn$ (subspaces of singular
  representation where all the edge coordinates are $\Id$)}
\begin{equation}
  \label{eq:2}
  \DSLdOn \subset \M(S, \SL(2,\R)\otimes \OO(n))
\end{equation}
the subset which is the image of
$\{ (\mathcal{F}_\mathcal{T}) \} \times \Loc(S, \OO(n))$.

\begin{teo}
  \label{teo:hom_type}
  \begin{enumerate}
  \item For every spanning tree $\mathcal{S}$,
    $\DfSLdOn$
    \ep{respectively $\Mf(S, \SL(2,\R)\otimes \OO(n))$}
    is the image by $\holSp$ of the subset of $\XplusSn$
    consisting of the elements $z$ such that, for every
    arrow~$a$ in~$A_2$, $z(a)=\Id$ \ep{respectively $z(a)\in \R_{>0} \Id$}.

  \item The spaces
    $\DfSLdOn$ and
    $\Mf(S, \SL(2,\R)\otimes \OO(n))$ are strong
    deformation retracts of $\Mf(S, \Sp(2n,\R))$.

  \item The spaces
    $\DfSLdOn$ and
    $\Mf(S, \SL(2,\R)\otimes \OO(n))$ are respectively
    diffeomorphic to $\OO(n)^{1-{\chi( \bar{S})}}/\OO(n)$ \ep{quotient by
      simultaneous conjugation} and to
    $\R^{r-3{\chi(\bar{S})}} \times (\OO(n)^{1-{\chi(
        \bar{S})}}/\OO(n))$.
  \end{enumerate}
\end{teo}

\section{Coordinates for the subspaces}
Using an analysis similar to the one in Chapter~\ref{sec:local-systems-their}, the space
$\Loc(S, \OO(n))$ also admits parametrizations via local systems on
the quiver~$\Gamma_\mathcal{T}$. We will use the following result: the holonomy
map gives a surjective map
\[ \hol\colon  \mathcal{U} \longrightarrow \Loc( S, \OO(n))\]
where $\mathcal{U}$ is the subset of~$u$ in~$\OO(n)^{A}$  with
$u(a)=\Id$ for all $a\in A_2\sqcup \mathcal{S}$ and $u(a_3)u(a_2)u(a_1)=\Id$ for
every $3$-cycle $(a_1, a_2, a_3)$ in $A_3$.

The maps between the moduli spaces of framed local systems can be promoted to
maps between the parameters spaces:
\begin{align*}
  \XplusS{1}
  & \longrightarrow \XplusSn\\
  (z,u) & \longmapsto (z(a)u(a))_{a\in  A}\\
  \intertext{and}
  \mathcal{U}
  & \longrightarrow \XplusSn\\
  u & \longmapsto (u(a))_{a\in  A}
\end{align*}
The image of the first map will be denoted by $\XplusSLdOn$. It is the set of tuples $z$ in
$\XplusSn$ such that $z(a)\in \R_{>0} \Id$ for all $a\in A_2$;
it is isomorphic to $\R_{>0}^{r-3{\chi(\bar{S})}} \times \OO(n)^{1-{\chi(\bar{S})}}$.

The image of the second map is denoted by $\XplusFTOn$. It is the set of tuples $z$ such that $z(a)=\Id$ for all
$a\in A_2$; it is isomorphic to $ \OO(n)^{1-{\chi(\bar{S})}}$.

The above maps are also compatible with the holonomy maps to the moduli
spaces. This proves the assertion in Theorem~\ref{teo:hom_type} concerning the
parametrizations of $\Mf(S, \SL(2,\R) \otimes \OO(n))$ and
of $\DfSLdOn$ and their topologies.

\section{Retractions}
\label{sec:retractions}

Recall that a subspace $A$ of a topological space $X$ is called a \emph{strong
deformation retract} if there exists a strong deformation retraction $H\colon
X\times [0,1]\to X$, i.e.\ a continuous map~$H$ such that $H(a,t)=a$ for all
$a$ in $A$ and $t\in[0,1]$, $H(x,1)=x$ for all $x\in X$, and $H(x,0)\in A$ for
all $x\in X$.

To prove Theorem~\ref{teo:hom_type}, it is thus enough to find
$\OO(n)$-equivariant retractions of
$\XplusSn$ on
$\XplusFTOn$ and on
$\XplusSLdOn$ respectively. Or,
more concretely, to find $\OO(n)$-equivariant retractions of
$\Sym^{+}(n,\R)^{r-3{\chi( \bar{S})}} \times \OO(n)^{1-{\chi(
    \bar{S})}}$ on $\OO(n)^{1-{\chi( \bar{S})}}$ and on
$\R_{>0}^{r-3{\chi( \bar{S})}} \times \OO(n)^{1-{\chi( \bar{S})}}$.

Since the action of~$\OO(n)$ respects the decompositions into products, it all boils down to finding $\OO(n)$-equivariant retractions of
$\Sym^+(n,\R)$ on~$\{\Id\}$ and on $\R_{>0} \Id$. As the exponential
$\mathrm{exp}: \Sym(n,\R) \to \Sym^+(n,\R)$ is an $\OO(n)$-equivariant diffeomorphism, the
question translates now to finding equivariant retractions of $\Sym(n,\R)$ on~$\{0\}$ and
on $\R\Id$. The first one is given by the family of linear maps $\Sym(n,\R)\to
\Sym(n,\R)$: $\{t\Id\}_{t\in [0,1]}$. The second one is obtained similarly
using first that $\Sym(n,\R)$ is the direct sum of two $\OO(n)$-invariant
subspaces: $\R\Id$ and $\Sym_0(n,\R)$, the subspace of traceless matrices.

\section{Connected components}
\label{sec:connected-components}

An immediate consequence of Theorem~\ref{teo:hom_type} is:

\begin{cor}
\begin{itemize}
\item The space
  $\Mf(S,\Sp(2n,\R))$ has $2^{1-{\chi( \bar{S})}}$
  connected components.
\item The space
  $\Mf(S,\PSp(2n,\R))$ has $2^{1-{\chi( \bar{S})}}$ connected components
  if~$n$ is even. If~$n$ is odd, it is connected.
\end{itemize}
\end{cor}

A more general statement can be found in
Corollary~\ref{cor:number-connected-components-central}; parametrizations of
the space $\Mf(S,\PSp(2n,\R))$ are also described in
Chapter~\ref{cent_ext}.

In order to get some information about the space of maximal representations we
prove:
\begin{teo}
  \label{teo:connected-components-dec-to-max}
  If~$r=0$ \ep{i.e.\ $R=\emptyset$ and $S=\bar{S}$},
  the natural map
  \[\Mf(S, \Sp(2n, \R))\to \M(S, \Sp(2n, \R))\]
  induces a bijection at the level of connected
  components.
\end{teo}

Theorem~\ref{teo:connected-components-dec-to-max} will be proved in Section~\ref{sec:decor-sing-repr}. In this way we obtain a new proof of the  following statement:

\begin{cor}[{\cite[Theorem~4]{Strubel}}]\label{con_comp_undecor}
  If~$r=0$, the number of components of $\M(S,\Sp(2n,\R))$ is
  $2^{1-{\chi( {S})}}$.
\end{cor}

\section{Path lifting property}
\label{sec:path-lift-prop}

A continuous map $f\colon X\to Y$ is said to have \emph{the path-lifting property}
if for every continuous $\sigma\colon  [0,1]\to Y$ and every $x\in
f^{-1}(\sigma(0))$ there exists, up to reparametrization, a lift
$\tilde{\sigma}$ of~$\sigma$ starting at~$x$, namely $\tilde{\sigma}\colon [0,1]\to X$ is
continuous, $\tilde{\sigma}(0)=x$ and there is a continuous, increasing,
surjective function $\psi\colon [0,1]\to [0,1]$ such that $f\circ \tilde{\sigma}=
\sigma \circ \psi$. This is the case, for example, when $f$~is a covering;
this is also the case when, for any~$\sigma$ as above,
the space $\sigma^*(f)=\{ (t,x) \in [0,1]\times X \mid \sigma(t)=f(x)\}$ is
path-connected. A piecewise linear map between simplicial spaces that is surjective and with connected fibers
has the path lifting property.

\begin{prop}\label{prop:path-lift-prop-dec-to-max}
  If~$r=0$, the map  $\Mf(S, \Sp(2n, \R))\to \M(
  S, \Sp(2n, \R))$ has the path-lifting property.
\end{prop}

We will work in this section with spaces of representations instead of local
systems (cf.\ Section~\ref{sec:decor-sympl-local}).

Let $\Hommax(S, \Sp(2n, \R))$ and $\Homfmax(S, \Sp(2n, \R))$ the spaces of
maximal representations and of framed maximal representations (cf.\ Lemma~\ref{lem:decor-sympl-local-holonomy}). The
proposition is a direct consequence of the following lemmas.

\begin{lem}
  \label{lem:path-lift-prop-hom-to-max}
  The map $p\colon \Hommax(S, \Sp(2n, \R))\to \M(S,
  \Sp(2n,\R))$ has the path lifting property.
\end{lem}

\begin{proof}
  As these spaces are real algebraic sets, con\-nect\-ed\-ness and path-con\-nect\-ed\-ness are
  equivalent notions here. In fact, the same is true for the spaces $\sigma^*
  (p)$ for any continuous path $\sigma\colon [0,1] \to \M(S,
  \Sp(2n,\R))$. Actually the fibers of $\sigma^* (p)\to [0,1]$ are connected,
  since they are orbits for the action of the connected group $\Sp(2n,\R)$,
  and thus the total space is indeed connected.
\end{proof}

\begin{lem}
  \label{lem:path-lift-prop-dec-to-max-hom}
  The map $\Homfmax(S, \Sp(2n, \R)) \to
  \Hommax(S, \Sp(2n, \R))$ has the path lifting property.
\end{lem}

\begin{proof}
  Since we can choose independently the framing at the punctures (see
  Lemma~\ref{lem:decor-sympl-local-holonomy}), it is enough to answer
  positively the following problem:
  \begin{quote}
    Let $\mathcal{P}\subset \Sp(2n,\R)$ be the set of elements $g$ having at
    least one invariant Lagrangian and $\mathcal{Q}=\{ (g,L) \in
      \mathcal{P}\times \Lag{n}\mid g\cdot L=L\}$. Then the natural map
      $\pi\colon \mathcal{Q} \to \mathcal{P}$ has the path lifting property.
  \end{quote}
  Let $L_0$ be the base point of $\Lag{n}$ and $P$ its stabilizer (see
  Section~\ref{sec:lagr-grassm-its}). We then have a surjective map
  $\Sp(2n,\R) \times P \to \mathcal{Q} \mid (h,p) \mapsto (hph^{-1}, h\cdot
  L_0)$ and it is enough to prove that its composition with~$\pi$ has the path
  lifting property, i.e.\ that
  $\Sp(2n,\R) \times P \to \mathcal{P} \mid (h,p) \mapsto hph^{-1}$ has the
  path lifting property. This last map is algebraic so that there exists
  simplicial structures on $\Sp(2n,\R)\times P$ and on~$\mathcal{P}$ such that
  the map is piecewise linear; since it is as well surjective, it has the path
  lifting property.
\end{proof}

\section{Framing of singular representations}
\label{sec:decor-sing-repr}
In this section we prove that any singular representation has a unique framing:

\begin{prop}
  \label{prop:decor-sing-repr-uniq}
  The map from Equation~\eqref{eq:2}
  \[\DfSLdOn
  \to \DSLdOn\]
  is an isomorphism.
\end{prop}

With Proposition~\ref{prop:decor-sing-repr-uniq} we first conclude the proof of Theorem~\ref{teo:connected-components-dec-to-max}:
\begin{proof}[Proof of Theorem~\ref{teo:connected-components-dec-to-max}]
  The map
  \[ \pi_0\bigl( \Mf( S, \Sp(2n,\R))\bigr) \longrightarrow \pi_0\bigl(
    \M(S, \Sp(2n,\R))\bigr)\]
  is surjective since $\Mf( S, \Sp(2n,\R)) \rightarrow
  \M(S, \Sp(2n,\R))$ is surjective. Let us prove its
  injectivity. By Theorem~\ref{teo:hom_type}, the map
  \begin{align*}
   \pi_0\bigl( \DfSLdOn\bigr) & \longrightarrow \pi_0\bigl( \Mf( S, \Sp(2n,\R))\bigr)\\
\intertext{  is an isomorphism and by Proposition~\ref{prop:decor-sing-repr-uniq}}
   \pi_0\bigl( \DfSLdOn\bigr) &\longrightarrow \pi_0\bigl( \DSLdOn \bigr)\\
    \intertext{is also bijective. Hence we are reduced to
  show that}
   \pi_0\bigl( \DSLdOn \bigr)
    &\longrightarrow \pi_0\bigl( \M(S, \Sp(2n,\R))\bigr)
  \end{align*}

  is injective.

  Assume that the two representations $\rho_0, \rho_1 \in \DSLdOn$ are connected by a
   path $\sigma\colon [0,1] \to \M(S, \Sp(2n,\R))$. By
  Proposition~\ref{prop:decor-sing-repr-uniq}, there is a
  lift~$\tilde{\sigma}$ (of a reparametrization) of~$\sigma$ to
  $\Mf( S, \Sp(2n,\R))$. The elements~$\tilde{\sigma}(0)$
  and~$\tilde{\sigma}(1)$ are the unique framings of~$\rho_0$ and~$\rho_1$ respectively. Applying Theorem~\ref{teo:hom_type} again, we get a continuous
  path in $\DfSLdOn$
  between $\tilde{\sigma}(0)$ and $\tilde{\sigma}(1)$, as a result~$\rho_0$
  and~$\rho_1$ belong to the same connected component of
  $\DSLdOn$. This concludes
  the injectivity and the proof.
\end{proof}

\begin{proof}[Proof of Proposition~\ref{prop:decor-sing-repr-uniq}]
  Let $\rho \in \DSLdOn$. We have to prove that, for each loop~$c$
  representing a boundary component, $\rho(c)$ has a unique invariant Lagrangian. By
  construction (see Section~\ref{sec:other-X-like}) and up to conjugation,
  $\rho(c)$ is a finite product of elements of the form
  \[  \begin{pmatrix} 0 & -\Id \\ \Id & 0 \end{pmatrix} \begin{pmatrix} -C &
      C \\ -C & 0 \end{pmatrix} = \begin{pmatrix} C &
      0 \\ -C & C \end{pmatrix}  = \begin{pmatrix} C &
      0 \\ 0 & C \end{pmatrix}  \begin{pmatrix} \Id &
      0 \\ -\Id & \Id \end{pmatrix}
  \]
  with $C\in \OO(n)$. Hence it is equal to
  \[ \begin{pmatrix} \Id &
      0 \\ m\Id & \Id \end{pmatrix}  \begin{pmatrix} B &
      0 \\ 0 & B \end{pmatrix}  \]
  for some $m\neq 0$ and $B\in \OO(n)$. Up to taking the inverse and conjugate
  by a block diagonal matrix, we can assume that $m=1$. The powers of the
  above matrix $M$ are then
  \[ M^k=\begin{pmatrix} \Id &
      0 \\ k\Id & \Id \end{pmatrix}  \begin{pmatrix} B^k &
      0 \\ 0 & B^k \end{pmatrix}= \begin{pmatrix} B^k &
      0 \\ kB^k & B^k \end{pmatrix}. \]

  Let us write elements in $\R^{2n}$ as pairs $(x,y)$ with $x$, $y$ in $\R^n$
  so that the symplectic form has the expression $\omega( (x,y), (x',y')) =
  {}^T\! x y' - {}^T\! y x'$.

  Suppose that~$L$ is an $M$-invariant Lagrangian. Suppose that there is~$x$
  in~$\R^n$ with $(x,0)$ in~$L$. For every $k\geq 0$, $M^k(x,0)= (B^k x, kB^k
  x)$ belongs to the Lagrangian~$L$ and the relation $\omega(M^k(x,0),
  M^\ell(x,0) )=0$ ($k$, $\ell$ in $\N$) says $(\ell-k) {}^T\! x
  B^{\ell-k}x=0$. It follows that, denoting~$V$  the $B$-invariant subspace of~$\R^n$ generated by $B^k
  x$ ($k\geq 0$), $B(V)$ is orthogonal to $V$. This
  is possible only if $V=0$ and $x=0$.

  Let $( \mathbf{e}, \mathbf{f})$ denote the standard symplectic basis of
  $\R^{2n}$.
  The Lagrangian $L$ is thus transverse to $\Span(\mathbf{e})$ and there
  exists a unique symmetric matrix $P$ such that $L=\Span(
  \mathbf{e}P+\mathbf{f})$. The equation $M(L)=L$ translate into $\Span(
  \mathbf{e}BP+\mathbf{f}B(P+\Id))=\Span(
  \mathbf{e}P+\mathbf{f})$ which implies that $P+\Id$ is invertible (hence $-1$
  does not belong to the spectrum of~$P$) and the equality $PBP+B=BP$. From
  this last equality, if $x$ is an eigenvector of $P$ for the eigenvalue
  $\lambda$, then~$Bx$ is an eigenvector of~$P$ for the eigenvalue
  $\lambda/(\lambda+1)$. Thus the spectrum of~$P$ is invariant by the Möbius
  transformation $\lambda\mapsto \lambda/(\lambda+1)$. However the only finite
  invariant set for this transformation is $\{0\}$ and we obtain that
  $P=0$. It turns out that $L=\Span(\mathbf{f})$, establishing the uniqueness.
\end{proof}

\chapter[Singularities of the space of framed maximal representations]{Singularities of the space of framed maximal representations into $\Sp(4,\mathbf{R})$}
\label{sec:singularities}

An interesting feature of the spaces of maximal representations is that they
have singularities. This contrasts with the situation studied by Fock and
Goncharov where the positive part of their moduli spaces is diffeomorphic to
$\R^n$. We will now turn our attention to the study of these
singularities. This problem seems rather intricate in general, so we restrict
our attention to the first interesting case: we analyze the singularities of
the space of framed maximal $\Sp(4,\R)$-local systems. In this first case,
orbifold singularities as well as non-orbifold singularities appear, and we can completely understand them. We show that the singular locus corresponds exactly to
local systems into subgroups that we describe first. The results in this chapter can be compared with the results in Alessandrini--Collier~\cite{AC}, who studied the singularities of the spaces of maximal representations of closed surfaces in $\Sp(4,\R)$.

\section{Space of block diagonal local systems}
\label{sec:anth-natur-subsp}

If $(e_1,e_2, f_1, f_2)$ denotes the canonical basis of~$\R^4$, the two
subspaces $V=\Span(e_1, f_1)$, $W=\Span(e_2, f_2)$ are in direct sum and this
decomposition induces a homomorphism
\[ \SL(2,\R) \times \SL(2,\R) \longrightarrow \Sp(4, \R)\]
as wells as an equivariant map $\R\PP^1 \times \R\PP^1 \to \Lag{2}$.
In particular we get an induced map
\[ \Mf( S, \SL(2,\R)) \times \Mf( S, \SL(2,\R)) \to
  \Mf( S, \Sp(4,\R))\]
which is a two-to-one ramified covering on its image. We will denote the image of this map
by $\Mf( S, \SL(2, \R)\times \SL(2,\R) )$.\index{notation}{96@$\Mf( S, \SL(2, \R)\times
  \SL(2,\R) )$ (subspace of $\Mf(S, \Sp(4,\R))$ coming from
  $\SL(2,\R)\times\SL(2,\R)\to \Sp(4,\R)$)}

The subspaces
$\Mf( S, \SL(2, \R) )$ and
$\Mf( S, \SL(2, \R)\otimes \OO(2) )$ were defined in Section~\ref{sec:subsp-enqu-repr}. Similarly, we can as well define
$\Mf( S, \SL(2, \R)\otimes \SO(2) )$.

\section{Singular points}
\label{sec:singular-points}

We will prove the following statement:

\begin{teo}
  \label{teo:singular-points-sp4}
  A point $x $ of the space $\Mf(S, \Sp(4,\R))$ is
  \begin{itemize}
  \item a smooth point if and only if $x$ does not belong to the union \[\Mf(
    S, \SL(2, \R)\times \SL(2,\R) ) \cup \Mf( S, \SL(2, \R)\otimes
    \SO(2) ).\]
  \item an orbifold point with isotropy group $\Z/2\Z$ if and only if $x$
    belongs to
    \[\Mf( S, \SL(2, \R)\times
      \SL(2,\R) ) \smallsetminus \Mf( S, \SL(2, \R) \otimes
    \SO(2));\] in which
    case a neighborhood of~$x$ is isomorphic to a neighborhood of~$0$ in
    $\R^{2r-6{\chi( \bar{S})}} \times ( \R^{r-4{\chi( \bar{S})} } / \{ \pm \Id\})$;
  \item a non-orbifold singular point if and only if the point~$x$ belongs to the
    connected subspace
    $\Mf(S, \SL(2,\R) \otimes \SO(2))$; furthermore
    \begin{itemize}
    \item if $x$ belongs to
      $\Mf(S, \SL(2,\R) \otimes \SO(2)) \smallsetminus
      \Mf(S, \SL(2,\R))$, then a neighborhood of~$x$ is isomorphic to a
      neighborhood of $0$ in the space
      $\R^{r-4{\chi( \bar{S})}+1} \times ( \C^{r-3{\chi( \bar{S})}}/\SO(2))$, where an element
      $\bigl(\begin{smallmatrix}
       \cos \theta & -\sin \theta \\
       \sin \theta & \cos \theta
       \end{smallmatrix}\bigr)\in\SO(2)$
      acts on $\C^{r-3{\chi( \bar{S})}}$ by simultaneous multiplication by $e^{2i\theta}$ in every factor.

    \item if $x$ belongs to $ \Mf(S, \SL(2,\R))$, then a neighborhood
      of~$x$ is isomorphic to a neighborhood of~$0$ in the quotient space
      $\R^{r-3{\chi( \bar{S})} } \times \bigl( \R^{1-{\chi( \bar{S})}}
      \times ( \C^{r-3{\chi( \bar{S})}}/\SO(2))\bigr)/\sigma $
      where~$\sigma$ is the involution on the space
      $\R^{1-{\chi( \bar{S})}} \times ( \C^{r-3{\chi( \bar{S})}
        }/\SO(2))$ induced by $-\Id$ on the factor
      $\R^{1-{\chi( \bar{S})}}$ and the complex conjugation on the factor
      $ \C^{r-3{\chi( \bar{S})}}/\SO(2)$.
    \end{itemize}
  \end{itemize}
\end{teo}

\begin{rem}
  \label{rem:singular-points-link}
  The claims about the singularity types in the above theorem are consequences
  of the following considerations:
  \begin{enumerate}
  \item The \emph{join}\index{definition}{join} $X\bigcurlyvee Y$ of two topological spaces is the
    quotient of $X\times [0,1]\times Y$ by the equivalence relation whose
    classes are $\{x\}\times \{0\}\times Y$, $X\times \{1\}\times \{y\}$, and
    $\{(x,t,y)\}$ ($x\in X$, $t\in (0,1)$, $y\in Y$), cf.\
    \cite[p.~9]{Hatcher}. If a neighborhood of a
    point~$a$ in a space~$A$ is the cone over~$X$ (i.e.\ the quotient of
    $[0,1]\times X$ by the equivalence relation whose only nontrivial class
    is $\{0\}\times X$), and a neighborhood of~$b\in B$ is the cone over~$Y$,
    then a neighborhood of $(a,b)$ in $A\times B$ is the join $X\bigcurlyvee
    Y$. Joins of spheres are spheres.

  \item   The join can also be written as $U \cup V$, where~$U$ is the image in $X\bigcurlyvee
    Y$ of $X\times
    [0,1)\times Y$, and~$V$ is the image of $X\times (0,1]\times Y$. The
    space $X$ is a deformation retract of~$U$, $Y$ is a deformation retract
    of~$V$, and $X\times Y$ is a deformation retract of $U\cap V$. From this
    and the Mayer--Vietoris long exact sequence, we get the following formula,
    relating the Euler characteristics~$\chi_F$ over a field~$F$ of these
    spaces:
    \[ \chi_F \bigl( X\bigcurlyvee Y\bigr) = -\chi_F(X) \chi_F(Y) +\chi_F(X)+\chi_F(Y).\]
  \item A neighborhood of~$0$ in $\R^{-4{\chi(S)}} / \{ \pm \Id\}$ is the 
    cone over $\R\PP^{-4{\chi(S)}-1}$ so that a neighborhood of~$0$ in
    $\R^{-6{\chi(S)}} \times ( \R^{-4{\chi(S)}} / \{ \pm \Id\})$ is the
    join $S^{-6{\chi(S)}-1} \bigcurlyvee \R\PP^{-4{\chi(S)}-1}$.  The
    above formula (and the knowledge of the cohomology of real projective
    spaces) can be used to show that the Euler characteristic of this join
    varies with the field~$F$ and therefore it cannot be homeomorphic to a
    sphere. Thus no neighborhood of the points in the second item of the
    theorem is homeomorphic to a
    manifold.
  \item By similar arguments, for the last case in the theorem, the
    neighborhood of the singularity is either the join of a sphere and the
    projective space $\C\PP^{-3{\chi(S)}-1}$ or a quotient of this join by
    $\Z/2\Z$. From this, again by cohomological argument, it can be seen that
    the singularity is not homeomorphic to an orbifold singularity.
  \end{enumerate}
\end{rem}

\section{Parameters}
\label{sec:parameters}

Analogous to the map between moduli spaces mentioned in Section~\ref{sec:anth-natur-subsp}, there is a map between the parameter spaces of
Theorem~\ref{teo:holonomy-ZS}:
\[ \XplusS{1} \times \XplusS{1} \to \XplusS{2}. \]
Writing $\XplusS{1} \simeq
\R_{>0}^{r-3{\chi(\bar{S})} } \times \OO(1)^{1-{\chi( \bar{S})}}$ and similarly for
$\XplusS{2}$, the map is explicitly given by
\begin{align*}
  \bigl( \R_{>0}^{r-3{\chi( \bar{S})} } \times \OO(1)^{1-{\chi(
  \bar{S})}} \bigr)^2
  & \longrightarrow  \Sym^{+}(2,\R)^{r-3{\chi( \bar{S})}
    }\times \OO(2)^{1-{\chi( \bar{S})}}\\
  ( \{\lambda_{j,i}\}_i, \{\epsilon_{j,\ell}\}_\ell
  )_{j=1,2}
  & \longmapsto \bigl( \{\diag(\lambda_{1,i},\lambda_{2,i})\}_i,
    \{\diag(\epsilon_{1,\ell},\epsilon_{2,\ell})\}_\ell\bigr).
\end{align*}
From this, an element $z=( \{s_i\}_i, \{r_\ell\}_{\ell})$ in
$\XplusS{2}$ represents a framed representation
in $\Mf(S, \SL(2,\R) \times \SL(2,\R))$ if, up to replacing $z$ by
$r\cdot z$ with~$r$ in~$\OO(2)$, all the coordinates~$s_i$, and~$r_\ell$ of~$z$ are diagonal.

In approaching the above theorem, it is useful to consider different
coordinates for the space of symmetric positive matrices. The map
\begin{align*}
  \R \times \C & \longrightarrow  \Sym^{+}(2,\R)\\
  (t, a+ib) & \longmapsto \exp\biggl( t\Id +
              \Bigl(\begin{matrix}
                a & b \\ b & -a
              \end{matrix}\Bigr)
\biggl)
\end{align*}
is a diffeomorphism. Via this isomorphism, the group~$\OO(2)$ acts on
$\R\times \C$. The action is trivial on the factor~$\R$ whereas on the factor~$\C$
\begin{itemize}
\item the element $
  \bigl(\begin{smallmatrix}
    \cos \theta & -\sin \theta \\ \sin \theta & \cos \theta
  \end{smallmatrix}\bigr)
$ acts by multiplication by $e^{2i\theta}$,
\item the element $
  \bigl(\begin{smallmatrix}
    1 & 0 \\ 0 & -1
  \end{smallmatrix}\bigr)
  $ acts by $z\mapsto \bar{z}$.
\end{itemize}

The space
$\XplusS{2}$ is then $\OO(2)$-isomorphic to the product space $
\R^{r-3{\chi( \bar{S})} } \times \C^{r-3{\chi( \bar{S})} }
\times \OO(2)^{1-{\chi( \bar{S})}}$.
Recapping the above and Section~\ref{sec:subsp-enqu-repr} we get:
\begin{prop}
  \label{prop:parameters-sbg-sp4}
  Let $z$ be an element
  \[( \{t_i\}_i, \{z_j\}_j, \{r_\ell\}_\ell) \in \R^{r-3{\chi( \bar{S})} }
  \times \C^{r-3{\chi( \bar{S})} }
  \times \OO(2)^{1-{\chi( \bar{S})}} \simeq \XplusS{2},\] and
  let~$x$ be the corresponding framed maximal representation:
  $x\in \Mf(S, \Sp(4,\R))$. Then
  \begin{enumerate}
  \item\label{item1:prop:parameters-sbg-sp4} $x$ belongs to
    $\Mf(\SL(2,\R))$ if and only if, for all~$j$, $z_j=0$, and, for
    all~$\ell$, $r_\ell = \pm \Id$;
  \item\label{item2:prop:parameters-sbg-sp4} $x$ belongs to $\Mf(\SL(2,\R) \otimes \SO(2)) \smallsetminus
    \Mf(\SL(2,\R))$ if and only if, for all~$j$, $z_j=0$, and
    $\{r_\ell\}_\ell \in \SO(2)^{1-{\chi( \bar{S})}} \smallsetminus \{\pm
    \Id\}^{1-{\chi( \bar{S})}}$;
  \item\label{item3:prop:parameters-sbg-sp4} $x$ belongs to $\Mf(\SL(2,\R) \times \SL(2,\R)) \smallsetminus
    \Mf(\SL(2,\R) \otimes \SO(2))$ if and only if, there is
    $\theta\in \R$ such that, for all~$j$, $z_j \in \R e^{i\theta}$, for
    all~$\ell$, $r_\ell \in \left\{ \pm\Id, \pm \bigl(
      \begin{smallmatrix}
        \cos \theta & \sin \theta \\ \sin \theta & -\cos \theta
      \end{smallmatrix}
\bigr)\right\}$, and either $\{z_j\}_j \neq 0$ or $\{r_\ell\}_\ell \notin
\{\pm \Id\}^{1-{\chi( \bar{S})}}$.
  \item\label{item4:prop:parameters-sbg-sp4} otherwise $x$ does not belong to the union $\Mf(
    S, \SL(2, \R)\times \SL(2,\R) ) \cup \Mf( S, \SL(2, \R)\otimes
    \SO(2) )$.
  \end{enumerate}
\end{prop}

In particular, the space $\Mf(\SL(2,\R) \otimes \SO(2))$ is the image of the
connected set $\R^{r-3{\chi( \bar{S})} }
  \times \{0\}
  \times \SO(2)^{1-{\chi( \bar{S})}}$ and is thus itself connected. This
  proves the connectedness claim in the third item of Theorem~\ref{teo:singular-points-sp4}.

\begin{rem}
  \label{rem:z-with-theta=0}
  Note that, in case (\ref{item3:prop:parameters-sbg-sp4}), the element $x$ is the holonomy of the element
  $z' = \bigl(
  \begin{smallmatrix}
    \cos \theta/2 & \sin \theta/2 \\ -\sin \theta/2 & \cos \theta/2
  \end{smallmatrix}
  \bigr)\cdot z = ( \{t'_i\}, \{z'_j\}, \{r'_\ell\})$ with $z'_j\in \R$, and
  $r'_\ell \in \bigl\{ \pm\Id, \pm \bigl(
  \begin{smallmatrix}
    1 & 0 \\ 0 & -1
  \end{smallmatrix}
  \bigr)\bigr\}$.
\end{rem}

\section{Stabilizers}
\label{sec:stabilizers}

We are going to calculate stabilizers of elements in
$\R^{r-3{\chi( \bar{S})} } \times \C^{r-3{\chi( \bar{S})} } \times
\OO(2)^{1-{\chi( \bar{S})}} \simeq
\XplusS{2}$. In the calculation we make use of the
following observations:
\begin{itemize}
\item the stabilizer (here the centralizer) of an element $s\in \OO(2)$ is
  equal to~$\OO(2)$ if and only if $s=\pm \Id$;
\item the stabilizer of an element $s\in \OO(2)$ is equal to~$\SO(2)$ if and
  only if $s\in \SO(2) \smallsetminus\pm \{\Id\}$;
\item the stabilizer of an element $s\in \OO(2)$ is finite if and
  only if $s\in \OO(2) \smallsetminus \SO(2)$; in this case the stabilizer
  is equal to $\{ \pm\Id, \pm s\}$;
\item The stabilizer of an element $z\in \C$ is:
  \begin{itemize}
  \item equal to $\OO(2)$ if $z=0$;
  \item equal to $\left\{ \pm\Id, \pm \bigl(
      \begin{smallmatrix}
        \cos \theta & \sin \theta \\ \sin \theta & -\cos \theta
      \end{smallmatrix}
      \bigr)\right\}$ if $z\in \R^* e^{i\theta}$.
  \end{itemize}
\end{itemize}

As a consequence:
\begin{prop}
  \label{prop:stabilizers-ZS-sp4}
  Let
  $z= (\{t_i\}_i, \{z_j\}_j, \{r_\ell\}_\ell) \in
  \XplusS{2}$. Then the stabilizer of~$z$
  in~$\OO(2)$ is equal to
  \begin{enumerate}
  \item\label{item1:prop:stabilizers-ZS-sp4} $\OO(2)$ if and only if, for all~$j$, $z_j=0$, and, for all~$\ell$,
    $r_\ell=\pm \Id$.
  \item\label{item2:prop:stabilizers-ZS-sp4} $\SO(2)$ if and only if, for all~$j$, $z_j=0$, and  $\{r_\ell\}_\ell
    \in \SO(2)^{1-{\chi( \bar{S})}} \smallsetminus \{\pm \Id\}^{1-{\chi( \bar{S})}}$.
  \item\label{item3:prop:stabilizers-ZS-sp4} the finite group $G_\theta= \left\{ \pm\Id, \pm \bigl(
      \begin{smallmatrix}
        \cos \theta & \sin \theta \\ \sin \theta & -\cos \theta
      \end{smallmatrix}
      \bigr)\right\}$ \ep{for some $\theta\in\R$} if and only if, for all~$j$,
    $z_j \in \R e^{i\theta}$, and, for all~$\ell$, $r_\ell \in G_\theta$, and if either
    at least one of the~$z_j$ is non-zero or at least one of the $r_\ell$ is not
    $\pm\Id$.
  \item\label{item4:prop:stabilizers-ZS-sp4} $\{ \pm\Id\} $ otherwise.
  \end{enumerate}
\end{prop}

Note that $-\Id$ acts trivially on~$\XplusS{2}$,
so the above stabilizers are more meaningful in the quotient group $\OO(2)/
\{\pm \Id\}$.

Observe also that the group~$\OO(2)$ acts diagonally on the product
$\R^{r-3{\chi( \bar{S})} } \times \C^{r-3{\chi( \bar{S})} } \times
\OO(2)^{1-{\chi( \bar{S})}}$ and that the action
is trivial on the first factor. Hence
\[ \XplusS{2} /\OO(2) = \R^{r-3{\chi(
      \bar{S})} } \times \bigl(
  \C^{r-3{\chi( \bar{S})} } \times \OO(2)^{1-{\chi( \bar{S})}}\bigr)/\OO(2).\]

\section{Singularity types}
\label{sec:singularity-types}
To conclude the proof of
Theorem~\ref{teo:singular-points-sp4}, it now remains to analyze the quotient
singularities in each case of Proposition~\ref{prop:parameters-sbg-sp4}.
Let $z=( \{t_i\}_i, \{z_j\}_j, \{r_\ell\}_\ell)$ be in
$\XplusS{2}$ and let $x$ be the corresponding
framed representation in $\Mf(S, \Sp(4,\R))$.

\subsection{Case~(\ref{item4:prop:parameters-sbg-sp4}) of Proposition~\ref{prop:parameters-sbg-sp4}}
\label{sec:case-4-proposition}

By
Proposition~\ref{prop:stabilizers-ZS-sp4}.(\ref{item4:prop:stabilizers-ZS-sp4}),
the stabilizer of~$z$ in $\OO(2)/\{ \pm \Id\}$ is trivial. Therefore a
neighborhood of~$x$ in $\Mf(S, \Sp(4,\R))$ is isomorphic to a neighborhood
of~$z$ in $\XplusS{2}$ and is thus a smooth point.

\subsection{Case~(\ref{item3:prop:parameters-sbg-sp4}) of Proposition~\ref{prop:parameters-sbg-sp4}}
\label{sec:case-3-proposition}

One can assume that~$\theta=0$ (see the remark~\ref{rem:z-with-theta=0}). We
treat this case under the assumption that
$\{r_\ell\}_\ell \notin \{ \pm\Id\}^{1-{\chi( \bar{S})}}$ and  $\{z_j\}_j =0$; the reasoning is similar under the assumption
that $\{z_j\}_j \neq 0$.

 By Proposition~\ref{prop:stabilizers-ZS-sp4}.(\ref{item3:prop:stabilizers-ZS-sp4}),
the stabilizer of~$z$ in $\OO(2)/\{ \pm \Id\}$ is isomorphic to the group
generated by the involution $\bigl(
\begin{smallmatrix}
  1 & 0 \\ 0 & -1
\end{smallmatrix}
\bigr)$. Therefore a
neighborhood of~$x$ in $\Mf(S, \Sp(4,\R))$ is isomorphic to a neighborhood
(of the class) of~$z$ in
\[\XplusS{2} / \{ \pm \bigl(
\begin{smallmatrix}
  1 & 0 \\ 0 & -1
\end{smallmatrix}
\bigr)\}.\]

Consider the following $\OO(2)$-equivariant transformations, for every $1\leq
i<j\leq 1-\chi(S)$, on the factor
$\OO(2)^{1-{\chi( \bar{S})}}$:
\begin{align*}
  \OO(2)^{1-{\chi( \bar{S})}} & \longrightarrow \OO(2)^{1-{\chi( \bar{S})}}\\
  (a_1, \dots, a_i, \dots, a_j,\dots, a_{1-\chi(S)}) & \longmapsto   (a_1, \dots, a_{i-1},
                                        a_i a_j, a_{i+1}, \dots, a_j,\dots, a_{1-\chi(S)})\\
  \intertext{or}
    (a_1, \dots, a_i, \dots, a_j,\dots, a_{1-\chi(S)}) & \longmapsto   (a_1, \dots, a_{i-1},
                                         a_j, a_{i+1}, \dots, a_i,\dots, a_{1-\chi(S)}).
\end{align*}

Since the stabilizer in~$\OO(2)$ of a tuple $(a_1, \dots, a_{1-\chi(S)})$ is
the centralizer of the group generated by the elements~$a_i$, these
transformations do not change the stabilizers.

Applying these transformations as many times as needed we can assume that $r_1=\pm \bigl(
\begin{smallmatrix}
  1 & 0 \\ 0 & -1
\end{smallmatrix}
\bigr)$ and $r_\ell= \pm\Id$ for all $\ell>1$.

We can parametrize a neighborhood of $z$ in
$\XplusS{2}$ by the map
\[
  \Psi\colon  \R^{r-3{\chi( \bar{S})} }\times \CC^{r-3{\chi( \bar{S})}
  } \times \R^{1-{\chi( \bar{S})}}
   \longrightarrow \XplusS{2}\]
defined as
  \[\Psi( \{\tau_i\}_i, \{\zeta_j\}_j, \{\theta_\ell\}_\ell)
   \coloneqq \Bigl( \{t_i+\tau_i\}_i, \{\zeta_j\}_j, \bigl\{  \bigl(
    \begin{smallmatrix}
      \cos \theta_\ell & -\sin \theta_\ell\\ \sin \theta_\ell & \cos \theta_\ell
    \end{smallmatrix}
\bigr) r_\ell\bigr\}_\ell \Bigr).
\]
In these coordinates, the action of $\bigl(
\begin{smallmatrix}
  1 & 0 \\ 0 & -1
\end{smallmatrix}
\bigr)$ is trivial on the first factor $\R^{r-3{\chi( \bar{S})} }$ and is
given by the complex conjugation on $\C^{r-3{\chi( \bar{S})} }$ and by $-\Id$ on
$\R^{1-{\chi( \bar{S})}}$. The locally free action of~$\SO(2)$ on a neighborhood of~$z$ in
$\XplusS{2}$ comes from a free action of~$\R$ on
these coordinates: for~$\theta\in\R$, the action of~$\theta$ is given by
\[ ( \{\tau_i\}_i, \{\zeta_j\}_j, \{\theta_\ell\}_\ell)
  \longmapsto ( \{\tau_i\}_i, \{\zeta_j+2i\theta\}_j, \{\theta_1+\theta,\theta_\ell\}_{\ell>1}).\]
We can then work on the hyperplane $\theta_1=0$ and grouping furthermore the
real and imaginary parts of the parameters in~$\C$, we conclude that a
neighborhood of~$x$ in $\Mf(S, \Sp(4,\R))$ is diffeomorphic to a
neighborhood of~$0$ in
\[ \R^{2r-6{\chi( \bar{S})} } \times \bigl( \R^{r-4{\chi( \bar{S})} }/\{\pm\Id\}\bigr).\]

\subsection{Case~(\ref{item2:prop:parameters-sbg-sp4}) of Proposition~\ref{prop:parameters-sbg-sp4}}
\label{sec:case-2-proposition}

In that case, since the stabilizer of~$z$ is $\SO(2)$, a neighborhood of $x$
in $\Mf(S, \Sp(4,\R))$ is isomorphic to a neighborhood of~$z$ in
$\XplusS{2}/\SO(2)$. Since, for all~$\ell$,
$r_\ell$ is in $\SO(2)$, we can restrict to the subspace $\R^{r-3{\chi(
    \bar{S})}} \times
\C^{r-3{\chi( \bar{S})} } \times \SO(2)^{1-{\chi( \bar{S})}}$. As~$\SO(2)$ is abelian, we can conclude
that a neighborhood of~$x$
in $\Mf(S, \Sp(4,\R))$ is isomorphic to a neighborhood of $(0,\dots, 0, \Id, \dots, \Id)$ in the space
\[ \R^{r-3{\chi( \bar{S})} } \times
\bigl(\C^{r-3{\chi( \bar{S})} }/\SO(2)\bigr) \times \SO(2)^{1-{\chi( \bar{S})}},\]
and this is what we wanted to prove.

\subsection{Case~(\ref{item1:prop:parameters-sbg-sp4}) of Proposition~\ref{prop:parameters-sbg-sp4}}
\label{sec:case-1-proposition}

The situation is similar to the previous case, but we have to take care of the
remaining action by $\OO(2)/\SO(2)\simeq\Z/2\Z$. Precisely, we conclude that a
neighborhood of $x$ in $\Mf(S, \Sp(4,\R))$ is isomorphic to a
neighborhood of $0$ in the quotient of the space
\[ \R^{r-3{\chi( \bar{S})} } \times
\bigl(\C^{r-3{\chi( \bar{S})} }/\SO(2)\bigr) \times \R^{1-{\chi( \bar{S})}},\]
by the involution whose action is trivial on the factor $\R^{r-3{\chi(
    \bar{S})} }$, is the
action induced by the complex conjugation on the factor $\C^{r-3{\chi( \bar{S})}}/\SO(2)$, and is $-\Id$ on $\R^{1-{\chi( \bar{S})}}$.

\chapter{General $\mathcal{X}$-coordinates}
\label{sec:generalX}

In this chapter we introduce general (i.e.\ not necessarily positive)
$\mathcal{X}$-coordinates with respect to a chosen ideal triangulation
$\mathcal{T}$ of~$S$. This relies on the analysis of pairs of nondegenerate
quadratic forms which is given in the appendix. We obtain a generically
finite-to-one parametrization of the space of transverse framed
local systems (Section~\ref{sec:from-coord-repr-1}) as well as a kind of
cellular decomposition of that space (Section~\ref{sec:pieces}). Relaxing the
condition on the parameter space (Section~\ref{sec:over-parametrization}), we
deduce topological results for the space of transverse framed
local systems (Section~\ref{sec:connected-components-1}).

\section{Pairs of real symmetric forms}
\label{sec:pairs-real-symmetric}

In Appendix~\ref{sec:normal-form-pair} we establish a number of results
about pairs of quadratic forms that we recap here. The reader might want to
read the appendix first where the statements are obtained more progressively.

\subsection{Parameter space}
\label{sec:parameter-space}

We denote by~$\mathcal{D}(n)$\index{notation}{98@$\mathcal{D}(n)$ (a set of
  \enquote{partitions} of the integer~$n$)}%
\index{definition}{path-lifting property}%
\index{definition}{property (path-lifting ---)}
the set of quintuples $\mathbf{n} = ( \{ \underline{n}_{x}\}_{x\in \{\pm 1\}^2}, 2\underline{m})$ of sequences of integers $\underline{n}_x =(n_{x,j})_{j=1,\dots, k_x} \in
  (\Z_{>0})^{k_x}$ (for $x\in \{\pm 1\}^2$) and
  $2\underline{m} =(2m_{j})_{j=1,\dots, k_0}\in (2\Z_{>0})^{k_0}$ of lengths $k_x$ ($x\in \{\pm 1\}^2$)
  and~$k_0$ in~$\Z_{\geq 0}$, such that the sum of all the integers in the~$5$ sequences in~$\mathbf{n}$ is
  equal to $n$.

There is an involution $\iota:\mathcal{D}(n) \rightarrow \mathcal{D}(n)$ defined as
\[\iota( \{ \underline{n}_{( \varepsilon, \eta)}\}_{ ( \varepsilon, \eta)\in \{\pm 1\}^2},
2\underline{m}) = ( \{
\underline{n}_{( \eta, \varepsilon)}\}_{( \varepsilon,
  \eta)\in \{\pm 1\}^2}, 2\underline{m}).\]

\medskip

We will denote $\mathcal{E}(n)$ the space of pairs $(\mathbf{n},
\boldsymbol{\lambda})$ with\index{notation}{100@$\mathcal{E}(n)$ (parameter space for
  pairs of quadratic forms)}\index{notation}{101@$\iota$ (the involution on
  $\mathcal{D}(n)$ and on $\mathcal{E}(n)$)}
\begin{itemize}
\item $\mathbf{n} \in \mathcal{D}(n)$, and
\item $\boldsymbol{\lambda}$ a quintuple
  $(\{ \underline{\lambda}_x\}, \underline{\lambda})$ of
   decreasing sequences
  $\underline{\lambda}_x =(\lambda_{x,j})_{j=1,\dots, k_x}\in \R^{k_x}$ ($x\in \{\pm 1\}^2$) of real numbers and
  a sequence $\underline{\lambda} =(\lambda_{j})_{j=1,\dots, k_0}\in \mathbb{H}^{k_0}$ of decreasing (for the
  lexicographic order on $\C$ and where $\mathbb{H}=\{ z \in \C \mid
  \im(z)>0\}$) complex numbers with positive imaginary part
  and of the same lengths as the sequences in~$\mathbf{n}$;
\item For any $x=(\varepsilon, \eta)\in \{ \pm 1\}^2$ and any $1\leq\ell \leq k_x$, the product $\varepsilon \eta \lambda_{x,\ell}$ is
positive, i.e.\ $\lambda_{1, 1,
    \ell}$, $\lambda_{-1, -1, \ell}$ are
  positive and $\lambda_{-1,1, \ell}$, $\lambda_{1,-1, \ell}$ negative;
\item where the sequences in~$\boldsymbol{\lambda}$ are constant, the
  corresponding sequences in~$\mathbf{n}$ are decreasing: if $r<s$,
  $\lambda_{x,r} = \lambda_{x,s}$ (resp.\ $\lambda_r=\lambda_s$), then
  $n_{x,r}\geq n_{x,s}$ (resp.\ $m_r\geq m_s$). Equivalently, if $r<s$ and
  $n_{x,r}< n_{x,s}$ (resp.\ $m_r< m_s$), then
  $\lambda_{x,r} >\lambda_{x,s}$ (resp.\ $\lambda_r>\lambda_s$) (where a pair
  in one of the sequences
  in~$\mathbf{n}$ is strictly increasing, the corresponding elements
  in~$\boldsymbol{\lambda}$ are strictly decreasing).
\end{itemize}

The space~$\mathcal{E}(n)$ has also an involution denoted by~$\iota$ and
defined by
\[ (\mathbf{n},
\boldsymbol{\lambda}) \longmapsto \bigl( \iota( \mathbf{n}), (\{
\underline{\lambda}_{( \eta, \varepsilon)}\}_{( \varepsilon,
  \eta)\in \{\pm 1\}^2}, \underline{\lambda})\bigr).\]

There is a natural projection $\pi\colon \mathcal{E}(n) \to \mathcal{D}(n)$; it is
$\iota$-equivariant and its
fibers are contractible:

\begin{lem}
  \label{lem:fiber-parameter-space}
  For every
  $\mathbf{n}= ( \{ \underline{n}_x\}_{ x\in \{ \pm 1\}^2}, 2\underline{m})$
  in~$\mathcal{D}(n)$, the
  fiber~$V(\mathbf{n}) \coloneqq \pi^{-1}( \mathbf{n})$ is a convex cone of
  dimension $d(\mathbf{n}) \coloneqq \sum_{x\in \{\pm 1\}^2} k_x + 2k_0$ where
  $k_x$ is the length of $\underline{n}_x$ \ep{$x\in \{ \pm 1\}^2$} and $k_0$
  is the length of $2\underline{m}$.
\end{lem}

Of course, the precise part of the closure of the cone $V( \mathbf{n})$ that must be included
is described by the conditions on $\boldsymbol{\lambda}$.

\subsection{Matrices}
\label{sec:matrices}

For each~$n$ in~$\Z_{>0}$, and $\lambda\in\C$, let $C_n$ and $J_n(\lambda)$ be
the following $n\times n$-matrices\index{notation}{102@$C_n$ (the antidiagonal
  matrix)}\index{notation}{103@$J_n(\lambda)$ (Jordan block)}
\begin{equation}
  \label{eq:C_nJ_nsec:matrices}
   C_n \coloneqq \begin{pmatrix}
0 & \dots & 0 & 1 \\
0 & \dots & 1 & 0 \\
\vdots  & \iddots &\iddots  & \vdots \\
1 & 0 &\dots  & 0
\end{pmatrix} \quad
\text{and} \quad
 J_n(\lambda) \coloneqq \begin{pmatrix}
\lambda & 1 & 0 & \dots   \\
0 & \lambda & 1 & \ddots  \\
\vdots & \vdots & \ddots &\ddots  \\
0 & 0 &\dots & \lambda
\end{pmatrix}.
\end{equation}

When $\lambda$ is real and non-zero, let us define the (symmetric) $n\times
n$-matrix\index{notation}{104@$\Phi_n(\lambda)$ (matrix transconjugating $C_n$ and $C_n J_n(\lambda)$)}
\begin{equation}
  \Phi_n(\lambda) \coloneqq |\lambda|^{1/2} C_n \sum_{\ell=0}^{n-1} a_\ell
  \lambda^{-\ell} J_n(0)^\ell
\label{eq:Phi_nsec:matrices}
\end{equation}
where $\sum_{\ell=0}^{\infty} a_\ell t^\ell$ is the Taylor series of
$t\mapsto (1+t)^{1/2}$ (i.e.{} for all~$\ell$,
$a_\ell = \frac{ \prod_{j=0}^{\ell-1}(1/2-j)}{\ell!}$), cf.\
Section~\ref{back_trafo}.

For sequences $\underline{n}=(n_1, \dots, n_\ell)$ in $(\Z_{>0})^\ell$
and\index{notation}{105@$C(\underline{n})$ (block diagonal matrix with blocks among the
  $C_n$)}\index{notation}{106@$J(\underline{n}, \underline{\lambda})$ (idem with blocks
  among the $J_n(\lambda)$)}\index{notation}{107@$\Phi(\underline{n},
  \underline{\lambda})$ (idem with blocks of type $\Phi_n(\lambda)$)}
$\underline{\lambda}=(\lambda_1, \dots, \lambda_\ell)$ in $\C^\ell$ we define
${C}{(\underline{n})}$ to be the block diagonal matrix whose blocks are $C_{n_1},
\dots{}, C_{n_\ell}$; and we define $J({\underline{n}},\underline{\lambda})$ to be
the block diagonal matrix whose blocks are $J_{n_1}( \lambda_1),
\dots{}, J_{n_\ell}(\lambda_\ell)$. When $\underline{\lambda}\in (\R^*)^\ell$, the
matrix $\Phi({\underline{n}}, \underline{ \lambda})$ is the block diagonal
matrix whose blocks are $\Phi_{n_1}(\lambda_1), \dots{}, \Phi_{n_\ell}( \lambda_\ell)$.

Let us denote, for an even number $2m$ and
$\lambda=a+ib \in \C$,  the
$2m\times 2m$-matrices\index{notation}{108@$C^{\prime}_{2m}$ (antidiagonal block matrix
  with blocks $\bigl(
  \begin{smallmatrix}
    1 & 0 \\ 0 & -1
  \end{smallmatrix}
\bigr)$)}\index{notation}{110@$J^{\prime}_{2m}(\lambda)$ (\enquote{real} Jordan block)}
\begin{equation}
 C^{\prime}_{2m} \coloneqq \begin{pmatrix}
0 & 0 & \dots & 1 & 0 \\
0 & 0 & \dots & 0 & -1 \\
\vdots  & \vdots   & \iddots & \iddots  & \iddots  \\
1 & 0  & \dots & 0 & 0 \\
0 & -1 & \dots & 0 & 0
\end{pmatrix},
J^{\prime}_{2m}(\lambda) \coloneqq\begin{pmatrix}
a & -b & 1 & 0 & \dots & 0 & 0 \\
b & a & 0 & 1 & \dots & 0 & 0 \\
0 & 0 & a & -b & \ddots & 0 & 0 \\
0 & 0 & b & a & \ddots & 0 & 0 \\
\vdots & \vdots & \ddots  & \ddots  & \ddots & \vdots  & \vdots  \\
0 & 0 & 0 & 0 & \dots & a & -b \\
0 & 0 & 0 & 0 & \dots & b & a
\end{pmatrix}.\label{eq:CprimeJprime_sec:matrices}
\end{equation}

If $\lambda \neq 0$, let $c+id$ the biggest (for the lexicographic order)
square root of~$\lambda$. We define (again with $\sum a_\ell t^\ell =
(1+t)^{1/2}$)\index{notation}{111@$\Psi_{2m}(\lambda)$ (matrix transconjugating
  $C^{\prime}_{2m}$ and $J^{\prime}_{2m}(\lambda)$)}
  \begin{equation}
\Psi_{2m}(\lambda) \coloneqq
  \left(\begin{smallmatrix}
    0 & 0 & \dots & 0 & 0 & c & -d \\
    0 & 0 & \dots & 0 & 0 & -d & -c \\
    0 & 0 & \iddots& c & -d  & \vdots & \vdots \\
    0 & 0 & \iddots & -d & -c & \vdots  & \vdots\\
    \vdots & \vdots & \iddots & \iddots & \iddots & \vdots & \vdots \\
    c & -d & 0 & 0 & \dots & 0  & 0\\
    -d & -c & 0 & 0 & \dots & 0  & 0
  \end{smallmatrix}\right)
  \sum_{\ell=0}^{m-1} \frac{ a_\ell }{ (a^2+b^2)^\ell}
   \left( \begin{smallmatrix}
      0 & 0 & a & b & 0 & 0 & \dots & \dots\\
      0 & 0 & -b & a & 0 & 0 & \dots &\dots \\
      0 & 0 & 0 & 0 & \ddots & \ddots & \vdots & \vdots\\
      0 & 0 & 0 & 0 & \ddots & \ddots & \vdots & \vdots\\
      \vdots & \vdots & \ddots & \ddots &0 & 0  & a & b\\
      \vdots & \vdots & \ddots & \ddots &  0 & 0 & -b & a\\
      0 & 0 & \dots & \dots & \dots & \dots & 0 & 0\\
      0 & 0 & \dots & \dots & \dots & \dots & 0 & 0
  \end{smallmatrix}\right)^\ell .\label{eq:Psi2m_sec:matrices}
\end{equation}

When $\underline{\lambda}= (\lambda_1, \dots, \lambda_\ell)\in \C^\ell$ and
$2\underline{m} = (2m_1, \dots, 2m_\ell)$ is a sequence of even integers,
let~$C^{\prime}({2 \underline{m}})$ be the matrix with diagonal
blocks\index{notation}{112@$C^{\prime}(2\underline{m})$, $J^{\prime}(2\underline{m},
    \underline{\lambda})$ , $\Psi(2\underline{m}, \underline{\lambda})$ (block
    diagonal matrices constructed from the $C^{\prime}_{2m}$,
    $J^{\prime}_{2m}(\lambda)$, and $\Psi_{2m}(\lambda)$ respectively)}
$C^{\prime}_{2m_1}, \dots, C^{\prime}_{2m_k}$ and $J^{\prime}( {2\underline{m}},
\underline{\lambda})$ the matrix with diagonal blocks $J^{\prime}_{2m_1}(
\lambda_1), \dots, J^{\prime}_{2m_k}( \lambda_k)$. When
$\underline{\lambda}\in (\C^*)^\ell$, the matrix $\Psi
({ 2 \underline{m}}, \underline{\lambda})$ has the diagonal blocks $\Psi_{2 m_1}( \lambda_1),
\dots, \Psi_{2 m_\ell}( \lambda_\ell)$.

\medskip

We now introduce the matrices that will serve as normal forms for pairs of
quadratic forms.

Let $(\mathbf{n}, \boldsymbol{\lambda})$ be in $\mathcal{E}(n)$, thus
$\mathbf{n}$ is a quintuple
$(\underline{n}_{1,1}, \underline{n}_{1,-1}, \underline{n}_{-1,1},
\underline{n}_{-1,-1}, 2\underline{m})$ and similarly
for~$\boldsymbol{ \lambda}$. We define $C(\mathbf{n})$ to be the block
diagonal matrix\index{notation}{113@$C( \mathbf{n})$ (nondegenerate symmetric matrix
  associated with the element $\mathbf{n}$ of $\mathcal{D}(n)$)}
\[ C( \mathbf{n}) \coloneqq
  \begin{pmatrix}
    C({ \underline{n}_{1,1}}) & & & & \\
 &C({ \underline{n}_{1,-1}}) & & & \\
 & & -C({ \underline{n}_{-1,1}}) & & \\
 & & & -C({ \underline{n}_{-1,-1}}) & \\
 & & & &C^{\prime}({ 2 \underline{m}})
  \end{pmatrix},
\]
 We define
 $J( \mathbf{n}, \boldsymbol{ \lambda})$ to be the block diagonal matrix
\[
  \begin{pmatrix}
    J({ \underline{n}_{1,1}}, \underline{\lambda}_{1,1}) & & & & \\
&\!\!\! J({ \underline{n}_{1,-1}}, \underline{\lambda}_{1,-1}) & & & \\
& &\!\!\! J({ \underline{n}_{-1,1}}, \underline{\lambda}_{-1,1}) & & \\
& & &\!\!\! J({ \underline{n}_{-1,-1}}, \underline{\lambda}_{-1,-1}) & \\
& & & & \!\!\!J^{\prime}({ 2 \underline{m}}, \underline{\lambda})
  \end{pmatrix}.
\]
Also we define
$D(\mathbf{n}, \boldsymbol{ \lambda}) \coloneqq C(\mathbf{n}) J(\mathbf{n}, \boldsymbol{ \lambda})$.\index{notation}{115@$D( \mathbf{n},\boldsymbol{ \lambda}) = C(\mathbf{n}) J(\mathbf{n},
\boldsymbol{ \lambda})$ (nondegenerate symmetric matrix
  associated with the element $(\mathbf{n},\boldsymbol{ \lambda})$ of $\mathcal{E}(n)$)}

Finally, since none of the elements in~$\boldsymbol{ \lambda}$ are zero, let
$\Phi( \mathbf{n}, \boldsymbol{ \lambda})$ be\index{notation}{116@$\Phi( \mathbf{n},
  \boldsymbol{ \lambda})$ (matrix transconjugating $C( \mathbf{n})$ and $D(
  \mathbf{n}, \boldsymbol{ \lambda})$)}
\[
  \begin{pmatrix}
            \Phi({\underline{n}_{1,1}}, \underline{\lambda}_{1,1}) & & & & \\
        & 0 &\!\!\! \Phi({\underline{n}_{1,-1}}, \underline{\lambda}_{1,-1}) & & \\
        &\!\!\! \Phi({\underline{n}_{-1,1}}, \underline{\lambda}_{-1,1}) & 0 & & \\
        & & &\!\!\! \Phi({\underline{n}_{-1,-1}}, \underline{\lambda}_{-1,-1}) & \\
        & & & &\!\!\! \Psi({2\underline{m}}, \underline{\lambda})
  \end{pmatrix}.
\]

\subsection{Normal forms}
\label{sec:normal-forms-1}

Theorem~\ref{teo:normal-forms-r} and Proposition~\ref{prop:normal-forms-appendix-back-trans} of the appendix imply:
\begin{teo}
  Let $V$ be a $n$-dimensional real vector space, let~$b_0$ and~$b_1$ be two
  symmetric forms on~$V$ with $b_0$ and $b_1$ nondegenerate.

  Then there is a unique element\index{notation}{118@$(\mathbf{n}(b_0, b_1), \boldsymbol{
    \lambda}(b_0, b_1))$ (the element in $\mathcal{E}(n)$ associated with the
  pair of quadratic forms $(b_0, b_1)$)}
  $(\mathbf{n}, \boldsymbol{ \lambda}) = (\mathbf{n}(b_0, b_1), \boldsymbol{
    \lambda}(b_0, b_1))$ in $\mathcal{E}(n)$ such that there is a
  basis~$\mathbf{e}$ of~$V$ for which the matrices of the symmetric
  forms~$b_0$ and~$b_1$ are $C( \mathbf{n})$ and
  $D(\mathbf{n}, \boldsymbol{ \lambda})$ respectively.

  For the pair $(b_{1}^{\ast}, b_{0}^{\ast})$ of quadratic forms on the dual
  space~$V^*$, one has
  \[(\mathbf{n}(b^{\ast}_{1}, b^{\ast}_{0}), \boldsymbol{
      \lambda}(b^{\ast}_{1}, b^{\ast}_{0})) = \iota(\mathbf{n}(b_0, b_1),
    \boldsymbol{ \lambda}(b_0, b_1)).\]
 If $\mathbf{v}^*$ is the basis of~$V^*$ dual to the basis
  $\mathbf{v}= \mathbf{e}\, \Phi(\mathbf{n}, \boldsymbol{ \lambda})$, then the
  matrices of the symmetric forms~$b^{\ast}_{1}$ and~$b^{\ast}_{0}$ in the
  basis~$\mathbf{v}^*$ are $C(\iota( \mathbf{n}))$ and
  $D(\iota(\mathbf{n}, \boldsymbol{ \lambda}))$ respectively.
\end{teo}

There is also a uniqueness statement for the standard bases: two such bases
are conjugate under the action of the subgroup of $\GL(n, \R)$ which is the
intersection of the orthogonal group of~$C(\mathbf{n})$ and the orthogonal
group of $D(\mathbf{n}, \boldsymbol{ \lambda})$; it is also the centralizer
of $\Phi(\mathbf{n}, \boldsymbol{ \lambda})$ in the orthogonal group of~$C(\mathbf{n})$. This group naturally
identifies with $\OO( b_0) \cap \OO( b_1)$ and is described in Section~\ref{sec:automorphism-groups}.

\section{Quadruples of transverse Lagrangians}
\label{sec:4-uple-transverse}

Based on the normal forms of the above theorem, one deduces normal forms for
quadruples of transverse Lagrangians, generalizing
Proposition~\ref{prop:standard-basis-positive-4-uple}.

\begin{prop}\label{prop:quadr-transv-lagr-normal-form}
  Let $x=(L_1, M_1, L_2, M_2)$ be a quadruple of Lagrangians in~$\R^{2n}$ such
  that $(L_1, M_1, L_2)$, and $(L_1, M_2, L_2)$ are transverse triples. Then
  there is unique element
  $( \mathbf{n}, \boldsymbol{\lambda}) =( \mathbf{n}(x),
  \boldsymbol{\lambda}(x))$ such that there is a symplectic basis
  $(\mathbf{e}, \mathbf{f})$ with\index{notation}{120@$(\mathbf{n}(x), \boldsymbol{
    \lambda}(x))$ (the element in $\mathcal{E}(n)$ associated with the
  quadruple of Lagrangians~$x$)}
  \[ L_1= \Span( \mathbf{e}),\ L_2 = \Span( \mathbf{f}), \ M_1 =\Span(
    \mathbf{e} + \mathbf{f} C( \mathbf{n})), \ M_2 =\Span( \mathbf{e} -
    \mathbf{f} D( \mathbf{n}, \boldsymbol{\lambda})).\]

  Let $y$ 
  be the quadruple $(L_2, M_2, L_1, M_1)$. Then $(
  \mathbf{n}(y), \boldsymbol{\lambda}(y)) = \iota( \mathbf{n},
  \boldsymbol{\lambda})$ and, setting $\mathbf{e}' = \mathbf{f} \,
  \Phi(\mathbf{n},
  \boldsymbol{ \lambda})$, and $\mathbf{f'} = - \mathbf{e} \, {}^T\! \Phi(\mathbf{n},
  \boldsymbol{ \lambda})^{-1}$ one has
  \begin{multline*}
     L_2= \Span( \mathbf{e}'),\ L_1 = \Span( \mathbf{f}'), \ M_2 =\Span(
    \mathbf{e}' + \mathbf{f}' C(\iota (\mathbf{n}))),\\ \ M_1 =\Span(
    \mathbf{e}' - \mathbf{f}' D( \iota(\mathbf{n}, \boldsymbol{\lambda}))),
  \end{multline*}
  and $( \mathbf{e}', \mathbf{f}')$ is a symplectic basis.
\end{prop}

A basis $(\mathbf{e}, \mathbf{f})$ as in the proposition is said in
\emph{standard position} with respect to~$x$.\index{definition}{standard
  position (basis in ---)}%
\index{definition}{basis! in standard position}%

\begin{rem}
  The cross ratio of~$x$ 
  is then $D( \mathbf{n},
  \boldsymbol{\lambda})^{-1} C( \mathbf{n})= J( \mathbf{n},
  \boldsymbol{\lambda})^{-1}$ (cf.\ the definition of the cross-ratio in Section~\ref{prop:properties_CR}).
\end{rem}

As a consequence (cf.\ Section~\ref{sec:conf-lagr} for the notation):
\begin{cor}
  \label{cor:quadr-transv-lagr-one-to-one-En}
  The map $x\mapsto ( \mathbf{n}(x), \boldsymbol{\lambda}(x))$ induces an
  isomorphism between $\Conf^{4\Diamond}( \mathcal{L}_n)$ and
  the space~$\mathcal{E}(n)$.
\end{cor}

The uniqueness statement for standard basis takes the following form:

\begin{prop}
  \label{prop:quadr-transv-lagr-normal-form-uniqueness}
  Let $(\mathbf{e}_1, \mathbf{f}_1)$ and $(\mathbf{e}_2, \mathbf{f}_2)$ be
  bases in standard position with respect to a quadruple $x = (L_1, M_1, L_2,
  M_2)$. Then
  \begin{enumerate}
  \item there is a \ep{unique} element $r$ of $\GL(n,\R)$ that is orthogonal with
    respect to the symmetric matrices $C(\mathbf{n}(x))$ and
    $D(\mathbf{n}(x), \boldsymbol{\lambda}(x))$ and such that
    $(\mathbf{e}_2, \mathbf{f}_2) = (\mathbf{e}_1 \cdot{ r}, \mathbf{f}_1
    \cdot{ r})$;
  \item for every
    $s$ in~$\GL(n, \R)$ orthogonal with respect to  the symmetric matrices $C(\mathbf{n}(x))$ and
    $D(\mathbf{n}(x), \boldsymbol{\lambda}(x))$, the basis $(\mathbf{e}_1
    \cdot s, \mathbf{f}_1 \cdot s)$ is in standard position with respect to~$x$.
  \end{enumerate}
\end{prop}

\section{Triple of decorated Lagrangians}
\label{sec:triple-fram-lagr}

We give now a generalization of Lemma~\ref{lem:triple-of-lagrangian}
and of Lemma~\ref{lem:CBA-to-triple-of-lagrangian} established in
Section~\ref{sec:maxim-tripl-fram} dropping the maximality assumption.

We adopt here a point of view closer to the framed local systems introduced
in Section~\ref{sec:decor-local-syst-gamma_T}.

\begin{lem}\label{lem:triple-fram-lagr}
  Let $F_a$, $F_b$, and~$F_c$ be three symplectic vector spaces dimension~$2n$
  and let $(\mathbf{e}_a, \mathbf{f}_a)$, $(\mathbf{e}_b, \mathbf{f}_b)$, and
  $(\mathbf{e}_c, \mathbf{f}_c)$ be symplectic bases of $F_a$, $F_b$,
  and~$F_c$. For $x=a,b,c$ set $L^{t}_{x}=\Span( \mathbf{e}_x)$ and
  $L^{b}_{x}=\Span( \mathbf{f}_x)$ \ep{$L^{t}_{x}$ and $L^{b}_{x}$ are thus Lagrangians in~$F_x$}. Let~$S_a$,
  $S_b$, and~$S_c$ be symmetric $n\times n$-matrices and let
  $A\colon F_b\to F_c$, $B\colon F_c\to F_a$, and $C\colon F_a \to F_b$ be
  symplectic isomorphisms such that
  \begin{enumerate}
  \item $CBA =-\Id$,
  \item\label{item:2:lem:triple-fram-lagr} $A(L^{b}_{b})=L^{t}_{c}$, $B(L^{b}_{c})=L^{t}_{a}$, and
    $C(L^{b}_{a})=L^{t}_{b}$,
  \item\label{item:3:lem:triple-fram-lagr} $A(L^{t}_{b})= \Span( \mathbf{e}_c + \mathbf{f}_c \cdot S_c )$,
    $B(L^{t}_{c})= \Span( \mathbf{e}_a + \mathbf{f}_a \cdot S_a )$,
    $C(L^{t}_{a})= \Span( \mathbf{e}_b + \mathbf{f}_b \cdot S_b )$.
  \end{enumerate}
  Then the matrices~$S_a$, $S_b$, and~$S_c$ are invertible; there are $Y_a$,
  $Y_b$, and~$Y_c$ in~$\GL(n,\R)$ such that, with respect to the symplectic
  bases, the matrices of~$A$, $B$, and~$C$ are respectively given by
  \begin{equation}
    \begin{pmatrix}
      S_{c}^{-1} Y_a & -{}^T Y_{a}^{-1} \\ Y_a & 0
    \end{pmatrix}, \quad
    \begin{pmatrix}
      S_{a}^{-1} Y_b & -{}^T Y_{b}^{-1} \\ Y_b & 0
    \end{pmatrix}, \quad
    \begin{pmatrix}
      S_{b}^{-1} Y_c & -{}^T Y_{c}^{-1} \\ Y_c & 0
    \end{pmatrix},\label{eq:lem:triple-fram-lagr}
  \end{equation}
  and the following relations hold
  \[ Y_c \,{}^T Y_{b}^{-1} Y_a = S_b, \quad Y_b \,{}^T Y_{a}^{-1} Y_c = S_a,
    \quad Y_a \,{}^T Y_{c}^{-1} Y_b = S_c.\]
\end{lem}

\begin{proof}
  From the hypothesis~(\ref{item:2:lem:triple-fram-lagr}), we get that the matrices of~$A$,
  $B$, and~$C$ have the following form
  \[
    \begin{pmatrix}
      M_{a} Y_a & -{}^T Y_{a}^{-1} \\ Y_a & 0
    \end{pmatrix}, \quad
    \begin{pmatrix}
      M_{b} Y_b & -{}^T Y_{b}^{-1} \\ Y_b & 0
    \end{pmatrix}, \quad
    \begin{pmatrix}
      M_{c} Y_c & -{}^T Y_{c}^{-1} \\ Y_c & 0
    \end{pmatrix},
  \]
  for some (uniquely defined) $Y_a$, $Y_b$, and~$Y_c$ in $\GL(n,\R)$ and
  symmetric matrices~$M_a$, $M_b$, and~$M_c$. Since $CBA=-\Id$, a small
  calculation (see Remark~\ref{rem:triple-lag}) implies that~$M_a$, $M_b$,
  and~$M_c$ are nonsingular and the relations $Y_c \,{}^T Y_{b}^{-1} Y_a = M_{c}^{-1}, \ Y_b \,{}^T Y_{a}^{-1} Y_c = M_{b}^{-1},
    \quad Y_a \,{}^T Y_{c}^{-1} Y_b = M_{a}^{-1}$ are satisfied. The relations between the
    symmetric matrices come from the equality $\Span( \mathbf{e}_c +
    \mathbf{f}_c \cdot S_c) = A( \Span( \mathbf{e_b})) = \Span( \mathbf{e}_c
    \cdot M_a Y_a + \mathbf{f}_c \cdot Y_a)= \Span( \mathbf{e}_c
    \cdot M_a + \mathbf{f}_c)$ which gives $S_c=M_{a}^{-1}$ and
    similarly, $S_a= M_{b}^{-1}$, and $S_b= M_{c}^{-1}$.
\end{proof}

\begin{rem}
  The matrices~$Y_x$ ($x\in \{a,b,c\}$) are isometries in the following
  sense: $Y_c S_{a}^{-1} \,{}^T Y_c = S_b$, $Y_a S_{b}^{-1} \,{}^T Y_a = S_c$,
  $Y_b S_{c}^{-1} \,{}^T Y_b = S_a$.
\end{rem}

Conversely:
\begin{lem}\label{lem:triple-fram-lagr-kind-of-converse}
  Let~$S_b$ be a nonsingular symmetric matrix and let~$Y_a$, and~$Y_b$ be
  in~$\GL(n, \R)$. Define $Y_c \coloneqq S_b Y_{a}^{-1} \, {}^T Y_{b}$,
  $S_c \coloneqq Y_a S_{b}^{-1} \,{}^T Y_a $, and
  $S_a \coloneqq Y_b S_{c}^{-1} \,{}^T Y_b$.

  Then~$S_a$ and~$S_c$ are symmetric and nonsingular and one has
  \[Y_b \,{}^T Y_{a}^{-1} Y_c = S_a,\ Y_a \,{}^T Y_{c}^{-1} Y_b = S_c,
    \textrm{ and }\,  S_b = Y_c S_{a}^{-1} \,{}^T Y_c.\] The matrices~$A$, $B$, and~$C$ defined by
  Equation~\eqref{eq:lem:triple-fram-lagr} satisfy $CBA=-\Id$.
\end{lem}

\begin{rem}
  Lemmas~\ref{lem:triple-G} and~\ref{lem:triple-G-converse} below generalize
  these statements.
\end{rem}

\section{Space of $\mathcal{X}$-coordinates}
\label{sec:parameter-spaces}

We denote by $\XETn$ the set of tuples\index{notation}{122@$\XETn$ (space of general $\mathcal{X}$-coordinates)}\index{definition}{$\mathcal{X}$-coordinates}
\[ \bigl( \{(\mathbf{n}_a, \boldsymbol{\lambda}_a) \}_{a\in A_2}, \{S_v\}_{v\in V},
  \{Y_a\}_{a\in A_3}\bigr)\]
such that
\begin{itemize}
\item for all~$a$ in~$A_2$, $(\mathbf{n}_a, \boldsymbol{\lambda}_a )$ belongs
  to $\mathcal{E}(n)$;
\item for all cycles~$\{a,a'\}$ in~$A_2$, $(\mathbf{n}_{a}, \boldsymbol{\lambda}_{a}
  ) = \iota(\mathbf{n}_{a'}, \boldsymbol{\lambda}_{a'} )$;
\item for all~$v$ in~$V$, $S_v$ is a symmetric matrix. If $v=v^+(a)$ for
  some $a\in A_2$, then $S_v= C(\mathbf{n}_a )$, and if not (i.e.\ when~$v$ is
  an external edge), $S_v$ is a diagonal matrix~$I_{p,q}=\bigl(
  \begin{smallmatrix}
    \Id_p & 0\\ 0 & -\Id_q
  \end{smallmatrix}
\bigr)$ ($p+q=n$);
\item for all~$a$ in~$A_3$, $Y_a$ belongs to~$\GL(n,\R)$; and
\item  for all cycle
  $(a,b,c)$ in~$A_3$, the equality $Y_c \,{}^T Y_{b}^{-1} Y_a = S_v$ holds
  where $v= v^+(c)= v^-(b)$.
\end{itemize}

Of course the family $\{ S_{v^+(a)}\}_{a\in A_2}$ is completely determined by
$\{ ( \mathbf{n}_a, \boldsymbol{ \lambda}_a)\}$ but it is helpful to keep
it. Likewise, the matrices~$S_v$, for~$v$ an external vertex are completely
determined by the signature associated with the triangle
containing~$v$. Unless $A_2=\emptyset$ (which happens only in the case of the
disk with~$\sharp R=3$), these signatures are equally determined by the
family~\mbox{$\{ ( \mathbf{n}_a, \boldsymbol{ \lambda}_a)\}$}.

\section{Positive locus}
\label{sec:positive-locus}

The subset of elements
\[ u=\bigl( \{(\mathbf{n}_a, \boldsymbol{\lambda}_a) \}_{a\in A_2},
  \{S_v\}_{v\in V}, \{Y_a\}_{a\in A_3}\bigr)\] of
$\XETn$ for which, for all~$v$ in~$V$, $S_v$ is positive
definite is called the \emph{positive locus} of the parameter space. In this
case, for every~$a$ in~$A_2$ the elements
$\mathbf{n}_a = (\{ \underline{n}_x\}_{x\in \{ \pm 1\}^2}, 2\underline{m}),
\boldsymbol{\lambda}_a = (\{ \underline{\lambda}_x\}_{x\in \{ \pm 1\}^2},
2\underline{\lambda})$ have a simpler form: only the sequence
$\underline{n}_{1,1}= (n_1, \dots, n_k)$ is nontrivial and all its entries are
equal to~$1$ (so $k=n$) and the sequence
$\underline{\lambda}_{1,1} = (\lambda_1, \dots, \lambda_n)$ is
decreasing. Hence, this data can be encoded by the diagonal matrix $x(a)$ with
entries $(\lambda_1, \dots, \lambda_n)$. Also, all the matrices~$Y_a$ ($a\in
A_3$) belong to~$\OO(n)$ and we set $x(a)=Y_a$. The tuple $f(u)\coloneqq
(x(a))_{a\in A}$ belongs then to $\XplusDeltaTn$ (cf.\
Section~\ref{rec_rep_max}) and the map~$f$ is an isomorphism between the
positive locus and $\XplusDeltaTn$.

In the sequel, we will rather consider $\XplusDeltaTn$ as a
subspace of $\XETn$ (i.e.\ the positive locus) without
reference to~$f$.

Our last observation is that, in this case, the matrix $\Phi (\mathbf{n}_a, \boldsymbol{\lambda}_a)$ ($a\in
A_2$) constructed above is equal to~$x(a)^{1/2}$.

\section{From coordinates to representations}
\label{sec:from-coord-repr-1}

Similarly to Section~\ref{sec:maxim-decor-sympl} and based 
on
Chapter~\ref{sec:local-systems-their}, we associate a framed
$\delta$-twisted symplectic local system on~$\Gamma_\mathcal{T}$ to every
$x= \bigl( \{(\mathbf{n}_a, \boldsymbol{\lambda}_a )\}_{a\in A_2},
\{S_v\}_{v\in V}, \{ Y_a\}_{a\in A_3}\bigr)$ by specifying the transition
matrices:
\begin{enumerate}
\item for every  $a$ in~$A_2$,
  \[ G_a  =
    \begin{pmatrix}
      0 & - {}^T\! \Phi(
  \mathbf{n}_{ a}, \boldsymbol{\lambda}_{ a})^{-1} \\
       \Phi(
  \mathbf{n}_{ a}, \boldsymbol{\lambda}_{ a}) & 0\\
    \end{pmatrix};
  \]
\item for every~$a$ in~$A_3$,
  \[ G_a =
    \begin{pmatrix}
      S_{v^+(a)}^{-1} Y_a & -{}^T Y_{a}^{-1} \\ Y_a & 0
    \end{pmatrix}.
  \]
\end{enumerate}

The fact that this procedure defines indeed a $\delta$-twisted
local system follows from the equalities,  for every $(
  \mathbf{n}, \boldsymbol{\lambda}) \in \mathcal{E}(n)$, ${}^T\! \Phi(
  \iota(\mathbf{n}, \boldsymbol{\lambda})) = \Phi(
  \mathbf{n}, \boldsymbol{\lambda})$ (in turn a consequence of the
  the fact that the matrices $\Phi_n(\lambda)$ and $\Psi_{2m}(\lambda)$ are symmetric) and the relations
  between the $Y_a$s and the $S_v$s.

We will denote by $\holXT(x)$ the element of
$\Locfdelta( \Gamma_\mathcal{T}, \Sp(2n,\R))$ constructed above.\index{notation}{124@$\holXT$ (the holonomy map from $\XETn$ to the space of
  framed local systems on~$\Gamma_{\mathcal T}$)}

\begin{rem}
  \label{rem:from-coord-repr-restricit-positive-locus}
  In restriction to the positive locus, the map
  $\holXT$ is exactly the map
  $\holXpT$ defined in Section~\ref{rec_rep_max}.
\end{rem}

We now state our main result for this chapter. It will be proved in Section \ref{sec:stand-bases-decor}.
\begin{teo}
  \label{teo:from-coord-repr-X-general}
  The map
  \[ \holXT \colon \XETn
    \longrightarrow \Locfdelta( \Gamma_\mathcal{T}, \Sp(2n,\R))\]
  is onto the subspace $\LocfdeltaT(
  \Gamma_\mathcal{T}, \Sp(2n,\R))$ of transverse framed local systems.

  Two elements $x= \bigl( \{(\mathbf{n}_a, \boldsymbol{\lambda}_a )\}, \{S_v\},
\{Y_a\}\bigr)$ and  $x'= \bigl( \{(\mathbf{n}^{\prime}_{a}, \boldsymbol{\lambda}^{\prime}_{a} )\}, \{S_{v}^{\prime}\},
\{Y_{a}^{\prime}\}\bigr)$ have the same image under
$\holXT$ if and only if
\begin{enumerate}
\item for all~$a$ in~$A_2$,
  $(\mathbf{n}_a, \boldsymbol{\lambda}_a ) = (\mathbf{n}^{\prime}_{a},
  \boldsymbol{\lambda}^{\prime}_{a} )$ \ep{thus $S_{v^+(a)} =
    S^{\prime}_{v^+(a)}$};
\item for every external vertex~$v$, $S_v=S'_v$;
\item and there is a family of matrices $\{ r_v\}_{v\in V}$, such that
  \begin{itemize}
  \item for all~$v$ in~$V$, $r_v$ is orthogonal with respect to~$S_v$;
  \item for all arrows~$a$ in~$A_2$,
    $r_{v^+(a)}$ commutes with $\Phi( \mathbf{n}_a, \boldsymbol{\lambda}_a)$,
    and $r_{v^{-}(a)} = r_{v^+(a)}$;
  \item for all~$a$ in~$A_3$, $Y^{\prime}_{a} = r_{v^{+}(a)} Y_a r_{v^{-}(a)}^{-1}$.
  \end{itemize}
\end{enumerate}
\end{teo}

\section{Standard bases for framed local systems}
\label{sec:stand-bases-decor}

Proposition~\ref{prop:quadr-transv-lagr-normal-form} gives a notion of
standard basis for a quadruple of Lagrangians. Here we generalize
this notion to framed local systems.

\noindent Let $(F_v, g_a, L^{t}_v, L^{b}_{v})$ be a framed $\delta$-twisted symplectic
 local system on the quiver~$\Gamma_\mathcal{T}$.

\begin{df}
  A generating symplectic basis $\{ (\mathbf{e}_v, \mathbf{f}_v)\}_{v\in V}$
  (see Def.~\ref{df:sympl-basis-decor}) is said to be \emph{in standard
    position} with respect to the framing if, for every \emph{internal}
  vertex~$v$ in~$V$,
  \begin{enumerate}
  \item if $x= q_v(F_v, g_a, L^{t}_v, L^{b}_{v})$ is the associated quadruple
    of Lagrangians in~$F_v$ (see Section~\ref{sec:conf-assoc-with}), then the
    basis~$( \mathbf{e}_v, \mathbf{f}_v)$ is in standard position with respect
    to~$x$ (cf. Proposition~\ref{prop:quadr-transv-lagr-normal-form});
  \item if $a$ is the arrow in $A_2$ such that $v=v^+(a)$, the matrix of~$g_a$
    is
    \[ G_a =
      \begin{pmatrix}
        0 & -{}^T\! \Phi( \mathbf{n}(x), \boldsymbol{\lambda}(x))^{-1} \\
        \Phi( \mathbf{n}(x), \boldsymbol{\lambda}(x)) & 0
      \end{pmatrix},
    \]
    where $( \mathbf{n}(x), \boldsymbol{\lambda}(x))$ is given by
    Proposition~\ref{prop:quadr-transv-lagr-normal-form};
  \end{enumerate}
  and if, for every \emph{external} vertex~$v$, there exists $(p,q)\in\N^2$
  with $p+q=n$ such that the triple $f_v(F_v, g_a,
  L^{t}_v, L^{b}_{v})$ of Lagrangians in~$F_v$ is $(\Span(\mathbf{e}_v ),
  \Span(\mathbf{e}_v + \mathbf{f}_v \cdot I_{p,q} ), \Span(\mathbf{f}_v ))$  (see Section~\ref{sec:conf-assoc-with}).
  This means in particular that, if $a$~is the element of~$A_3$ such that $v=v^+(a)$, then
  $g_{a}(L^{t}_{v^-(a)}) = \Span(\mathbf{e}_v + \mathbf{f}_v \cdot I_{p,q} )$
  (the other equalities $L^{t}_{v} = \Span(\mathbf{e}_v )$ and $L^{b}_{v} =
  \Span(\mathbf{f}_v )$ are already satisfied by the assumption that the basis is
  generating).
\end{df}\index{definition}{standard position (basis in ---)}%
\index{definition}{basis! in standard position}

The moduli space of framed local systems equipped with a standard basis will
be denoted by $\Locfstdelta( \Gamma_\mathcal{T}, \Sp(2n,
\R))$. The holonomy construction performed above
(Section~\ref{sec:from-coord-repr-1}) defines in fact a
map\index{notation}{126@$\holXstT$ (an enhancement of $\holXT$)}
\[ \holXstT \colon
  \XETn \longrightarrow \Locfstdelta( \Gamma_\mathcal{T}, \Sp(2n,
\R)),\]
(cf.\ Section~\ref{sec:local-systems-basis}). Since the transition matrices
completely determine the parameters, and since the bases entirely determine the
transition matrices, the map~$\holXstT$
is a one-to-one correspondence.

The following proposition readily implies
 Theorem~\ref{teo:from-coord-repr-X-general} stated above.

\begin{prop}
  \label{prop:stand-bases-decor-exist-uniq}
  {}\hspace*{0cm}

  \begin{enumerate}
  \item Every transverse framed $\delta$-twisted symplectic local
    system $(F_v, g_a, L^{t}_v, L^{b}_v)$ admits a standard basis
    $(\mathbf{e}_v, \mathbf{f}_v)$.
  \item \label{item:2:prop:stand-bases-decor-exist-uniq} For every family
    $(r_v)_{v\in V}$ such that
    \begin{itemize}
    \item for every external vertex~$v$ in~$V$, $r_v$ is orthogonal with respect to
      $I_{p,q}$ where $p=(n+s_T)/2$, $q=(n-s_T)/2$ and $s_T=\mu^T(F_v, g_a,
      L^{t}_v, L^{b}_v)$ the Maslov index for the triangle~$T$
      containing~$v$;
    \item for every internal vertex~$v$ in~$V$, $r_v$ is orthogonal with respect to $C(
      \mathbf{n}(x))$ and with respect to $D(
      \mathbf{n}(x), \boldsymbol{\lambda}(x))$ where $x$ is the quadruple
      $q_v( F_v, g_a, L^{t}_{v}, L^{b}_{v})$;
    \item for every~$a$ in~$A_2$, $r_{v^-(a)} = r_{v^+(a)}$;
    \end{itemize}
    the family $\{ (\mathbf{e}_v \cdot r_v, \mathbf{f}_v \cdot {}^T\!
    r^{-1}_{v})\}$ is in standard position.
  \item For every basis $\{ ( \mathbf{e}'_{v}, \mathbf{f}'_{v})\}$ in standard
    position, there is a \ep{unique} family $(r_v)_{v\in V}$ as
    in~(\ref{item:2:prop:stand-bases-decor-exist-uniq}) and such that
    $\{ ( \mathbf{e}'_{v}, \mathbf{f}'_{v})\} = \{ (\mathbf{e}_v \cdot r_v,
    \mathbf{f}_v \cdot {}^T\! r^{-1}_{v})\}$.
  \end{enumerate}
\end{prop}
\begin{proof}
This is a direct consequence of the existence and
uniqueness of standard bases for quadruples of Lagrangians
(Proposition~\ref{prop:quadr-transv-lagr-normal-form} and
Proposition~\ref{prop:quadr-transv-lagr-normal-form-uniqueness}) and of the
fact that triples of pairwise transverse Lagrangians are classified by their
Maslov index.
\end{proof}

\section{Maslov indices}
\label{sec:maslov-indices}

A framed transverse local system gives a family of integers $\{ s_T\}_{T\in
  \mathcal{T}}$ where $s_T$ is the Maslov index of the configuration of three Lagrangians associated
with the triangle~$T$. The integer~$s_T$ belongs to $\{ -n, -n+2, \dots, n\}$.

When the local system is the holonomy of an element
$( \{(\mathbf{n}_a, \boldsymbol{\lambda}_a )\}, \{S_v\}, \{ Y_a\})$, then, for every
triangle~$T$ and every vertex $v$ contained in~$T$, the integer~$s_T$ is the
signature of the symmetric matrix $S_v$, which is equal to $ C( \mathbf{n}_a)$
if $v=v^+(a)$ for some~$a\in A_2$. It can be easily
calculated noting that, for every integer~$m$, the signature of~$C_{2m}$ and
of $C'_{2m}$ are~$0$ and the signature of $C_{2m+1}$ is~$1$.

To state a precise result, for every~$\mathbf{n} = ( \{ \underline{n}_x\}_{
  x\in \{ \pm 1\}^2}, 2\underline{m})$ in~$\mathcal{D}(n)$, set, for $x\in
\{\pm 1\}^2$,
\[ p_x( \mathbf{n}) \coloneqq \sharp \bigl\{ \ell \in \{ 1, \dots, k_x\} \mid n_{x,
    \ell} =1 \mod 2 \bigr\}.\]

\begin{lem}
  \label{lem:signa-Cn}
  For every~$ \mathbf{n}$ in~$\mathcal{D}(n)$, the signature of $C(
  \mathbf{n})$ is equal to
  \[ p_{1,1}( \mathbf{n}) + p_{1,-1}( \mathbf{n}) -  p_{-1,1}( \mathbf{n}) -
    p_{-1,-1}( \mathbf{n}),\]
  and the signature of $C( \iota(
  \mathbf{n}))$ is equal to
  \[ p_{1,1}( \mathbf{n}) - p_{1,-1}( \mathbf{n}) + p_{-1,1}( \mathbf{n}) - p_{-1,-1}( \mathbf{n}).\]
\end{lem}

Let us introduce $\mathcal{D}(\mathcal{T}, n)$ the subspace of
$\{-n, -n+2, \dots, n-2, n\}^{\mathcal{T}} \times \mathcal{D}(n)^{A_2}$
consisting of tuples $( \{s_T\}_{T\in \mathcal{T}}, \{ \mathbf{n}_a\}_{a\in
  A_2})$ such that\index{notation}{130@$\mathcal{D}(\mathcal{T}, n)$ (the possible
  tuples of invariants for elements in $\XETn$)}
\begin{itemize}
\item for all cycle~$\{a,a'\}$ in~$A_2$, $\mathbf{n}_{a} = \iota(
  \mathbf{n}_{a'})$;
\item for all~$a$ in~$A_2$, if $v^+(a)$  belongs to the
  triangle~$T$ of~$\mathcal{T}$, the signature of $C( \mathbf{n}_{a})$ is
  equal to~$s_T$.
\end{itemize}

We note that for every possible choice of the indices~$\{ s_T\}_{T\in
  \mathcal{T}}$ there exist elements in~$\mathcal{D}( \mathcal{T}, n)$
realizing this choice:

\begin{lem}
  \label{lem:maslov-indices-every-poss}
  For every $\{ s_T\}_{T\in \mathcal{T}}$ in
  $\{ -n, 2-n, \dots, n\}^\mathcal{T}$, there exists
  $\{ \mathbf{n}_a\}_{a\in A_2}$ in $\mathcal{D}(n)^{A_2}$ such that the tuple
  $(\{ s_T\}_{T\in \mathcal{T}}, \{ \mathbf{n}_a\}_{a\in A_2} )$ is in
  $\mathcal{D}( \mathcal{T},n)$.
\end{lem}

\begin{proof}
  Let $E\subset A_2$ be a subset containing exactly one of the elements in
  every cycle in~$A_2$. By
  definition of~$\mathcal{D}( \mathcal{T},n)$ it is enough to specify $\{
  \mathbf{n}_a\}_{a\in E}$.

  Let thus~$a$ be in~$E$. Let $T$ be the triangle containing~$v^+(a)$ and let
  $T'$ be the triangle containing~$v^-(a)$. By Lemma~\ref{lem:signa-Cn}, we have
  \begin{align*}
     p_{1,1}( \mathbf{n}_a) + p_{1,-1}( \mathbf{n}_a) -  p_{-1,1}( \mathbf{n}_a) -
    p_{-1,-1}( \mathbf{n}_a) & = s_T \\
   p_{1,1}( \mathbf{n}_a) - p_{1,-1}( \mathbf{n}_a) + p_{-1,1}( \mathbf{n}_a) -
    p_{-1,-1}( \mathbf{n}_a)  & = s_{T'}.
  \end{align*}
  Therefore, we have to prove that there exists~$\mathbf{n}_a$
  in~$\mathcal{D}(n)$ such that
  \begin{align*}
     p_{1,1}( \mathbf{n}_a)  -  p_{-1,-1}( \mathbf{n}_a) & =\frac{ s_T + s_{T'}}{2} \\
     p_{1,-1}( \mathbf{n}_a)  -  p_{-1,1}( \mathbf{n}_a) & =\frac{ s_T-
                                                           s_{T'}}{2}.
  \end{align*}
  Define the element $\mathbf{n}_a = ( \{ \underline{n}_x\}_{x \in \{ \pm 1\}^2},
  2\underline{m})$ by
  \begin{itemize}
  \item $\underline{n}_{1,1}$ is a sequence of~$1$ whose length is $\max( 0, (s_T +
    s_{T'})/2)$,
  \item $\underline{n}_{-1,-1}$ is a sequence of~$1$ whose length is $\max( 0, -(s_T +
    s_{T'})/2)$,
  \item $\underline{n}_{1,-1}$ is a sequence of~$1$ whose length is $\max( 0, (s_T -
    s_{T'})/2)$,
  \item $\underline{n}_{-1,1}$ is a sequence of~$1$ whose length is $\max( 0, (-s_T +
    s_{T'})/2)$,
  \item $2\underline{m}$ is a sequence whose length is~$0$ or~$1$ chosen so that
    $\mathbf{n}_v$ belongs to~$\mathcal{D}(n)$.
  \end{itemize}
  The bounds on $s_T$ and $s_{T'}$ and their parity properties imply that the
  above construction is legitimate.
\end{proof}

\section{Pieces}
\label{sec:pieces}

The natural projection $\pi\colon \XETn \to \{ -n, -n+2,
\dots, n-2, n\}^{\mathcal{T}} \times \mathcal{D}(n)^{A_2}
$ associates to $x= ( \{(\mathbf{n}_a, \boldsymbol{\lambda}_a) \}, \{S_v\},
  \{Y_a\})$ the family $( \{s_T\}, \{ \mathbf{n}_a\})$ where, for each
  triangle~$T$ of~$\mathcal{T}$, $s_T$ is the common signature of the
  matrices~$S_v$ for~$v$ in~$T$; it
takes values in $\mathcal{D}( \mathcal{T}, n)$.
For every~$\mathbf{x}=(\{s_T\}, \{ \mathbf{n}_a\})$ in~$\mathcal{D}( \mathcal{T}, n)$, we denote by
$\XEbfxTn$ the fiber $\pi^{-1}(
\mathbf{x})$. This subspace will be called a \emph{piece} of
$\XETn$.\index{notation}{132@$\XEbfxTn$ (a piece of $\XETn$)}\index{definition}{piece}

Let $E\subset A_2$ be a subset containing exactly one of the elements in every
cycle in~$A_2$.
Using the notation of Lemma~\ref{lem:fiber-parameter-space}, we define
\[ V( \mathbf{x}) \coloneqq \prod_{a\in E} V( \mathbf{n}_a).\]
Then $ V( \mathbf{x})$ is a convex cone of dimension~$d( \mathbf{x}) \coloneqq\sum_{a\in E} d(
\mathbf{n}_a)$.

Building on
Lemma~\ref{lem:triple-fram-lagr-kind-of-converse}, for every $3$-cycle $\{ a,
b, c\}$ in~$A_3$, contained in a triangle~$T$ of~$\mathcal{T}$, the space of
triples of matrices $(Y_a, Y_b, Y_c)$ satisfying the hypothesis of the lemma
(or the conditions in Section~\ref{sec:parameter-spaces}) is isomorphic
to
\[ G_{s_T} \coloneqq \OO\Bigl( \frac{ n+s_T}{2}, \frac{ n-s_T}{2}\Bigr)^2.\]
The product of these groups, for~$T$ in~$\mathcal{T}$, is
\[ G( \mathbf{x}) \coloneqq \prod_{ T\in \mathcal{T}} G_{s_T}. \]
This Lie group is of dimension $(r-2{ \chi( \bar{S})}  )n(n-1)$ as all the orthogonal
groups involved have the same dimension $n(n-1)/2$ and as $\sharp \mathcal{T}
= r-2{ \chi( \bar{S})}$.

The following result is a consequence of the fact that, in the description of the
space of $\mathcal{X}$-coordinates (Section~\ref{sec:parameter-spaces}), the
conditions on the family $\{ \boldsymbol{\lambda}_a\}_{a\in A_2}$ depend only
on~$\{ \mathbf{n}_a\}_{a\in A_2}$ and the conditions on the family $\{
Y_a\}_{a\in A_3}$ depend only on~$\{ S_v\}_{v\in V}$.

\begin{prop}
  \label{prop:pieces}
  \begin{enumerate}
  \item The piece $\XEbfxTn$ is isomorphic to
    $V( \mathbf{x}) \times G( \mathbf{x})$ and is of dimension
    $d( \mathbf{x}) + (r-2{ \chi( \bar{S})}  ) n(n-1)$.
  \item   The restriction of $\holXT$ to
    $\XEbfxTn$ is continuous.
  \item  Furthermore, if two elements of~$\XETn$ have the
    same image under~$\holXT$, then they belong to
    the same piece.
  \end{enumerate}
\end{prop}

\begin{proof}
  Only the continuity needs a comment. It results from the continuity of the
  maps $\boldsymbol{ \lambda} \mapsto \Phi( \mathbf{n}, \boldsymbol{\lambda})$.
\end{proof}

From this proposition, we obtain the following corollary.
\begin{cor}
  \label{cor:pieces-open-hol-generic}
  For $\mathbf{x} =(\{ s_T\}_{T }, \{ \mathbf{n}_a\}_{a })$ in $\mathcal{D}( \mathcal{T},
  n)$, the piece $\XEbfxTn$ has nonempty
  interior if and only if, for all~$a$ in~$A_2$, denoting $\mathbf{n}_a = ( \{
  \underline{n}_x\}_{ x\in \{ \pm 1\}^2}, 2\underline{m})$ all the integers in the
  sequences $\underline{n}_x$ \ep{$x\in \{ \pm 1\}^2$} and $\underline{m}$ are
  equal to~$1$.

  In this situation, the restriction of $\holXT$ to $\XEbfxTn$ is generically finite-to-one.
\end{cor}

The decomposition of $\XETn$ into pieces induces a
decomposition of the space of transverse local systems $\LocfdeltaT(
\Gamma_\mathcal{T}, \Sp(2n, \R))$. This gives a local description of the space. However, from this decomposition it seems difficult to obtain information about the global topology.

\section{Over-parametrization}
\label{sec:over-parametrization}

The formulas used in Section~\ref{sec:from-coord-repr-1} are valid on a wider
set; let\index{notation}{134@$\XTn$ (a bigger $\mathcal{X}$-space)}\index{definition}{$\mathcal{X}$-coordinates}
\[ \XTn \subset
  \Sym(n, \R)^V
  \times \GL(n,\R)^{A}\]
be the subset of tuples $z= (
\{ S_v\}_{v\in V}, \{
Y_a\}_{a \in A})$ such that
\begin{itemize}
\item for all cycle $(a,a')$ in $A_2$, $Y_{a'} = {}^T \! Y_a$;
\item for all~$v$ in~$V$, $S_v$ is nonsingular;
\item  for all cycle
  $(a,b,c)$ in~$A_3$, the equality $Y_c \,{}^T Y_{b}^{-1} Y_a = S_v$ holds
  where $v= v^+(c)= v^-(b)$.
\end{itemize}

For such a~$z$, we will denote again by $\holXT(z)$ the
framed $\delta$-twisted symplectic local system arising
from the following family $\{ G_a(z)\}_{a\in A}$ of transition matrices:
\begin{itemize}
\item for~$a$ in~$A_2$, $G_a(z) =
  \begin{pmatrix}
    0 & -{}^T \! Y_{a}^{-1} \\ Y_a & 0
  \end{pmatrix}$;
\item for~$a$ in~$A_3$, $G_a(z) =
  \begin{pmatrix}
    S_{v^+(a)}^{-1} Y_a & -{}^T Y_{a}^{-1} \\ Y_a & 0
  \end{pmatrix}
  $.
\end{itemize}

The group\index{notation}{136@$G_\mathcal{X}$ (the group of transformations of $\XTn$)}
\[ G_{ \mathcal{X}} \coloneqq \GL(n, \R)^V\]
acts on~$\XTn$ via the following formula: if $z= ( \{ S_v\}, \{ Y_a\})$ belongs to~$\XTn$ and
$g = \{ g_v\}_{v\in V}$ belong to~$G_\mathcal{X}$, then
\[ g\cdot z \coloneqq \bigl( \{
  \{ g_v
  S_v \, {}^T \! g_v\}_{v\in V}, \{ g_{ v^+(a)} Y_a\, {}^T \! g_{ v^-(a)}\}_{a \in A}\bigr).\]

Collecting the information about equivalent local systems
(Section~\ref{sec:local-systems-basis}) and using
Theorem~\ref{teo:from-coord-repr-X-general}, we get:

\begin{teo}
  \label{teo:over-parametrization-Z-GZ}
  The map
  \[ \holXT\colon \XTn
    \longrightarrow \LocfdeltaT(
    \Gamma_\mathcal{T}, \Sp(2n,\R))\]
  is continuous, onto, and its fibers are the orbits of the action
  of~$G_\mathcal{X}$.
\end{teo}

\section{Connected components}
\label{sec:connected-components-1}

The previous theorem~\ref{teo:over-parametrization-Z-GZ} can be used to
determined the number of connected components of the moduli space of
transverse framed local systems.

For this, let $H \coloneqq \{\pm 1\} \simeq \pi_0( \GL(n,\R))$ be the group of
connected components of $\GL(n, \R)$, the quotient map $\GL(n, \R) \to H$ will
be denoted by~$\pi_0$. Let\index{notation}{138@$Z( \mathcal{T},n)$ (the set
  parametrizing the connected components of $\XTn$)}
\[ Z( \mathcal{T}, n) \subset
  \{ -n, -n+2, \dots, n\}^V \times
  H^{A}\]
be the set of tuples $(
\{ s_v\}_{v\in V}, \{ h_a\}_{a\in
  A_3})$ such that
\begin{itemize}
\item for all $a$ in~$A_3$, $s_{ v^{+}(a)} = s_{ v^{-}(a)}$;
\item for all cycle $(a, a')$ in~$A_2$, $h_{a'}= h_a$;
\item for all cycle $(a,b,c)$ in $A_3$, $h_c h_b h_a = (-1)^{(n-s_{v^+(a)})/2}$.
\end{itemize}

The group $F_{Z}\coloneqq H^{V}$ acts on $Z( \mathcal{T}, n)$: if
$h=\{ h_v\} $ is in $ F_{Z}$ and
$z= (
\{ s_v\} , \{ h_a\} ) $ is in $ Z( \mathcal{T}, n)$, then\index{notation}{140@$F_{Z}$
  (its group of symmetries)}
\[ h\cdot z \coloneqq \bigl(
  \{
    s_v\}_{v\in V}, \{ h_a h_{v^+(a)} h_{v^-(a)}\}_{a\in A}\bigr).\]

Using the fact that the space of nonsingular symmetric matrices of a given
signature is connected (it is an orbit under the action of the connected group $\GL^+(n, \R)$),
the explicit description of $\XTn$ gives:

\begin{prop}\label{prop:connected-components-LocdT}
  The map
  \begin{align*}
    \pi_\mathcal{X} \colon \XTn & \longrightarrow Z(
                             \mathcal{T}, n) \\
    ( \{\Phi_a\}, \{ S_v\}, \{ Y_a\}) & \longmapsto ( \{
\pi_0(\Phi_a)\}, \{ \sgn(S_v)\}, \{ \pi_0(Y_a)\})
  \end{align*}
  induces a bijection between the space of connected components
  of~$\XTn$ and the set~$Z( \mathcal{T},n)$. This map is equivariant with respect to
  the morphism
  \begin{align*}
    \pi_G \colon  G_\mathcal{X} & \longrightarrow F_{Z}\\
    \{ g_v\}_{v\in V} & \longmapsto \{ \pi_0( g_v)\}_{v\in V}.
  \end{align*}
  In turn, there is a well defined map
  \[ \LocfdeltaT( \Gamma_\mathcal{T},
    \Sp(2n,\R)) \simeq G_\mathcal{X} \backslash \XTn
    \longrightarrow F_{Z} \backslash Z( \mathcal{T}, n)\]
  that induces a bijection between $F_{Z} \backslash Z( \mathcal{T}, n)$ and the set of connected components of $\LocfdeltaT( \Gamma_\mathcal{T},
    \Sp(2n,\R))$.
\end{prop}

\begin{proof}
  Let $A'\subset A_2$ be a subset containing exactly one of the arrows of every
  $2$-cycle, let $T\subset V$ be a subset containing, for every triangle~$f$, exactly one of the three
  vertices of $\Gamma_\mathcal{T}$ that are in~$f$. Let~$B$ be a subset of~$A_3$ containing exactly two of the arrows of every $3$-cycle.

  Thanks to Lemma~\ref{lem:triple-fram-lagr-kind-of-converse}, the map
  \begin{align*}
       \XTn & \longrightarrow
                                      \Sym^\ast(n, \R)^{T} \times \GL(n,
                                      \R)^{A' \sqcup B}\\
    (
    \{ S_v\}_{v\in V}, \{
    Y_a\}_{a \in A})
    & \longmapsto (
      \{ S_v\}_{v\in T}, \{
Y_a\}_{a \in A' \sqcup B})
  \end{align*}
  is a diffeomorphism where $\Sym^\ast(n,\R)$ is the space of nonsingular
  symmetric matrices. Similarly, the map
  \begin{align*}
    Z( \mathcal{T}, n)
    & \longrightarrow  \{ -n, -n+2, \dots, n\}^T \times
      H^{A' \sqcup B}\\
    (  \{ s_v\}_{v\in V}, \{ h_a\}_{a\in
    A})
    & \longmapsto (  \{ s_v\}_{v\in T}, \{ h_a\}_{a\in
  A' \sqcup B})
  \end{align*}
  is a bijection. Using these isomorphisms, the map~$\pi_\mathcal{X}$ become
  \begin{align*}
     \Sym^\ast(n, \R)^{T} \times \GL(n, \R)^{A' \sqcup B}
    & \longrightarrow
      \{ -n, -n+2, \dots, n\}^T \times
      H^{A'\sqcup B}\\
    (
    \{ S_v\}_{v\in T}, \{
    Y_a\}_{a \in A' \sqcup B})
    & \longmapsto (
      \{ \sgn(S_v)\}_{v\in T},
      \{ \pi_0( Y_a)\}_{a\in A'\sqcup B}).
  \end{align*}
  It induces a bijection at the level of connected components since
  $\pi_0\colon \GL(n, \R)\to H$ and $\sgn \colon \Sym^\ast(n,\R) \to \{ -n,
  -n+2, \dots, n\}$ do. The other statements follow from similar consideration.
\end{proof}

Finally, noting that the diagonal subgroup $H\subset F_{Z}$ acts
trivially on $Z( \mathcal{T}, n)$ and that the quotient group $F_{Z} /
H$ acts freely on $Z( \mathcal{T}, n)$, one obtains:
\begin{cor}
  \label{cor:number-connected-components-LocdT}
  The number of 
  components of 
  $\LocfdeltaT( \Gamma_\mathcal{T},
    \Sp(2n,\R))$ is equal to $2^{ 1-{\chi( \bar{S})} } \times (n+1)^{r-2{
        \chi( \bar{S})}}$.
\end{cor}

\chapter{$\mathcal{X}$-coordinates for representations into isogenic groups}\label{cent_ext}

In this chapter we investigate framed local systems for groups that are
isogenic to~$\Sp(2n, \R)$ as well as coordinates on their
moduli space. We first describe the groups in Section~\ref{sec:groups} and
then the moduli spaces in Section~\ref{sec:representations}. In order
to introduce their twisted version,
we need a few special elements of the Lie groups, which we introduce in Section~\ref{sec:some-elements}. Finally we describe the corresponding local systems on the quiver~$\Gamma_\mathcal{T}$ and \enquote{parametrize} them. From
this we draw a few topological consequences.

\section{Groups}
\label{sec:groups}

Let~$G$ be a connected finite cover of~$\PSp(2n, \R)$. Thus~$G$ is a Lie group
that is isomorphic to the quotient $\widetilde{\Sp}(2n, \R) / \Lambda_G$
for\index{notation}{141@$G$ (in Chapter~\ref{cent_ext}, a finite connected cover of $\PSp(2n,\R)$)}\index{notation}{142@$\Lambda_G$ (the subgroup of $\pi_1(\PSp(2n,\R))$ determining
  the cover~$G$)}
some uniquely determined subgroup~$\Lambda_G$ of~$\pi_1( \PSp(2n, \R)) \simeq
Z( \widetilde{\Sp}(2n, \R)) \subset Z( \widetilde{\U}(n))$. As a subgroup
of~$\widetilde{\U}(n)$ (see Section~\ref{sec:universal-coverings}), the
group~$\pi_1( \PSp(2n, \R))$ is the subgroup generated by the elements
\[ (\Id, 2\pi), \quad \textrm{and}\quad ( -\Id, n\pi).\]
Depending on the parity of~$n$, a generating system for that group is
\begin{enumerate}
\item \label{item:pi_1_PSp_n_odd} $( -\Id, \pi)$ if $n$ is odd in which case
  $\pi_1( \PSp(2n, \R))$ is isomorphic to~$\Z$;
\item \label{item:pi_1_PSp_n_even} $(\Id, 2\pi)$ and $( -\Id, 0)$ if $n$ is even in which case
  $\pi_1( \PSp(2n, \R))$ is isomorphic to~$\Z\times \Z/2\Z$ and its
  $2$-torsion is generated by $(-\Id, 0)$.
\end{enumerate}

Thus the subgroup~$\Lambda_G$ is either\index{notation}{144@$x_G$ (the element
  of~$\Z_{>0}$ characterizing $\Lambda_G$)}
\begin{enumerate}[(I)]
\item \label{item:Lambda_G_n_odd} generated by $( (-1)^{x_G} \Id, x_G\pi)$ for some
  uniquely determined~$x_G\in \Z_{>0}$ when~$n$ is odd;
\item \label{item:Lambda_G_n_even_torsion} generated by $(-\Id, 0)$ and  $( \Id, x_G\pi)$ for some
  uniquely determined~$x_G\in 2\Z_{>0}$ when~$n$ is even and~$\Lambda_G$ contains
  the torsion of $\pi_1( \PSp(2n, \R))$;
\item \label{item:Lambda_G_n_even_no_torsion} generated by $( \Id, x_G\pi)$ or
  by $( -\Id, x_G\pi)$ for some uniquely determined~$x_G\in 2\Z_{>0}$ when~$n$
  is even and~$\Lambda_G$ does not contain the torsion of
  $\pi_1( \PSp(2n, \R))$.
\end{enumerate}

We will designate by~$L$ and by~$K$ the subgroups of~$G$ that are the
preimages of the subgroups $\PGL(n, \R)$ and $\PO(n)$ by the homomorphism
$\pi\colon G \to \PSp(2n,\R)$. Thus~$K$ is a maximal compact subgroup in~$L$ and the
polar decomposition induces a $K$-equivariant diffeomorphism between~$L$ and
$K\times \Sym(n, \R)$. Let also $\hat{L}$ and $\hat{K}$ be the preimages  of
$\PGL(n, \R)$ and $\PO(n)$ in $\widetilde{ \PSp}(2n,\R)$. The group~$K$ is
hence isomorphic to the quotient $\hat{K}/  \Lambda_G $ (note that $\Lambda_G$
is contained in the kernel of $\widetilde{\Sp}(2n, \R)\to \PSp(2n, \R)$ and
therefore in~$\hat{K}$).

As a subgroup of $\widetilde{\U}(n)$ (see Section~\ref{sec:universal-coverings}), $\hat{K}$ is
\[ \hat{K} = \bigl\{ (u, \theta) \in \widetilde{ \U}(n) \mid u\in
  \OO(n)\bigr\} = \SO(n)\times 2\pi\Z \bigsqcup \bigl(\OO(n) \setm \SO(n)\bigr)\times (
  \pi + 2\pi\Z),\]
and its neutral component $\hat{K}_0$ is $\SO(n)\times \{0\}$.
Let $r$ be any element in $\OO(n) \setm \SO(n)$.
The homomorphism
\begin{align*}
 \tau\colon \Z & \longrightarrow \hat{K}\\
  n & \longmapsto ( r^n, n\pi)
\end{align*}
induces an isomorphism $\Z \simeq \pi_0( \hat{K}) = \hat{K}/ \hat{K}_0 $ that
does not depend on the choice of~$r$.
The natural map $\pi_0( \hat{K}) \to \pi_0(K)$ is onto and its kernel is
$\Lambda_G \hat{K}_0$ (seen as a subgroup of $\pi_0( \hat{K})$).
Using again the fact that $\SO(n)$ is connected, one gets that in
cases~(\ref{item:Lambda_G_n_odd}), (\ref{item:Lambda_G_n_even_torsion})
and~(\ref{item:Lambda_G_n_even_no_torsion}) above:
\begin{itemize}
\item $\tau^{-1}( \Lambda_G \hat{K}_0)$ is the subgroup $x_G\Z$.
\end{itemize}
As a conclusion:
\begin{lem}
  \label{lem:group_K_connnected_comp}
  The group~$\pi_0(K) \simeq \pi_0(L)$ is isomorphic to $\Z/x_G\Z$.
\end{lem}

\section{Local systems}
\label{sec:representations}

There is a natural action of~$G$ on~$\Lag{n}$ through the morphism $G\to
\PSp(2n, \R)$. A \emph{framing} for a $G$-local system~$\mathcal{F}$ on~$S$
is a section~$\sigma$ of the restriction of~$\mathcal{F}_{\Lag{n}}$
to~$\partial S$. The pair $(\mathcal{F}, \sigma)$ is called a \emph{framed
  local system}. The moduli space of framed
local system
will be denoted by\index{notation}{146@$\Locf( S, G)$ (moduli of framed
  $G$-local system)}%
\index{definition}{framing}%
\index{definition}{framed!local system}%
\index{definition}{local system!framed ---}
\[ \Locf( S, G).\]
Any framed $G$-local system $(\mathcal{F}, \sigma)$ induces a framed
$\PSp(2n,\R)$-local system $(\mathcal{F}'\!, \sigma)$; the framed local
system $(\mathcal{F}, \sigma)$ is said \emph{maximal} if $(\mathcal{F}'\!,
\sigma)$ is maximal.\index{definition}{maximal!local system}%
\index{definition}{local system!maximal ---}
 The subspace of maximal framed local systems will be denoted by
$\Mf( S, G)$.\index{notation}{148@$\Mf( S, G)$ (subspace of maximal framed local systems)} When an ideal triangulation~$\mathcal{T}$ of~$S$ is
given, the definition of a transverse framing with respect to~$\mathcal{T}$ is as in
Section~\ref{sec:transv-local-syst} and the moduli space of transverse framed
local systems will be denoted $\LocfT(S, G)$.\index{notation}{150@$\LocfT(S, G)$
  (subspace of transverse framed local systems)}

\begin{rem}
  When~$R=\emptyset$, we also say that a $G$-local system is \emph{maximal} if
  the associated $\PSp(2n, \R)$-local system is maximal.
\end{rem}

\section{Some elements in~$G$}
\label{sec:some-elements}

The exponential map $\exp_G\colon \mathfrak{sp}(2n, \R)\to G$ enables us to
define the following elements:\index{notation}{155@$s_G$, $\delta_G$, $u_M$, $v_M$ (some
  elements of~$G$)}
\begin{align*}
  s_G \coloneqq & \exp_G \left( \frac{\pi}{2}
                  \left(\begin{matrix}
                    0 & -\Id \\ \Id & 0
                  \end{matrix}\right)\right),
  \quad \delta_G \coloneqq s_{G}^{2},\\
  u_M \coloneqq & \exp_G
                  \begin{pmatrix}
                    0 & M \\ 0 & 0
                  \end{pmatrix}, \quad
  v_M \coloneqq \exp_G
                  \begin{pmatrix}
                    0 & 0 \\ M & 0
                  \end{pmatrix} \quad (M \in \Sym(n, \R)).
\end{align*}

\begin{rem}
  \label{rem:some-elements-sp2n}
  In the case when $G=\Sp(2n, \R)$, one has $s_{\Sp(2n, \R)} = \bigl(
  \begin{smallmatrix}
    0 & -\Id \\ \Id & 0
  \end{smallmatrix}
\bigr)$, $\delta_{\Sp(2n, \R)} = -\Id$,  $u_M = \bigl(
\begin{smallmatrix}
  \Id & M \\ 0 & \Id
\end{smallmatrix}
\bigr)$, and $v_M = \bigl(
\begin{smallmatrix}
  \Id & 0 \\ M & \Id
\end{smallmatrix}
\bigr)$.
\end{rem}

\begin{lem}
  \label{lem:delatG-sG-central}
  The element~$\delta_G$ belongs to the center of~$G$.  More precisely, as elements
  of~$\widetilde{\U}(n)/ \Lambda_G$, $s_G$~is represented by
  $(i\Id, n\pi/2)$ and $\delta_G$
  by $(-\Id, n\pi)$.
\end{lem}
\begin{proof}
  The matrix~$\bigl(
  \begin{smallmatrix}
    0 & -\Id \\ \Id & 0
  \end{smallmatrix}
  \bigr)$ in $\mathfrak{sp}(2n, \R)$ correspond to the matrix $i\Id$ in
  $\mathfrak{u}(n)$. Since, for every~$\theta$ in~$\R$,
  $\exp_{ \widetilde{\U}(n)}(i \theta \Id) = ( e^{i\theta}\Id, n\theta)$, the
  sought for equality follows applying
  $\theta = \pi/2$ and~$\pi$.
\end{proof}

The automorphism of~$G$\index{notation}{160@$g^* = s_G g s_{G}^{-1}$ (Cartan involution of~$G$)}
\[ g \longmapsto g^* \coloneqq s_G g s_{G}^{-1}\]
is therefore an involution that stabilizes the subgroup~$L$. For the case
 when $G$~is $\Sp(2n, \R)$, the restriction of the involution to~$L=\GL(n,\R)$ is
$g\mapsto {}^T\! g^{-1}$. Using this last fact, for every~$S$ in $\Sym(n, \R)\subset
\mathfrak{gl}(n,\R) =\mathrm{Lie}(L)$, one has $\exp_L( S)^* = \exp_L(
-S)$. Using the representative of~$s_G$ in $\widetilde{\U}(n)$
(Lemma~\ref{lem:delatG-sG-central}), we get that the restriction of $g\mapsto
g^*$ to~$K$ is the trivial automorphism. Consequently, $g\mapsto g^*$ is a
Cartan involution of~$L$. (In fact it is already a Cartan involution of~$G$.)

Furthermore, for every~$M$ in~$\Sym(n, \R)$
\[ s_G u_M s_{G}^{-1} = v_{-M} \quad \textrm{and} \quad s_G v_M s_{G}^{-1} = u_{-M}.\]
The group~$L$ acts on $\Sym(n, \R)$ via the homomorphism $\pi\colon L \to
\PGL(n, \R)$: for every~$g$ in~$L$ and every~$M$ in~$\Sym(n, \R)$,
\[ g\cdot M = \pi(g) M\, {}^T \! \pi(g).\]
(This formula involves a lift of~$\pi(g)$ to $\GL(n,\R)$ but the result does
not depend on the lift.)
The stabilizer of~$\Id \in \Sym(n,\R)$ is the group~$K$.
This action is related to the adjoint action of~$L$ since, for~$g$ and~$M$ as
above, $g u_M g^{-1} = u_{ g\cdot M}$ and $g v_M g^{-1} = v_{ g^*\cdot
  M}$. Furthermore, one can check that $g\cdot M$ is nonsingular if and only
if $M$ is, and in that case $(g\cdot M)^{-1}= g^{\ast -1} \cdot M^{-1}$.

The following are a direct generalizations of
Lemma~\ref{lem:triple-fram-lagr}
and Lemma~\ref{lem:triple-fram-lagr-kind-of-converse}.

\begin{lem}
  \label{lem:triple-G}
  Let $M_a$, $M_b$, and~$M_c$ be in~$\Sym(n, \R)$ and let~$\ell_a$, $\ell_b$,
  and~$\ell_c$ be in~$L$. Define the elements of~$G$
  \[ A\coloneqq u_{M_a} s_G \ell_a, \ B\coloneqq u_{M_b} s_G \ell_b, \ \textrm{and }
    C\coloneqq u_{M_c} s_G \ell_c.\]
  Suppose that $CBA=\delta_{G}$. Then
  \begin{itemize}
  \item the matrices $M_a$, $M_b$, and~$M_c$ are nonsingular;
  \item one has $M_{c}^{-1}= \ell_c \cdot M_b$,  $M_{b}^{-1}= \ell_b \cdot M_a$, and
    $M_{a}^{-1}= \ell_a \cdot M_c$;
  \item the element $u_{M_c} v_{-\ell_c \cdot M_b} u_{\ell_{c}^{\ast} \ell_b \cdot M_a}
    s_G$ belongs to~$L$ and is equal to
    $ (\ell_c \ell_{b}^{\ast} \ell_a)^{-1}$.
  \end{itemize}
\end{lem}

\begin{rem}
  When $G=\Sp(2n,\R)$, the product $u_M s_{\Sp(2n, \R)} g$ is equal to $\bigl(
  \begin{smallmatrix}
    M Y & -{}^T Y^{-1} \\ Y & 0
  \end{smallmatrix}
  \bigr)$ where $g=\bigl(
  \begin{smallmatrix}
    Y & 0 \\ 0 & {}^T Y^{-1}
  \end{smallmatrix}
\bigr)$.
\end{rem}

\begin{proof}
  One has
  \begin{align*}
    CBA   &= u_{M_c} s_G \ell_c\  u_{M_b} s_G \ell_b\ u_{M_a} s_G \ell_a\\
    & = u_{M_c} s_G  (u_{\ell_c \cdot M_b} \ell_c)  s_G \ell_b u_{M_a}
      s_G \ell_a  \quad (\text{as }\ell_c u_{M_b} \ell_{c}^{-1} = u_{
      \ell_c \cdot M_b} ) \\
    &= u_{M_c} s_G  u_{\ell_c \cdot M_b}  (s_G \ell^{\ast}_{c})  \ell_b
      u_{M_a} s_G \ell_a\quad (\text{using that }s_{G}^{-1}\ell_{c} s_G =
      \ell_{c}^{\ast}) \\
    & = u_{M_c} (v_{-\ell_c \cdot M_b}  (s_G)^2) \ell^{\ast}_{c}  \ell_b
      u_{M_a} s_G \ell_a \quad (s_G u_{\ell_c \cdot
      M_b} s_{G}^{-1} = v_{-\ell_c \cdot M_b}) \\
    & = u_{M_c} v_{-\ell_c \cdot M_b}  \delta_G  (u_{\ell^{\ast}_{c} \ell_b
      \cdot M_a} \ell^{\ast}_{c} \ell_b) s_G \ell_a \quad (\ell^{\ast}_{c}
      \ell_b u_{ M_a} (\ell^{\ast}_{c} \ell_b)^{-1}=  u_{\ell^{\ast}_{c}
      \ell_b \cdot M_a}) \\
    &= u_{M_c} v_{-\ell_c \cdot M_b} \delta_G u_{\ell^{\ast}_{c} \ell_b \cdot
      M_a} s_G \ell_{c} \ell^{\ast}_{b} \ell_a,
  \end{align*}
thus $u_{M_c} v_{-\ell_c \cdot M_b} u_{\ell^{\ast}_{c} \ell_b \cdot M_a} s_G$ is equal
to $ ( \ell_{c} \ell^{\ast}_{b} \ell_a)^{-1}$ and hence must belongs
to~$L$. Denote by $N_b = \ell_c \cdot M_b$ and $N_a= \ell^{\ast}_{c} \ell_b \cdot
M_a$, then the following product of symplectic matrices
\[
  \begin{pmatrix}
    \Id & M_c \\ 0 & \Id
  \end{pmatrix}
  \begin{pmatrix}
    \Id & 0 \\ -N_b & \Id
  \end{pmatrix}
  \begin{pmatrix}
    \Id & N_a \\ 0 & \Id
  \end{pmatrix}
  \begin{pmatrix}
    0 & -\Id \\ \Id & 0
  \end{pmatrix}
\]
projects to $\pi ( u_{M_c} v_{-\ell_c \cdot M_b} u_{\ell^{\ast}_{c} \ell_b \cdot M_a}
s_G)$ in $\PSp(2n, \R)$ and must thus belong to the subgroup~$\GL(n,\R ) $ of $\Sp(2n,
\R)$. Since the above product of matrices is equal to:
\[
  \begin{pmatrix}
    (\Id -M_c N_b)N_a + M_c & M_c N_b -\Id \\ -N_bN_a +\Id & N_b
  \end{pmatrix}\, ,
\]
it turns out that $M_c N_b=\Id$ and $N_b N_a =\Id$. This implies the first two
assertions 
of the lemma and the third one was already observed.
\end{proof}

Conversely:
\begin{lem}
  \label{lem:triple-G-converse}
  Let $M_a$ be in $\Sym^\ast(n, \R)$ and let $\ell_b$, and~$\ell_c$ be
  in~$L$. Define $M_b \coloneqq \ell_{b}^{\ast} \cdot M_{a}^{-1}$ and $M_c
  \coloneqq \ell_{c}^\ast \ell_b \cdot M_{a}$. Then the product $h\coloneqq u_{M_c}
  v_{ -\ell_c \cdot M_b} u_{ \ell_{c}^{\ast} \ell_b \cdot M_a} s_G$ belongs to~$L$.

  Furthermore denoting $\ell_a \coloneqq \ell_{b}^{\ast -1}
  \ell_{c}^{-1}  h^{-1}$, $A\coloneqq u_{M_a} s_G \ell_a$, $B\coloneqq
  u_{M_b} s_G \ell_b$, and $C\coloneqq u_{M_c} s_G \ell_c$, then $CBA=\delta_{G}$.
\end{lem}

\begin{rem}
  \label{rem:triple-continuous}
  The element~$h$ is of course equal to $u_{M_c} v_{-M_{c}^{-1}} u_{M_c} s_G$
  and varies continuously with~$M_c$.
\end{rem}

The element $s_G$ could be expressed as a product of $u_M$s and $v_M$s,
precisely:
\begin{lem}
  \label{lem:uIdvId_and_sG}
  The product $u_{\Id} v_{-\Id} u_{\Id} s_G$ is trivial in~$G$.\par
  The product $u_{-\Id} v_{\Id} u_{-\Id} s_G$ is equal to~$\delta_G$.
\end{lem}
\begin{proof}
  Let $u$, $v$, and~$s$ be the elements of $\widetilde{\Sp}(2n,\R)$ obtained
  by exponentiating $\bigl(
  \begin{smallmatrix}
    0 & \Id \\ 0 & 0
  \end{smallmatrix}
\bigr)$, $ \bigl(
  \begin{smallmatrix}
    0 & 0 \\ -\Id & 0
  \end{smallmatrix}
\bigr)$, and  $ \bigl(
  \begin{smallmatrix}
    0 & -\pi/2\Id \\ \pi/2 \Id & 0
  \end{smallmatrix}
\bigr)$ respectively.
It is then enough to prove the equality $uvus=e_{ \widetilde{\Sp}(2n,
  \R)}$. Using embeddings as in Section~\ref{sec:universal-coverings}, one can
assume furthermore that $n=1$, hence the equality has to be checked in
$\widetilde{ \SL}(2, \R)$ or even in $\widetilde{\GL}{}^+(2, \R)$.

Elements of $\widetilde{\GL}{}^+(2, \R)$ will be represented by paths in $\GL\!{}^+(2,
\R)$. A path representing~$u$ is
\[ \theta \in [0, \pi/2] \longmapsto
  \begin{pmatrix}
    1 & \sin \theta \\ 0 & 1
  \end{pmatrix},
\]
 a path representing~$v$ is
\[ \theta \in [0, \pi/2] \longmapsto
  \begin{pmatrix}
    1 & 0 \\  -\sin \theta & 1
  \end{pmatrix},
\]
and a path representing~$s$ is
\[ \theta \in [0, \pi/2] \longmapsto
  \begin{pmatrix}
    \cos \theta & -\sin \theta \\ \sin \theta & \cos \theta
  \end{pmatrix}.
\]
Thus the element $uvus$ of $\widetilde{\GL}{}^+(2, \R)$ is represented by the
loop
\[ \theta \in [ 0, \pi/2] \longmapsto
  \begin{pmatrix}
    \cos^3 \theta +  \cos^2 \theta \sin^2 \theta + \sin^2 \theta & \cos \theta \sin \theta ( \cos^2 \theta
    -\cos \theta +1) \\ \sin \theta \cos \theta ( \cos \theta-1) & \sin^2 \theta + \cos^3 \theta
  \end{pmatrix}.
\]
This loop is contained in the subspace $T=\{ A \in \GL\!{}^+(2, \R) \mid
\tr A>0\}$. Since, for every matrix~$A$, and every~$t$
in~$\R$, $\tr( t A + \Id) = t \tr A + 2$ and $\det( tA + \Id) = t^2 \det(A)
+ t \tr A + 1$, the space~$T$ is contractible and the above loop is
homotopically trivial in $\GL\!{}^+(2, \R)$. This concludes the equality $uvus =
e_{ \widetilde{\SL}(2, \R)}$.

Conjugating with~$s_G$ gives $s_G u_{\Id} v_{-\Id} u_{\Id}= e_G$; taking inverses we obtain
$u_{-\Id} v_{\Id} u_{-\Id} s_{G}^{-1}=e_G$, which implies the second equality.
\end{proof}

\begin{rem}
  \label{rem:triple-Ipq}
  For any decomposition $n=p+q$, let $I_{p,q}$ be the symmetric matrix $\bigl(
  \begin{smallmatrix}
    \Id_p & 0 \\ 0 & -\Id_q
  \end{smallmatrix}
\bigr)$, then the element $u_{I_{p,q}} v_{-I_{p,q}} u_{I_{p,q}} s_G$ belongs to~$K$ and is
represented by the element $( I_{p,q}, q\pi)$ of $\hat{K}\subset \widetilde{\U}(n)$.
\end{rem}

\section{Twisted local systems}
\label{sec:twist-repr}

Twisted local systems were defined in
Section~\ref{sec:twist-local-syst}, their moduli space is denoted $\Locdelta(S,G)$. Using trivializations $T'S \simeq S \times
\C^*$, or, what amounts to the same, global sections of $T'S$, one gets a
bijective correspondence with local systems on~$S$:\index{notation}{162@$\Locdelta(S,G)$
(moduli twisted $G$-local systems)}

\begin{prop}
  \label{prop:delta-widehatpi_1-and-rep-pi1S=G}
  Any nonvanishing vector field~$\vec{x}$ on~$S$ induces, via pull back, an
  isomorphism
  \begin{equation*}
    \Locdelta(S, G) \longrightarrow \Loc(S, G).
  \end{equation*}
\end{prop}

For a $\delta$-twisted $G$-local system~$\mathcal{F}$, a \emph{framing}
of~$\mathcal{F}$ is a flat section of the restriction of
$\mathcal{F}_{\Lag{n}}$ to $T'S|_{\partial S}$. The pair $(\mathcal{F},
\sigma)$ is called a \emph{framed local system}.
Equivalently (cf.\ Section~\ref{sec:decor-twist-local})
 a section of the restriction of
$\mathcal{F}_{\Lag{n}}$ to~$\vec{\partial} S$ can be called a \emph{framing}
of~$\mathcal{F}$.%
\index{definition}{framing}%
\index{definition}{framed!local system}%
\index{definition}{local system!framed ---}

Similar to Proposition~\ref{prop:delta-widehatpi_1-and-rep-pi1S=G}, one has

\begin{prop}
  \label{prop:twist-local-syst-decorated}
  Any nonvanishing vector field gives rise \ep{via pull-back} to an isomorphism between
  the space $\Locfdelta(S,G)$ of framed twisted $G$-local systems
  and $\Locf(S,G)$.\index{notation}{164@$\Locfdelta(S,G)$ (moduli of framed twisted
    $G$-local systems)}
\end{prop}

The respective images of $\LocfT(S,G)$ and
$\Mf(S,G)$ will be denoted $\LocfdeltaT(S,G)$ and
$\Mfdelta(S,G)$; their elements are called as well
\emph{transverse} (respectively \emph{maximal}).\index{notation}{166@$\Mfdelta(S,G)$
  (subspace of maximal framed local systems)}\index{notation}{167@$\LocfdeltaT(S,G)$
  (subspace of transverse framed local systems)}

\section{Local systems on the quiver~$\Gamma_\mathcal{T}$}
\label{sec:local-systems}

The discussion of Chapter~\ref{sec:local-systems-their} can be adapted to
representations into~$G$. If~$\Gamma=(V,A)$ is a quiver, a \emph{$G$-local
  system} on~$\Gamma$ is the data $( \{ H_v\}_{v\in V}, \{ g_a\}_{a \in A})$
where, for all~$v$ in~$V$, $H_v$ is a right $G$-space with simply transitive
$G$-action (thus isomorphic to the space~$G$ with the right action coming from the
multiplication), and, for all~$a$ in~$A$, $g_a$ is a $G$-morphism from
$H_{v^-(a)}$ to $H_{v^+(a)}$ (hence an isomorphism).\index{definition}{local system}

Two local systems $(H_v, g_a)$ and $( H'_v, g'_a)$ are \emph{equivalent} if
there is a family $\{ \psi_v\}_{v\in V}$ such that, for every~$v$ in~$V$,
$\psi_v\colon H_v \to H'_v$ is a $G$-morphism, and, for every~$a$ in~$A$,
$\psi_{v^+(a)} \circ g_a =  g'_a \circ \psi_{v^-(a)}$. The moduli space of
$G$-local systems is denoted $\Loc( \Gamma, G)$.\index{notation}{170@$\Loc( \Gamma, G)$
  (moduli of $G$-local systems on a quiver~$\Gamma$)}

Fixing a {base} point~$b_v$ in~$H_v$, for every~$v$ in~$V$, one
obtains, for every~$a$ in~$A$, an element~$G_a$ of~$G$ via the equality: $g_a(
b_{v^-(a)}) = b_{v^+(a)} \cdot G_a$. The tuples $( \{ H_v\}, \{ b_v\}, \{g_a\})$
and $( \{ H_v\}, \{ b_v\}, \{g_a\}, \{G_a\})$ or even the family $\{G_a\}$
will be called $G$-local systems. In fact, there is a one-to-one
correspondence between equivalence classes of \emph{based} local systems and
the space~$G^A$.

For the quiver~$\Gamma_\mathcal{T}$ associated with an ideal
triangulation~$\mathcal{T}$ on the surface~$S$ (cf.\
Sections~\ref{sec:vect-fields-triang} and~\ref{sec:orient-graph-gamm}), one set:
\begin{df}
  \label{def:comp-G-local-systems}
  A local system $( \{H_v\}, \{g_a\})$ on the quiver~$\Gamma_\mathcal{T}$ is
  \emph{$\delta$-twisted} if
  \begin{itemize}
  \item for every $2$-cycle $(a,a')$ in~$A_2$, $g_a \circ g_{a'} = \delta_{G}^{-1}$
    (i.e.\ for every~$b$ in~$H_{v^+(a)}$, $g_a\circ g_{a'}(b)= b\cdot \delta_{G}^{-1}$;
    this holds if and only if there exists~$b$ in~$H_{v^+(a)}$ such that
    $g_a\circ g_{a'}(b)= b\cdot \delta_{G}^{-1}$);
  \item for every $3$-cycle $(a,b,c)$ in $\Gamma_\mathcal{T}$, $g_c \circ
    g_b\circ g_a = \delta_{G} $.
  \end{itemize}
\end{df}\index{definition}{twisted!local system}%
\index{definition}{local system!twisted ---}

Via the restriction map, the moduli space $\Locdelta( \Gamma_\mathcal{T},
G)$ of $\delta$-twisted local systems is isomorphic to the space
$\Locdelta( S, G)$.\index{notation}{172@$\Locdelta( \Gamma_\mathcal{T},
G)$ (moduli of twisted $G$-local systems on~$\Gamma_\mathcal{T}$)}

\section{Framed local systems}
\label{sec:decor-local-syst}

For a $G$-local system $( \{H_v\}, \{g_a\})$, the space~$\Lag{v}$ of Lagrangians
in~$H_v$ is the quotient $G\backslash (H_v \times \Lag{n})$ by the
diagonal action of~$G$: $g\cdot ( b, L) \coloneqq ( b\cdot g^{-1}, g\cdot
L)$. For every~$a$ in~$A$, the isomorphism $g_a \colon H_{v^-(a)} \to
H_{v^+(a)}$ induces an isomorphism $\Lag{ v^-(a)} \to \Lag{ v^+(a)}$ denoted
again by~$g_a$.

When the local system  $( \{H_v\}, \{g_a\})$ is on the
quiver~$\Gamma_\mathcal{T}$ and is $\delta$-twisted, a \emph{framing}
is a family $\{ (L^{t}_{v}, L^{b}_{v})\}_{v\in V}$ such that, for every~$v$
in~$V$, $L^{t}_{v}$ and $L^{b}_{v}$ belong to~$\Lag{v}$, and, for every~$a$
in~$A$, $g_a( L^{b}_{v^-(a)}) = L^{t}_{v^+(a)}$. The tuples $( \{H_v\},
\{g_a\}, \{ (L^{t}_{v}, L^{b}_{v})\} )$ will then be called a framed local system.%
\index{definition}{framing}%
\index{definition}{framed!local system}%
\index{definition}{local system!framed ---}

For~$a$ in~$A_2$, let~$a'$ be
in~$A_2$ the arrow such that $(a,a')$ is a $2$-cycle, then one has
 $g_a( L^{t}_{v^-(a)}) = g_a\circ g_{a'}( L^{b}_{v^+(a)}) = L^{b}_{v^+(a)}$
 since the action of $\delta^{-1}_{G} = g_a\circ g_{a'}$ is trivial on the Lagrangian variety.

The notion of equivalence for framed local systems
is
adapted directly from Section~\ref{sec:decor-local-syst-gamma_T}.
Via the restriction map, the moduli space
$\Locfdelta( \Gamma_\mathcal{T}, G)$ of framed $\delta$-twisted
local systems is isomorphic to $\Locfdelta( S, G)$.\index{notation}{174@$\Locfdelta(
  \Gamma_\mathcal{T}, G)$ (moduli of framed twisted $G$-local systems)}

A framed local system $( \{ H_v\}, \{g_a\}, \{ (L^{t}_{v},
L^{b}_{v})\})$ is called \emph{transverse} if, for every~$v$ in~$V$, the
Lagrangians~$L^{t}_{v}$ and~$L^{b}_{v}$ are transverse. The space $\LocfdeltaT( \Gamma_\mathcal{T}, G)$ of transverse framed
$\delta$-twisted local systems is isomorphic to $\LocfdeltaT(
S, G)$.\index{notation}{176@$\LocfdeltaT( \Gamma_\mathcal{T}, G)$ (subspace of
  transverse local systems)}%
\index{definition}{transverse!framed twisted local system}%
\index{definition}{framed!transverse --- twisted local system}%
\index{definition}{twisted!transverse framed --- local system}%
\index{definition}{local system!transverse framed twisted ---}

A framed local system $( \{ H_v\}, \{g_a\}, \{ (L^{t}_{v},
L^{b}_{v})\})$ is called \emph{maximal} if, for every~$a$ in~$A_3$, the triple
of Lagrangians $( L^{t}_{ v^+(a)}, g_a( L^{t}_{v^-(a)}), L^{b}_{ v^+(a)})$ is
maximal. Such a local system is automatically transverse. The moduli space
$\Mfdelta( \Gamma_\mathcal{T}, G)$ of maximal framed
$\delta$-twisted local systems is isomorphic to
$\Mfdelta(S, G )$.\index{notation}{178@$\Mfdelta( \Gamma_\mathcal{T}, G)$ (subspace of
  maximal ones)}\index{definition}{maximal!framed local system}%
\index{definition}{framed!maximal --- local system}%
\index{definition}{local system!maximal framed ---}%

\section{Parameters}
\label{sec:parameters-1}

We will denote by\index{notation}{180@$\XTG$ (space of $\mathcal{X}$-coordinates)}\index{definition}{$\mathcal{X}$-coordinates}
\[ \XTG \subset L^{A_2} \times \Sym(n, \R)^V \times L^{A_3}\]
the subspace of tuples $z= \bigl( \{ \phi_a\}_{a\in A_2}, \{ M_v\}_{v\in V},
\{ \ell_a\}_{a \in A_3}\bigr)$ such that
\begin{itemize}
\item for all $2$-cycle $(a,b)$ in~$A_2$, $\phi_{a}^{\ast} = \phi_{b}^{-1}$,
\item for all~$v$ in~$V$, $M_v$ is nonsingular,
\item for all~$a$ in~$A_3$, $M_{v^+(a)}^{-1} = \ell_a \cdot M_{v^-(a)}$,
\item for all $3$-cycle $(a,b,c)$, $u_{M_{v^+(c)}} v_{-\ell_c \cdot M_{v^+(b)}} u_{\ell_{c}^{\ast} \ell_b \cdot M_{v^+(a)}}
    s_G \,  \ell_c \ell_{b}^{\ast} \ell_a = e_{G}$.
\end{itemize}

The subspace $\XplusTG$ consists of the tuples~$z$ in\index{notation}{182@$\XplusTG$
  (subspace of positive coordinates)}
$\XTG$ such that
\begin{itemize}
\item for every~$a$ in~$A_2$, $\phi_a$ belongs to $\exp_L( \Sym(n, \R))$, and
\item for every~$v$ in~$V$, $M_v=\Id$.
\end{itemize}

Let~$z$ be in $\XplusTG$. For all~$a$ in~$A_3$ the element $\ell_a$ belongs to~$K$
(since $K$ is the stabilizer of~$\Id$ for the action of~$L$ on~$\Sym(n, \R)$)
and, for all $3$-cycle $(a,b,c)$, $ \ell_c \ell_b \ell_a = e_{G}$ since, by
Lemma~\ref{lem:uIdvId_and_sG}, $u_{\Id} v_{-\Id} u_{\Id} s_G = e_G$ and since
$g^*=g$ for every~$g$ in~$K$. Also for every~$a$ in~$A_2$, there is a
unique~$S_a$ in~$\Sym(n, \R)$, such that $\phi_a = \exp_L( S_a)$ and, for
every $2$-cycle $(a,b)$ in~$A_2$, one has $S_a = S_b$ since $\phi_{b}^{-1} = \phi_{a}^{\ast}$.

Accordingly the space $\XplusTG$ can be also described as the
set of tuples $( \{S_a\}_{a\in A_2}, \{ \ell_a\}_{a\in A_3})$ in $
\Sym(n,\R)^{A_2} \times K^{A_3}$ such that, for every cycle $(a,b)$ in $A_2$,
$S_a=S_b$, and, for every $3$-cycle $(a,b,c)$ in~$A_3$, $\ell_c \ell_b \ell_a
= e_{G}$.

\smallskip

The group $G_\mathcal{X} \coloneqq L^V$ acts on~$\XTG$
via the following rule: for all $m= \{ m_v\}_{v\in V}$ and all $z= \bigl( \{ \phi_a\}_{a\in A_2}, \{ M_v\}_{v\in V},
\{ \ell_a\}_{a \in A_3}\bigr)$,\index{notation}{185@$G_\mathcal{X}$ (symmetry group of $\XTG$)}
\[ m \cdot z \coloneqq \bigl( \{ m_{v^+(a)} \phi_a
  m_{v^-(a)}^{\ast-1}\}_{a\in A_2}, \{ m_{v}\cdot M_v\}_{v\in V},
\{ m_{v^+(a)} \ell_a m_{v^-(a)}^{\ast-1}\}_{a \in A_3}\bigr).\]

\begin{rem}
  The meaningless difference with respect to the formula for the action in
  Section~\ref{sec:over-parametrization} involves just the precomposition with
  $m\mapsto m^\ast$.
\end{rem}

The subgroup $K_\mathcal{X} \subset K^V \subset G_\mathcal{X}$ consisting on
tuples $\{ k_v\}_{v\in V}$ for\index{notation}{187@$K_\mathcal{X}$ (symmetry group of $\XplusTG$)} which, for every~$a$ in~$A_2$, $k_{v^+(a)} =
k_{v^-(a)}$ stabilizes the subspace $\XplusTG$ and there
is therefore an induced action of $K_\mathcal{X}$ on
$\XplusTG$. An element $k=\{ k_v\}_{v\in V}$ acts on $z=
(\{S_a \}_{a\in A_2}, \{ \ell_a\}_{ a\in A_3})$ in $\XplusTG \subset \Sym(n,\R)^{A_2} \times K^{A_3}$ by
\[ k\cdot z = \bigl( \{ k_{v^+(a)} S_a k_{v^-(a)}^{-1} \}_{a\in A_2}, \{
  k_{v^+(a)} \ell_a k_{v^-(a)}^{-1} \}_{ a\in A_3}\bigr)\]

\smallskip

Let us fix a subset $E\subset A_2$ containing one element of every $2$-cycle
and a subset $W\subset V$ containing one of the three vertices in every
triangle of~$\mathcal{T}$. Let us denote by $B$ the subset of $A_3$ of the
arrows~$a$ having one of its endpoints in~$W$. The group~$G_\mathcal{X}$ acts
on $L^{E} \times \Sym^{\ast}(n, \R)^{ W} \times L^B$: for all $m= \{ m_v\}_{v\in V}$ and all $z= \bigl( \{ \phi_a\}_{a\in E}, \{ M_v\}_{v\in W},
\{ \ell_a\}_{a \in B}\bigr)$,
\[ m \cdot z \coloneqq \bigl( \{ m_{v^+(a)} \phi_a
  m_{v^-(a)}^{\ast-1}\}_{a\in E}, \{ m_{v}\cdot M_v\}_{v\in W},
\{ m_{v^+(a)} \ell_a m_{v^-(a)}^{\ast-1}\}_{a \in B}\bigr).\]

Using Lemma~\ref{lem:triple-G} and Lemma~\ref{lem:triple-G-converse} we get:

\begin{lem}
  \label{lem:topo-parameters-ZG}
  The map
  \begin{align*}
    \XTG & \longrightarrow L^{E} \times \Sym^{\ast}(n,
                                   \R)^{ W} \times L^B\\
    \bigl( \{ \phi_a\}_{a\in A_2}, \{ M_v\}_{v\in V},
    \{ \ell_a\}_{a \in A_3}\bigr)
    & \longmapsto \bigl( \{ \phi_a\}_{a\in E}, \{ M_v\}_{v\in W},
\{ \ell_a\}_{a \in B}\bigr)
  \end{align*}
  is a $G_\mathcal{X}$-equivariant diffeomorphism.
\end{lem}

Similarly, there is a natural action of $K_\mathcal{X}$ on $\Sym(n,
\R)^{E}\times K^B$, and we have:

\begin{lem}
  \label{lem:topo-parameters-ZGplus}
  The map
  \begin{align*}
    \XplusTG & \longrightarrow \Sym(n,\R)^{ E} \times  K^B\\
    \bigl( \{ S_a\}_{a\in A_2},
    \{ \ell_a\}_{a \in A_3}\bigr)
    & \longmapsto \bigl( \{ S_a\}_{a\in E},
\{ \ell_a\}_{a \in B}\bigr)
  \end{align*}
  is a $K_\mathcal{X}$-equivariant diffeomorphism.
\end{lem}

\section{Holonomy}
\label{sec:holonomy-central}

Let $z= ( \{ \phi_a\}_{a\in A_2}, \{ M_v\}_{v\in V},
\{ \ell_a\}_{a \in A_3})$ be in $\XTG$ (respectively let $z= ( \{ S_a\}_{a\in A_2},
\{ \ell_a\}_{a \in A_3})$ in $\XplusTG$). We
associate to~$z$ the (based) framed $\delta$-twisted $G$-local system
$\holXGT(z) \coloneqq( \{ H_v\}_{v\in V}, \{ g_a\}_{a\in A}, \{ (L^{t}_{v}, L^{b}_{v})\}_{v\in
  V})$ (respectively $\holXpGT$) where:\index{notation}{189@$\holXGT$, $\holXpGT$ (the
  holonomy maps)}
\begin{itemize}
\item for every~$v$ in~$V$, $H_v = G$ (as a right $G$-space); thus $\Lag{v} =
  \Lag{n}$ is the space of Lagrangians in~$\R^{2n}$;
\item for every~$v$ in~$V$, $L^{t}_{v} = \Span( \mathbf{e}_0)$ and $L^{b}_{v}
  = \Span( \mathbf{f}_0)$ where $( \mathbf{e}_0, \mathbf{f}_0)$ is the
  standard symplectic basis of~$\R^{2n}$;
\item for every~$a$ in~$A_2$, $g_a$ is the map $G \to G \mid h \mapsto h s_{G}^{-1} \phi_a$
  (respectively $h\mapsto h s_{G}^{-1} \exp_L( S_a)$);
\item for every~$a$ in~$A_3$, $g_a$ is the right multiplication by
  $u_{M_{v^+(a)}} s_G \ell_a$ (respectively by $u_{\Id} s_G \ell_a$).
\end{itemize}
The fact that this construction gives indeed a framed twisted local
system follows directly from the conditions defining $\XTG$ and the fact that the action of~$s_G$ on~$\Lag{n}$ permutes
$\Span( \mathbf{e}_0)$ and $\Span( \mathbf{f}_0)$. This local system is
clearly transverse (and maximal in the case $z\in \XplusTG$).

The following results are direct generalizations of
Theorem~\ref{teo:over-parametrization-Z-GZ} and of
Section~\ref{sec:over-param}.

\begin{teo}
  \label{teo:holonomy-ZG}
  The map
  \[ \holXGT\colon \XTG \longrightarrow
    \LocfdeltaT( \Gamma_\mathcal{T}, G)\]
  is onto and its fibers are the orbits of~$G_\mathcal{X}$, hence the quotient
  $G_\mathcal{X} \backslash \XTG$ is isomorphic to $\LocddeltaT( \Gamma_\mathcal{T}, G)$.
\end{teo}

\begin{teo}
  \label{teo:holonomy-ZGplus}
  The map
  \[ \holXpGT\colon \XplusTG \longrightarrow
    \Mfdelta( \Gamma_\mathcal{T}, G)\]
  is onto and its fibers are the orbits of~$K_\mathcal{X}$, hence the quotient
  $K_\mathcal{X} \backslash \XplusTG$ is isomorphic to $\Mfdelta( \Gamma_\mathcal{T}, G)$.
\end{teo}

\section{Connected components}
\label{sec:connected-components-central}

Let~$H$ be the group
$\pi_0(L)\simeq \pi_0(K)$; $H$~is isomorphic to~$\Z/x_G\Z$ (cf.\ Lemma~\ref{lem:group_K_connnected_comp}, where $x_G \in \Z_{>0}$ is the
integer that determines the group~$\Lambda_G$).

For~$s$ in~$\{-n, -n+2, \dots, n\}$, the class in~$H$ of the element $u_{M}
v_{-M^{-1}} u_{M} s_G$ does not depend on the nonsingular symmetric
matrix~$M$ of signature~$s$ (cf.\ Remark~\ref{rem:triple-continuous}). We will
denote by ~$d(s)$ this element of~$H$. In the isomorphism
$H\simeq \Z/x_G\Z$ given by
Lemma~\ref{lem:group_K_connnected_comp}, $d(s)$ is represented by the integer
$(n-s)/2$ modulo~$x_G$ (cf.\ Remark~\ref{rem:triple-Ipq}).

Similarly to Section~\ref{sec:connected-components-1}, let $Z(\mathcal{T}, G)$
be the subset of tuples
\[( \{ h_a\}_{a\in A_2}, \{ s_v\}_{v\in V}, \{ h_a\}_{a\in A_3})\] in
$H^{A_2} \times \{ -n, -n+2, \dots, n\}^V \times H^{A_3}$ such
that\index{notation}{191@$Z(\mathcal{T}, G)$ (space parametrizing the connected
  components of $\XTG$)}
\begin{itemize}
\item for all $2$-cycle $(a,b)$ in~$A_2$, $h_a  h_b = e_{H}$;
\item for all~$v$ and~$v'$ in~$V$ that are in the same triangle
  of~$\mathcal{T}$, $s_v = s_{v'}$;
\item for all $3$-cycle $(a,b,c)$ in~$A_3$, $ h_c h_b h_a = d({s_{v^+(a)}})$.
\end{itemize}
The group $F_{Z}= H^V$ acts in an \enquote{obvious} way on
$Z(\mathcal{T}, G)$. Generalizing
Proposition~\ref{prop:connected-components-LocdT} and
Corollary~\ref{cor:number-connected-components-LocdT}, we
have:\index{notation}{193@$F_Z$ (its symmetry group)}

\begin{prop}
  \label{prop:connected-components-central}
  The natural map $\XTG \to Z( \mathcal{T}, G)$ is
  equivariant with respect to the natural homomorphism $G_\mathcal{X} \to
  F_{Z}$. This map and this homomorphism induce a bijection at the level
  of connected components. As a result, the corresponding map between the set
  of connected components of $\LocfdeltaT(
  \Gamma_\mathcal{T}, G)$ and $F_{Z} \backslash Z(\mathcal{T}, G)$ is
  a bijection.
\end{prop}

As the cardinality of~$H =\pi_0(L)$ is equal to~$x_G$, we
get:

\begin{cor}
  \label{cor:number-connected-components-central}
  The number of connected components of  $\LocfdeltaT(
  \Gamma_\mathcal{T}, G)$ is equal to \[x_{G}^{1-{ \chi( \bar{S})}} \times
  (n+1)^{ r-2{ \chi( \bar{S})}}.\]

  The number of connected components of $\Mfdelta(
  \Gamma_\mathcal{T}, G)$ is equal to $x_{G}^{ 1-{ \chi( \bar{S})}}$.
\end{cor}

\section{Homotopy type of the space of maximal framed local systems}
\label{sec:homot-type-maxim}

Theorem~\ref{teo:holonomy-ZG} and the existence of a $K$-equivariant
retraction of~$\Sym(n,\R)$ on~$\{0\}$ imply

\begin{cor}
  \label{cor:homotopy-type-max-central}
  Let $\mathcal{X}^{+}_{0}( \mathcal{T}, G)$ be the space of tuples $( \{
  S_a\}_{a\in A_2}, \{ \ell_a\}_{a\in A_3})$ such that $S_a=0$ for all~$a\in
  A_2$. Then $K_\mathcal{X}\backslash \mathcal{X}^{+}_{0}( \mathcal{T}, G)$ is
  a strong deformation retract of $\Mfdelta(
  \Gamma_\mathcal{T}, G)$.

  Furthermore the quotient $K_\mathcal{X}\backslash \mathcal{X}^{+}_{0}(
  \mathcal{T}, G)$ is isomorphic to $K \backslash K^{1-{ \chi( \bar{S})}}$
  \ep{quotient by the diagonal conjugation action}.
\end{cor}

In the notation of Section~\ref{sec:subsp-enqu-repr}, when $G=\Sp(2n,\R)$, the
image of $\mathcal{X}^{+}_{0}( \mathcal{T}, G)$ by the holonomy map is the
subspace $\DfSLdOn$. Since, for any~$G$, framed representations in the subspace $K_\mathcal{X}\backslash
\mathcal{X}^{+}_{0}( \mathcal{T}, G)$ projects in $\PSp(2n, \R)$ to the
subspace $\PO(n)_{\mathcal{Z}} \backslash \mathcal{X}^{+}_{0}( \mathcal{T},
\PSp(2n,\R))$, one can apply, in the case $R=\emptyset$, Proposition~\ref{prop:decor-sing-repr-uniq} to
deduce that these representations have a unique framing. From this, exactly
as in Theorem~\ref{teo:connected-components-dec-to-max}, one deduces that, in
this case,
$\Mf(S,G)\to \M(S,G)$ induces a bijection at the level of
connected components.

\begin{cor}
  \label{cor:number-cc-max-G}
  Suppose that $R=\emptyset$.  The space $\M(S,G)$ has
  $x_{G}^{1-{\chi(S)}}$ connected components. Here $x_G \in \Z_{>0}$ is
  the integer that determines the group~$\Lambda_G$
  \ep{Section~\ref{sec:groups}}.
\end{cor}

\chapter{$\mathcal{A}$-coordinates for decorated local systems}\label{sec:Acoordinates}

In this chapter we investigate the coordinates on the space of decorated
$\delta$-twisted local systems given by symplectic $\Lambda$-lengths. We
prove that the $\Lambda$-lengths with respect to a triangulation give a
one-to-one parametrization of the space of transverse (with respect to that
triangulation) decorated representations. Furthermore we show that the
$\Lambda$-lengths produce a geometric realization of the \enquote{noncommutative
surfaces} introduced by Berenstein and Retakh~\cite{BR}; those are
noncommutative algebras associated with the surfaces~$S$ and that exhibit
a mapping class group invariant noncommutative cluster structure.

\section{Symplectic $\Lambda$-length}
\label{sec:sympl-lambda-length}

For every arc~$\alpha$ in~$T'S$ --- this means here that
$\alpha$ is (the homotopy class of) a path $\alpha\colon[0,1] \to T'S$ with
$\alpha( \{0,1\}) \subset \vec{\partial} S$ --- one associates
a $\Lambda$-length $\Lambda_\alpha$ on the space $\Locddelta( S, \Sp(2n,
\R))$. Namely, to every decorated twisted local system  $(\mathcal{F}, \beta)$, its
pull back by~$\alpha$ gives a pair $(\mathbf{v}^t, \mathbf{v}^b)$ of decorated
Lagrangians in~$\R^{2n}$ well defined up to the action of $\Sp(2n, \R)$
(compare with Section~\ref{sec:conf-assoc-with-1}). Thus
the matrix $\Lambda_\alpha(\mathcal{F}, \beta) =\omega( \mathbf{v}^t,
\mathbf{v}^b)$ is well defined.\index{notation}{194@$\Lambda_\alpha$ ($\Lambda$-length
  function associated with the arc~$\alpha$)}\index{definition}{$\Lambda$-length (symplectic ---)}%
\index{definition}{symplectic!$\Lambda$-length}%

\section{Decorated local systems on~$\Gamma_\mathcal{T}$}
\label{sec:framed-local-systems}

Let~$\mathcal{T}$ be an ideal triangulation of~$S$ (cf.\ Section~\ref{sec:triangulations-general}). Using the vector field
$\vec{x}_\mathcal{T}$,
the
space of decorated local systems is isomorphic to the moduli space
$\Locddelta( \Gamma_\mathcal{T}, \Sp(2n,\R))$ of decorated $\delta$-twisted
symplectic local systems on the quiver~$\Gamma_\mathcal{T}$ where we used the
following definition ($V$ denotes the vertex set of~$\Gamma_\mathcal{T}$ and
$A= A_2 \sqcup A_3$ its arrow set):\index{notation}{197@$\Locddelta( \Gamma_\mathcal{T},
  \Sp(2n,\R))$ (moduli of decorated local systems on the quiver~$\Gamma_\mathcal{T}$)}

\begin{df}
  \label{df:framed-local-systems-on-quiver}
  A tuple $( \{ F_v, \mathbf{f}^{t}_{v}, \mathbf{f}^{b}_{v}\}_{v\in V}, \{ g_a\}_{a\in A})$ is a
  \emph{decorated}  $\delta$-twisted
symplectic local system if
\begin{enumerate}
\item $( \{ F_v\}_{v\in V}, \{ g_a\}_{a\in A})$ is a $\delta$-twisted
symplectic local system (cf.\ Definition~\ref{df:pi1S-hat-compat-local-sys});
\item for all~$v$ in~$V$, $f_{v}^{t}$ and $f_{v}^{b}$ are decorated Lagrangians
  in~$F_v$;
\item for all~$a$ in~$A$, $g_a( f_{v^{-}(a)}^{b})= f_{v^{+}(a)}^{t}$.
\end{enumerate}
\end{df}\index{definition}{decorated! twisted local system}%
\index{definition}{twisted!decorated --- local system}%
\index{definition}{local system!decorated twisted ---}%

The associated framed local system is
$( \{ F_v, L^{t}_{v}, L^{b}_{v}\}_{v\in V}, \{ g_a\}_{a\in A})$ where, for
all~$v$, $L^{t}_{v} = \Span( \mathbf{f}^{t}_{v})$ and $L^{b}_{v} = \Span(
\mathbf{f}^{b}_{v})$.

The decorated local system will be called transverse (or
$\mathcal{T}$-transverse) if the associated framed local system is
transverse, i.e.\ if and only if, for every~$v$ in~$V$, the pair $(\mathbf{f}^{t}_{v},
\mathbf{f}^{b}_{v})$ is a basis of~$F_v$. This happens if and only if, for any
arc~$\alpha$ in~$T'S$ that projects (in~$S$) to an oriented edge of~$\mathcal{T}$,
the symplectic $\Lambda$-length $\Lambda_\alpha$ is an invertible matrix.

\section{Lifting arcs}
\label{sec:lift-orient-edges}

Let $\alpha\colon ([0,1], \{0,1\}) \to (S, \partial S)$ be an arc in~$S$. We
construct a \emph{lift}\index{notation}{199@$r(\alpha)$ (the lift of the arc~$\alpha$)}
\[r(\alpha) \colon ([0,1], \{0,1\}) \to (T'S,
  \vec{\partial} S)\]
via the following procedure: choose first a $\mathcal{C}^1$-representative
of~$\alpha$ having the minimal possible self-intersections (for example the
geodesic representative for a hyperbolic structure) and then smooth this
representative at its endpoints so that it becomes tangent at the boundary
there; then $r(\alpha):[0,1]\to T'S$ is the tangent curve of this
last curve.

It is easy to observe that one can choose representatives of $r(\alpha)$ and
of $r(\bar{\alpha})$ so that $r(\alpha) \sqcup r(\bar{\alpha})$ is homotopic
to a fiber $T'S\to S$. In particular, the holonomy around this loop of a
twisted local system is~$-\Id$, therefore the equality $\Lambda_{ \overline{
    r(\alpha)}} = - \Lambda_{r(\bar{\alpha})}$
     follows.

\section{$\mathcal{A}$-space}
\label{sec:a-space}

Note that the following properties hold:
\begin{enumerate}
\item \label{item:1:sec:a-space} for every arc~$\alpha$ in~$S$,
  $\Lambda_{r(\bar{\alpha})} = {}^T \!
  \Lambda_{r(\alpha)}$, simply since $\Lambda_{r(\bar{\alpha})} = -
  \Lambda_{ \overline{r(\alpha)}}$ (see above) and since $\Lambda_{ \overline{r(\alpha)}} = - {}^T \! \Lambda_{r(\alpha)}$;
\item \label{item:2:sec:a-space} for every triangle
  $(\alpha_1, \alpha_2, \alpha_3)$ in~$S$, the matrix
  $\Lambda_{r(\alpha_1)} (\Lambda_{r(\bar{\alpha}_2)})^{-1}
  \Lambda_{r(\alpha_3)} $ is symmetric and is equal to
  $\Lambda_{r(\bar{\alpha}_3)} (\Lambda_{r(\alpha_2)})^{-1}
  \Lambda_{r(\bar{\alpha}_1)} $ (cf.\ Corollary~\ref{coro:tri-rel}).
\end{enumerate}

Let~$\mathcal{T}$ an ideal triangulation of~$S$. Recall that we
 denote by~$\edgeT$ its set of oriented edges (cf.\
 Section~\ref{sec:triangulations-general}). Hence if $\alpha\in \edgeT$, the
 reverse $\bar{\alpha}$ also belongs to~$\edgeT$.

We define $\mathcal{A}( \mathcal{T}, n)$ to be the space of tuples, indexed
 by $\edgeT$ 
the set of oriented edges of~$\mathcal{T}$,
$\{ G_\alpha\}_{\alpha \in \edgeT} \in \GL(n, \R)^{\edgeT}$
satisfying the above equations:\index{notation}{201@$\mathcal{A}( \mathcal{T}, n)$
  (space of $\mathcal{A}$-coordinates)}\index{definition}{$\mathcal{A}$-coordinates}
\begin{itemize}
\item for all~$\alpha$ in~$\edgeT$,
  $G_{\bar{\alpha}} = {}^T \! G_{\alpha}$, and
\item  for every triangle
  $(\alpha_1, \alpha_2, \alpha_3)$ in~$\mathcal{T}$,
  $G_{\alpha_1} (G_{\bar{\alpha}_2})^{-1} G_{\alpha_3} = G_{\bar{\alpha}_3}
  (G_{\alpha_2})^{-1} G_{\bar{\alpha}_1}$.
\end{itemize}

We call $\mathcal{A}( \mathcal{T}, n)$ the \emph{space of $\mathcal{A}$-coordinates}.

Our main result is
\begin{teo}
  \label{teo:a-space-1to1}
  The map\index{notation}{203@$\Psi_\mathcal{T}$ (the parametrization of $\LocddeltaT( S, \Sp(2n,\R))$)}
  \[ \Psi_\mathcal{T} = \{ \Lambda_{r(\alpha)}\}_{\alpha\in
      \edgeT} \colon \LocddeltaT( S, \Sp(2n,\R)) \longrightarrow
    \mathcal{A}( \mathcal{T}, n)\]
  is a one-to-one correspondence.
\end{teo}

\begin{proof}
  We need first to prove that a transverse decorated $\delta$-twisted
  symplectic local system $x=( F_v, f_{v}^{t}, f_{v}^{b}, g_a)$ is completely
  determined by the matrices $H_\alpha \coloneqq \Lambda_{r(\alpha)}(x)$
  ($\alpha \in \edgeT$).

  For every~$v$ in~$V$, the arc~$\alpha_v$ (cf.\
  Section~\ref{sec:orient-graph-gamm}) is an oriented edge of~$\mathcal{T}$ and we will
  write~$H_v$ instead of~$H_{\alpha_v}$.

  By definition of the symplectic $\Lambda$-length (and the correspondence
  between decorated local systems and decorated representations) one has $H_v =
  \omega( \mathbf{f}^{t}_{v}, \mathbf{f}^{b}_{v})$ where~$\omega$ is the symplectic form on the
  space~$F_v$. Also the pair $(\mathbf{f}^{t}_{v}, \mathbf{f}^{b}_{v})$ is a basis of~$F_v$.

  To prove injectivity of~$\Psi_\mathcal{T}$, it is enough to show that the
  transition maps~$g_a$ are uniquely determined. This is obvious when~$a$
  belongs to~$A_2$ since, the local system being $\delta$-twisted, one
  has $g_a( f_{v^-(a)}^{t}) = - f_{v^+(a)}^{b}$. Thus the matrix of~$g_a$ in
  the bases $(\mathbf{f}^{t}_{v^-(a)}, \mathbf{f}^{b}_{v^-(a)})$ and  $(\mathbf{f}^{t}_{v^+(a)},
  \mathbf{f}^{b}_{v^+(a)})$ is $ \bigl(
  \begin{smallmatrix}
    0 & \Id \\ -\Id & 0
  \end{smallmatrix}
  \bigr)$.

  Suppose now that~$a$ belongs to a $3$-cycle $(a,b,c)$ of~$A_3$, so that
  $v^+(a)=v^-(b)$, $v^+(b)=v^-(c)$, $v^+(c)=v^-(a)$, and $g_c g_b g_a
  =-\Id$. Together with the condition
  $g_a( \mathbf{f}^{b}_{ v^-(a)}) = \mathbf{f}^{t}_{v^+(a)}$, one has
  $g_a( \mathbf{f}^{t}_{ v^-(a)}) = g_a \circ g_c ( \mathbf{f}^{b}_{ v^-(c)}) = -g_{b}^{-1}(
  \mathbf{f}^{b}_{ v^-(c)})$. Also
  \[\omega\bigl( g_a( \mathbf{f}^{t}_{ v^-(a)}), \mathbf{f}^{t}_{v^-(b)}\bigr) = \omega\bigl(
    g_a( \mathbf{f}^{t}_{  v^-(a)}), g_a(\mathbf{f}^{b}_{v^-(a)})\bigr)= \omega( \mathbf{f}^{t}_{ v^-(a)},
    \mathbf{f}^{b}_{v^-(a)}) = H_{v^-(a)};\]
  similarly
  \[\omega\bigl( \mathbf{f}^{b}_{v^-(b)}, g_a( \mathbf{f}^{t}_{v^-(a)})\bigr) = -\omega\bigl(
  \mathbf{f}^{b}_{v^-(b)}, g_{b}^{-1}( \mathbf{f}^{b}_{v^+(c)})\bigr) = -H_{v^-(c)}.\] By
  Lemma~\ref{lamb} this implies that
  \[g_a( f_{v^-(a)}^{t}) = \mathbf{f}^{t}_{v^-(b)}\cdot {}^T \! H_{v^-(b)}^{-1}
  H_{v^-(c)} - \mathbf{f}^{b}_{v^-(b)} \cdot H_{v^-(b)}^{-1} {}^T\! H_{v^-(a)}\] and the
  matrix of~$g_a$ is
  \[
    \begin{pmatrix}
       {}^T \! H_{v^-(b)}^{-1} H_{v^-(c)} & \Id \\
       - H_{v^-(b)}^{-1} {}^T\! H_{v^-(a)} & 0
    \end{pmatrix}.
  \]
  (Of course this matrix does not belong to the group $\Sp(2n, \R)$, it is
  nevertheless the matrix of a symplectic isomorphism $F_{v^-(a)}\to
  F_{v^+(a)}$ in the given bases.)

  \smallskip

  Conversely given a family $\{G_\alpha\}_{\alpha\in \edgeT}$ in
  $\mathcal{A}( \mathcal{T}, n)$, we can define
  \begin{itemize}
  \item for each~$v$ in~$V$, a symplectic vector space~$F_v$ with a basis $(
    \mathbf{f}^{t}_{v}, \mathbf{f}^{b}_{v})$ such that $\omega( \mathbf{f}^{t}_{v}, \mathbf{f}^{t}_{v}) =0$,
    $\omega( \mathbf{f}^{b}_{v}, \mathbf{f}^{b}_{v}) =0$, $\omega( \mathbf{f}^{t}_{v}, \mathbf{f}^{b}_{v})
    =G_{\alpha_v}$. To simplify a little the notation, the latter matrix will
    be denoted~$G_v$;
  \item for each~$a$ in~$A_2$, $g_a$ to be the linear map
    $F_{v^-(a)}\to F_{v^+(a)}$ whose matrix (in the given bases) is $\bigl(
    \begin{smallmatrix}
      0 & \Id \\ -\Id & 0
    \end{smallmatrix}
    \bigr)$;
  \item for each~$a$ in~$A_3$, hence belonging to a $3$-cycle $(a,b,c)$, the
    matrix of $g_a\colon F_{v^-(a)}\to F_{v^+(a)}$ is $\Bigl(
    \begin{smallmatrix}
      {}^T \! G_{v^-(b)}^{-1} G_{v^-(c)} & \Id \\
      -G_{v^-(b)}^{-1} {}^T \! G_{v^-(a)} & 0
    \end{smallmatrix}
    \Bigr)$.
  \end{itemize}
  Then $( F_v, \mathbf{f}^{t}_{v}, \mathbf{f}^{b}_{v}, g_a)$ is a decorated $\delta$-twisted
  symplectic local system on~$\Gamma_\mathcal{T}$ and its image under the
  map~$\Psi_\mathcal{T}$ is the family $\{ G_\alpha\}$. This concludes the
  proof of the theorem.
\end{proof}

\section{Change of coordinates: flips}
\label{sec:change-coord-flips}

For every pair $(\mathcal{T}_0, \mathcal{T}_1)$ of triangulations, we would
like to understand the change of coordinates from $\mathcal{A}( \mathcal{T}_0,
n)$ to $\mathcal{A}( \mathcal{T}_1, n)$. It is enough to do this
when~$\mathcal{T}_0$ and~$\mathcal{T}_1$ differ by a flip, i.e.\ when there
are internal nonoriented edges~$e_0$ in~$\mathcal{T}_0$ and~$e_1$ in~$\mathcal{T}_1$ such
that $\mathcal{T}_0 \setm \{ e_0\} = \mathcal{T}_1 \setm \{ e_1\}$. We will
furthermore choose oriented edges~$\alpha_0$ and~$\alpha_1$ representing~$e_0$
and~$e_1$ respectively and such that the starting point of~$\alpha_0$ is
in the triangle of~$\mathcal{T}_1$ that is to the left of~$\alpha_1$ (cf.\ Figure~\ref{flip}).

\begin{figure}[ht]
\begin{center}
\begin{tikzpicture}
  \coordinate (L1) at (0,-1) ;
  \coordinate (L2) at (2,.5) ;
  \coordinate (L3) at (-0.5,2) ;
  \coordinate (L4) at (-2,0) ;
  \draw[middlearrow={latex}] (L1)--(L2) node[midway,below]{$b$} ;
  \draw[secondthirdarrow={latex}] (L1)--(L3) node[pos=0.6, right]{$\alpha_1$};
  \draw[middlearrow={latex}] (L4)--(L1)  node[midway,below]{$a$} ;
  \draw[middlearrow={latex}] (L3)--(L2)  node[midway,above]{$d$} ;
  \draw[secondthirdarrow={latex}] (L4)--(L2) node[pos=0.65,below]{$\alpha_0$} ;
  \draw[middlearrow={latex}] (L4)--(L3)  node[midway,above]{$c\ $} ;
\end{tikzpicture}
\caption{The oriented edges involved in a flip}
\label{flip}
\end{center}
\end{figure}

The oriented edge~$\alpha_0$ bounds two triangles in~$\mathcal{T}_0$, one containing the
starting point of~$\alpha_1$ and whose other oriented edges are called~$a$, $b$:
precisely~$a$, $b$, and~$\bar{\alpha}_0$ form a cycle bounding the corresponding
triangle. The other triangle contains the endpoint of~$\alpha_1$, its oriented edges~$c$,
$d$, and~$\bar{\alpha}_0$ forming as well a cycle in~$\mathcal{T}_0$.

The following is an easy consequence of Proposition~\ref{prop:Ptolemy} (one
has to take care of a few sign changes due to the fact that we are working
with $\delta$-twisted system, however it is easily verified that each term
in Proposition~\ref{prop:Ptolemy} is changed with the same sign).

\begin{prop}
  \label{prop:change-coord-flips}
  With the above notation, one has the identity \ep{between functions defined on
  $\LocddeltaTone( S, \Sp(2n,
  \R))$}
\[ \Lambda_{r(\alpha_0)} = \Lambda_{r(a)} (\Lambda_{r(\bar{\alpha}_1)})^{-1}
  \Lambda_{r(d)} + \Lambda_{r(c)} (\Lambda_{r(\alpha_1)})^{-1} \Lambda_{r(b)}.\]
\end{prop}

\begin{rem}
  Of course, this result holds for any quadrilateral (cf.\ below Section~\ref{sec:algebra-berenst-reta}) and not only for
  those coming from triangulations.
\end{rem}

As a consequence:
\begin{teo}
  \label{teo:change-coord-flips}
The image by~$\Psi_{\mathcal{T}_1}$ \ep{cf.\
  Theorem~\ref{teo:a-space-1to1}} of the space
  \[ \LocddeltaTzero( S, \Sp(2n,
    \R)) \cap \LocddeltaTone( S,
    \Sp(2n, \R))\]
  is the set of tuples $\{ G_\alpha\}_{\alpha\in \edgeTsub{1}}$ in~$\mathcal{A}(
  \mathcal{T}_1, n)$ such that $G_a (G_{\bar{\alpha}_1})^{-1} G_d + G_c
  (G_{\alpha_1})^{-1} G_b$ is invertible. The composition $\Psi_{\mathcal{T}_0}
  \circ \Psi_{\mathcal{T}_1}^{-1}$ is then $\{ G_\alpha\}_{\alpha\in \edgeTsub{1}}
  \mapsto \{ H_\alpha\}_{\alpha\in \mathcal{T}_0}$ where $H_\alpha=G_\alpha$ if $\alpha\notin \{ \alpha_0,
  \bar{\alpha}_0\}$, $H_{\alpha_0} = G_a (G_{\bar{\alpha}_1})^{-1} G_d + G_c
  (G_{\alpha_1})^{-1} G_b$, and $H_{\bar{\alpha}_0} = {}^T \! H_{\alpha_0}$.
\end{teo}

As pointed out in Section~\ref{sec:Ptolemy} the formula for the flip in Proposition~\ref{prop:change-coord-flips} can be seen as a noncommutative Ptolemy-relation. Thus the $\mathcal{A}$-space we introduce here is a noncommutative generalization of Penner's parametrization of the decorated Teichmüller space.

\begin{rem}
  Since the $\Lambda$-lengths completely determine the decorated local system
  which in turn determine a framed local system, they also determine
  completely the possible $\mathcal{X}$-coordinates.  However deriving the
  $\mathcal{X}$-coordinates from the $\Lambda$-lengths involves diagonalizing
  matrices, hence this nice formula for the flip does not easily descend to
  $\mathcal{X}$-coordinates. In Section \ref{sec:mapAX} we will show how to use
  $\mathcal{A}$-coordinates to compute the formula for the flip of the cross ratios,
  this partially describes the behavior of the $\mathcal{X}$-coordinates under flips.
\end{rem}

\section{$\mathcal{A}$-coordinates for maximal decorated twisted local systems}
Using the results from Section~\ref{sec:symp_cross} we can describe the subspace of $\mathcal{A}$-coordinates that parametrizes maximal decorated twisted representations.

\begin{prop}\label{prop:Amax}
  Let~$\mathcal{T}$ be an ideal triangulation of~$S$.

A decorated twisted representation in $\Locddelta( S, \Sp(2n,\R))$ is
maximal if and only if it is $\mathcal{T}$-transverse and, for every triangle
$(e_1, e_2, e_3)$ in~$\edgeT$, the symmetric matrix
 \[\Lambda_{r(e_1)} \Lambda_{r(\bar{e}_2)}^{-1} \Lambda_{r(e_3)}\]
is positive definite.
\end{prop}

Theorem~\ref{teo:toledo_maslov} implies that these conditions are invariant
under a flip, and thus independent of the triangulation (this can be also
checked by a direct calculation using the formulas of Theorem~\ref{teo:change-coord-flips}). Therefore Proposition~\ref{prop:Amax} gives a parametrization of the space of maximal decorated twisted representation.

\section{From $\mathcal{A}$-coordinates to $\mathcal{X}$-coordinates}
\label{sec:mapAX}
We now describe the relation between $\mathcal{X}$-coordinates and  $\mathcal{A}$-coordinates, and derive explicit formulas for the flip.

In Section~\ref{sec:symp_cross} we expressed the cross ratio of four pairwise transverse decorated Lagrangians in terms of the symplectic $\Lambda$-lengths, namely let $(L_i,\mathbf{v_i})\in \Lagd{n}$, with $i\in\{1,2,3,4\}$, be four pairwise transverse framed Lagrangians. Then
\[[L_1,L_2,L_3,L_4]_{\mathbf{v_1}}
=-\Lambda_{41}^{-1}\Lambda_{43}\Lambda_{23}^{-1}\Lambda_{21},\]
where $[L_1,L_2,L_3,L_4]_{\mathbf{v_1}}$ denotes the cross ratio expressed in the basis $\mathbf{v_1}$.

We think of the cross ratio $[L_1,L_2,L_3,L_4]_{\mathbf{v_1}}$ as being
associated to an oriented edge/an arc $\alpha$ from the decorated Lagrangian
$(L_1, \mathbf{v}_1)$ to the decorated Lagrangian $(L_3, \mathbf{v}_3)$. We
write $\CR_{\alpha}$ for this cross ratio, and call this the cross ratio
of~$\alpha$.\index{notation}{205@$\CR_{\alpha}$ (cross ration function associated with
  the arc~$\alpha$)}\index{definition}{cross ratio}

This allows to define a map from the space of $\mathcal{A}$-coordinates to the
space of $\mathcal{X}$-coordinates:

Given a decorated twisted local system $(\mathcal{F}, \beta)$ we get for every
oriented edge $\alpha \in \edgeT$ a cross ratio $\CR_\alpha$,
(or~$\CR^{\mathcal{T}}_{\alpha}$ if we need to remember the
triangulation). Lemma~\ref{lem:symp_cross} implies that
$\CR^{\mathcal{T}}_{\bar{ \alpha}} = \Lambda_{\alpha}^{-1}
\,{}^T\!\!\CR^{\mathcal{T}}_{\alpha} \Lambda_{\alpha}$.

The next proposition shows that the formulas for the change of the cross ratios~$\CR_\alpha$ under a flip
have a nice form, that is just a noncommutative analog of the classical formula.

Oriented edges and triangles in the $8$-gon are completely determined by
their
extremities and we will designate them by their extremities: e.g.\ $62$ is the
oriented edge from the vertex~$6$ to the vertex~$2$.
Consider the following two triangulations~$\mathcal{T}$
and~$\mathcal{T}'$ of the $8$-gon: the triangles $234$, $456$,
$678$, and~$812$ belong to both~$\mathcal{T}$ and~$\mathcal{T}'$, the edge $62$
belongs to~$\mathcal{T}$ and the edge~$84$ belongs to~$\mathcal{T}'$ \ep{cf.\
Figure~\ref{flipA}}. Hence~$\mathcal{T}$ and~$\mathcal{T}'$ differ by a flip
in the quadrilateral $2468$.

\begin{prop}\label{cross_ratio_flip}
Consider eight framed Lagrangians $(L_i,\mathbf{v_i})$
\ep{$i\in\{1,\dots,8\}$} thought as being associated with the vertices of a
$8$-gon.
We have the following formulas for the flip along the edge $62$:

\begin{figure}[ht]
\begin{center}
\begin{tikzpicture}
  \coordinate (L1) at (-1.8,1.8) ;
  \coordinate (L2) at (0,2) ;
  \coordinate (L3) at (1.8,1.8) ;
  \coordinate (L4) at (2,0) ;
  \coordinate (L5) at (1.8,-1.8) ;
  \coordinate (L6) at (0,-2) ;
  \coordinate (L7) at (-1.8,-1.8) ;
  \coordinate (L8) at (-2,0) ;
  \draw (L1) node[above left]{$L_1$} ;
  \draw (L2) node[above]{$L_2$} ;
  \draw (L3) node[above right]{$L_3$} ;
  \draw (L4) node[right]{$L_4$} ;
  \draw (L5) node[below right]{$L_5$} ;
  \draw (L6) node[below]{$L_6$} ;
  \draw (L7) node[below left]{$L_7$} ;
  \draw (L8) node[left]{$L_8$} ;
  \draw (L1)--(L2) ;
  \draw (L2)--(L3) ;
  \draw (L3)--(L4) ;
  \draw (L4)--(L5) ;
  \draw (L5)--(L6) ;
  \draw (L6)--(L7) ;
  \draw (L7)--(L8) ;
  \draw (L8)--(L1) ;
  \draw[middlearrow={latex}] (L8)--(L2) node[midway, sloped,
  above]{$\CR^{\mathcal{T}}_{82}$} ;
  \draw[middlearrow={latex}] (L6)--(L8) node[midway, sloped,
  above]{$\CR^{\mathcal{T}}_{68}$} ;
  \draw[middlearrow={latex}] (L6)--(L4) node[midway, sloped,
  above]{$\CR^{\mathcal{T}}_{64}$} ;
  \draw[middlearrow={latex}] (L4)--(L2) node[midway, sloped,
  above]{$\CR^{\mathcal{T}}_{42}$} ;
  \draw[middlearrow={latex}] (L6)--(L2) node[midway, sloped,
  above]{$\CR^{\mathcal{T}}_{62}$} ;
  \coordinate (Lp1) at (4.2,1.8) ;
  \coordinate (Lp2) at (6,2) ;
  \coordinate (Lp3) at (7.8,1.8) ;
  \coordinate (Lp4) at (8,0) ;
  \coordinate (Lp5) at (7.8,-1.8) ;
  \coordinate (Lp6) at (6,-2) ;
  \coordinate (Lp7) at (4.2,-1.8) ;
  \coordinate (Lp8) at (4,0) ;
  \coordinate (e1) at (2.3,1) ;
  \coordinate (e2) at (3.7,1) ;
  \draw [->] (e1)--(e2) node[midway, above]{flip};
  \draw (Lp1) node[above left]{$L_1$} ;
  \draw (Lp2) node[above]{$L_2$} ;
  \draw (Lp3) node[above right]{$L_3$} ;
  \draw (Lp4) node[right]{$L_4$} ;
  \draw (Lp5) node[below right]{$L_5$} ;
  \draw (Lp6) node[below]{$L_6$} ;
  \draw (Lp7) node[below left]{$L_7$} ;
  \draw (Lp8) node[left]{$L_8$} ;
  \draw (Lp1)--(Lp2) ;
  \draw (Lp2)--(Lp3) ;
  \draw (Lp3)--(Lp4) ;
  \draw (Lp4)--(Lp5) ;
  \draw (Lp5)--(Lp6) ;
  \draw (Lp6)--(Lp7) ;
  \draw (Lp7)--(Lp8) ;
  \draw (Lp8)--(Lp1) ;
  \draw[middlearrow={latex}] (Lp8)--(Lp2) node[midway, sloped,
  above]{$\CR^{\mathcal{T}'}_{82}$} ;
  \draw[middlearrow={latex}] (Lp6)--(Lp8) node[midway, sloped,
  above]{$\CR^{\mathcal{T}'}_{68}$} ;
  \draw[middlearrow={latex}] (Lp6)--(Lp4) node[midway, sloped,
  above]{$\CR^{\mathcal{T}'}_{64}$} ;
  \draw[middlearrow={latex}] (Lp4)--(Lp2) node[midway, sloped,
  above]{$\CR^{\mathcal{T}'}_{42}$} ;
  \draw[middlearrow={latex}] (Lp8)--(Lp4) node[midway, sloped,
  above]{$\CR^{\mathcal{T}'}_{84}$} ;
\end{tikzpicture}
\caption{The flip in a $8$-gon}
\label{flipA}
\end{center}
\end{figure}

\begin{align*}
  \CR^{\mathcal{T}'}_{84} & = \Lambda_{68}^{-1}
\, {}^{T}\!\!\CR^{\mathcal{T}-1}_{62}  \Lambda_{68} \\
  \CR^{\mathcal{T}'}_{64} & =\CR^{\mathcal{T}}_{64} (\Id+\CR^{\mathcal{T}-1}_{62})^{-1},
\quad
\CR^{\mathcal{T}'}_{82}= (\Id+ \Lambda_{68}^{-1} \,
                            {}^T\!\!\CR^{\mathcal{T}-1}_{62} \Lambda_{68}) \CR^{\mathcal{T}}_{82} \\
\CR^{\mathcal{T}'}_{68} & =
                        (\Id+\CR^{\mathcal{T}}_{62})  \CR^{\mathcal{T}}_{68}, \quad
\CR^{\mathcal{T}'}_{42}=\CR^{\mathcal{T}}_{42} (\Id+ \Lambda_{64}^{-1}
                          \CR^{\mathcal{T}}_{62} \Lambda_{64}) .
\end{align*}
\end{prop}

Note that in the case $n=1$, these are precisely the formulas for the flip, see for example~\cite[Formula (1.30)]{FG}, and here we have a noncommutative generalization of them.

\begin{rem}
The same formulas apply to the computation of a flip in the case when some of
the edges 82, 68, 64, 42 are external. Since no cross ratios are associated with external edges, in this case it is only necessary to ignore the external edges.
\end{rem}

\begin{proof}
  One has
  $\CR^{ \mathcal{T}}_{62} = -\Lambda_{46}^{-1} \Lambda_{42} \Lambda_{82}^{-1}
  \Lambda_{86}$ and
  $\CR^{ \mathcal{T}'}_{84} = -\Lambda_{68}^{-1} \Lambda_{64} \Lambda_{24}^{-1}
  \Lambda_{28}$ so that, using the identities
  $\Lambda_{ji} = -{}^{T}\!\! \Lambda_{ij}$,
  ${}^T\!\!  \CR^{\mathcal{T}-1}_{62} = -\Lambda_{64} \Lambda_{24}^{-1}
  \Lambda_{28} \Lambda_{68}^{-1}$. Thus
  $\CR^{ \mathcal{T}'}_{84} = -\Lambda_{68}^{-1} (\Lambda_{64}
  \Lambda_{24}^{-1} \Lambda_{28} \Lambda_{68}^{-1}) \Lambda_{68} =
  \Lambda_{68}^{-1} {}^T\!\!  \CR^{\mathcal{T}-1}_{62} \Lambda_{68}$.

  The flip relation for $\Lambda$-lengths implies
  \begin{align*}
    \Lambda_{84} & = \Lambda_{86} \Lambda_{26}^{-1} \Lambda_{24} +
                   \Lambda_{82} \Lambda_{62}^{-1} \Lambda_{64} = \Lambda_{86}(
                   \Id + \Lambda_{86}^{-1} \Lambda_{82} \Lambda_{62}^{-1}
                   \Lambda_{64}  \Lambda_{24}^{-1} \Lambda_{26})
                   \Lambda_{26}^{-1} \Lambda_{24} \\
    \intertext{ by the triangle relation $ \Lambda_{62}^{-1}
                   \Lambda_{64}  \Lambda_{24}^{-1} \Lambda_{26} =
    -\Lambda_{42}^{-1} \Lambda_{46}$:}
    & = \Lambda_{86}(
                   \Id - \Lambda_{86}^{-1} \Lambda_{82} \Lambda_{42}^{-1} \Lambda_{46})
                   \Lambda_{26}^{-1} \Lambda_{24} = \Lambda_{86}(
                   \Id + \CR^{\mathcal{T}-1}_{62})
                   \Lambda_{26}^{-1} \Lambda_{24}.
  \end{align*}
  Therefore
  \begin{align*}
    \CR^{\mathcal{T}'}_{64} &= -\Lambda_{56}^{-1} \Lambda_{54} \Lambda_{84}^{-1}
  \Lambda_{86} = -\Lambda_{56}^{-1} \Lambda_{54} \Lambda_{24}^{-1} \Lambda_{26}
    ( \Id + \CR^{\mathcal{T}-1}_{62})^{-1} \\
    &= \CR^{\mathcal{T}}_{64} ( \Id +
  \CR^{\mathcal{T}-1}_{62})^{-1}.
  \end{align*}

  The proof for other cross ratios is similar.
\end{proof}

\section{Berenstein--Retakh's algebra}
\label{sec:algebra-berenst-reta}

In~\cite{BR}, Berenstein and Retakh introduced the following
noncommutative instance of cluster algebras, which we shortly recall. Let $\Gamma(S)$ be the set of
homotopy classes (relative to the boundary) of arcs in~$S$: elements of
$\Gamma(S)$ are homotopy class of maps between pairs $\alpha\colon ( [0,1],
\{0,1\}) \to (S, \partial S)$. (Actually
Berenstein and Retakh have a more refined version incorporating orbifolds
points of order~$2$.)\index{notation}{207@$\Gamma(S)$ (homotopy classes of arcs)}

Recall (Section~\ref{sec:conf-assoc-with-1}) that, for every~$p$ in
$\{2,3, \dots\}$, a $p$-gon is (the homotopy class of) a map
$f\colon ( \mathbb{D}, \mu_p) \to ({S}, \partial S)$ where
$\mu_p = \{ z\in \C \mid z^p =1\}$. \index{notation}{209@$\mu_p$ (the $p$-th roots of unity)}

Every $3$-gon~$t$ defines~$6$ elements in~$\Gamma(S)$: there are the
$\alpha_{k,\ell}(t)$ for $k\neq \ell$ in $\Z/3\Z$ where~$\alpha_{k,\ell}(t)$ is
the restriction of~$t$ to an arc going from~$\mathbf{e}(k/3)=e^{2ik\pi/3}$
to~$\mathbf{e}(\ell/3)=e^{2i\ell\pi/3}$ (one can choose the segment $[\mathbf{e}(k/3),
\mathbf{e}(\ell/3)]$ in~$\overline{\mathbb{D}}$). Of course $\overline{ \alpha_{\ell, k}(t)} =
\alpha_{k,\ell}(t)$.\index{notation}{211@$\mathbf{e}(x) =e^{2i\pi x}$ (exponential functions)}

A $4$-gon~$q$
defines~$12$ elements $\alpha_{k,\ell}(q)$ in~$\Gamma(S)$ for
$k\neq \ell $ in~$\Z/4\Z$.

Similarly,
for every $p$-gon~$f$, $\alpha_{k,\ell}(f)$ ($k\neq \ell$ in $\Z/p\Z$) will
 denote the class of the restriction of~$f$ to the segment
 $[ \mathbf{e}(k/p), \mathbf{e}(\ell/p)]$ in  $\mathbb{D}$.

\begin{df}[Noncommutative surface]
  \label{df:algebra-berenst-reta}
  The \emph{noncommutative surface} associated with~$S$ is the
  algebra~$\mathcal{A}_S$ over~$\Q$ generated by the elements $x_\alpha$,
  $x_{\alpha}^{-1}$ ($\alpha \in \Gamma(S)$) subject to the
  relations:\index{notation}{213@$\mathcal{A}_S$ (Berenstein ans Retakh's algebra)}
  \begin{enumerate}
  \item[(T)] \label{itemT:df:algebra-berenst-reta} for every
    $3$-gon~$t\colon ( \mathbb{D}, \mu_3) \to
    ({S}, \partial S)$, abbreviating
    $x_{k\ell}^{\pm 1} = x_{ \alpha_{k,\ell}(t)}^{\pm 1}$,
    \[x_{12} x_{32}^{-1} x_{31} = x_{13} x_{23}^{-1} x_{21};\]
  \item[(Q)] \label{itemQ:df:algebra-berenst-reta} for every immersed $4$-gon
    $q\colon ( \mathbb{D}, \mu_4) \to ({S}, \partial S)$, abbreviating
$x_{k\ell}^{\pm 1} = x_{ \alpha_{k,\ell}(q)}^{\pm 1}$,
\[x_{13} = x_{12}  x_{42}^{-1} x_{43} +  x_{14}
    x_{24}^{-1} x_{23}.\]
  \end{enumerate}
\end{df}\index{notation}{214@$\mathcal{A}_S$}%
\index{definition}{algebra (Berenstein and Retakh's ---)}%
\index{definition}{Berenstein and Retakh's algebra}%
\index{definition}{noncommutative surface}%
\index{definition}{surface (noncommutative ---)}

\section{Geometric realization of the noncommutative surface}
\label{sec:geom-real-nonc}

For this subsection, we select an irreducible component of the space of decorated $\delta$-twisted local systems that contains a maximal local system.
 Let $\mathcal{K}$ be the field of rational functions on this irreducible component. The $\Lambda$-lengths over this space
enable us to give a \enquote{geometric} realization of the
algebra~$\mathcal{A}_S$:

\begin{teo}   \label{thm:realization BR}
  \label{thn:geom-real-AS}
  There is a \ep{unique} algebra homomorphism
  \[ \Psi\colon \mathcal{A}_S \longrightarrow M_n (\mathcal{K})\]
  such that, for every arc~$\alpha$ in~$S$,
 $\Psi( x_\alpha) = \Lambda_{r(\alpha)}$.
\end{teo}

\begin{proof}
  We first remark that every element $\Lambda_{r(\alpha)}$ is invertible in $\mathcal{K}$. This is because the selected irreducible component contains a maximal local system, and every maximal local system is $\alpha$-transverse for every $\alpha$, see Corollary~\ref{cor:maxim-decor-local-transverse}. Hence the determinant of $\Lambda_{r(\alpha)}$ is non-zero in $\mathcal{K}$.

  Now, by the universal property of the algebra~$\mathcal{A}_S$, the
  assignment $x^{\pm 1}_{\alpha} \mapsto \Lambda_{r(\alpha)}^{\pm 1}$ extends
  to an algebra homomorphism if and only if the defining relations between
  the elements $\{x_\alpha\}_{\alpha\in \Gamma(S)}$ (Definition~\ref{df:algebra-berenst-reta}) are satisfied by
  the family $\{ \Lambda_{r(\alpha)}\}_{\alpha\in \Gamma(S)}$. This follows
  from point~(\ref{item:2:sec:a-space}) in Section~\ref{sec:a-space} and
  Proposition~\ref{prop:change-coord-flips}.
\end{proof}

\section{Zigzag path and expression of the $\Lambda$-lengths}
\label{sec:zig-zag-path}

For any arc~$\alpha$ in~$\Gamma(S)$, there is a unique \enquote{minimal}
triple $(p, f, k)$ where~$p$ belongs to~$\N$,
$f\colon ( \mathbb{D}, \mu_p) \to ({S}, \partial S)$ is a $p$-gon, and~$k$ is
in $\Z/p\Z$ such that, for all~$i$ in~$\Z/p\Z$, $\alpha_{i,i+1}(f)$ belongs
to~$\mathcal{T}$ and $\alpha= \alpha_{1k}(f)$. This can be proved, for
example, using an hyperbolic structure on~$S$ with totally geodesic boundary
and the geodesic realizations of the arcs.

\begin{df}[Zigzag sequence]
  \label{df:admi-sequ}
  A \emph{zigzag} sequence for~$\alpha$ (or 
  an $\alpha$-zigzag sequence) is an odd length sequence
  $\boldsymbol{\delta} =( \delta_1, \delta_2, \dots, \delta_{2m},
  \delta_{2m+1})$ in~$\mathcal{T}^{2m+1}$ such that:
  \begin{itemize}
  \item there is $(j_0, j_1, \dots, j_{2m}, j_{2m+1})\in (\Z/p\Z)^{2m+2}$ with
    $j_0=1$, $j_{2m+2}=k$, and for all~$\ell = 1, \dots, 2m+1$, $\delta_\ell =
    \alpha_{ j_{\ell-1} j_\ell}(f)$;
  \item for all odd~$\ell$, the segments (in~$\mathbb{D}$) $[\mathbf{e}(
      j_{\ell-1}/p) , \mathbf{e}({ j_{\ell}/p})]$ and $[\mathbf{e}(1/p), \mathbf{e}( k/p)]$ do not intersect (besides endpoints for $\ell=1$ of
    $\ell=2m+1$);
  \item for all even~$\ell$, the segments $[\mathbf{e}(j_{\ell-1}/p) ,
    \mathbf{e}( j_{\ell}/p)]$ and $[\mathbf{e}(1/p), \mathbf{e}( k/p)]$
    intersect at some point $u_\ell$;
  \item the sequence $(u_2, u_4, \dots, u_{2m})$ in $[\mathbf{e}(1/p), \mathbf{e}( k/p)]$ is increasing.
  \end{itemize}
\end{df}\index{definition}{zigzag path}%
\index{definition}{zigzag sequence}%
Note that our terminology differs from~\cite{BR}.

For such a $\boldsymbol{\delta}$ let us denote $x_{\boldsymbol{\delta}} =
x_{\delta_1} x_{\bar{\delta}_2}^{-1} x_{\delta_3} \cdots
x_{\bar{\delta}_{2m}}^{-1} x_{\delta_{2m+1}}$. This elements belongs thus to
the subalgebra generated by the~$x_{\alpha}^{\pm 1}$ for~$\alpha$
in~$\mathcal{T}$. Similarly, we denote by $\Lambda_{\boldsymbol{\delta}}$ the product $
\Lambda_{\delta_1} \Lambda_{\bar{\delta}_2}^{-1} \Lambda_{\delta_3} \cdots
\Lambda_{\bar{\delta}_{2m}}^{-1} \Lambda_{\delta_{2m+1}}$ of $\GL( n, \mathcal{K})$.

One spectacular result from~\cite{BR} is the noncommutative Laurent phenomenon:

\begin{teo}[{\cite[Theorem~3.30]{BR}}]
  \label{teo:noncom-Laurent}
  For every~$\alpha$ in~$\Gamma(S)$, one has
  \[ x_\alpha = \sum_{\boldsymbol{\delta}} x_{\boldsymbol{\delta}}\]
  where the sum runs over the $\alpha$-zigzag sequences~$\boldsymbol{\delta}$.
\end{teo}

As an immediate corollary we get:

\begin{cor}
  For every~$\alpha$ in~$\Gamma(S)$, one has, in $\GL(n, \mathcal{K})$,
  \[ \Lambda_\alpha = \sum_{\boldsymbol{\delta}} \Lambda_{\boldsymbol{\delta}}\]
  where the sum runs over the $\alpha$-zigzag sequences~$\boldsymbol{\delta}$.
\end{cor}

\appendix

\chapter{Normal form for pair of quadratic forms}
\label{sec:normal-form-pair}

The well known spectral theorem says that for two symmetric bilinear forms $b_0$, $b_1$
on an $n$-dimensional real vector space $V$ such that $b_0$ is positive definite, there exists
a basis $\mathbf{e}$ such that $[b_0]_{\mathbf e}=\Id_n$,
$[b_1]_{\mathbf e}=\diag(\lambda_1,\dots,\lambda_n)$ where
$\lambda_1\geq \cdots\geq\lambda_n$. Therefore, the tuple
$(\lambda_1,\dots,\lambda_n)$ defines the pair $(b_0,b_1)$ up to change of
basis of~$V$. We can define the standard form of the pair of bilinear forms
$(b_0,b_1)$ to be the pair of matrices
$(\Id_n,\diag(\lambda_1,\dots,\lambda_n))$ and say that the basis~$\mathbf{e}$
puts $(b_0,b_1)$ to the standard form. We used this standard form to define
edge invariants for maximal representations in
Chapter~\ref{sec_def_coord_max}.

In this appendix, we define the standard form for a pair of bilinear forms
$(b_0,b_1)$ assuming only nondegeneracy of $b_0$. This standard form will be
used to define edge invariants for general representations in
Chapter~\ref{sec:4-uple-transverse}.

\section{Bilinear forms and selfadjoint linear maps}

Let~$V$ be a $n$-dimensional vector space over a field~$K$ of
characteristic~$0$, and let $b_0$, $b_1$ be symmetric bilinear forms on
$V$. We assume that $b_0$ is not degenerate. We denote by
$b_{i}^{\dagger}\colon V\to V^*$ the linear map corresponding to~$b_i$, i.e.\
$b_{i}^{\dagger}(x)(y) = b_i(x,y)$ for all $x,y\in V$. Then we can consider
$f=(b_{0}^{\dagger})^{-1}\circ b_{1}^{\dagger}\colon V\to
V$.\index{notation}{215@$b^{\dagger}$ (linear map associated with the bilinear form~$b$)}

\begin{lem}
  \label{lem:f-sym}
  For all $x,y\in V$, one has $b_1(x,y)=b_0(fx,y)$.

  The map $f$ is selfadjoint with respect to $b_0$.
\end{lem}
\begin{proof}
  The sought for equality can be written: $\forall x\in V$,
  $b_{1}^{\dagger}(x) = b_{0}^{\dagger}(fx)$ and follows immediately from the
  definition of~$f$. Since $b_0$ and $b_1$ are symmetric, this implies
  that~$f$ is selfadjoint.
\end{proof}

\section{More bilinear forms}
\label{sec:more-bilinear-forms}

It will be useful to consider the bilinear forms, for all $k\in\Z_{\geq 0}$,
\[ b_k( v,w) = b_0( f^k v, w) \quad (v,w \in V).\] For all $p,q\geq 0$, one
has $b_{p+q}(v,w) = b_0( f^p v, f^q w)$ and, for all~$k$, $f$ is selfadjoint
with respect to~$b_k$.

More generally, for every polynomial~$P$ in~$K[X]$, $b_P( v,w) =
b_0(P(f)v,w)$ is a symmetric form and $f$ is selfadjoint with respect
to~$b_P$.

\begin{lem}
  \label{lem:ker-bP}
  The kernel of the symmetric form~$b_P$ is equal to the kernel of the
  endomorphism~$P(f)$.
\end{lem}

\begin{proof}
  Indeed, the kernel of~$b_P$ is
  \begin{align*}
    \ker(b_P)
    &= \{ v\in V \mid \forall w\in V, \ b_P(v,w)=0\}\\
    &= \{ v\in V \mid \forall w\in V, \ b_0(P(f)v,w)=0\}\\
\intertext{and since $b_0$ is nondegenerate}
    &= \{ v\in V \mid P(f)v=0\} =\ker P(f). \qedhere
  \end{align*}
\end{proof}

Let us equip the vector space~$V$ with its $K[X]$-module structure inherited
from the endomorphism~$f$, namely, for every $P\in K[X]$, $P\cdot v= P(f)v$.

\begin{rem}
  \label{rem:ringKn}
  \begin{enumerate}
  \item When $f$ is nilpotent, $V$ will be a module over the local ring
    $K_{[n]}=K[X]/(X^n)$.

  \item     The homomorphism $\epsilon\colon  K_{[n]}\to K[X]/(X)\simeq K$ is often called the
    augmentation. An element $z$ in $K_{[n]}$ is invertible if and only if
    $\epsilon(z)\in K^*$ (one can check that, if $\epsilon(z)\in K^*$, the finite
    sum $\epsilon(z)^{-1}\sum_{\ell=0}^{n-1} (1-\epsilon(z)^{-1}z)^\ell$ is the
    inverse of~$z$ in~$K_{[n]}$). Furthermore $z$ is a square if and only if
    $\epsilon(z)$ is a square: one first gets back to the case $\epsilon(z)=1$
    and then checks that, if $\sum_\ell a_\ell t^\ell$ is the Taylor series of
    $t\mapsto (1+t)^{1/2}$ (the exact formula being
    $a_\ell = \frac{\prod_{j=0}^{\ell-1}(1/2-j)}{\ell!}$), then
    $\sum_{\ell=0}^{n} a_\ell (z-1)^\ell$ is a square root of $z$. It follows
    that $K_{[n]}^* /K_{[n]}^{\ast 2}\simeq K^*/K^{\ast 2}$.
  \end{enumerate}
\end{rem}

An encompassed way to put together all those bilinear forms is to consider
forms with values in the following $K[X]$-module\index{notation}{217@$K\langle\langle
  X^{-1}\rangle\rangle$ (\enquote{truncated} Laurent series)}
\[ K\langle\!\langle X^{-1}\rangle\!\rangle= K(\!(X^{-1})\!)/ XK[X],\] the quotient of the field of Laurent series in $X^{-1}$
modulo the polynomials in~$X$ without constant term. Every element in~$K\langle\!\langle X^{-1}\rangle\!\rangle$ is
represented uniquely as a series $\sum_{\ell\geq 0} a_\ell X^{-\ell}$. The
relevance of~$K\langle\!\langle X^{-1}\rangle\!\rangle$ is justified by the following lemma.
\begin{lem}
  \label{lem:all-form-in-one}
  With the above notation, the map
  \begin{align*}
   b_{K\langle\!\langle X^{-1}\rangle\!\rangle}\colon  V \times V & \longrightarrow K\langle\!\langle X^{-1}\rangle\!\rangle \\
    (v,w) & \longmapsto \sum_{k\geq 0} b_k(v,w) X^{-k}
  \end{align*}
  is $K[X]$-bilinear and symmetric.
\end{lem}

\begin{proof}
  The symmetry and $K$-linearity follow at once from the same properties of
  the forms $b_k$. The $K[X]$-linearity is a consequence of the fact that, for all
  $k\geq 0$, and all $v,w\in V$, $b_{k+1}(v,w)= b_k(fv,w)$.
\end{proof}

\begin{rem}
  \label{rem:more-bilinear-forms}
  Let $Q$ be a polynomial such that $Q(f)=0$. Then the above map $b_{K\langle\!\langle X^{-1}\rangle\!\rangle}$ takes
  values in the submodule $\mathrm{Tor}_Q =\{ S\in K\langle\!\langle X^{-1}\rangle\!\rangle \mid QS=0\}$ (i.e.\ if
  $\tilde{S}\in K(\!(X^{-1})\!)$ represents~$S$, then $Q\tilde{S}$ is a
  polynomial without constant term). Furthermore the $K[X]$-module
  $\mathrm{Tor}_Q$ is isomorphic to the cyclic module $K[X]/(Q)$. We will
  exploit this additional fact mainly when~$f$ is nilpotent (i.e.\ $Q=X^n$), in which case we
  will shift the indices for $b_{K\langle\!\langle X^{-1}\rangle\!\rangle}$ in the positive range (see the proof of
  Lemma~\ref{lem:normal-form-one-block}).
\end{rem}

\section{When $f$ has one split Jordan block}
\label{sec:when-module-v}

In this section, we start our investigation of normal forms under the
additional assumption that the Jordan normal form of~$f$ has only one
block. That is to say, there is a basis~$\mathbf{e}$ of $V$ and $\lambda\in
K$, such that the matrix $[f]_\mathbf{e}$ of~$f$ in that basis is
$J_n(\lambda) = \lambda \Id +N$, where $N=( \delta_{i+1,j})_{1\leq i,j\leq n}$
the regular nilpotent matrix and $\delta_{\cdot,\cdot}$ is the Kronecker symbol.

\begin{rem}
  \label{rem:basis-jordan}
  \begin{enumerate}
  \item Such a basis is entirely determined by its last vector $v=e_n$ and the
    rule $e_i= f(e_{i+1})-\lambda e_{i+1}$ ($i=n-1, n-2, \dots, 1$). The
    vector~$v$ must be chosen in $V \smallsetminus \ker (f-\lambda)^{n-1}$ in order for
    the previous construction to lead to a basis.

  \item When~$\lambda=0$, the matrix of the bilinear form $b_0$ in that basis
    is then \[[b_0]_\mathbf{e} = ( b_{2n-i-j}(v,v))_{i,j=1,\dots,n}.\]
  \end{enumerate}
\end{rem}

Let also $C_n$ be the \enquote{antidiagonal} matrix
$( \delta_{n+1,i+j})_{1\leq i,j\leq n}$ (see
Equation~\eqref{eq:C_nJ_nsec:matrices} in Section~\ref{sec:matrices}).

To keep the discussion for a general field~$K$, let us fix
$\mathfrak{c}\subset K^*$ a set of representatives of $K^*/ K^{\ast 2}$; we are mainly interested in
the case when $K=\R$ ($\mathfrak{c}=\{\pm 1\}$) and $K=\C$ ($\mathfrak{c}=\{1\}$).

\begin{lem}
  \label{lem:normal-form-one-block}
  Under the assumption that the Jordan normal form of~$f$ has only one block,
  there is a basis~$\mathbf{e}$ of $V$ and $\varepsilon\in \mathfrak{c}$, such
  has $[f]_\mathbf{e}$ is the Jordan block $J_n(\lambda)$ and
  $[b_0]_\mathbf{e}$ is $\varepsilon C_n$.

  The basis~$\mathbf{e}$ is unique up to multiplication by~$\pm 1$.
\end{lem}

\begin{rem}
  \label{rem:b1-oneblock}
  The matrix of the symmetric bilinear form~$b_1$ is then
  $[b_1]_\mathbf{e} = \varepsilon C_n J_n(\lambda) = \varepsilon {}^T\!
  J_n(\lambda) C_n $.
\end{rem}

\begin{proof}
  Up to changing $f$ into $f-\lambda \Id$ (which amounts to changing $b_1$
  into $b_1-\lambda b_0$), we can assume that $\lambda=0$. In this situation, $b_\ell=0$ for every $\ell\geq n$. We will consider
  $V$ as a module over the algebra $K_{[n]}=K[X]/(X^n)$, and the $K_{[n]}$-bilinear
  symmetric form (compare with
  Section~\ref{sec:more-bilinear-forms}):\index{notation}{221@$K_{[n]}=K[X]/(X^n)$ }
  \begin{align*}
    b_{K_{[n]}}\colon  V\times V
    &\longrightarrow K_{[n]}\\
    (v,w) & \longmapsto \sum_{i=0}^{n-1} b_{n-1-i}(v,w) X^i.
  \end{align*}

  Following Remark~\ref{rem:basis-jordan}, we are searching for a vector $v$
  such that $b_{n-1}(v,v)\neq 0$ and $b_j(v,v)=0$ for all $j\neq n-1$, i.e.\
  such that the element $b_{K_{[n]}}(v,v)$ belongs to~$K^*\subset K_{[n]}$.

  The kernel of the form $b_{n-1}$ is equal to the kernel of~$f^{n-1}$
  (Lemma~\ref{lem:ker-bP}) and therefore is not zero by the assumption that
  the Jordan decomposition of~$f$ has only one block. Let thus~$w$ be in~$V$
  with $b_{n-1}(w,w)\neq 0$. The element $b_{K_{[n]}}(w,w)$ is then in $K_{[n]}^{*}$
  since $\epsilon( b_{K_{[n]}}(w,w) ) =b_{n-1}(w,w)\in K^\ast$, and there is an
  element $u\in K_{[n]}$ such that $\varepsilon\coloneqq  u^2 b_{K_{[n]}}(w,w)$
  belongs to~$\mathfrak{c}$ (see
  Remark~\ref{rem:ringKn}). Define $v= u\cdot w\in V$. Then
  $b_{K_{[n]}}(v,v) =\varepsilon$ which is the sought for equality.

  Let now $v'$ be in~$V$ such that
  $b_{K_{[n]}}(v',v')\in \mathfrak{c} \subset K^*$. Since the $K_{[n]}$-module~$V$
  is cyclic generated by $v$, there exists $t\in K_{[n]}$ such that
  $v'=t\cdot v$. Therefore $t^2 \varepsilon = b_{K_{[n]}}(v',v')$.  This equality implies
  that $t$ belongs to~$K$ (assume that it is not the case, then
  $t=\nu + \alpha X^k $ and $t^2 = \nu^2 + 2\nu \alpha X^k \mod X^{k+1}$ cannot belong
  to~$K$). Since $\mathfrak{c}$ is a set of representatives of
  $K^*/ K^{\ast 2}$, this implies that $t^2=1$ and $t=\pm 1$. The uniqueness
  (up to sign) of the basis follows.
\end{proof}

\begin{rem}
  \label{rem:bilin-Kn}
  Pursuing a little more the proof, one observes that $V$ can be identified
  with~$K_{[n]}$ and the bilinear for $b_{K_{[n]}}\colon  K_{[n]}\times K_{[n]} \to K_{[n]}$ is
  simply $\varepsilon$ times the multiplication in the algebra~$K_{[n]}$. Shifting back to the point
  of view of Section~\ref{sec:more-bilinear-forms}, the bilinear form $b_{K\langle\!\langle X^{-1}\rangle\!\rangle}$ is
  in that case:
  \begin{align*}
    b_{K\langle\!\langle X^{-1}\rangle\!\rangle}\colon  K_{[n]} \times K_{[n]} &\longrightarrow K\langle\!\langle X^{-1}\rangle\!\rangle \\
    (u,v) & \longmapsto \varepsilon X^{1-n} u v.
  \end{align*}
\end{rem}

\section{(Over the reals or the complex) Back transformation
  when $f$ has one split Jordan block}
\label{sec:back-transf-when}
\label{back_trafo}

In this section we make the additional hypothesis that $b_1$ is nondegenerate
and that $f$ has one Jordan block and $K=\R$, or $\C$.

The dual vector space~$V^*$ has two bilinear forms~$b_{0}^{\ast}$
and~$b_{1}^{\ast}$ and, by definition of these forms, $b_{0}^{\ast\dagger} =
(b_{0}^{\dagger})^{-1}$ and $b_{1}^{\ast\dagger} = (b_{1}^{\dagger})^{-1}$. In particular
$(b_{1}^{\ast\dagger})^{-1} b_{0}^{\ast\dagger} = b_{1}^{\dagger}  (b_{0}^{\dagger})^{-1}$
is conjugate to $f= (b_{0}^{\dagger})^{-1} b_{1}^{\dagger}$.
Applying
Lemma~\ref{lem:normal-form-one-block} gives $\eta\in\{\pm 1\}$, a
basis of~$V^*$, and hence a \enquote{predual}
basis~$\mathbf{v}$ of~$V$, such that
\[ [b_{1}^{*}]_{\mathbf{v}^*} = \eta C_n, \quad \text{and}\quad
  [b_{0}^{*}]_{\mathbf{v}^*} = \eta C_n J_n(\lambda).\]
Thus
\[ [b_{1}]_{\mathbf{v}} = {}^T\! (\eta C_n)^{-1} = \eta C_n, \quad \text{and}\quad
  [b_{0}]_{\mathbf{v}} = \eta  C_n \, {}^T\!\! J_n(\lambda)^{-1} = \eta J_n(\lambda)^{-1} C_n.\]
Let $\Phi$ be the change-of-basis matrix from~$\mathbf{e}$ to~$\mathbf{v}$,
i.e.\ $\mathbf{e}= \mathbf{v}\Phi$. The matrix~$\Phi$ must satisfy the
following two equalities:
\[ {}^T\! \Phi \eta C_n \Phi = \varepsilon C_n J_n(\lambda), \quad
  \text{and}\quad  {}^T\! \Phi \eta J_n(\lambda)^{-1} C_n \Phi =
  \varepsilon C_n .\]
Denote again~$N$ the  nilpotent matrix of
maximal rank $( \delta_{i+1,j})_{1\leq i,j\leq n}$.
\begin{lem}
  \label{lem:back-trans-one-block}
  With the above notation,
  \begin{enumerate}
  \item If $K=\R$, $\varepsilon \eta \lambda$ is positive \ep{and in particular
    equal to $|\lambda|$}.
  \item The matrix~$\Phi$ is equal to, up to sign,
    \begin{equation}
      \label{eq:back-trans-one-block}
      \sqrt{ \varepsilon \eta \lambda} \sum_{\ell=0}^{n} a_\ell \lambda^{-\ell} C_n N^{\ell},
    \end{equation}
    where $\sum_{\ell=0}^{\infty} a_\ell t^\ell$ is the Taylor series of
    $t\mapsto (1+t)^{1/2}$ \ep{i.e.{} for all~$\ell$, $a_\ell = \frac{
      \prod_{j=0}^{\ell-1}(1/2-j)}{\ell!}$} \ep{cf.\ also Equation~\eqref{eq:Phi_nsec:matrices}}.
  \end{enumerate}
\end{lem}

\begin{proof}
  Let $L= \varepsilon \eta J_n(\lambda) = \varepsilon
  \eta \lambda \Id + \varepsilon \eta N$.
  The two equations for~$\Phi$ imply
  \[ {}^T\! \Phi^{-1} = C_n \Phi L^{-1} C_n, \quad \text{and} \quad {}^T\!
    \Phi^{-1} = L^{-1} C_n \Phi C_n,\]
  so that $C_n\Phi$ and $L^{-1}$ commute and hence $C_n\Phi$ and $N$
  commute. This implies that $C_n \Phi$ is a polynomial in $N$: there are
  $c_0, \dots, c_{n-1} $ in~$K$ such that $C_n\Phi = c_0\Id +\cdots
  +c_{n-1} N^{n-1}$. Thus $\Phi = c_0 C_n + c_1 C_n N + \cdots
  +c_{n-1} C_n N^{n-1}$ is a symmetric matrix (for all~$\ell$, $C_n N^\ell$ is
  symmetric). The first equation for~$\Phi$
  can now be written
  \[ ( C_n \Phi)^2 = \varepsilon \eta J_n(\lambda)\]
  and as $( C_n \Phi)^2 = c_{0}^2 \Id + d_1 N+ \cdots+ d_{n-1}N^{n-1}$ (for
  some $d_1, \dots, d_{n-1}$ in~$K$), we get that $\varepsilon
  \eta \lambda =c_{0}^{2}$ and is positive if $K=\R$. Dividing the last equation
  by~$\varepsilon \eta \lambda$ give
  \[ ( (\varepsilon \eta \lambda)^{-1/2} C_n \Phi)^2 = \Id + \frac{1}{\lambda} N,\]
  and this leads to the desired result.
\end{proof}

\begin{df}
  \label{df:back-trans-1}
  The matrix given in Equation~\eqref{eq:back-trans-one-block} is called the
  \emph{back transformation} and denoted by $\Phi_n(\lambda)$. When $K=\C$, so
  that $\varepsilon=\eta=1$, we always choose for $\sqrt{\lambda}$
  the biggest square root of~$\lambda$ for the lexicographic order on
  $\C\simeq \R^2$.
\end{df}

\section{When the minimal polynomial of~$f$ has one root}
\label{sec:when-minim-polyn}

In this section we assume that all the blocks of the Jordan decomposition
of~$f$ are associated with the same eigenvalue~$\lambda$ in~$K$ or, what amounts to
the same, that $(f-\lambda)^n=0$ ($n=\dim V$). We denote again by~$\mathfrak{c}$ a set of representatives of $K^*/K^{\ast 2}$.

\begin{prop}
  \label{prop:normal-form-one-eigenval}
  There is a basis $\mathbf{e}$ of~$V$, a sequence $(n_1, \dots, n_p)$ of
  integers, and a sequence $(\varepsilon_1, \dots, \varepsilon_p)$ in~$\mathfrak{c}$ such that $[f]_\mathbf{e}$ and
  $[b_0]_\mathbf{e}$ are block diagonal with diagonal blocks being
  $J_{n_1}(\lambda), \dots , J_{n_p}(\lambda)$, and $\varepsilon_1 C_{n_1}, \dots,
  \varepsilon_p C_{n_p}$ respectively.
\end{prop}

\begin{proof}
  We work by induction on~$n$. Up to replacing $f$ by $f-\lambda$, we may
  assume that $\lambda=0$. Let $m\geq 1$ be the order of nilpotency of $f$,
  hence $f^m=0$, $b_\ell=0$ for all $\ell\geq m$, and $\ker f^{m-1}\neq V$. There is then a vector $w\in V$ such
  that $c=b_{m-1}(w,w)\neq 0$ (Lemma~\ref{lem:ker-bP}). The $K[X]$-module $W$
  generated by $w$ has for basis $w$, $ fw, \dots, f^{m-1}w$ and the matrix
  of $b_0|_W$ in that basis is
  \[
    \begin{pmatrix}
      0 & \dots & 0 & c \\
      0 & \dots & c & \ast \\
      \vdots & \iddots & \ast & \ast \\
      c & \ast &\dots & \ast
    \end{pmatrix}
  \]
  (see Remark~\ref{rem:basis-jordan}) so that $b_0|_W$ is nondegenerate. Applying
  Lemma~\ref{lem:normal-form-one-block} to~$W$ and the induction hypothesis to
  $W^{\perp_{b_0}}$ gives the result.
\end{proof}

\section{Orthogonality of generalized eigenspaces}
\label{sec:orth-gener-eigensp}

We return to the general setting: $b_0$ and~$b_1$ are symmetric bilinear forms
on a $K$-vector space~$V$ of dimension~$n$, $K$ is a field of
characteristic~$0$, $b_0$ is nondegenerate, and $f= (b_{0}^{\dagger})^{-1}
b_{1}^{\dagger}$.

The generalized eigenspaces of~$f$ are the $V_\lambda = \ker(f-\lambda)^n$ for
$\lambda\in K$.

\begin{lem}
  \label{lem:orth-gener-eigensp}
  Generalized eigenspaces are orthogonal:
  if $\lambda\neq \mu$, then $V_\mu \subset V_{\lambda}^{\perp_{b_0}}$.
\end{lem}

\begin{proof}
  The proof relies on the $K[X]$-module structure of~$V$ and on the bilinear form
  $b_{K\langle\!\langle X^{-1}\rangle\!\rangle}$
  (Section~\ref{sec:more-bilinear-forms}).  Let $P$ and~$Q$ be in $K[X]$
  establishing a Bezout relation between $(X-\lambda)^n$ and $(X-\mu)^n$:
  $(X-\lambda)^n P + (X-\mu)^n Q = 1$.

  Let $v$ be in~$V_\lambda$ and~$w$ be in~$V_\mu$ so that $(X-\lambda)^n v=
  (X-\mu)^n w =0$. One has
  \begin{align*}
    b_{K\langle\!\langle X^{-1}\rangle\!\rangle} (v,w)
    &=(X-\lambda)^n P  b_{K\langle\!\langle X^{-1}\rangle\!\rangle}(v,w) + (X-\mu)^n Q
      b_{K\langle\!\langle X^{-1}\rangle\!\rangle} (v,w) \\
    &=P  b_{K\langle\!\langle X^{-1}\rangle\!\rangle} ((X-\lambda)^n v,w) + Q
      b_{K\langle\!\langle X^{-1}\rangle\!\rangle} (v,(X-\mu)^nw) \\
    &=0,
  \end{align*}
  and, for the least, $b_0(v,w)=0$. This is what had to be proved.
\end{proof}

\begin{rem}
  \label{rem:orth-gener-eigensp}
  More generally for every irreducible polynomial $P\in
  K[X]$, one can consider the subspace $V_P = \ker P(f)^n$. The subspaces $V_P$ are in direct orthogonal
  sum.
\end{rem}

As a corollary, when the minimal polynomial of~$f$ is split over~$K$, there is
a basis~$\mathbf{e}$ of~$V$ where the matrices of $[b_0]_\mathbf{e}$ and
$[f]_\mathbf{e}$ are block diagonal with blocks as in
Lemma~\ref{lem:normal-form-one-block}.

Elements commuting with~$f$ also stabilize the generalized eigenspaces. A more
precise statement is the following:

\begin{prop}
  \label{prop:auto-b0-b1}
  Let $g$ be in $\mathrm{End}_K(V)$. The following are equivalent:
  \begin{enumerate}
  \item $g$ belongs to the intersection of the orthogonal groups $\OO(b_0)$
    and $\OO(b_1)$;
  \item $g$ belongs to $\OO(b_0)$ and commutes with $f$;
  \item $g$ is $K[X]$-linear and is orthogonal with respect to the form $
      b_{K\langle\!\langle X^{-1}\rangle\!\rangle}$ \ep{i.e.\ for all $v,w$
    in~$V$, $b_{K\langle\!\langle X^{-1}\rangle\!\rangle}(gv,gw)=
    b_{K\langle\!\langle X^{-1}\rangle\!\rangle}(v,w)$};
  \item under the hypothesis that the minimal polynomial of~$f$ is split
    over~$K$, $g$ stabilizes every generalized eigenspace $V_\lambda$ of~$f$
    and the restriction $g|_{V_\lambda}$ is $b_0$-orthogonal.
  \item \ep{with no assumption on~$f$}, for every irreducible polynomial $P\in
    K[X]$, $g$ stabilizes $V_P = \ker P(f)^n$ and its restriction to $V_P$ is
    $b_0$-orthogonal.
  \end{enumerate}
\end{prop}

\section{(Over the reals) When $f$ has two conjugate eigenvalues}
\label{sec:over-reals-when}

To have a complete understanding of the pair $(b_0, f)$ over $\R$,
we need to investigate the case when $V=V_P=\ker P(f)^n$, where $P=
(X-\lambda)(X-\bar{\lambda})$ for $\lambda = a+ib \notin \R$.

The complexification of~$V$ will be denoted~$V_\C = V+iV$. It is a $\C$-vector
space equipped with an antilinear involution (the complex conjugation) $v\mapsto \bar{v}$ whose fixed
points set is precisely $V\subset V_\C$. Any $\R$-linear endomorphism $g\colon V\to V$
admits a complexification $g_\C\colon V_\C \to V_\C$. The $\R$-bilinear form
$b_0\colon V\times V\to \R$ has also a $\C$-bilinear complexification $b_{0,\C}\colon V_\C \times V_\C
\to \C$. In particular, $f_\C$ has two eigenvalues: $\lambda$
and~$\bar{\lambda}$ and the space $V_\C$ is the direct $b_{0,\C}$-orthogonal
sum of $V_\lambda = \ker (f_\C -\lambda)^n$ and of $V_{\bar\lambda} = \ker
(f_\C -\bar\lambda)^n$. The spaces $V_\lambda$ and $V_{\bar{\lambda}}$ are
exchanged by the complex conjugation $v\mapsto \bar{v}$, they intersect~$V$ trivially.

\begin{lem}
  \label{lem:c-to-r}
  The map $g\mapsto g_\C | V_\lambda$ induces an isomorphism between
  \begin{enumerate}
  \item the group of elements $g$ in~$\OO(b_0)$ that commute with $f$, and
  \item the group of elements $\alpha$ in~$\OO(V_\lambda,b_{0,\C})$ that
    commute with $f_\C$.
  \end{enumerate}
\end{lem}

\begin{proof}
  If $g$ commutes with $f$, $g_\C$ commutes with $f_\C$ and stabilizes
  $V_\lambda$. If furthermore $g$ is $b_0$-orthogonal, then $g_\C$ is
  $b_{0,\C}$-orthogonal. Therefore the mentioned map from the centralizer
  $Z_{\OO(b_0)}(f)$ to the centralizer $Z_{\OO(V_\lambda, b_{0,\C})}(f_\C)$ is
  well defined and easily seen to be a homomorphism.

  Let us construct its inverse. For
  $\alpha \in Z_{\OO(V_\lambda, b_{0,\C})}(f_\C)$, let
  $\beta \colon V_{\bar{\lambda}} \to V_{\bar{\lambda}} \mid v \mapsto
  \overline{ \alpha(\bar{v})}$. Then~$\beta$ is $\C$-linear,
  $b_{0,\C}$-orthogonal, and commutes with~$f_\C$. The pair $(\alpha,\beta)$
  combines therefore into a $\C$-linear and $b_{0,\C}$-orthogonal map
  $\gamma\colon V_\C=V_\lambda \oplus V_{ \bar{\lambda}} \to V_\C $ that
  commutes with~$f_\C$. The map~$\gamma$ commutes with the complex conjugation
  and hence comes from a unique $\R$-linear map $g\colon V\to V$ ($g$ is
  simply the restriction of~$\gamma$ to~$V$). The map~$g$
  is furthermore $b_0$-orthogonal and commutes with~$f$ since $\gamma$ is
  $b_{0,\C}$-orthogonal and commutes with~$f_\C$.
\end{proof}

Applying this lemma to $g=f$ and the results of
Section~\ref{sec:when-module-v}, we can give a normal form for the pair
$(b_0,f)$ when $V$ is cyclic, i.e.\ when $f_\C|_{V_\lambda}$ has one Jordan
block.
For this, let us denote, for every matrix~$M$ in~$M_m(\C)$, $r(M)$
(respectively $s(M)$)\index{notation}{224@$r(M)$, $s(M)$ (real matrices associated with
  a complex matrix~$M$)} the matrix in~$M_{2m}(\R)$ where each coefficient
$x+iy$ of~$M$ is replaced by the bloc $\bigl(
\begin{smallmatrix}
  x & -y \\ y & x
\end{smallmatrix}
\bigr)$ (respectively by $\bigl(
\begin{smallmatrix}
  x & -y \\-y & -x
\end{smallmatrix}
\bigr)$). The following identities are easily verified: $r(MN) = r(M) r(N)$,
$s(M) = B r(M) = r( \overline{M}) B$ where~$B$ is the diagonal matrix
\[
  \begin{pmatrix}
    1 & 0 & \dots & 0 & 0 \\
    0 & -1 & \dots & 0 & 0 \\
    \vdots & \ddots & \ddots & \vdots & \vdots \\
    0 & 0 & \dots & 1 & 0 \\
    0 & 0 & \dots & 0 & -1
  \end{pmatrix},
\]
and also: $s( {}^T\! M) =  {}^T\! s(M)$, $s( M^{-1}) = s( \overline{M})^{-1}$.
\begin{lem}
  Let $(W, b)$ be a complex vector space equipped with a bilinear symmetric
  form. Denote by~$W_\R$ the underlying real vector space and $b_\R \coloneqq
  \Re( b)$ the real part of~$b$. If $S$ is the matrix of~$b$ in a basis
  $\mathbf{e}=( e_1, \dots, e_m)$ of~$W$, then $s(M)$ is the matrix of $b_\R$
  in the basis $\mathbf{f} = ( e_1, i e_1, \dots, e_m, i e_m) $ of~$W_\R$.

      Let $f$ be an endomorphism of~$W$ and let~$F$ be the matrix of~$f$ in the
  basis~$\mathbf{e}$. Then the matrix of the endomorphism~$f$ of~$W_\R$ in the
  basis~$\mathbf{f}$ is~$r(F)$.
\end{lem}

We denote $C^{\prime}_{2m} = s(C_m)$ and $J^{\prime}_{2m}( \lambda) = r(
J_m(\lambda))$ (see Equation~\eqref{eq:CprimeJprime_sec:matrices} in
Section~\ref{sec:matrices}). Then, using the basis of the complex vector
space~$V_\lambda$ furnished by Lemma~\ref{lem:normal-form-one-block} for
$b_{0, \C}$ and $f_{ \C}$, we get

\begin{lem}
  With the above hypothesis, there is a basis~$\mathbf{e}$ of the real vector
  space~$V$, such that~$[b_0]_\mathbf{e} = C^{\prime}_{2m}$ and
  $[f]_{\mathbf{e}} = J^{\prime}_{2m}(\lambda)$. The basis~$\mathbf{e}$ is
  unique up to multiplication by~$\pm 1$.
\end{lem}

For any two complex $m\times m$ matrices~$M$ and~$N$, the following equalities
are direct consequences of the identities recalled above:
\begin{align*}
  s( {}^T \! NMN) & = {}^T \! r(N) s(M) r(N) = {}^T \! s(N) s(\overline{M})
                    s(N)\\
  s( {}^T \! N {}^T\! M^{-1} N) &  = {}^T \! s(N)\, {}^T\!
                                  s({M})^{-1}  s(N).
\end{align*}
Let us denote $\Psi_{2m}( \lambda) \coloneqq s( \Phi_m(\lambda))$ (cf.\
Equation~\eqref{eq:Psi2m_sec:matrices}). Then the matrix $\Psi_{2m}( \lambda)$
is symmetric and satisfies \[{}^T \! \Psi_{2m}( \lambda)\, {}^T\! C^{\prime
  -1}_{2m} \Psi_{2m}( \lambda) = C^{\prime}_{2m} J^{\prime}_{2m}(\lambda)\]
and \[{}^T \! \Psi_{2m}( \lambda) \, {}^T\! (C^{\prime }_{2m}
J^{\prime}_{2m}(\lambda))^{-1} \Psi_{2m}( \lambda) = C^{\prime}_{2m}.\] Up to
sign, this is the unique matrix satisfying these two equations.

\begin{df}
  \label{df:back-trans-2}
  We call also the matrix $\Psi_{2m}(\lambda)$ the \emph{back transformation}.
\end{df}

\section{Normal forms}
\label{sec:normal-forms}

We now collect the results from the previous sections to establish normal
forms for pairs of quadratic forms over~$\R$.

For a finite sequence $\underline{n}=(n_1, \dots, n_k)$ of positive integers,
we denote by $C({\underline{n}})$ the square matrix of size $n_1+\cdots +n_k$,
diagonal by blocks, and whose blocks are $C_{n_1}, \dots{}, C_{n_k}$. If
$\underline{\lambda}=(\lambda_1, \dots, \lambda_k)$ is a further sequence of
real numbers, we will denote by $J({\underline{n}}, \underline{\lambda})$ the
block diagonal matrix whose blocks are $J_{n_1}(\lambda_1), \dots{},
J_{n_k}(\lambda_k)$. Under the hypothesis that no $\lambda_i$ is~$0$, similar notation
$\Phi({\underline{n}},\underline{\lambda})$ is adopted for the matrix whose
blocks are the back transformations $\Phi_{n_i}(\lambda_i)$ (see
Definition~\ref{df:back-trans-1}). When $2\underline{m}=(2m_1, \dots, 2m_k)$
is a sequence of even integers and $\underline{\lambda}=(\lambda_1, \dots,
\lambda_k)\in \C^k$, we introduce also, via the same procedure, the matrices
$C^{\prime}({2\underline{m}})$,
$J^{\prime}({2\underline{m}}, \underline{\lambda})$, and if
$\underline{\lambda} \in (\C^*)^k$
$\Psi({2\underline{m}}, \underline{\lambda})$.

\begin{teo}
  \label{teo:normal-forms-r}
  Let~$V$ be a real vector space of dimension~$n$. Let~$b_0$ and~$b_1$ be two
  symmetric bilinear forms on~$V$ with~$b_0$ nondegenerate. Let
  $b_{i}^{\dagger}$ \ep{$i=0,1$} be the induced morphisms $V\to V^*$ and
  let~$f=(b_{0}^{\dagger})^{-1} \circ b_{1}^{\dagger}$.
  \begin{enumerate}
  \item\label{item1:teo:normal-forms-r} There are uniquely determined sequences
    $(\underline{n}_{x}, \underline{\lambda}_{x}) \in \Z_{>0}^{k_{x}} \times
    \R^{k_{x}}$ \ep{$x\in \{\pm 1\}^2$}, and
    $(2\underline{m}, \underline{\lambda}) \in (2\Z_{>0})^{k_0} \times \mathbb{H}^{k_0}$
    satisfying the following normalization:
    \begin{itemize}
    \item for all
      $x=(\varepsilon, \eta) \in \{\pm 1\}^2$, and
      all $1\leq \ell \leq k_x$, $\varepsilon \eta \lambda_{x,\ell} \geq 0$;
    \item the sequences $\underline{\lambda}_{x}$
      \ep{$x\in \{\pm 1\}^2 $}, and $\underline{\lambda}$ are
      decreasing \ep{for the last one, the lexicographic order is understood on
      $\C\simeq \R^2$};
    \item for all $x\in \{\pm 1\}^2$ and all
      $1\leq k<\ell\leq k_x$, if $\lambda_{x, k}= \lambda_{x, \ell}$, then
      $n_{x, k}\geq n_{x, \ell}$;
    \item for all $1\leq k<\ell\leq k_0$, if $\lambda_{k}= \lambda_{\ell}$, then
      $m_{ k}\geq m_{\ell}$;
    \item the sum of the sequences $\underline{n}_{x}$
      \ep{$x\in \{\pm 1\}^2$}, and $2\underline{m}$ is equal
      to~$n$;
    \end{itemize}
    and such that
    \begin{itemize}
    \item  there is 
      a basis~$\mathbf{e}$ of~$V$ such that the
      matrix~$[b_0]_\mathbf{e}$ is the block diagonal matrix whose blocks are
      $C({\underline{n}_{1,1}})$, $C({\underline{n}_{1,-1}})$,
          $-C({\underline{n}_{-1,1}})$, $-C({\underline{n}_{-1,-1}})$, and $C^{\prime}({2\underline{m}})$
and $ [f]_\mathbf{e}$ is the block diagonal matrix whose blocks are
$J({\underline{n}_{1,1}}, \underline{\lambda}_{1,1})$, $J({\underline{n}_{1,-1}}, \underline{\lambda}_{1,-1})$,
$J({\underline{n}_{-1,1}}, \underline{\lambda}_{-1,1})$, $J({\underline{n}_{-1,-1}}, \underline{\lambda}_{-1,-1})$, and
$J^{\prime}({2\underline{m}}, \underline{\lambda})$.
\end{itemize}
\item The form~$b_1$ is nondegenerate if and only if no zero appears in the
  sequences $\underline{\lambda}_{x}$ \ep{$x\in \{ \pm 1\}^2$}.
\item If the form~$b_1$ is nondegenerate, the sequences associated to the
  pair $(b_{1}^{*}, b_{0}^{*})$ on $V^*$ are, for $x=(\varepsilon,
  \eta)$,  $(\underline{n}_{\eta, \varepsilon}, \underline{\lambda}_{\eta, \varepsilon})$, and
    $(2\underline{m}, \underline{\lambda})$.
  \end{enumerate}
\end{teo}

A basis~$\mathbf{e}$ as in~(\ref{item1:teo:normal-forms-r}) will be called \emph{in standard position}
with respect to $(b_0, b_1)$.

\begin{prop}\label{prop:normal-forms-appendix-back-trans}
  Under the same assumption than Theorem~\ref{teo:normal-forms-r}, let~$\Phi$
  the block  matrix
  \[
    \begin{pmatrix}
      \Phi(\underline{n}_{1,1}, \underline{\lambda}_{1,1}) & & & & \\
      & &\!\!\! \Phi({\underline{n}_{1,-1}}, \underline{\lambda}_{1,-1})  & & \\
      &\!\!\! \Phi({\underline{n}_{-1,1}}, \underline{\lambda}_{-1,1}) & & & \\
      & & & \!\!\!\Phi(\underline{n}_{-1,-1}, \underline{\lambda}_{-1,-1}) & \\
      & & & &\!\!\! \Psi({2\underline{m}}, \underline{\lambda})
    \end{pmatrix}.
  \]
  If $\mathbf{e}$ is a basis in standard position with respect
  to~$(b_0, b_1)$, let $\mathbf{v}= \mathbf{e} \Phi$ and $\mathbf{v}^*$ be the
  basis of~$V^*$ dual to~$\mathbf{v}$. Then $\mathbf{v}^*$ is in
  standard position with respect to $(b_{1}^{\ast}, b_{0}^{\ast})$.
\end{prop}

The work done so far implies the existence of such basis and sequences. The
uniqueness will follow from our analysis of the automorphism groups.

\section{Automorphism groups}
\label{sec:automorphism-groups}

It is desirable to have an understanding on how much the basis in the above
theorem can vary, or, what amounts to the same, to have an understanding of
the centralizer of~$f$ in the orthogonal group $\OO(b_0)$. By
Proposition~\ref{prop:auto-b0-b1}, Lemma~\ref{lem:c-to-r}, and up to replacing
$f$ by $f-\lambda\Id$, we are reduced
to the following situation:
\begin{quote}
  $V$ is a $K$-vector space ($K=\R$ or $\C$) and $f$ is nilpotent.
\end{quote}
From the analysis done in Section~\ref{sec:when-minim-polyn}, we can assume
that there is a decreasing sequence $(m_1, \dots, m_r)$ of integers and
sequences $(p_1, \dots, p_r)$, $(q_1, \dots, q_r)$ such that, the
$K[X]$-module~$V$ is the direct $b_{ K \langle\!\langle
  X^{-1}\rangle\!\rangle}$-orthogonal sum
\[ V = K_{[m_1]}^{p_1+q_1}\oplus \cdots \oplus K_{[m_r]}^{p_r+q_r}\]
and, for each $i=1,\dots, r$, the bilinear form on the factor
$K_{[m_i]}^{p_i+q_i}$ is
\begin{align*}
  K_{[m_i]}^{p_i+q_i} \times K_{[m_i]}^{p_i+q_i}  \longrightarrow& K \langle\!\langle
                                             X^{-1}\rangle\!\rangle \\
  ( ( x_1, \dots, x_{p_i+q_i}), ( y_1, \dots, y_{p_i+q_i}))
   \longmapsto&  X^{1-m_i}\, {}^T\! x
    I_{p_i,q_i} y \\ = X^{1-m_i}( x_1 y_1 + \cdots + x_{p_i} y_{p_i}& - x_{p_i+1}
    y_{p_i+1} - \cdots - x_{p_i +q_i} y_{p_i+q_i}),
\end{align*}
where $I_{p,q} =\bigl(
\begin{smallmatrix}
  \Id_p & \\ & -\Id_q
\end{smallmatrix}\bigr)
$.

An element $g$ in $\mathrm{End}_{K[X]}(V)$ can be considered as a matrix
$(g_{i,j})_{1\leq i,j\leq r}$ where, for all $i,j$, $g_{i,j}$ belongs to
$\Hom_{K[X]}( K_{[m_j]}^{p_j+q_j}, K_{[m_i]}^{p_i+q_i})$. The element
$g_{i,j}$ can 
be considered as a matrix of type $(p_i+q_i)\times
(p_j+q_j)$ with entries in $\Hom_{K[X]}( K_{[m_j]}, K_{[m_i]})$. As $K_{[m_j]}$
is a cyclic $K[X]$-module, one has $\Hom_{K[X]}( K_{[m_j]}, K_{[m_i]})
\simeq \{ x \in K_{[m_i]} \mid X^{m_j} x =0\}$ which is equal to $K_{[m_i]}$ if
$m_j\geq m_i$ (i.e.\ if $j\leq i$) and to $X^{m_i-m_j} K_{[m_j]}$ otherwise
(i.e.\ if $i<j$). If $i>j$ we denote by $h_{i,j}$ the $(p_i+q_i)\times
(p_j+q_j)$ matrix with entries in $K_{[m_j]}$ such that $g_{i,j} = X^{m_i-m_j} h_{i,j}$.

The element $g$ belongs to the orthogonal group~$G$ with respect to the bilinear
form $b_{ K\langle \!\langle X^{-1}\rangle\!\rangle}$ if and only if the
following equalities hold in $K\langle \!\langle X^{-1}\rangle\!\rangle$
\begin{itemize}
\item for all~$j$,
  \begin{multline*}
     \sum_{i>j} X^{1-m_i}\, {}^T\! g_{i,j} I_{p_i,q_i} g_{i,j} + X^{1-m_j}\, {}^T\!
    g_{j,j} I_{p_j,q_j} g_{j,j} \\ +\sum_{i<j} X^{1+m_i-2m_j}\, {}^T\! h_{i,j}
    I_{p_i,q_i} h_{i,j}  = X^{1-m_j} I_{p_j,q_j}
  \end{multline*}
\item for all~$j\neq j'$,
    \[ \sum_{i} X^{1-m_i}\, {}^T\! g_{i,j} I_{p_i,q_i} g_{i,j'}=0.\]
\end{itemize}

For every integer $s\geq 0$, let us denote by $G_{(s)}$ the subgroup of $G$
consisting of the elements $g$ equal to the identity modulo $X^s$.

Recall also the augmentation map $\epsilon\colon K_{[m]}\to K$ that we promote to maps
between spaces of matrices. Thanks to the above
description of the elements of $G$, one has
\begin{lem}\label{lem:automorphism-groups-levi}
  The assignment $\pi\colon g\mapsto \epsilon(g)$ induces an isomorphism between
  $G/G_{(1)}$ and the group $H=\prod_{i=1}^{r} \OO( I_{p_i,q_i})$. The inclusions
  $K\hookrightarrow K_{[m]}$ induce a morphism $\iota \colon H\to G$ that is a section
  of~$\pi$.
\end{lem}
In particular, the pairs $(p_i,q_i)$ are determined by~$G$ and this implies
the sought for uniqueness in Theorem~\ref{teo:normal-forms-r}.

\begin{lem}\label{lem:automorphism-groups-central-series}
  The assignment $\Id + X^s A \mapsto \epsilon(A)$ induces an isomorphism
  between the group $G_{(s)}/ G_{(s+1)}$ and the abelian group underlying the vector
  space of block matrices $a=(a_{i,j})$ \ep{i.e.\ for all $i,j$,
  $a_{i,j}\in M_{p_i+q_i,p_j+q_j}(K)$} satisfying:
  \begin{itemize}
  \item $a$ is antisymmetric with respect to $\diag( I_{p_1,q_1},
    \dots, I_{p_r, q_r})$;
  \item for all $i$ and $j$, if $m_i\leq s$, then $a_{i,j}=0$;
  \item for all $i$ and $j$, if $m_i-m_j>s$, then $a_{i,j}=0$.
  \end{itemize}
\end{lem}

Thus $G_{(1)}$ is a nilpotent Lie group, contained in the
unipotent radical of~$G$ and
the Levi factor of~$G$ is a product of orthogonal groups:

\begin{prop}
  \label{prop:automorphism-groups-decomposition}
  The group~$G$ of automorphisms of the pair $(b_0,f)$ is the semi-direct
  product of a Levi factor $\iota\bigl( \prod_{p_i+q_i>2} \OO(I_{p_i,q_i})\bigr)$ and
  its unipotent radical $\pi^{-1}\bigl( \prod_{p_i+q_i\leq 2} \OO(I_{p_i,q_i})\bigr)$; its
  radical is equal to its unipotent radical.
\end{prop}

\backmatter

\Printindex{definition}{Index}
\Printindex{notation}{Index of Notation}

\bibliographystyle{amsplain}
\bibliography{sp2n_memo}

\begin{bibdiv}
\begin{biblist}

\bib{ABRRW}{article}{
      author={Alessandrini, Daniele},
      author={Berenstein, Arkady},
      author={Retakh, Vladimir},
      author={Rogozinnikov, Eugen},
      author={Wienhard, Anna},
       title={Symplectic groups over noncommutative algebras},
        date={2021},
      eprint={2106.08736},
}

\bib{AC}{article}{
      author={Alessandrini, Daniele},
      author={Collier, Brian},
       title={The geometry of maximal components of the {$\PSp(4, \mathbb{R})$}
  character variety},
        date={2019},
        ISSN={1465-3060},
     journal={Geom. Topol.},
      volume={23},
      number={3},
       pages={1251\ndash 1337},
         url={https://doi.org/10.2140/gt.2019.23.1251},
      review={\MR{3956893}},
}

\bib{AOS}{article}{
      author={Arthamonov, Semeon},
      author={Ovenhouse, Nicholas},
      author={Shapiro, Michael},
       title={Noncommutative networks on a cylinder},
        date={2021},
      eprint={2008.02889},
}

\bib{BR}{article}{
      author={Berenstein, Arkady},
      author={Retakh, Vladimir},
       title={Noncommutative marked surfaces},
        date={2018},
        ISSN={0001-8708},
     journal={Adv. Math.},
      volume={328},
       pages={1010\ndash 1087},
         url={https://doi.org/10.1016/j.aim.2018.02.014},
      review={\MR{3771148}},
}

\bib{Biallas}{article}{
      author={Biallas, Dieter},
       title={Verallgemeinerte {D}oppelverh\"{a}ltnisse und {E}ndomorphismen
  von {V}ektorr\"{a}umen},
        date={1966},
        ISSN={0025-5858},
     journal={Abh. Math. Sem. Univ. Hamburg},
      volume={29},
       pages={263\ndash 291},
         url={https://doi.org/10.1007/BF03016052},
      review={\MR{0211314}},
}

\bib{BB}{article}{
      author={Boralevi, Ada},
      author={Buczy\'{n}ski, Jaros\l~aw},
       title={Secants of {L}agrangian {G}rassmannians},
        date={2011},
        ISSN={0373-3114},
     journal={Ann. Mat. Pura Appl. (4)},
      volume={190},
      number={4},
       pages={725\ndash 739},
         url={https://doi.org/10.1007/s10231-010-0171-0},
      review={\MR{2861067}},
}

\bib{BGPG}{article}{
      author={Bradlow, Steven~B.},
      author={Garc\'{\i}a-Prada, Oscar},
      author={Gothen, Peter~B.},
       title={Maximal surface group representations in isometry groups of
  classical {H}ermitian symmetric spaces},
        date={2006},
        ISSN={0046-5755},
     journal={Geom. Dedicata},
      volume={122},
       pages={185\ndash 213},
         url={https://doi.org/10.1007/s10711-007-9127-y},
      review={\MR{2295550}},
}

\bib{BILW}{article}{
      author={Burger, Marc},
      author={Iozzi, Alessandra},
      author={Labourie, Fran\c{c}ois},
      author={Wienhard, Anna},
       title={Maximal representations of surface groups: symplectic {A}nosov
  structures},
        date={2005},
        ISSN={1558-8599},
     journal={Pure Appl. Math. Q.},
      volume={1},
      number={3, Special Issue: In memory of Armand Borel. Part 2},
       pages={543\ndash 590},
         url={https://doi.org/10.4310/PAMQ.2005.v1.n3.a5},
      review={\MR{2201327}},
}

\bib{BIW}{article}{
      author={Burger, Marc},
      author={Iozzi, Alessandra},
      author={Wienhard, Anna},
       title={Surface group representations with maximal {T}oledo invariant},
        date={2010},
        ISSN={0003-486X},
     journal={Ann. of Math. (2)},
      volume={172},
      number={1},
       pages={517\ndash 566},
         url={https://doi.org/10.4007/annals.2010.172.517},
      review={\MR{2680425}},
}

\bib{FrP}{article}{
      author={Clerc, Jean-Louis},
      author={Koufany, Khalid},
       title={Primitive du cocycle de {M}aslov g\'{e}n\'{e}ralis\'{e}},
        date={2007},
        ISSN={0025-5831},
     journal={Math. Ann.},
      volume={337},
      number={1},
       pages={91\ndash 138},
         url={https://doi.org/10.1007/s00208-006-0028-4},
      review={\MR{2262778}},
}

\bib{FG}{article}{
      author={Fock, Vladimir},
      author={Goncharov, Alexander},
       title={Moduli spaces of local systems and higher {T}eichm\"{u}ller
  theory},
        date={2006},
        ISSN={0073-8301},
     journal={Publ. Math. Inst. Hautes \'{E}tudes Sci.},
      number={103},
       pages={1\ndash 211},
         url={https://doi.org/10.1007/s10240-006-0039-4},
      review={\MR{2233852}},
}

\bib{GPGM}{article}{
      author={Garc\'{\i}a-Prada, Oscar},
      author={Gothen, Peter~B.},
      author={Mundet~i Riera, Ignasi},
       title={Higgs bundles and surface group representations in the real
  symplectic group},
        date={2013},
        ISSN={1753-8416},
     journal={J. Topol.},
      volume={6},
      number={1},
       pages={64\ndash 118},
         url={https://doi.org/10.1112/jtopol/jts030},
      review={\MR{3029422}},
}

\bib{Gothen}{article}{
      author={Gothen, Peter~B.},
       title={Components of spaces of representations and stable triples},
        date={2001},
        ISSN={0040-9383},
     journal={Topology},
      volume={40},
      number={4},
       pages={823\ndash 850},
         url={https://doi.org/10.1016/S0040-9383(99)00086-5},
      review={\MR{1851565}},
}

\bib{GLW}{article}{
      author={Guichard, Olivier},
      author={Labourie, Fran\c{c}ois},
      author={Wienhard, Anna},
       title={Positivity and representations of surface groups},
        date={2021},
      eprint={2106.14584},
}

\bib{GWpos}{unpublished}{
      author={Guichard, Olivier},
      author={Wienhard, Anna},
       title={Generalizing {L}usztig's positivity},
        note={in preparation},
}

\bib{GW}{incollection}{
      author={Guichard, Olivier},
      author={Wienhard, Anna},
       title={Positivity and higher {T}eichm\"{u}ller theory},
        date={2018},
   booktitle={European {C}ongress of {M}athematics},
   publisher={Eur. Math. Soc., Z\"{u}rich},
       pages={289\ndash 310},
      review={\MR{3887772}},
}

\bib{HS}{article}{
      author={Hartnick, Tobias},
      author={Strubel, Tobias},
       title={Cross ratios, translation lengths and maximal representations},
        date={2012},
        ISSN={0046-5755},
     journal={Geom. Dedicata},
      volume={161},
       pages={285\ndash 322},
         url={https://doi.org/10.1007/s10711-012-9707-3},
      review={\MR{2994044}},
}

\bib{Hatcher}{book}{
      author={Hatcher, Allen},
       title={Algebraic topology},
   publisher={Cambridge University Press, Cambridge},
        date={2002},
        ISBN={0-521-79160-X; 0-521-79540-0},
      review={\MR{1867354}},
}

\bib{Hua-GeoMAtI}{article}{
      author={Hua, Loo-Keng},
       title={Geometries of matrices. {I}. {G}eneralizations of von {S}taudt's
  theorem},
        date={1945},
        ISSN={0002-9947},
     journal={Trans. Amer. Math. Soc.},
      volume={57},
       pages={441\ndash 481},
         url={https://doi.org/10.2307/1990185},
      review={\MR{12679}},
}

\bib{Kashaev}{article}{
      author={Kashaev, Rinat~M.},
       title={Coordinates for the moduli space of flat {${\rm PSL}(2,{\mathbb
  R})$}-connections},
        date={2005},
        ISSN={1073-2780},
     journal={Math. Res. Lett.},
      volume={12},
      number={1},
       pages={23\ndash 36},
         url={https://doi.org/10.4310/MRL.2005.v12.n1.a3},
      review={\MR{2122727}},
}

\bib{Lion}{book}{
      author={Lion, G\'{e}rard},
      author={Vergne, Mich\`ele},
       title={The {W}eil representation, {M}aslov index and theta series},
      series={Progress in Mathematics},
   publisher={Birkh\"{a}user, Boston, Mass.},
        date={1980},
      volume={6},
        ISBN={3-7643-3007-4},
         url={https://doi.org/10.1016/0012-365x(79)90168-7},
      review={\MR{573448}},
}

\bib{NeretinLecturesGaussian}{book}{
      author={Neretin, Yurii~A.},
       title={Lectures on {G}aussian integral operators and classical groups},
      series={EMS Series of Lectures in Mathematics},
   publisher={European Mathematical Society (EMS), Z\"{u}rich},
        date={2011},
        ISBN={978-3-03719-080-7},
         url={https://doi.org/10.4171/080},
      review={\MR{2790054}},
}

\bib{Penner}{article}{
      author={Penner, Robert~C.},
       title={The decorated {T}eichm\"{u}ller space of punctured surfaces},
        date={1987},
        ISSN={0010-3616},
     journal={Comm. Math. Phys.},
      volume={113},
      number={2},
       pages={299\ndash 339},
         url={http://projecteuclid.org/euclid.cmp/1104160216},
      review={\MR{919235}},
}

\bib{RetakhNonComCR}{article}{
      author={Retakh, Vladimir},
       title={Noncommutative cross-ratios},
        date={2014},
        ISSN={0393-0440},
     journal={J. Geom. Phys.},
      volume={82},
       pages={13\ndash 17},
         url={https://doi.org/10.1016/j.geomphys.2014.04.001},
      review={\MR{3206637}},
}

\bib{Retakh2019NoncommutativeCA}{article}{
      author={Retakh, Vladimir},
      author={Rubtsov, Vladimir},
      author={Sharygin, Georgy},
       title={Noncommutative cross-ratio and schwarz derivative},
        date={2019},
     journal={arXiv:1905.01366},
       pages={1\ndash 24},
         url={https://arxiv.org/abs/1905.01366},
}

\bib{Siegel}{article}{
      author={Siegel, Carl~Ludwig},
       title={Symplectic geometry},
        date={1943},
        ISSN={0002-9327},
     journal={Amer. J. Math.},
      volume={65},
       pages={1\ndash 86},
         url={https://doi.org/10.2307/2371774},
      review={\MR{0008094}},
}

\bib{Souriau}{incollection}{
      author={Souriau, Jean-Marie},
       title={Construction explicite de l'indice de {M}aslov. {A}pplications},
        date={1976},
   booktitle={Group theoretical methods in physics ({F}ourth {I}nternat.
  {C}olloq., {N}ijmegen, 1975)},
       pages={117\ndash 148. Lecture Notes in Phys., Vol. 50},
      review={\MR{0478229}},
}

\bib{Strubel}{article}{
      author={Strubel, Tobias},
       title={Fenchel-{N}ielsen coordinates for maximal representations},
        date={2015},
        ISSN={0046-5755},
     journal={Geom. Dedicata},
      volume={176},
       pages={45\ndash 86},
         url={https://doi.org/10.1007/s10711-014-9959-1},
      review={\MR{3347573}},
}

\bib{WienhardICM}{article}{
      author={Wienhard, Anna},
       title={An invitation to higher {T}eichm\"uller theory},
        date={2018},
     journal={Proceedings of the ICM 2018},
}

\bib{Zelikin}{article}{
      author={Zelikin, Mikhail~I.},
       title={Geometry of the cross ratio of operators},
        date={2006},
        ISSN={0368-8666},
     journal={Mat. Sb.},
      volume={197},
      number={1},
       pages={39\ndash 54},
         url={https://doi.org/10.1070/SM2006v197n01ABEH003745},
      review={\MR{2230131}},
}

\end{biblist}
\end{bibdiv}

\end{document}